\documentclass[a4paper,reqno]{amsart}

\usepackage{lmodern}
\usepackage[utf8]{inputenc}
\usepackage[T1]{fontenc}

\usepackage[english]{babel}

\usepackage{amsmath,amsfonts,amssymb,amsthm}
\usepackage{mathtools}

\numberwithin{equation}{section}

\usepackage{tikz}

\usepackage{hyperref}
\hypersetup{
  pdftitle={Quasi-crystals for arbitrary root systems and associated generalizations of the hypoplactic monoid},
  pdfauthor={Alan J. Cain, Ricardo P. Guilherme, Ant\'{o}nio Malheiro},
  pdfsubject={05E16; 05E10, 20M05, 20M10},
  pdfkeywords={quasi-crystal, hypoplactic monoid, crystal, plactic monoid, Kashiwara operator, weight labelled graph}
}

\theoremstyle{plain}
\newtheorem{thm}{Theorem}[section]
\newtheorem{cor}[thm]{Corollary}
\newtheorem{lem}[thm]{Lemma}
\newtheorem{prop}[thm]{Proposition}

\theoremstyle{definition}

\newtheorem{dfn}[thm]{Definition}
\newtheorem{exa}[thm]{Example}

\theoremstyle{remark}
\newtheorem{rmk}[thm]{Remark}

\newcounter{thmenumi}[thm]

\newcounter{exaenumi}[thm]  

\newcommand*{\exaitem}{%
  \refstepcounter{exaenumi}%
  \textup{(\theexaenumi)}%
}

\newcommand*{\dtgterm}[1]{\emph{#1}}  

\newcommand*{\avoidrefbreak}{\nolinebreak[3] }   
\newcommand*{\avoidcitebreak}{\nolinebreak[3] }  

\newcommand*{\itmform}[1]{\textup{(#1)}}

\newcommand*{\itmref}[1]{\itmform{\ref{#1}}}

\newcommand*{\comboref}[3][]{%
  \ifdefined\hyperref%
    \if\relax\detokenize{#1}\relax%
      \hyperref[#3]{#2\avoidrefbreak \textup{\ref*{#3}}}%
    \else%
      \hyperref[#1]{#2\avoidrefbreak \textup{\ref*{#3}(\ref*{#1})}}%
    \fi%
  \else%
    \if\relax\detokenize{#1}\relax%
      #2\avoidrefbreak \textup{\ref{#3}}%
    \else%
      #2\avoidrefbreak \textup{\ref{#3}(\ref{#1})}%
    \fi%
  \fi%
}

\newcommand*{\itmcomboref}[3]{\comboref[#3]{#1}{#2}}

\newcounter{cstlabeli}
\newcommand*{\cstlabel}[1]{%
  \renewcommand*{\thecstlabeli}{#1}%
  \refstepcounter{cstlabeli}%
}

\newcommand*{\itmcstlabel}[1]{%
  \cstlabel{#1}%
  \itmform{#1}%
}

\newcommand*{\cstitem}[2]{%
  \item[%
    \itmcstlabel{#1}%
    \label{#2}%
  ]%
}

\newcommand*{\Z}{\mathbb{Z}}               
\newcommand*{\R}{\mathbb{R}}               

\DeclarePairedDelimiter{\parens}{\lparen}{\rparen}

\providecommand*{\given}{\relax}
\DeclarePairedDelimiterX{\set}[1]{\{}{\}}{%
  \renewcommand*{\given}{\mathclose{}\mathrel{\setsymbol[\delimsize]}\mathopen{}}%
  #1%
}
\DeclarePairedDelimiterXPP{\setcardin}[1]{\#}{\lvert}{\rvert}{}{#1}

\newcommand*{\powersetsymbol}{\mathcal{P}}
\DeclarePairedDelimiterXPP{\powerset}[1]{\powersetsymbol}{\lparen}{\rparen}{}{#1}

\DeclarePairedDelimiterX{\gnrt}[1]{\langle}{\rangle}{%
  \renewcommand*{\given}{\mathclose{}\mathrel{\setsymbol[\delimsize]}\mathopen{}}%
  #1%
}

\newcommand*{\pressymbol}[1][]{#1\vert}
\DeclarePairedDelimiterX{\pres}[1]{\langle}{\rangle}{%
  \renewcommand*{\given}{\mathclose{}\mathrel{\pressymbol[\delimsize]}\mathopen{}}%
  #1%
}

\DeclarePairedDelimiter{\wlng}{\lvert}{\rvert}  
\newcommand*{\ew}{\epsilon}                     

\newcommand*{\lbedge}[1]{\xrightarrow[\hphantom{n-g}]{#1}}  

\DeclarePairedDelimiterX{\innerp}[2]{\langle}{\rangle}{#1,#2}  

\newcommand*{\vc}[1]{\mathbf{#1}}

\newcommand*{\wbar}[1]{\overline{#1}}  

\DeclarePairedDelimiterX{\Liebs}[2]{\lbrack}{\rbrack}{#1,#2}   

\newcommand*{\glin}{\mathfrak{gl}}  
\newcommand*{\symp}{\mathfrak{sp}}  

\newcommand*{\alphav}{\alpha^{\vee}}
\newcommand*{\betav}{\beta^{\vee}}

\newcommand*{\tA}{\mathfrak{A}}
\newcommand*{\tB}{\mathfrak{B}}
\newcommand*{\tC}{\mathfrak{C}}
\newcommand*{\tD}{\mathfrak{D}}

\newcommand*{\tAn}{\tA_{n}}
\newcommand*{\tBn}{\tB_n}
\newcommand*{\tCn}{\tC_n}
\newcommand*{\tDn}{\tD_n}

\DeclarePairedDelimiterXPP{\qnum}[1]{}{\lbrack}{\rbrack}{_q}{#1}

\newcommand*{\crst}[1]{\mathcal{#1}}
\newcommand*{\crstB}{\crst{B}}

\DeclareMathOperator{\wt}{wt}
\newcommand*{\Koe}{\tilde{e}}

\newcommand*{\undf}{\bot}  

\DeclareMathOperator{\sgn}{sgn}       

\newcommand*{\crtA}{\crst{A}}

\newcommand*{\crtAn}{\crtA_n}

\newcommand*{\qcrst}[1]{\mathcal{#1}}
\newcommand*{\qcrstQ}{\qcrst{Q}}

\newcommand*{\qKoe}{\ddot{e}}
\newcommand*{\qKof}{\ddot{f}}
\newcommand*{\qKoec}{\ddot{\varepsilon}}
\newcommand*{\qKofc}{\ddot{\varphi}}

\newcommand*{\dotimes}{\mathbin{\ddot{\otimes}}}  

\newcommand*{\qtpsgn}{\sgn}              

\DeclareMathOperator{\inv}{inv}  

\newcommand*{\qctA}{\qcrst{A}}

\newcommand*{\qctC}{\qcrst{C}}

\newcommand*{\qctAn}{\qctA_n}

\newcommand*{\qctCn}{\qctC_n}

\newcommand*{\qcmon}[1]{\qcrst{#1}}
\newcommand*{\qcrstM}{\qcmon{M}}

\newcommand*{\fqcms}{{\ddot{*}}}      

\newcommand*{\RpAn}{R}

\newcommand*{\classicalhypo}{\mathrm{hypo}}
\newcommand*{\hypon}{\classicalhypo_n}
\newcommand*{\hycon}{\sim_{\hypon}}

\DeclareMathOperator{\hypo}{hypo}
\newcommand*{\hyco}{\mathrel{\ddot{\sim}}}
\newcommand*{\nhyco}{\mathrel{\not\hyco}}

\newcommand*{\RhAn}{R}

\newlength\basiccrystalx
\newlength\basiccrystaly
\newlength\widecrystalx

\makeatletter
\if 0\@ptsize
\else%
  \if 1\@ptsize
  \else%
    \if 2\@ptsize
    \else%
    \fi%
  \fi%
\fi
\makeatother

\tikzset{%
  basiccrystal/.style={%
  x=\basiccrystalx,%
  y=\basiccrystaly,%
    every node/.append style={%
      execute at begin node=$,%
      execute at end node=$,%
    },%
    every path/.append style={%
      ->,%
      auto,%
      every node/.append style={%
        execute at begin node=\scriptstyle,%
      },%
    },%
  },%
  widecrystal/.style={%
  x=\widecrystalx,%
  y=\basiccrystaly,%
    every node/.append style={%
      execute at begin node=$,%
      execute at end node=$,%
    },%
    every path/.append style={%
      ->,%
      auto,%
      every node/.append style={%
        execute at begin node=\scriptstyle,%
      },%
    },%
  },%
}

\begin{document}

\title[Quasi-crystals for arbitrary root systems]{Quasi-crystals for arbitrary root systems and associated generalizations of the hypoplactic monoid}

\author{Alan J. Cain}
\address{%
Center for Mathematics and Applications (NovaMath)\\
FCT NOVA\\
2829--516 Caparica\\
Portugal
}
\email{a.cain@fct.unl.pt}

\author{Ricardo P. Guilherme}
\address{%
Center for Mathematics and Applications (NovaMath)\\
FCT NOVA\\
and
Department of Mathematics\\
FCT NOVA\\
2829--516 Caparica\\
Portugal
}
\email{rj.guilherme@campus.fct.unl.pt}
\thanks{The second author was funded by national funds through the FCT -- Funda\c{c}\~{a}o para a Ci\^{e}ncia e a Tecnologia, I.P., under grant reference SFRH/BD/121819/2016.}

\author{Ant\'{o}nio Malheiro}
\address{%
Center for Mathematics and Applications (NovaMath)\\
FCT NOVA\\
and
Department of Mathematics\\
FCT NOVA\\
2829--516 Caparica\\
Portugal
}
\email{ajm@fct.unl.pt}
\thanks{For all three authors, this work was funded by national funds through the FCT -- Funda\c{c}\~{a}o para a Ci\^{e}ncia e a Tecnologia, I.P., under the scope of the projects UIDB/00297/2020 and UIDP/00297/2020 (Center for Mathematics and Applications), and under the scope of the SemiComb project PTDC/MAT-PUR/31174/2017.}

\begin{abstract}
The hypoplactic monoid was introduced by Krob and Thibon through a presentation and through quasi-ribbon tableaux and an insertion algorithm.
Just as Kashiwara crystals enriched the structure of the plactic monoid and allowed its generalization, the first and third authors of this paper introduced a construction of the hypoplactic monoid by identifying vertices in a quasi-crystal graph derived from the crystal graph associated to the general linear Lie algebra.
Although this construction is based on Kashiwara's work, it cannot be extended to other crystal graphs, since the analogous quasi-Kashiwara operators on words do not admit a recursive definition.
This paper addresses these issues.
A general notion of quasi-crystal is introduced, followed by a study of its properties and relation with crystals.
A combinatorial study of quasi-crystals is then made by associating a quasi-crystal graph to each quasi-crystal, which for the class of seminormal quasi-crystals results in a one-to-one correspondence.
To model the binary operation of the hypoplactic monoid by quasi-crystals, a notion of quasi-tensor product of quasi-crystals is introduced, along with a combinatorial way of computing it similar to the signature rule for the tensor product of crystals.
This framework allows the generalization of the classical hypoplactic monoid to a family of hypoplactic monoids associated to the various simple Lie algebras.
The quasi-crystal structure is then used to establish algebraic properties of the hypoplactic monoid associated to the symplectic Lie algebra.
\end{abstract}

\keywords{Quasi-crystal, hypoplactic monoid, crystal, plactic monoid, Kashiwara operator, weight labelled graph}

\subjclass[2020]{Primary 05E16; Secondary 05E10, 20M05, 20M10}

\maketitle

\tableofcontents

\section{Introduction}\label{sec:introduction}

The plactic monoid, formally introduced  by Lascoux and Sch\"{u}tzenberger\avoidcitebreak \cite{LS81}, is an algebraic object of great interest, with connections to several fields such as representation theory, combinatorics\avoidcitebreak \cite{Ful97}, symmetric functions, and Schubert polynomials\avoidcitebreak \cite{LS85,LS89}.
It was also used to give a first rigorous proof of the Littlewood--Richardson rule\avoidcitebreak \cite{LR34}.
This led Sch\"{u}tzenberger\avoidcitebreak \cite{Sch97} to consider it one of the most fundamental monoids in algebra.
There are numerous ways of obtaining the plactic monoid; we highlight three of them.
First, it originally emerged from Young tableaux and the Schensted insertion algorithm\avoidcitebreak \cite{Sch61}.
Second, it also has a presentation by the so-called Knuth relations\avoidcitebreak \cite{Knu70}.
Third, it can be obtained by identifying words in the same position of isomorphic connected components of a certain crystal graph.

Kashiwara\avoidcitebreak \cite{Kas90,Kas91,Kas94} introduced crystal bases for modules of quantized universal enveloping algebras, discovered independently by Drinfel'd\avoidcitebreak \cite{Dri85} and Jimbo\avoidcitebreak \cite{Jim85}, and showed that the plactic monoid arises from the crystal basis associated with the vector representation of the quantized universal enveloping general linear Lie algebra.
This result allowed a deeper study of the plactic monoid and its generalization, because the underlying construction still results in a monoid for crystal bases associated with other quantized universal enveloping algebras.
Thus, Kashiwara and Nakashima\avoidcitebreak \cite{KN94} studied crystal graphs for the Cartan types $\tAn$, $\tBn$, $\tCn$ and $\tDn$, leading to a notion of Kashiwara--Nakashima tableaux.
Based on this, Lecouvey\avoidcitebreak \cite{Lec02,Lec03} presented comprehensive descriptions of the plactic monoids for the Cartan types $\tBn$, $\tCn$, and $\tDn$, which later appeared in a survey\avoidcitebreak \cite{Lec07}.
In recent works, Cain, Gray and Malheiro\avoidcitebreak \cite{CGM15f,CGM19} presented rewriting systems and biautomatic structures for these monoids.
In an independent work and by an alternative approach, Hage\avoidrefbreak \cite{Hag15} described a finite convergent presentation of the plactic monoid for type $\tCn$.

The hypoplactic monoid was introduced by Krob and Thibon\avoidcitebreak \cite{KT97} from representation-theoretical interpretations of quasi-symmetric functions and noncommutative symmetric functions.
It emerged from a noncommutative realization of quasi-symmetric functions analogous to the realization of symmetric functions by the plactic monoid presented by Lascoux and Sch\"{u}tzenberger\avoidcitebreak \cite{LS81}.
This led to a construction of the hypoplactic monoid through quasi-ribbon tableaux and an insertion algorithm, and to a presentation consisting of the Knuth relations and the quartic relations.
A detailed study of the hypoplactic monoid was done by Novelli\avoidcitebreak \cite{Nov00}.
A comparative study with other monoids was done by Cain, Gray and Malheiro in\avoidcitebreak \cite{CGM15r}, where a rewriting system and a biautomatic structure for the hypoplactic monoid is presented.
Recently, following the work in\avoidcitebreak \cite{Rib22}, a complete description of the identities satisfied by the hypoplactic monoid was presented by Cain, Malheiro and Ribeiro\avoidcitebreak \cite{CMR22}.

A first notion of quasi-crystal graph was introduced by Krob and Thibon\avoidcitebreak \cite{KT99} to encode the full structure of the modules that give rise to the hypoplactic monoid.
The vertex set of such a graph is formed by the quasi-ribbon words over the alphabet $\set{12,\ldots,n}$, which also form a complete set of representatives for the hypoplactic congruence.
Therefore, these quasi-crystal graphs do not allow a construction of the hypoplactic monoid analogous to the construction of the plactic monoid from crystal graphs, because they do not have isomorphic connected components.

To overcome the limitations of the first notion of quasi-crystal graph,
Cain and Malheiro\avoidcitebreak \cite{CM17crysthypo} described a new quasi-crystal graph, derived from the crystal graph for type $\tAn$, that allows a construction of the hypoplactic monoid by identifying words in the same position of isomorphic connected components,
and induces the definition of an analogue of Kashiwara operators on words over the alphabet $\set{1,2,\ldots,n}$, called quasi-Kashiwara operators.
However, this construction is purely combinatorial and does not have an algebraic foundation.
It cannot be used to construct a monoid starting with a crystal graph of another type\avoidcitebreak \cite[Remark~6.17]{Gui22}.
It is therefore natural to ask whether quasi-Kashiwara operators on words can be defined recursively.

The main goal of this paper is to establish a general theory of quasi-crystals that allows a generalization of the hypoplactic monoid.
It addresses the problems discussed above, while showing that the construction in\avoidcitebreak \cite{CM17crysthypo} can be placed in the context of this new theory.
It follows the work in\avoidcitebreak \cite{Gui22} and presents a more consolidated theory with new and improved results.

This paper is structured as follows.
\comboref{Section}{sec:preliminaries} introduces notation and discusses preliminaries relating to monoids, root systems, and graphs.
\comboref{Section}{sec:qch} states the definitions of quasi-crystals and homomorphisms between quasi-crystals, which give rise to a category, and is devoted to making an algebraic study of them.
\comboref{Section}{sec:qcg} presents the notion of the quasi-crystal graph associated to a quasi-crystal, leading to a combinatorial study of quasi-crystals, and describes a one-to-one correspondence between the class of seminormal quasi-crystals and a class of weighted labelled graphs.
\comboref{Section}{sec:qcqtp} states the definition of the quasi-tensor product of quasi-crystals and describes a practical method to compute it.
\comboref{Section}{sec:qcm} states the definition of the quasi-crystal monoid and is devoted to making an algebraic study of it, concerning homomorphisms, congruences and free objects.
It is shown that a free quasi-crystal monoid satisfies a universal property which defines it up to isomorphism, and that congruences on a quasi-crystal monoid form a lattice.
Homomorphism theorems for quasi-crystal monoids are also proven.
\comboref{Section}{sec:hyco} shows that identifying elements in isomorphic connected components of a free quasi-crystal monoid gives rise to a congruence, called the hypoplactic congruence, which leads to the definition of hypoplactic monoid associated to a quasi-crystal.
It is shown that the central elements of a hypoplactic monoid correspond to the isolated elements of the free quasi-crystal monoid, and the idempotents correspond to isolated elements of weight zero, leading to the conclusion that the idempotents of a hypoplactic monoid commute.
\comboref{Section}{sec:crystclassicalhypo} proves that the hypoplactic monoid associated to the standard quasi-crystal of type $\tAn$ is isomorphic to the classical hypoplactic monoid of rank $n$, indicating that this approach results in a genuine generalization of the classical hypoplactic monoid, by showing that the construction in\avoidcitebreak \cite{CM17crysthypo} can be placed in the context of the developed framework.
\comboref{Section}{sec:hypotCn} is devoted to the study of the hypoplactic monoid associated to the standard quasi-crystal of type $\tCn$.
Highest-weight and isolated words are characterized, allowing an identification of central and idempotent elements of this monoid.
Relations satisfied by the hypoplactic monoid of type $\tCn$ are then studied, in particular, it is investigated whether this monoid satisfies the Knuth relations.
it is shown that the hypoplactic monoid of type $\tCn$ satisfies a non-trivial identity if and only if $n = 2$, in contrast to the classical hypoplactic monoid, which satisfies a non-trivial identity independently of rank.
It is proven that the hypoplactic monoid of type $\tCn$ is not finitely presented, for any $n \geq 2$.
Finally, embeddings of the hypoplactic monoids of types $\tA_{n-1}$ and $\tC_{n-1}$ into the hypoplactic monoid of type $\tCn$ are presented, and it is shown that the `obvious' approach to defining such embeddings does not work.

\section{Preliminaries}
\label{sec:preliminaries}

We assume some familiarity with the basic concepts related with monoids and graphs, so we will not make a proper introduction to them.
For background on monoids see\avoidcitebreak \cite{How95}, on presentations see\avoidcitebreak \cite{Hig92}, and on graphs see\avoidcitebreak \cite{Bol98}.

We will introduce crystals as a subclass of quasi-crystals, and so, we will not need to present a complete introduction to crystals.
We refer
to\avoidcitebreak \cite{Kas95} for an introduction to crystals as they originally emerged in connection to quantized universal envelopping algebras (also called quantum groups)
or\avoidcitebreak \cite{HK02} for a comprehensive background on this approch,
to\avoidcitebreak \cite{BS17} for a study of crystals detached from their origin,
and to\avoidcitebreak \cite{CGM19} for the relations between crystals and plactic monoids for the infinite Cartan types.

In this section, we give the essential background on root systems, as these algebraic structures will be used throughout this paper.
Root systems are commonly found in representation theory, in particular, they arise on the study of Lie groups and Lie algebras, but we will detach them from this context, as our aim is to construct an algebraic structure for defining crystals and quasi-crystals.
Thus, we only introduce the necessary notions needed for this purpose.
For further context see for example\avoidcitebreak \cite{FH91,Bou02,EW06,Bum13}.

Let $V$ be a Euclidean space, that is, a real vector space with an inner product $\innerp{\,{\cdot}\,}{\,{\cdot}\,}$.
For $\alpha \in V$ other than $0$, denote by $r_{\alpha}$ the \dtgterm{reflection} in the hyperplane orthogonal to $\alpha$, which is given by
\[
r_{\alpha} (v) = v - \innerp[\big]{v}{\alphav} \alpha,
\quad \text{where} \quad
\alphav = \frac{2}{\innerp{\alpha}{\alpha}} \alpha,
\]
for each $v \in V$. Note that $r_\alpha$ is bijective, as $r_{\alpha} \parens[\big]{ r_{\alpha} (v) } = v$, for all $v \in V$. Also, $r_\alpha$ preserves the inner product, as $\innerp[\big]{r_\alpha (u)}{r_\alpha (v)} = \innerp{u}{v}$ for any $u, v \in V$.

A \dtgterm{root system} in $V$ is a subset $\Phi$ of $V$ satisfying the following conditions:
\begin{itemize}
\cstitem{RS1}{itm:rootsystemz}
$\Phi$ is nonempty, finite, and $0 \notin \Phi$;

\cstitem{RS2}{itm:rootsystemr}
$r_\alpha (\beta) \in \Phi$, for all $\alpha, \beta \in \Phi$;

\cstitem{RS3}{itm:rootsystemi}
$\innerp[\big]{\alpha}{\betav} \in \Z$, for all $\alpha, \beta \in \Phi$,

\cstitem{RS4}{itm:rootsystemm}
if $\alpha \in \Phi$ and $k \alpha \in \Phi$, then $k = \pm 1$.
\end{itemize}
The elements of $\Phi$ are called \dtgterm{roots}, and the elements $\alphav$, with $\alpha \in \Phi$, are called \dtgterm{coroots}.
Note that the definition of root system may differ in the literature, as some authors omit some of the conditions above and use them to characterize root systems. For instance, some authors say that a root system is crystallographic when\avoidrefbreak \itmref{itm:rootsystemi} is satisfied, or that it is reduced when\avoidrefbreak \itmref{itm:rootsystemm} is satisfied.
On the other hand, some authors require $\Phi$ to span $V$, we say that a root system is \dtgterm{semisimple} when this happens.

Together with a root system, we always fix an index set $I$ and \dtgterm{simple roots} $(\alpha_i)_{i \in I}$, that is, a collection of roots satisfying the following conditions:
\begin{itemize}
\cstitem{SR1}{itm:simplerootsli}
$\set{\alpha_i \given i \in I}$ is a linearly independent subset of $V$; and

\cstitem{SR2}{itm:simplerootsspen}
every root $\beta \in \Phi$ can be expressed as $\beta = \sum_{i \in I} k_i \alpha_i$, where all $k_i$ are either nonnegative or nonpositive integers.
\end{itemize}
For each $i \in I$, the reflection $r_{\alpha_i}$ is called a \dtgterm{simple reflection} and is denoted by $s_i$.
We also fix a \dtgterm{weight lattice} $\Lambda$, that is, a $\Z$-submodule of $V$ satisfying the following conditions:
\begin{itemize}
\cstitem{WL1}{itm:weightlatticespans}
$\Lambda$ spans $V$;

\cstitem{WL2}{itm:weightlatticecontains}
$\Phi \subseteq \Lambda$;

\cstitem{WL3}{itm:weightlatticeint}
$\innerp[\big]{\lambda}{\alphav} \in \Z$, for any $\lambda \in \Lambda$ and $\alpha \in \Phi$.
\end{itemize}
The elements of $\Lambda$ are called \dtgterm{weights} and are compared using the following partial order
\begin{equation}
\lambda \geq \mu \iff \lambda - \mu = \sum_{i \in I} k_i \alpha_i, \text{ for some $k_i \in \R_{\geq 0}$, $i \in I$.}
\label{eq:rswpo}
\end{equation}

Finally, we draw attention to the root systems associated to Cartan types $\tAn$ and $\tCn$, which will be the only non-arbitrary root systems considered in the subsequent sections.
Let $n \geq 2$. Consider $V$ to be the real vector space $\R^n$ with the usual inner product,
and denote by $\vc{e_i} \in \R^n$ the $n$-tuple with $1$ in the $i$-th position, and $0$ elsewhere, $i = 1, 2, \ldots, n$.
The root system associated to Cartan type $\tAn$ based on the general linear Lie algebra $\glin_n$ consists of
$\Phi = \set{ \vc{e_i} - \vc{e_j} \given i \neq j }$,
the index set for the simple roots is $I = \set{1, 2, \ldots, n-1}$,
the simple roots are
$\alpha_i = \vc{e_i} - \vc{e_{i+1}}$, $i=1,2,\ldots,n-1$,
and the weight lattice is $\Lambda = \Z^{n}$.

The root system associated to Cartan type $\tCn$ based on the symplectic Lie algebra $\symp_{2n}$ consists of
$\Phi = \set{ \pm \vc{e_i} \pm \vc{e_j} \given i < j } \cup \set{ \pm 2\vc{e_i} \given i = 1, 2, \ldots, n }$,
the index set for the simple roots is $I = \set{1, 2, \ldots, n}$, the simple roots are
$\alpha_i = \vc{e_i} - \vc{e_{i+1}}$, $i=1,2,\ldots,n-1$, and $\alpha_n = 2\vc{e_n}$,
and the weight lattice is $\Lambda = \Z^{n}$.
For more examples of root systems see\avoidcitebreak \cite[Examples~2.4 to~2.10]{BS17}.

\section{Quasi-crystals and homomorphisms}
\label{sec:qch}

In this section we introduce the notion of quasi-crystals associated to a root system.
We then study some basic properties satisfied by quasi-crystals, some of which correspond to generalizations of properties verified by crystals.
Finally, we introduce the notion of quasi-crystal homomorphisms and study their properties.

Although we rely on root systems (\comboref{Section}{sec:preliminaries}) to define quasi-crystals, we only make use of properties that are also satisfied by other algebraic structures commonly used to define crystals. Thus, all subsequent definitions and results can be reinterpreted using the algebraic data in\avoidcitebreak \cite{Kas95} or a Cartan datum as in\avoidcitebreak \cite{HK02}.

Consider $\Z \cup \set{ {-\infty}, {+\infty} }$ to be the usual set of integers where we add a minimal element $-\infty$ and a maximal element $+\infty$, that is, ${-\infty} < m < {+\infty}$ for all $m \in \Z$.
Also, set $m + (-\infty) = (-\infty) + m = -\infty$ and $m + (+\infty) = (+\infty) + m = +\infty$, for all $m \in \Z$.

\begin{dfn}
\label{dfn:qc}
Let $\Phi$ be a root system with weight lattice $\Lambda$ and index set $I$ for the simple roots $(\alpha_i)_{i \in I}$.
A \dtgterm{quasi-crystal} $\qcrstQ$ of type $\Phi$ consists of a set $Q$ together with maps ${\wt} : Q \to \Lambda$, $\qKoe_i, \qKof_i : Q \to Q \sqcup \{\undf\}$ and $\qKoec_i, \qKofc_i : Q \to \Z \cup \{ {-\infty}, {+\infty} \}$, for each $i \in I$, satisfying the following conditions:
\begin{enumerate}
\item\label{dfn:qcwt}
$\qKofc_i (x) = \qKoec_i (x) + \innerp[\big]{\wt (x)}{\alphav_i}$;

\item\label{dfn:qcqKoe}
if $\qKoe_i (x) \in Q$, then
$\wt \parens[\big]{ \qKoe_i (x) } = \wt (x) + \alpha_i$,
$\qKoec_i \parens[\big]{ \qKoe_i (x) } = \qKoec_i (x) - 1$, and
$\qKofc_i \parens[\big]{ \qKoe_i (x) } = \qKofc_i (x) + 1$;

\item\label{dfn:qcqKof}
if $\qKof_i (x) \in Q$, then
$\wt \parens[\big]{ \qKof_i (x) } = \wt (x) - \alpha_i$,
$\qKoec_i \parens[\big]{ \qKof_i (x) } = \qKoec_i (x) + 1$, and
$\qKofc_i \parens[\big]{ \qKof_i (x) } = \qKofc_i (x) - 1$;

\item\label{dfn:qciff}
$\qKoe_i (x) = y$ if and only if $x = \qKof_i (y)$;

\item\label{dfn:qcminfty}
if $\qKoec_i (x) = -\infty$ then $\qKoe_i (x) = \qKof_i (x) = \undf$;

\item\label{dfn:qcpinfty}
if $\qKoec_i (x) = +\infty$ then $\qKoe_i (x) = \qKof_i (x) = \undf$;
\end{enumerate}
for $x, y \in Q$ and $i \in I$. The set $Q$ is called the \dtgterm{underlying set} of $\qcrstQ$, and the maps $\wt$, $\qKoe_i$, $\qKof_i$, $\qKoec_i$ and $\qKofc_i$ ($i \in I$) form the \dtgterm{quasi-crystal structure} of $\qcrstQ$. Also, the map $\wt$ is called the \dtgterm{weight map}, where $\wt (x)$ is said to be the \dtgterm{weight} of $x \in Q$, and the maps $\qKoe_i$ and $\qKof_i$ ($i \in I$) are called the \dtgterm{raising} and \dtgterm{lowering quasi-Kashiwara operators}, respectively.
\end{dfn}

In this definition, $\undf$ is an auxilary symbol. In the definition of crystals, $0$ is oftenly used instead of $\undf$, but since some well-known crystals have $0$ as an element, we have adopted this notation for quasi-crystals to avoid ambiguity.
For $x \in Q$, by $\qKoe_i (x) = \undf$ (or $\qKof_i (x) = \undf$) we mean that $\qKoe_i$ (resp., $\qKof_i$) is \dtgterm{undefined} on $x$. On the other hand, we say that $\qKoe_i$ (or $\qKof_i$) is \dtgterm{defined} on $x$ whenever $\qKoe_i (x) \in Q$ (resp., $\qKof_i (x) \in Q$).
So, alternatively one can consider the quasi-Kashiwara operators $\qKoe_i$ and $\qKof_i$ ($i \in I$) to be partial maps from $Q$ to $Q$.
When this point of view is more suitable to describe quasi-Kashiwara operators, we will make use of it.

In comparison with the definition of crystal\avoidcitebreak \cite[Definition~2.12]{BS17}, we have that $\qKoec_i$ and $\qKofc_i$ ($i \in I$) can also take the value $+\infty$.
This leads to the addition of condition\avoidrefbreak \itmref{dfn:qcpinfty}, as conditions\avoidrefbreak \itmref{dfn:qcwt} to\avoidrefbreak \itmref{dfn:qcminfty} coincide in both definitions.
Thus, we can take the following as the definition of crystal.

\begin{rmk}
\label{rmk:crstqc}
A \dtgterm{crystal} is a quasi-crystal $\crstB$ where $\qKoec_i (x) \neq +\infty$ and $\qKofc_i (x) \neq +\infty$, for all $x \in B$ and $i \in I$.
\end{rmk}

From condition\avoidrefbreak \itmref{dfn:qcwt} of \comboref{Definition}{dfn:qc}, we get that $\qKoec_i (x) = \pm\infty$ if and only if $\qKofc_i (x) = \pm\infty$. And if so, $\qKoec_i (x) = \qKofc_i (x)$.
Thus, conditions\avoidrefbreak \itmref{dfn:qcminfty} and\avoidrefbreak \itmref{dfn:qcpinfty} could have been stated replacing $\qKoec_i$ by $\qKofc_i$.
Moreover, we could have only stated one of conditions\avoidrefbreak \itmref{dfn:qcqKoe} and\avoidrefbreak \itmref{dfn:qcqKof} as justified by the following result.

\begin{prop}
\label{prop:dfnqcsimp}
Let $\Phi$ be a root system with weight lattice $\Lambda$ and index set $I$ for the simple roots $(\alpha_i)_{i \in I}$.
Consider a set $Q$ and maps ${\wt} : Q \to \Lambda$, $\qKoe_i, \qKof_i : Q \to Q \sqcup \{\undf\}$ and $\qKoec_i, \qKofc_i : Q \to \Z \cup \{ {-\infty}, {+\infty} \}$, for each $i \in I$, satisfying \itmcomboref{Definition}{dfn:qc}{dfn:qciff}.
Then \itmcomboref{Definition}{dfn:qc}{dfn:qcqKoe} holds if and only if \itmcomboref{Definition}{dfn:qc}{dfn:qcqKof} holds.
\end{prop}

\begin{proof}
Assume \itmcomboref{Definition}{dfn:qc}{dfn:qcqKoe} holds.
Let $x \in Q$ and $i \in I$ such that $\qKof_i (x) \in Q$. By \itmcomboref{Definition}{dfn:qc}{dfn:qciff}, we have that $\qKoe_i \parens[\big]{ \qKof_i (x) } = x$, and so,
\begin{align*}
\wt (x) &= \wt \parens[\big]{ \qKoe_i (\qKof_i (x)) } = \wt \parens[\big]{ \qKof_i (x) } - \alpha_i,\\
\qKoec_i (x) &= \qKoec_i \parens[\big]{ \qKoe_i (\qKof_i (x)) } = \qKoec_i \parens[\big]{ \qKof_i (x) } - 1,
\displaybreak[0]\\
\shortintertext{and}
\qKofc_i (x) &= \qKofc_i \parens[\big]{ \qKoe_i (\qKof_i (x)) } = \qKofc_i \parens[\big]{ \qKof_i (x) } + 1,
\end{align*}
by \itmcomboref{Definition}{dfn:qc}{dfn:qcqKoe}. Hence, \itmcomboref{Definition}{dfn:qc}{dfn:qcqKof} holds.

The converse implication is analogous.
\end{proof}

In the same way, by conditions\avoidrefbreak \itmref{dfn:qcwt} and\avoidrefbreak \itmref{dfn:qciff} of \comboref{Definition}{dfn:qc} we have that a quasi-crystal is determined by a set $Q$ and the weight map $\wt$ together with either $\qKoe_i$ or $\qKof_i$, and either $\qKoec_i$ or $\qKofc_i$, for each $i \in I$.
However, for a purpose of clarity, we usually give explicit definitions for each map when defining a quasi-crystal.

\begin{exa}
\label{exa:qc}
\exaitem\label{exa:qctAn}
Consider the root system of type $\tAn$.
By \comboref{Remark}{rmk:crstqc}, the standard crystal of type $\tAn$ gives rise to the quasi-crystal $\qctAn$ defined as follows.
The underlying set is the ordered set $A_n = \set{1 < 2 < \cdots < n}$.
For $x \in A_n$, the weight of $x$ is $\wt(x) = \vc{e_x}$.
For $i=1,2,\ldots,n-1$, the quasi-Kashiwara operators $\qKoe_i$ and $\qKof_i$ are only defined on $i+1$ and $i$, respectively, where $\qKoe_i (i+1) = i$ and $\qKof_i (i) = i+1$.
Finally, $\qKoec_i (x) = \delta_{x, i+1}$ and $\qKofc_i (x) = \delta_{x, i}$,
where $\delta_{k, l} = 1$ if $k = l$, and $\delta_{k, l} = 0$ whenever $k \neq l$.
We call $\crtAn$ the \dtgterm{standard quasi-crystal of type $\tAn$}.

\exaitem\label{exa:qctCn}
Consider the root system of type $\tCn$.
By \comboref{Remark}{rmk:crstqc}, the standard crystal of type $\tCn$ gives rise to the quasi-crystal $\qctCn$ defined as follows.
The underlying set is $C_n = \set{1 < 2 < \cdots < n < \wbar{n} < \wbar{n-1} < \cdots < \wbar{1}}$.
For $x \in \set{1,2,\ldots,n}$, the weight of $x$ is $\wt(x) = \vc{e_x}$,
and the weight of $\wbar{x}$ is $\wt(\wbar{x}) = -\vc{e_x}$.
For $i=1,2,\ldots,n-1$, the quasi-Kashiwara operators $\qKoe_i$ and $\qKof_i$ are only defined on the following cases:
$\qKoe_i (i+1) = i$, $\qKoe_i (\wbar{i}) = \wbar{i+1}$, $\qKof_i (i) = i+1$, and $\qKof_i (\wbar{i+1}) = \wbar{i}$.
The quasi-Kashiwara operators $\qKoe_n$ and $\qKof_n$ are only defined in $\wbar{n}$ and $n$, respectively, where $\qKoe_n (\wbar{n}) = n$ and $\qKof_n (n) = \wbar{n}$.
Finally, for $y \in C_n$, $\qKoec_i (y) = \delta_{y, i+1} + \delta_{y, \wbar{i}}$, $\qKoec_n (y) = \delta_{y, \wbar{n}}$, $\qKofc_i (y) = \delta_{y, i} + \delta_{y, \wbar{i+1}}$, and $\qKofc_n (y) = \delta_{y, n}$.
We call $\qctCn$ the \dtgterm{standard quasi-crystal of type $\tCn$}.

\exaitem\label{exa:qcA32}
Consider the root system of type $\tA_3$.
We have a quasi-crystal $\qctA_3^2$ of type $\tA_3$ whose underlying set is $A_3^2 = A_3 \times A_3$ and whose quasi-crystal structure is given as follows.
\begin{center}
  \begin{tabular}{c|c|c|c|c|c|c|c|c|c}
    $x$     & $\wt(x)$            & $\qKoe_1(x)$ & $\qKoe_2(x)$ & $\qKof_1(x)$ & $\qKof_2(x)$ & $\qKoec_1(x)$ & $\qKoec_2(x)$ & $\qKofc_1(x)$ & $\qKofc_2(x)$\\ \hline
    $(1,1)$ & $2\vc{e_1}$         & $\undf$      & $\undf$      & $(2,1)$      & $\undf$      & $0$           & $0$           & $2$           & $0$          \\
    $(1,2)$ & $\vc{e_1}+\vc{e_2}$ & $\undf$      & $\undf$      & $\undf$      & $(1,3)$      & $+\infty$     & $0$           & $+\infty$     & $1$          \\
    $(1,3)$ & $\vc{e_1}+\vc{e_3}$ & $\undf$      & $(1,2)$     & $(2,3)$      & $\undf$      & $0$           & $1$           & $1$           & $0$          \\
    $(2,1)$ & $\vc{e_1}+\vc{e_2}$ & $(1,1)$      & $\undf$      & $(2,2)$      & $(3,1)$      & $1$           & $0$           & $1$           & $1$          \\
    $(2,2)$ & $2\vc{e_2}$         & $(2,1)$      & $\undf$      & $\undf$      & $(3,2)$      & $2$           & $0$           & $0$           & $2$          \\
    $(2,3)$ & $\vc{e_2}+\vc{e_3}$ & $(1,3)$      & $\undf$      & $\undf$      & $\undf$      & $1$           & $+\infty$     & $0$           & $+\infty$    \\
    $(3,1)$ & $\vc{e_1}+\vc{e_3}$ & $\undf$      & $(2,1)$      & $(3,2)$      & $\undf$      & $0$           & $1$           & $1$           & $0$          \\
    $(3,2)$ & $\vc{e_2}+\vc{e_3}$ & $(3,1)$      & $(2,2)$      & $\undf$      & $(3,3)$      & $1$           & $1$           & $0$           & $1$          \\
    $(3,3)$ & $2\vc{e_3}$         & $\undf$      & $(3,2)$      & $\undf$      & $\undf$      & $0$           & $2$           & $0$           & $0$
  \end{tabular}
\end{center}

\exaitem\label{exa:qcQ2}
Consider the root system of type $\tA_2$.
We have a quasi-crystal $\qcrstQ$ of type $\tA_2$ consisting of a set $Q = \{a, b\}$ and maps defined as follows.
\begin{center}
  \begin{tabular}{c|c|c|c|c|c}
    $x$ & $\wt(x)$   & $\qKoe(x)$ & $\qKof(x)$ & $\qKoec(x)$ & $\qKofc(x)$ \\ \hline
    $a$ & $\vc{e_1}$ & $\undf$    & $\undf$    & $0$         & $1$         \\
    $b$ & $\vc{e_2}$ & $\undf$    & $\undf$    & $1$         & $0$
  \end{tabular}
\end{center}
Since the root system of type $\tA_2$ has exactly one simple root, we omit the subscript index in the maps, for instance $\qKoe$ instead of $\qKoe_1$.
\end{exa}

In the previous example we only introduce quasi-crystals that will be relevant below.
As crystals are quasi-crystals (\comboref{Remark}{rmk:crstqc}), more examples can be found in\avoidcitebreak \cite[Examples~2.21 to~2.25]{BS17}, where the standard crystals for types $\tBn$ and $\tDn$ are included.

Recall the partial order defined on a weight lattice $\Lambda$ described in\avoidrefbreak \eqref{eq:rswpo}. The following result justifies the terminology of \dtgterm{raising} and \dtgterm{lowering} used to characterize the quasi-Kashiwara operators $\qKoe_i$ and $\qKof_i$ ($i \in I$).

\begin{prop}\label{prop:qcrlqKo}
Let $\qcrstQ$ be a quasi-crystal, and let $x \in Q$ and $i \in I$.
If $\qKoe_i (x) \in Q$, then $\wt \parens[\big]{ \qKoe_i (x) } > \wt(x)$.
If $\qKof_i (x) \in Q$, then $\wt(x) > \wt \parens[\big]{ \qKof_i (x) }$.
\end{prop}

\begin{proof}
If $\qKoe_i (x) \in Q$, then
\[ \wt \parens[\big]{ \Koe_i (x) } - \wt(x) = \wt(x) + \alpha_i - \wt(x) = \alpha_i, \]
by \itmcomboref{Definition}{dfn:qc}{dfn:qcqKoe}, and so, $\wt \parens[\big]{ \qKoe_i (x) } \geq \wt(x)$.
Since $\alpha_i$ is a root, we have that $\alpha_i \neq 0$, which implies that $\wt \parens[\big]{ \qKoe_i (x) } \neq \wt(x)$.
Hence, $\wt \parens[\big]{ \qKoe_i (x) } > \wt(x)$.

If $\qKof_i (x) \in Q$, then $x = \qKoe_i \parens[\big]{ \qKof_i (x) }$, by \itmcomboref{Definition}{dfn:qc}{dfn:qciff}. As proved above, we have that $\wt(x) = \wt \parens[\big]{ \qKoe_i (\qKof_i (x)) } > \wt \parens[\big]{ \qKof_i (x) }$.
\end{proof}

From the previous result, we have that, like the Kashiwara operators in crystals, the raising quasi-Kashiwara operators $\qKoe_i$ ($i \in I$) increase the weight of elements, whenever defined, and the lowering quasi-Kashiwara operators $\qKof_i$ ($i \in I$) decrease the weight of elements, whenever they are defined.
Thus, the notions of highest- and lowest-weight elements from crystals can be generalized in a natural way.

\begin{dfn}
\label{dfn:qchlwe}
Let $x \in Q$ be an element of a quasi-crystal $\qcrstQ$.
\begin{enumerate}
\item\label{dfn:qchlweh}
$x$ is said to be of \dtgterm{highest weight} if $\qKoe_i (x) = \undf$, for all $i \in I$.

\item\label{dfn:qchlwel}
$x$ is said to be of \dtgterm{lowest weight} if $\qKof_i (x) = \undf$, for all $i \in I$.
\end{enumerate}
\end{dfn}

Similar to crystals, notice that a quasi-crystal may have a highest-weight element whose weight is less than or equal to the weight of an element that is not of highest weight.
For instance, consider the quasi-crystal $\qctA_3^2$ described in \itmcomboref{Example}{exa:qc}{exa:qcA32},
take $x = (1, 1)$, $y = (1, 2)$ and $z = (2, 1)$,
then $x$ and $y$ are of highest weight, $z$ is not of highest weight, $\wt(x) > \wt(y)$ and $\wt(y) = \wt(z)$.
Moreover, if a quasi-crystal $\qcrstQ$ has an element $x \in Q$ such that $\qKoec_i (x) \in \{ {-\infty}, {+\infty} \}$, for all $i \in I$, then we can change the weight $\wt(x)$ of $x$ to any weight in $\Lambda$, and the resulting structure is still a quasi-crystal.
However, if we extend this definition to the weights as follows, we get some more natural results.

\begin{dfn}
\label{dfn:qchlw}
Let $\qcrstQ$ be a quasi-crystal, and let $\lambda \in \Lambda$ be a weight.
\begin{enumerate}
\item\label{dfn:qchlwh}
$\lambda$ is called a \dtgterm{highest weight} in $\qcrstQ$ if there exists a highest-weight element $x \in Q$ such that $\lambda = \wt(x)$.

\item\label{dfn:qchlwl}
$\lambda$ is called a \dtgterm{lowest weight} in $\qcrstQ$ if there exists a lowest-weight element $x \in Q$ such that $\lambda = \wt(x)$.
\end{enumerate}
\end{dfn}

\begin{prop}
\label{prop:qchlwMm}
Let $\qcrstQ$ be a quasi-crystal, and let $\lambda$ be a weight in $\wt(Q) = \set{ \wt(x) \given x \in Q }$.
\begin{enumerate}
\item\label{prop:qchlwMmh}
If $\lambda$ is maximal among weights in $\wt(Q)$, then $\lambda$ is a highest weight, and any element $x \in Q$ such that $\wt(x) = \lambda$ is of highest weight.

\item\label{prop:qchlwMml}
If $\lambda$ is minimal among weights in $\wt(Q)$, then $\lambda$ is a lowest weight, and any element $x \in Q$ such that $\wt(x) = \lambda$ is of lowest weight.
\end{enumerate}
\end{prop}

\begin{proof}
\itmref{prop:qchlwMmh} Let $x \in Q$ be such that $\wt(x) = \lambda$. If $x$ is not of highest weight, then $\qKoe_i (x) \in Q$, for some $i \in I$, and by \comboref{Proposition}{prop:qcrlqKo}, $\wt \parens[\big]{ \qKoe_i (x) } > \wt(x)$. Therefore, $\lambda$ is not maximal among weights in $\wt(Q)$.

\itmref{prop:qchlwMml} Let $x \in Q$ be such that $\wt(x) = \lambda$. If $x$ is not of lowest weight, then $\qKof_i (x) \in Q$, for some $i \in I$, and by \comboref{Proposition}{prop:qcrlqKo}, $\wt(x) > \wt \parens[\big]{ \qKof_i (x) }$. Hence, $\lambda$ is not minimal among weights in $\wt(Q)$.
\end{proof}

Since the quasi-Kashiwara operators of a quasi-crystal $\qcrstQ$ can be regarded as partial maps from $Q$ to $Q$, we can compose them in a natural way. As usual, for $i \in I$, set $\qKoe_i^0$ and $\qKof_i^0$ to be the identity map on $Q$, and recursively, define $\qKoe_i^{k+1} = \qKoe_i \qKoe_i^k$ and $\qKof_i^{k+1} = \qKof_i \qKof_i^k$, for $k \geq 0$.

\begin{dfn}
\label{dfn:snqc}
A quasi-crystal $\qcrstQ$ is said to be \dtgterm{seminormal} if for any $x \in Q$ and $i \in I$,
\[ \qKoec_i (x) = \max \set[\big]{ k \in \Z_{\geq 0} \given \qKoe_i^k (x) \in Q } \]
and
\[ \qKofc_i (x) = \max \set[\big]{ k \in \Z_{\geq 0} \given \qKof_i^k (x) \in Q }, \]
whenever $\qKoec_i (x) \neq +\infty$.
\end{dfn}

The quasi-crystals described in items\avoidrefbreak \itmref{exa:qctAn} to\avoidrefbreak \itmref{exa:qcA32} of \comboref{Example}{exa:qc} are seminormal.
On the other hand, the quasi-crystal described in item\avoidrefbreak \itmref{exa:qcQ2} is not seminormal.

As pointed out in \comboref{Remark}{rmk:crstqc}, a crystal $\crstB$ satisfies $\qKoec_i (x) \neq +\infty$, for all $x \in B$ and $i \in I$.
If $\crstB$ is seminormal, then the equalities in \comboref{Definition}{dfn:snqc} are verified for any $x \in B$ and $i \in I$, and so, $\crstB$ is seminormal as a crystal\avoidcitebreak \cite[formula~(2.6)]{BS17}.
Thus, the seminormal property for quasi-crystals generalize the one for crystals in the following sense.

\begin{rmk}
\label{rmk:sncrstsnqc}
For a crystal $\crstB$, we have that $\crstB$ is seminormal as a crystal if and only if it is seminormal as a quasi-crystal.
\end{rmk}

We have just seen that the seminormal property for quasi-crystals is consistent with the corresponding property for crystals.
The exception when $\qKoec_i (x)$ takes the value $+\infty$ is crucial.
Without this exception, in the case $\qKoec_i (x) = +\infty$ we would have that
$\max \set[\big]{ k \in \Z_{\geq 0} \given \qKoe_i^k (x) \in Q } = 0$,
by \itmcomboref{Definition}{dfn:qc}{dfn:qcpinfty}, and hence the class of seminormal quasi-crystals would coincide with the class of seminormal crystals, and we would not have a proper generalization of the seminormal property as intended.
Thus this exception is vital. However, it has deep implications, as some common results for seminormal crystals are not satisfied by seminormal quasi-crystals. For example, we can no longer guarantee the weight of a highest-weight element to be dominant. Instead we have the following result.

\begin{prop}
\label{prop:snqchlw}
Let $x \in Q$ be an element of a seminormal quasi-crystal $\qcrstQ$.
\begin{enumerate}
\item\label{prop:snqchlwh}
If $x$ is of highest weight and $\innerp[\big]{\wt(x)}{\alphav_i} < 0$, for some $i \in I$, then $\qKoec_i (x) = \qKofc_i (x) = +\infty$.

\item\label{prop:snqchlwl}
If $x$ is of lowest weight and $\innerp[\big]{\wt(x)}{\alphav_i} > 0$, for some $i \in I$, then $\qKoec_i (x) = \qKofc_i (x) = +\infty$.
\end{enumerate}
\end{prop}

\begin{proof}
\itmref{prop:snqchlwh} Suppose that $\qKoec_i (x) \neq +\infty$ (or equivalently, $\qKofc_i (x) \neq +\infty$), for some $i \in I$.
As $\qcrstQ$ is seminormal, we have that $\qKoec_i (x), \qKofc_i (x) \in \Z_{\geq 0}$ and by \itmcomboref{Definition}{dfn:qc}{dfn:qcwt},
\[ \innerp[\big]{\wt(x)}{\alphav_i} = \qKofc_i(x) - \qKoec_i(x). \]
Thus, if $\innerp[\big]{\wt(x)}{\alphav_i} < 0$, then $\qKoec_i(x) > 0$, which implies that $\qKoe_i(x) \in Q$, because $\qcrstQ$ is seminormal. Hence, $x$ is not of highest weight.

\itmref{prop:snqchlwl} As above, if $\qKoec_i (x) \neq +\infty$ and $\innerp[\big]{\wt(x)}{\alphav_i} > 0$, then $\qKofc_i (x) \in \Z_{> 0}$, which implies that $\qKof_i (x) \in Q$. And therefore, $x$ is not of lowest weight.
\end{proof}

Now, we introduce the definition of a homomorphism between quasi-crystals, which is analogous to the one for crystals.

\begin{dfn}
\label{dfn:qch}
Let $\qcrstQ$ and $\qcrstQ'$ be quasi-crystals of the same type. A \dtgterm{quasi-crystal homomorphism} $\psi$ from $\qcrstQ$ to $\qcrstQ'$, denoted by $\psi : \qcrstQ \to \qcrstQ'$, is a map $\psi : Q \sqcup \{\undf\} \to Q' \sqcup \{\undf\}$ that satisfies the following conditions:
\begin{enumerate}
\item\label{dfn:qchundf}
$\psi (\undf) = \undf$;

\item\label{dfn:qchwtc}
if $\psi(x) \in Q'$, then
$\wt \parens[\big]{ \psi(x) } = \wt(x)$,
$\qKoec_i \parens[\big]{ \psi(x) } = \qKoec_i (x)$, and
$\qKofc_i \parens[\big]{ \psi(x) } = \qKofc_i (x)$;

\item\label{dfn:qchqKoe}
if $\qKoe_i (x) \in Q$ and $\psi(x), \psi \parens[\big]{ \qKoe_i (x) } \in Q'$, then
$\psi \parens[\big]{ \qKoe_i (x) } = \qKoe_i \parens[\big]{ \psi (x) }$;

\item\label{dfn:qchqKof}
if $\qKof_i (x) \in Q$ and $\psi(x), \psi \parens[\big]{ \qKof_i (x) } \in Q'$, then
$\psi \parens[\big]{ \qKof_i (x) } = \qKof_i \parens[\big]{ \psi (x) }$;
\end{enumerate}
for $x \in Q$ and $i \in I$.

A \dtgterm{quasi-crystal isomorphism} $\psi$ between $\qcrstQ$ and $\qcrstQ'$ is a bijection $\psi : Q \sqcup \{\undf\} \to Q' \sqcup \{\undf\}$ such that $\psi : \qcrstQ \to \qcrstQ'$ and $\psi^{-1} : \qcrstQ' \to \qcrstQ$ are quasi-crystal homomorphisms.
We say that $\qcrstQ$ and $\qcrstQ'$ are \dtgterm{isomorphic} if there exists a quasi-crystal isomorphism between $\qcrstQ$ and $\qcrstQ'$.
\end{dfn}

Due to condition\avoidrefbreak \itmref{dfn:qchundf}, when defining a quasi-crystal homomorphism $\psi$, we omit the explicit mention to $\psi(\undf) = \undf$.
Moreover, as $\undf$ is an auxilary symbol which stands for undefinition, alternatively a quasi-crystal homomorphism $\psi : \qcrstQ \to \qcrstQ'$ can be regarded as a partial map $\psi$ from $Q$ to $Q'$ satisfying conditions\avoidrefbreak \itmref{dfn:qchwtc} to\avoidrefbreak \itmref{dfn:qchqKof}.
Thus, when defining a quasi-crystal homomorphism, we usually only give the images for the elements $x \in Q$ such that $\psi(x) \in Q'$.
For the sake of simplicity, by saying that a map $\psi : Q \to Q'$ is a quasi-crystal homomorphism from $\qcrstQ$ to $\qcrstQ'$, we mean that the map $\psi' : Q \sqcup \{\undf\} \to Q' \sqcup \{\undf\}$, given by $\psi' (\undf) = \undf$ and $\psi' (x) = \psi (x)$, for each $x \in Q$, is a quasi-crystal homomorphism from $\qcrstQ$ to $\qcrstQ'$.

The notion of crystal homomorphism can be placed in the context of quasi-crystals in the following way.

\begin{rmk}
\label{rmk:crstqch}
A \dtgterm{crystal homomorphism} is a quasi-crystal homomorphism between two crystals.
\end{rmk}

At this point we defined quasi-crystals and homomorphisms between them.
It is immediate from \comboref{Definition}{dfn:qch} that given a quasi-crystal $\qcrstQ$, the identity map on $Q$ is a quasi-crystal homomorphism from $\qcrstQ$ to $\qcrstQ$.
The following result follows by a straightforward application of the definitions.

\begin{prop}
\label{prop:qchcomposition}
Let $\qcrstQ_1$, $\qcrstQ_2$ and $\qcrstQ_3$ be quasi-crystals of the same type,
and let $\psi_1 : \qcrstQ_1 \to \qcrstQ_2$ and $\psi_2 : \qcrstQ_2 \to \qcrstQ_3$ be quasi-crystal homomorphisms.
Then, $\psi_2 \circ \psi_1$ is a quasi-crystal homomorphism from $\qcrstQ_1$ to $\qcrstQ_3$.
\end{prop}

Thus we obtain a category whose objects are quasi-crystals of the same type and morphisms are quasi-crystal homomorphisms.


We say that a quasi-crystal homomorphism $\psi : \qcrstQ \to \qcrstQ'$ is injective, surjective or bijective if the map $\psi : Q \sqcup \{\undf\} \to Q' \sqcup \{\undf\}$ is injective, surjective or bijective, respectively. As the following example shows, a bijective quasi-crystal homomorphosm is not necessarily a quasi-crystal isomorphism.

\begin{exa}
\label{exa:qch}
Let $\qctA_2$ and $\qcrstQ$ be the quasi-crystals of type $\tA_2$ described respectively in items\avoidrefbreak \itmref{exa:qctAn} and\avoidrefbreak \itmref{exa:qcQ2} of \comboref{Example}{exa:qc}.
Define a map $\psi : Q \to A_2$ by $\psi (a) = 1$ and $\psi (b) = 2$. Then, $\psi$ is a quasi-crystal homomorphism from $\qcrstQ$ to $\qctA_2$. But $\psi$ is not a quasi-crystal isomorphism as $\psi^{-1}$ does not verify conditions\avoidrefbreak \itmref{dfn:qchqKoe} and\avoidrefbreak \itmref{dfn:qchqKof} of \comboref{Definition}{dfn:qch}.
\end{exa}

By \comboref{Remarks}{rmk:crstqc} and\avoidrefbreak \ref{rmk:crstqch}, we have that $\psi$ is a bijective crystal homomorphism that is not a crystal isomorphism.
Also, notice that $\qcrstQ$ is not seminormal.
So, in the following results we present an alternative characterization of quasi-crystal isomorphisms for seminormal quasi-crystals.

\begin{lem}
\label{lem:qchcom}
Let $\qcrstQ$ and $\qcrstQ'$ be quasi-crystals of the same type, and let $\psi : \qcrstQ \to \qcrstQ'$ be a bijective quasi-crystal homomorphism.
The following conditions are equivalent
\begin{enumerate}
\item\label{lem:qchcomqKoe}
$\psi \parens[\big]{ \qKoe_i (x) } = \qKoe_i \parens[\big]{ \psi (x) }$ for all $x \in Q$ and $i \in I$;

\item\label{lem:qchcomqKof}
$\psi \parens[\big]{ \qKof_i (x) } = \qKof_i \parens[\big]{ \psi (x) }$ for all $x \in Q$ and $i \in I$.
\end{enumerate}
\end{lem}

\begin{proof}
Suppose that $\psi$ satisfies\avoidrefbreak \itmref{lem:qchcomqKoe}.
Let $x \in Q$ and $i \in I$.
Since $\psi$ is bijective and $\psi (\undf) = \undf$ by \itmcomboref{Definition}{dfn:qch}{dfn:qchundf}, then $\psi (y) \in Q'$ for all $y \in Q$.
Thus, if $\qKof_i (x) \in Q$, then $\psi \parens[\big]{ \qKof_i (x) } \in Q'$ which implies $\psi \parens[\big]{ \qKof_i (x) } = \qKof_i \parens[\big]{ \psi (x) }$ by \itmcomboref{Definition}{dfn:qch}{dfn:qchqKof}.
If $\qKof_i \parens[\big]{ \psi (x) } \in Q'$, or equivalently, $\psi^{-1} \parens[\big]{ \qKof_i (\psi (x)) } \in Q$, then
\[ \psi \parens[\big]{ \qKoe_i \parens[\big]{ \psi^{-1} \parens[\big]{ \qKof_i (\psi (x)) } } } = \qKoe_i \parens[\big]{ \psi \psi^{-1} \parens[\big]{ \qKof_i (\psi (x)) } } = \qKoe_i \qKof_i \parens[\big]{ \psi (x) } = \psi (x), \]
as we assumed that $\psi$ satisfies\avoidrefbreak \itmref{lem:qchcomqKoe}, and so, $x = \qKoe_i \parens[\big]{ \psi^{-1} \parens[\big]{ \qKof_i (\psi (x)) } }$ which implies
\[ \psi \parens[\big]{ \qKof_i (x) } = \psi \parens[\big]{ \qKof_i \qKoe_i \parens[\big]{ \psi^{-1} \parens[\big]{ \qKof_i (\psi (x)) } } } = \psi \psi^{-1} \parens[\big]{ \qKof_i (\psi (x)) } = \qKof_i \parens[\big]{ \psi (x) }. \]
Hence, $\psi$ satisfies\avoidrefbreak \itmref{lem:qchcomqKof}.

The fact that \itmref{lem:qchcomqKof} implies \itmref{lem:qchcomqKoe} follows analogously.
\end{proof}

\begin{thm}
\label{thm:qciso}
Let $\qcrstQ$ and $\qcrstQ'$ be quasi-crystals of the same type, and let $\psi : \qcrstQ \to \qcrstQ'$ be a quasi-crystal homomorphism.
Then, $\psi$ is a quasi-crystal isomorphism if and only if $\psi$ is bijective and satisfies\avoidrefbreak \itmref{lem:qchcomqKoe} or\avoidrefbreak \itmref{lem:qchcomqKof} of \comboref{Lemma}{lem:qchcom}.
\end{thm}

\begin{proof}
Suppose that $\psi$ is a quasi-crystal isomorphism.
By \comboref{Definition}{dfn:qch}, $\psi$ is bijective.
Let $x \in Q$ and $i \in I$.
If $\qKoe_i (x) \in Q$, we also have that $\psi (x), \psi \parens[\big]{ \qKoe_i (x) } \in Q'$ as $\psi$ is bijective and $\psi (\undf) = \undf$, and so, $\psi \parens[\big]{ \qKoe_i (x) } = \qKoe_i \parens[\big]{ \psi (x) }$ by \itmcomboref{Definition}{dfn:qch}{dfn:qchqKoe}.
Similarly, since $\psi^{-1}$ is also a quasi-crystal isomorphism, if $\qKoe_i \parens[\big]{ \psi (x) } \in Q'$, then
\[ \psi^{-1} \parens[\big]{ \qKoe_i (\psi (x)) } = \qKoe_i \parens[\big]{ \psi^{-1} \psi (x) } = \qKoe_i (x), \]
which implies $\qKoe_i \parens[\big]{ \psi (x) } = \psi \parens[\big]{ \qKoe_i (x) }$. Hence, $\psi$ satisfies \itmcomboref{Lemma}{lem:qchcom}{lem:qchcomqKoe}.

Conversely, by \comboref{Lemma}{lem:qchcom}, we can assume that $\psi$ is bijective and satisfies conditions\avoidrefbreak \itmref{lem:qchcomqKoe} and\avoidrefbreak \itmref{lem:qchcomqKof} of that lemma.
Clearly, $\psi^{-1} (\undf) = \undf$.
Let $x' \in Q'$ and $i \in I$. Since $\psi$ is a quasi-crystal homomorphism, by \itmcomboref{Definition}{dfn:qch}{dfn:qchwtc} we have that
\[ \wt \parens[\big]{ \psi^{-1} (x') } = \wt \parens[\big]{ \psi \parens[\big]{ \psi^{-1} (x') } } = \wt (x'). \]
Similarly, we get $\qKoec_i \parens[\big]{ \psi^{-1} (x') } = \qKoec_i (x')$ and $\qKofc_i \parens[\big]{ \psi^{-1} (x') } = \qKofc_i (x')$.
Since $\psi$ satisfies \itmcomboref{Lemma}{lem:qchcom}{lem:qchcomqKoe}, then
\[ \qKoe_i \parens[\big]{ \psi^{-1} (x') } = \psi^{-1} \psi \parens[\big]{ \qKoe_i \parens[\big]{ \psi^{-1} (x') } } = \psi^{-1} \parens[\big]{ \qKoe_i \parens[\big]{ \psi \psi^{-1} (x') } } = \psi^{-1} \parens[\big]{ \qKoe_i (x') }. \]
And since $\psi$ satisfies \itmcomboref{Lemma}{lem:qchcom}{lem:qchcomqKof}, then
\[ \qKof_i \parens[\big]{ \psi^{-1} (x') } = \psi^{-1} \psi \parens[\big]{ \qKof_i \parens[\big]{ \psi^{-1} (x') } } = \psi^{-1} \parens[\big]{ \qKof_i \parens[\big]{ \psi \psi^{-1} (x') } } = \psi^{-1} \parens[\big]{ \qKof_i (x') }. \]
Hence, $\psi^{-1}$ is a quasi-crystal homomorphism from $\qcrstQ'$ to $\qcrstQ$,
and therefore, $\psi$ is a quasi-crystal isomorphism between $\qcrstQ$ and $\qcrstQ'$.
\end{proof}

\begin{cor}
\label{cor:snqciso}
Let $\qcrstQ$ and $\qcrstQ'$ be seminormal quasi-crystals of the same type, and let $\psi : \qcrstQ \to \qcrstQ'$ be a bijective quasi-crystal homomorphism.
Then, $\psi$ is a quasi-crystal isomorphism.
\end{cor}

\begin{proof}
Let $x \in Q$ and $i \in I$.
As $\psi$ is bijective, we get that $\psi(x) \in Q'$.
Since $\qcrstQ$ and $\qcrstQ'$ are seminormal and $\qKoec_i \parens[\big]{ \psi (x) } = \qKoec_i (x)$, we have that $\qKoe_i (x) \in Q$ if and only if $\qKoe_i \parens[\big]{ \psi (x) } \in Q'$.
So, if $\qKoe_i (x) \in Q$, then $\psi \parens[\big]{ \qKoe_i (x) } \in Q$, as $\psi$ is bijective, and $\psi \parens[\big]{ \qKoe_i (x) } = \qKoe_i \parens[\big]{ \psi (x) }$ by \itmcomboref{Definition}{dfn:qch}{dfn:qchqKoe}.
Otherwise, $\qKoe_i (x) = \undf = \qKoe_i \parens[\big]{ \psi (x) }$, which implies that $\psi \parens[\big]{ \qKoe_i (x) } = \undf = \qKoe_i \parens[\big]{ \psi (x) }$, by \itmcomboref{Definition}{dfn:qch}{dfn:qchundf}.
Hence, $\psi$ satisfies \itmcomboref{Lemma}{lem:qchcom}{lem:qchcomqKoe}, and by \comboref{Theorem}{thm:qciso}, $\psi$ is a quasi-crystal isomorphism.
\end{proof}

\section{Quasi-crystal graphs}
\label{sec:qcg}

In this section we present a combinatorial approach to quasi-crystals, which results in a generalization of the notion of crystal graph.
In this framework we are able to characterize some substructures of quasi-crystals, generalizing similar structures described for crystals based on crystal graphs.
Finally, as a crystal graph of a seminormal crystal completely determines its crystal structure, we show a similar connection between quasi-crystal graphs and seminormal quasi-crystals.

\begin{dfn}
  Let $\Lambda$ be a weight lattice. A \dtgterm{weight map} on a graph $\Gamma$ with vertex set $X$ is a map
  ${\wt} : X \to \Lambda$. For a vertex $x \in X$ of $\Gamma$, $\wt(x)$ is called the \dtgterm{weight} of $x$. In this
  case, we say the graph $\Gamma$ is \dtgterm{$\Lambda$-weighted}.
\end{dfn}

\begin{dfn}
\label{dfn:qcg}
Let $\Phi$ be a root system with weight lattice $\Lambda$ and index set $I$ for the simple roots $(\alpha_i)_{i \in I}$.
The \dtgterm{quasi-crystal graph} $\Gamma_\qcrstQ$ of a quasi-crystal $\qcrstQ$ of type $\Phi$ is a $\Lambda$-weighted $I$-labelled directed graph with vertex set $Q$ and an edge $x \lbedge{i} y$ from $x \in Q$ to $y \in Q$ labelled by $i \in I$ whenever $\qKof_i (x) = y$, and a loop on $x \in Q$ labelled by $i \in I$ whenever $\qKoec_i (x) = +\infty$.
For $x \in Q$, let $\Gamma_\qcrstQ (x)$ denote the connected component of $\Gamma_\qcrstQ$ containing the vertex $x$.
\end{dfn}

In comparison with crystal graphs, by requiring quasi-crystal graphs to be $\Lambda$-weighted, we accommodate the weight map $\wt$ of a quasi-crystal directly in the definition of its quasi-crystal graph.
Also, we have that a quasi-crystal graph may not be simple.
Moreover, a quasi-crystal graph is simple if the maps $\qKoec_i$ ($i \in I$) do not take the value $+\infty$, and from \comboref{Remark}{rmk:crstqc}, we observe the following.

\begin{rmk}
\label{rmk:crstgqcg}
For a quasi-crystal $\qcrstQ$, the quasi-crystal graph $\Gamma_{\qcrstQ}$ is simple if and only if $\qcrstQ$ is a crystal.
And if so, the quasi-crystal graph of $\qcrstQ$ coincides with its crystal graph.
\end{rmk}

\begin{exa}
\label{exa:qcg}
\exaitem\label{exa:qcgtAn}
The quasi-crystal graph $\Gamma_{\qctAn}$ of the standard quasi-crystal $\qctAn$ of type $\tAn$, described in \itmcomboref{Example}{exa:qc}{exa:qctAn}, is the following.
\[ 1 \lbedge{1} 2 \lbedge{2} 3 \lbedge{3} \cdots \lbedge{n-1} n \]

\exaitem\label{exa:qcgtCn}
The quasi-crystal graph $\Gamma_{\qctCn}$ of the standard quasi-crystal $\qctCn$ of type $\tCn$, described in \itmcomboref{Example}{exa:qc}{exa:qctCn}, is the following.
\[ 1 \lbedge{1} 2 \lbedge{2} \cdots \lbedge{n-1} n \lbedge{n} \wbar{n} \lbedge{n-1} \wbar{n-1} \lbedge{n-2} \cdots \lbedge{1} \wbar{1} \]

\exaitem\label{exa:qcgA32}
The quasi-crystal graph $\Gamma_{\qctA_3^2}$ of the quasi-crystal $\qctA_3^2$ of type $\tA_3$, described in \itmcomboref{Example}{exa:qc}{exa:qcA32}, is the following.
\[
\begin{tikzpicture}[widecrystal,baseline=(33.base)]
  %
  \node (11) at (1, 3) {(1, 1)};
  \node (21) at (2, 3) {(2, 1)};
  \node (31) at (3, 3) {(3, 1)};
  \node (12) at (1, 2) {(1, 2)};
  \node (22) at (2, 2) {(2, 2)};
  \node (32) at (3, 2) {(3, 2)};
  \node (13) at (1, 1) {(1, 3)};
  \node (23) at (2, 1) {(2, 3)};
  \node (33) at (3, 1) {(3, 3)};
  \path (11) edge node {1} (21)
        (21) edge node {1} (22)
        (21) edge node {2} (31)
        (31) edge node {1} (32)
        (12) edge [loop left] node {1} ()
        (12) edge node {2} (13)
        (22) edge node {2} (32)
        (32) edge node {2} (33)
        (13) edge node {1} (23)
        (23) edge [loop right] node {2} ();
\end{tikzpicture}
\]
\end{exa}

Note that\avoidrefbreak \itmref{exa:qcgtAn} and\avoidrefbreak \itmref{exa:qcgtCn} of the previous example are crystal graphs (\comboref{Remark}{rmk:crstgqcg}).
For the crystal graphs associated to the standard crystals of type $\tBn$ and $\tDn$, see for example\avoidcitebreak \cite[\S~3.3]{CGM19}.

Let $x \lbedge{i} y$ be an edge of a quasi-crystal graph $\Gamma_\qcrstQ$ of a quasi-crystal $\qcrstQ$.
If $x \neq y$, then $y = \qKof_i(x)$ by \comboref{Definition}{dfn:qcg}, or equivalently, $x = \qKoe_i (y)$ due to \itmcomboref{Definition}{dfn:qc}{dfn:qciff}.
Otherwise, we have that $x = y$, that is, $x$ has a loop labelled by $i$, and so, $\qKoec_i (x) = +\infty$ by \comboref{Definition}{dfn:qcg}, which implies that $\qKoe_i$ and $\qKof_i$ are undefined on $x$, by \itmcomboref{Definition}{dfn:qc}{dfn:qcpinfty}.
In either case, we have that if $x \lbedge{i} y'$ is an edge of $\Gamma_\qcrstQ$, then $y = y'$, and similarly, if $x' \lbedge{i} y$ is an edge of $\Gamma_\qcrstQ$, then $x = x'$.
Hence, for $i \in I$, a vertex of $\Gamma_\qcrstQ$ is the start of at most one edge, and is the end of at most one edge labelled by $i$.

We will show that quasi-crystal graphs provide a combinatorial framework to study quasi-crystals, analogous to the tools that crystal graphs provide for crystals.
For instance, \comboref{Proposition}{prop:qcrlqKo} is equivalent to state that for a quasi-crystal $\qcrstQ$, if $x \lbedge{i} y$ is an edge of $\Gamma_{\qcrstQ}$ with $x \neq y$, then $\wt(x) > \wt(y)$.
Also, \comboref{Definition}{dfn:qchlwe} is equivalent to stating that an element $x \in Q$ is of highest (or lowest) weight if the only edges of $\Gamma_{\qcrstQ}$ ending (resp., starting) at $x$ are loops.

From \comboref{Remark}{rmk:crstgqcg}, the combinatorial framework formed by quasi-crystal graphs is a genuine generalization of the framework formed by crystal graphs.
This allows a natural generalization of structures such as connected components.

\begin{dfn}
\label{dfn:qccc}
Let $\qcrstQ$ be a quasi-crystal.
A \dtgterm{connected component} of $\qcrstQ$ is a subset $Q'$ of $Q$ that satisfies the following conditions:
\begin{enumerate}
\item\label{dfn:qccccon}
for each $x, y \in Q'$ there exist $g_1, \ldots, g_m \in \set{ \qKoe_i, \qKof_i \given i \in I }$ such that $g_1 \cdots g_m (x) = y$;

\item\label{dfn:qcccmax}
$\qKoe_i (x), \qKof_i (x) \in Q' \sqcup \{\undf\}$ for all $x \in Q'$ and $i \in I$.
\end{enumerate}
We also use the term \dtgterm{connected component} to refer to the quasi-crystal $\qcrstQ'$ consisting of $Q'$ together with the maps $\wt$, $\qKoe_i$, $\qKof_i$, $\qKoec_i$ and $\qKofc_i$ ($i \in I$) of $\qcrstQ$ restricted to $Q'$.
For each $x \in Q$, the connected component of $\qcrstQ$ containing $x$ is denoted by $Q (x)$, and the associated quasi-crystal is denoted by $\qcrstQ (x)$.
\end{dfn}

As a justification for this terminology, we check that connected components of a quasi-crystal $\qcrstQ$ and the vertex sets of connected components of the quasi-crystal graph $\Gamma_\qcrstQ$ identify the same subsets of $Q$.

\begin{prop}
\label{prop:qcccqcgcc}
Let $\qcrstQ$ be a quasi-crystal, and let $Q' \subseteq Q$.
Then, $Q'$ is a connected component of $\qcrstQ$ if and only if $Q'$ is the vertex set of a connected component of $\Gamma_{\qcrstQ}$.
\end{prop}

\begin{proof}
Assume that $Q'$ is a connected component of $\qcrstQ$.
Let $x, y \in Q'$.
Then, by \itmcomboref{Definition}{dfn:qccc}{dfn:qccccon}, there exist $g_1, \ldots, g_m \in \set{ \qKoe_i, \qKof_i \given i \in I }$ such that $g_1 \cdots g_m (x) = y$.
Set $x_0 = x$, and $x_{k+1} = g_{m-k} (x_k)$, for $k = 0, 1, \ldots, m-1$.
Note that $x_{k+1} = g_{m-k} \cdots g_m (x) \in Q'$, by \itmcomboref{Definition}{dfn:qccc}{dfn:qcccmax}.
In particular, $x_m = g_1 \cdots g_m (x) = y$.
If $g_{m-k} = \qKof_i$, for some $i \in I$, then $x_k \lbedge{i} x_{k+1}$ is an edge of $\Gamma_{\qcrstQ}$.
Otherwise, $g_{m-k} = \qKoe_i$, for some $i \in I$, and so, $x_{k+1} \lbedge{i} x_k$ is an edge of $\Gamma_{\qcrstQ}$.
For $z \in Q$, if $x \lbedge{i} z$ or $z \lbedge{i} x$ is an edge of $\Gamma_\qcrstQ$, then $z = \qKof_i(x)$ or $z = \qKoe_i(x)$, respectively, which implies that $z \in Q'$, by \itmcomboref{Definition}{dfn:qccc}{dfn:qcccmax}.
Therefore, the subgraph of $\Gamma_\qcrstQ$ induced by $Q'$ is a connected component.

Conversely, assume that $Q'$ is a vertex set of a connected component of $\Gamma_\qcrstQ$.
Let $x, y \in Q'$.
Then there exist $x_0, \ldots, x_m \in Q'$ such that $x = x_0$, $y = x_m$, and for $k=0,\ldots,m-1$, $x_{k} \lbedge{i} x_{k+1}$ or $x_{k+1} \lbedge{i} x_k$ is an edge of $\Gamma_\qcrstQ$, for some $i \in I$.
If $x_{k} \lbedge{i} x_{k+1}$ is an edge of $\Gamma_{\qcrstQ}$, for some $i \in I$, set $g_{m-k} = \qKof_i$.
Otherwise, $x_{k+1} \lbedge{i} x_k$ is an edge of $\Gamma_{\qcrstQ}$, for some $i \in I$, and so, set $g_{m-k} = \qKoe_i$.
In any case we have that $x_{k+1} = g_{m-k} (x_k)$, which implies that $g_1 \cdots g_m (x) = y$.
For $i \in I$, if $\qKof_i (x) \in Q$, then $x \lbedge{i} \qKof_i (x)$ is an edge of $\Gamma_\qcrstQ$, and so, $\qKof_i (x) \in Q'$.
Similarly, if $\qKoe_i (x) \in Q$, then $\qKoe_i (x) \lbedge{i} x$ is an edge of $\Gamma_\qcrstQ$, which implies that $\qKoe_i (x) \in Q'$.
Therefore, $Q'$ is a connected component of $\qcrstQ$.
\end{proof}

Given $\Lambda$-weighted $I$-labelled directed graphs $\Gamma_1$ and $\Gamma_2$ with vertex sets $X_1$ and $X_2$,
respectively, a \dtgterm{homomorphism} $\psi$ from $\Gamma_1$ to $\Gamma_2$ is a map $\psi : X_1 \to X_2$ such that
$\wt \parens[\big]{ \psi(x) } = \wt(x)$, for all $x \in X_1$, and $\psi(x) \lbedge{i} \psi(y)$ is an edge of $\Gamma_2$, whenever
$x \lbedge{i} y$ is an edge of $\Gamma_1$. If $\psi$ is also bijective and $\psi^{-1}$ is a homomorphism from $\Gamma_2$
to $\Gamma_1$, then $\psi$ is said to be an \dtgterm{isomorphism} between $\Gamma_1$ and $\Gamma_2$.

We also have the following relation between quasi-crystal homomorphisms and graph homomorphisms.

\begin{lem}
\label{lem:qchqcgh}
Let $\qcrstQ$ and $\qcrstQ'$ be quasi-crystals of the same type, and let $\psi : \qcrstQ \to \qcrstQ'$ be a quasi-crystal homomorphism such that $\psi(Q) \subseteq Q'$.
Then, $\psi$ is a graph homomorphism from $\Gamma_\qcrstQ$ to $\Gamma_{\qcrstQ'}$.
\end{lem}

\begin{proof}
Let $x, y \in Q$ and $i \in I$.
By \itmcomboref{Definition}{dfn:qch}{dfn:qchwtc}, we have that $\wt \parens[\big]{ \psi(x) } = \wt(x)$.
Suppose that $x \lbedge{i} y$ is an edge of $\Gamma_{\qcrstQ}$.
If $x = y$, that is, $x$ has a loop labelled by $i$, then $\qKoec_i(x) = +\infty$, and by \itmcomboref{Definition}{dfn:qch}{dfn:qchwtc}, $\qKoec_i \parens[\big]{ \psi(x) } = +\infty$, which implies that $\psi(x)$ has a loop labelled by $i$ in $\Gamma_{\qcrstQ'}$.
Otherwise, $x \neq y$, we have by \comboref{Definition}{dfn:qcg} that $\qKof_i(x) = y$,
and since $\psi(x), \psi(y) \in \psi(Q) \subseteq Q'$, we get by \itmcomboref{Definition}{dfn:qch}{dfn:qchqKof} that $\qKof_i \parens[\big]{ \psi(x) } = \psi(y)$,
which implies that $\psi(x) \lbedge{i} \psi(y)$ is an edge of $\Gamma_{\qcrstQ'}$.
Therefore, $\psi$ is a graph homomorphism.
\end{proof}

Notice that the converse of the previous result does not hold, as $\psi$ may be a graph homomorphism from $\Gamma_{\qcrstQ}$ to $\Gamma_{\qcrstQ'}$ and not be a quasi-crystal homomorphism from $\qcrstQ$ to $\qcrstQ'$.

\begin{exa}
\label{exa:ghisnotqch}
Consider the root system of type $\tA_2$.
Take $\qcrstQ$ consisting of the set $Q = \set{x}$, where $\wt(x) = 0$, $\qKoe(x) = \qKof(x) = \undf$ and $\qKoec(x) = \qKofc(x) = 0$,
and take $\qcrstQ'$ consisting of the set $Q' = \set{x'}$, where $\wt(x') = 0$, $\qKoe(x') = \qKof(x') = \undf$ and $\qKoec(x') = \qKofc(x') = +\infty$.
The quasi-crystal graphs of $\qcrstQ$ and $\qcrstQ'$ are respectively
\[
\begin{tikzpicture}[basiccrystal, baseline=(x.base)]
  %
  \node (x) at (0, 0) {x};
\end{tikzpicture}
\qquad \text{and} \qquad
\begin{tikzpicture}[basiccrystal, baseline=(x.base)]
  %
  \node[inner sep=2.5pt] (x) at (0, 0) {x'};
  \path (x) edge[loop right] node {1} ();
\end{tikzpicture}
.
\]
The map $\psi : Q \to Q'$, defined by $\psi(x) = x'$, is a graph homomorphism, but not a quasi-crystal homomorphism, because $\qKoec \parens[\big]{ \psi(x) } = +\infty \neq 0 = \qKoec(x)$.
\end{exa}

A quasi-crystal isomorphism $\psi$ between quasi-crystals $\qcrstQ$ and $\qcrstQ'$ satisfies the property $\psi(Q) = Q'$.
Thus, it is immediate from \comboref{Lemma}{lem:qchqcgh} that $\psi$ and $\psi^{-1}$ are graph homomorphisms, which implies that $\psi$ is a graph isomorphism between $\Gamma_{\qcrstQ}$ and $\Gamma_{\qcrstQ'}$.
This leads to the following result.

\begin{prop}
\label{prop:qccciso}
Let $\qcrstQ$ and $\qcrstQ'$ be two quasi-crystals of the same type, and let $\psi : \qcrstQ \to \qcrstQ'$ be a quasi-crystal isomorphism.
Then, $Q_0$ is a connected component of $\qcrstQ$ if and only if $\psi (Q_0)$ is a connected component of $\qcrstQ'$.
Furthermore, for each $x \in Q$, the restriction of $\psi$ to $Q (x)$ is a quasi-crystal isomorphism between $\qcrstQ (x)$ and $\qcrstQ' \parens[\big]{ \psi (x) }$.
\end{prop}

\begin{proof}
Suppose $Q_0$ is a connected component of $\qcrstQ$.
Let $x', y' \in \psi(Q_0)$.
Set $x = \psi^{-1} (x')$ and $y = \psi^{-1} (y')$.
As $x, y \in Q_0$, there exist $g_1, \ldots, g_m \in \set{ \qKoe_i, \qKof_i \given i \in I }$ such that $g_1 \cdots g_m (x) = y$.
By \comboref{Lemma}{lem:qchcom} and \comboref{Theorem}{thm:qciso}, we have that
\[ g_1 \cdots g_m (x') = g_1 \cdots g_m \parens[\big]{ \psi (x) } = \psi \parens[\big]{ g_1 \cdots g_m (x) } = \psi (y) = y', \]
which implies that $\psi(Q_0)$ satisfies \itmcomboref{Definition}{dfn:qccc}{dfn:qccccon}.
By the same results, for $i \in I$, we have that
$\psi \parens[\big]{ \qKoe_i (x) } = \qKoe_i \parens[\big]{ \psi (x) } = \qKoe_i (x')$ and
$\psi \parens[\big]{ \qKof_i (x) } = \qKof_i \parens[\big]{ \psi (x) } = \qKof_i (x')$,
and since $\qKoe_i (x), \qKof_i (x) \in Q_0 \sqcup \{\undf\}$,
we get that
$\qKoe_i (x'), \qKof_i (x') \in \psi(Q_0) \sqcup \{\undf\}$.
Therefore, $\psi(Q_0)$ is a connected component of $\qcrstQ'$.

Since $\psi^{-1}$ is a quasi-crystal isomorphism between $\qcrstQ'$ and $\qcrstQ$, by the previous implication, if $\psi(Q_0)$ is a connected component of $\qcrstQ'$, then $\psi^{-1} \parens[\big]{ \psi (Q_0) } = Q_0$ is a connected component of $\qcrstQ$.

Finally, let $z \in Q$.
Since $Q(z)$ is a connected component of $\qcrstQ$, then $\psi \parens[\big]{ Q(z) }$ is a connected component of $\qcrstQ'$.
As $\psi(z) \in \psi \parens[\big]{ Q(z) }$, we get that $\psi \parens[\big]{ Q(z) } = Q' \parens[\big]{ \psi (z) }$.
The restriction of $\psi$ to $Q(z)$ is a bijective quasi-crystal homomorphism from $\qcrstQ(z)$ to $\qcrstQ' \parens[\big]{ \psi(z) }$, because $\psi$ is a quasi-crystal homomorphism from $\qcrstQ$ to $\qcrstQ'$.
Similarly, the restriction of $\psi^{-1}$ to $Q' \parens[\big]{ \psi(z) }$ is a bijective quasi-crystal homomorphism from $\qcrstQ' \parens[\big]{ \psi(z) }$ to $\qcrstQ(z)$.
And therefore, the restriction of $\psi$ to $Q(z)$ is a quasi-crystal isomorphism between $\qcrstQ(z)$ and $\qcrstQ' \parens[\big]{ \psi(z) }$.
\end{proof}

For a quasi-crystal $\qcrstQ$, it is immediate that the quasi-Kashiwara operators $\qKoe_i$ and $\qKof_i$ ($i \in I$) are completely determined by the quasi-crystal graph $\Gamma_{\qcrstQ}$,
because given $x, y \in Q$ with $x \neq y$, we have that $x \lbedge{i} y$ is an edge of $\Gamma_\qcrstQ$ if and only if $\qKof_i(x) = y$ and $\qKoe_i(y) = x$.
Now, we show that if $\qcrstQ$ is seminormal, then also the maps $\qKoec_i$ and $\qKofc_i$ ($i \in I$) are completely determined by $\Gamma_{\qcrstQ}$.

\begin{lem}
\label{lem:qcgwu}
Let $\qcrstQ$ be a quasi-crystal, and let $i \in I$.
Given an $i$-labelled walk
\[ x_0 \lbedge{i} x_1 \lbedge{i} \cdots \lbedge{i} x_m \]
on $\Gamma_{\qcrstQ}$, then either
\begin{enumerate}
\item $x_0 = x_1 = \cdots = x_m$; or

\item $x_k = \qKof_i^k (x_0)$, for $k=0,1,\ldots,m$, and thus, $x_0, x_1, \ldots, x_m$ form the unique $i$-labelled path on $\Gamma_\qcrstQ$ starting at $x_0$ and ending at $x_m$.
\end{enumerate}
\end{lem}

	\begin{proof}
If $x_0 = x_1$, then $x_0$ has an $i$-labelled loop, and so, $\qKoec_i (x_0) = +\infty$.
Thus, $\qKof_i (x_1) = \qKof_i(x_0) = \undf$, which implies that $x_1 = x_2$. And recursively, we obtain that $x_0 = x_1 = \cdots = x_m$.

Otherwise, we have that $x_0 \neq x_1$, which implies that $\qKof_i(x_0) = x_1$ (or equivalently, $\qKoe_i(x_1) = x_0$), because $x_0 \lbedge{i} x_1$ is an edge of $\Gamma_{\qcrstQ}$.
Since $\qKoe_i$ is defined on $x_1$, then $\qKoec_i(x_1) \neq +\infty$, and so, $x_1$ does not have an $i$-labelled loop.
Hence, $x_1 \neq x_2$, and so, $\qKof_i (x_1) = x_2$, because $x_1 \lbedge{i} x_2$ is an edge of $\Gamma_\qcrstQ$.
Recursively, we get that $\qKof_i (x_k) = x_{k+1}$, for $k = 0, 1, \ldots, m-1$.
By \comboref{Proposition}{prop:qcrlqKo}, we have that $\wt(x_0) > \wt(x_1) > \cdots > \wt(x_m)$, which implies that $x_0, x_1, \ldots, x_m$ are pairwise distinct, and thus, form an $i$-labelled path on $\Gamma_\qcrstQ$.
It is the unique $i$-labelled path on $\Gamma_\qcrstQ$ starting at $x_0$ and ending at $x_m$, because in a quasi-crystal graph every vertex is the start of at most an edge and is the end of at most an edge labelled by $i$.
\end{proof}

\begin{prop}
\label{prop:snqcgqKoc}
Let $\qcrstQ$ be a seminormal quasi-crystal.
For $x \in Q$ and $i \in I$, we have that
\begin{enumerate}
\item\label{prop:snqcgqKofc}
$\qKofc_i (x)$ is the supremum among nonnegative integers $m \in \Z_{\geq 0}$ such that there exists an $i$-labelled walk on $\Gamma_\qcrstQ$ starting at $x$ with length $m$;

\item\label{prop:snqcgqKoec}
$\qKoec_i (x)$ is the supremum among nonnegative integers $m \in \Z_{\geq 0}$ such that there exists an $i$-labelled walk on $\Gamma_\qcrstQ$ ending at $x$ with length $m$.
\end{enumerate}
\end{prop}

\begin{proof}
\itmref{prop:snqcgqKofc} Let $z \in \Z_{\geq 0} \cup \{ {+\infty} \}$ be the supremum among nonnegative integers $m \in \Z_{\geq 0}$ such that there exists an $i$-labelled walk on $\Gamma_\qcrstQ$ starting on $x$ with length $m$.
If $\qKofc_i (x) = +\infty$, then $x$ has an $i$-labelled loop on $\Gamma_\qcrstQ$.
And so, for any $m \in \Z_{\geq 0}$, the sequence $x_0, \ldots, x_m$, where $x_0 = \cdots = x_m = x$, is an $i$-labelled walk starting on $x$ with length $m$.
Hence, $z = +\infty = \qKofc_i(x)$.

Otherwise, we have that
\[ \qKofc_i (x) = \max \set[\big]{ k \in \Z_{\geq 0} \given \qKof_i^k (x) \in Q }, \]
because $\qcrstQ$ is seminormal. Since
\[ x \lbedge{i} \qKof_i (x) \lbedge{i} \cdots \lbedge{i} \qKof_i^{\qKofc_i(x)} (x) \]
is an $i$-labelled path on $\Gamma_\qcrstQ$ starting on $x$, we have that $\qKofc_i(x) \leq z$.
Since $x$ has no $i$-labelled loops, if
\[ x_0 \lbedge{i} x_1 \lbedge{i} \cdots \lbedge{i} x_m \]
is a walk on $\Gamma_\qcrstQ$ such that $x_0 = x$, then $x_k = \qKof_i^k (x)$, for $k=0,1,\ldots,m$, by \comboref{Lemma}{lem:qcgwu}.
And since $\qKof_i^{\qKofc_i(x)+1} (x) = \undf$, we get that $m \leq \qKofc_i (x)$.
Hence, $z \leq \qKofc_i(x)$, and therefore, $\qKofc_i(x) = z$.

\itmref{prop:snqcgqKoec} Analogously to\avoidrefbreak \itmref{prop:snqcgqKofc}, we have that if $\qKoec_i(x) = +\infty$, then the lengths of $i$-labelled walks on $\Gamma_\qcrstQ$ ending on $x$ are unbounded.
And otherwise,
\[ \qKoe_i^{\qKoec_i(x)} (x) \lbedge{i} \qKoe_i^{\qKoec_i(x)-1} (x) \lbedge{i} \cdots \lbedge{i} x \]
is the longest $i$-labelled walk on $\Gamma_\qcrstQ$ ending on $x$.
\end{proof}

We have shown how the maps $\qKoe_i$, $\qKof_i$, $\qKoec_i$ and $\qKofc_i$ ($i \in I$) of a seminormal quasi-crystal can be described by an $I$-labelled directed graph.
And so a seminormal quasi-crystal can be completely described by a $\Lambda$-weighted $I$-labelled directed graph.
From this correspondence between seminormal quasi-crystals and weighted $I$-labelled directed graphs, we can identify a subclass of graphs which leads to a purely combinatorial description of seminormal quasi-crystals.

\begin{rmk}
\label{rmk:qcgotosnqcg}
By translating \comboref{Definitions}{dfn:qc} and\avoidrefbreak \ref{dfn:snqc} for weighted labelled directed graphs we obtain a subclass of graphs, whose elements are call \dtgterm{seminormal quasi-crystal graphs}.
Consider a root system $\Phi$ with weight lattice $\Lambda$ and index set $I$ for the simple roots $(\alpha_i)_{i \in I}$.
A $\Lambda$-weighted $I$-labelled directed graph $\Gamma$ is a \dtgterm{seminormal quasi-crystal graph} if for any vertices $x$ and $y$, and any $i \in I$, the following conditions are satisfied:
\begin{enumerate}
\item $x$ is the start of at most one edge labelled by $i$, and is the end of at most one edge labelled by $i$;

\item any $i$-labelled path of $\Gamma$ is finite;

\item if $x \lbedge{i} y$ is an edge of $\Gamma$ with $x \neq y$, then $\wt(y) = \wt(x) - \alpha_i$;

\item $\qKofc_i (x) = \qKoec_i (x) + \innerp[\big]{\wt (x)}{\alphav_i}$, where
$\qKofc_i (x)$ is the supremum among nonnegative integers $k \in \Z_{\geq 0}$ such that there exists an $i$-labelled walk on $\Gamma$ starting on $x$ with length $k$, and
$\qKoec_i (x)$ is the supremum among nonnegative integers $l \in \Z_{\geq 0}$ such that there exists an $i$-labelled walk on $\Gamma$ ending on $x$ with length $l$.
\end{enumerate}
Note that if $\Gamma$ is a seminormal quasi-crystal with vertex set $Q$, we can define partial maps $\qKoe_i$ and $\qKof_i$ ($i \in I$) on $Q$, by setting $\qKoe_i (y) = x$ and $\qKof_i (x) = y$, whenever $x \lbedge{i} y$ is an edge of $\Gamma$ with $x \neq y$. Then, we get a seminormal quasi-crystal $\qcrstQ$, and $\Gamma$ coincides with $\Gamma_\qcrstQ$.
\end{rmk}

Due to this relation between seminormal quasi-crystals and weighted labelled directed graphs we obtain the following result.

\begin{thm}
\label{thm:snqcisoqcg}
Let $\qcrstQ$ and $\qcrstQ'$ be seminormal quasi-crystals of the same type, and let $\psi : Q \to Q'$.
Then, $\psi$ is a quasi-crystal isomorphism between $\qcrstQ$ and $\qcrstQ'$ if and only if $\psi$ is a graph isomorphism between the weighted labelled directed graphs $\Gamma_\qcrstQ$ and $\Gamma_{\qcrstQ'}$.
\end{thm}

\begin{proof}
If $\psi$ is a quasi-crystal isomorphism between $\qcrstQ$ and $\qcrstQ'$, then $\psi$ and $\psi^{-1}$ are graph homomorphisms, by \comboref{Lemma}{lem:qchqcgh}.
Hence, $\psi$ is a graph isomorphism between $\Gamma_{\qcrstQ}$ and $\Gamma_{\qcrstQ'}$.

Conversely, suppose that $\psi$ is a graph isomorphism between $\Gamma_\qcrstQ$ and $\Gamma_{\qcrstQ'}$.
Let $x \in Q$ and $i \in I$.
By definition, $\psi$ is weight-preserving, that is, $\wt \parens[\big]{ \psi(x) } = \wt(x)$.
If $\qKof_i(x) \in Q$, then $x \lbedge{i} \qKof_i(x)$ is an edge of $\Gamma_\qcrstQ$, which implies that $\psi(x) \lbedge{i} \psi \parens[\big]{ \qKof_i (x) }$ is an edge of $\Gamma_{\qcrstQ'}$.
Since $x \neq \qKof_i(x)$ and $\psi$ is bijective, then $\psi(x) \neq \psi \parens[\big]{ \qKof_i (x) }$, which implies that $\qKof_i \parens[\big]{ \psi(x) } = \psi \parens[\big]{ \qKof_i (x) }$.
Analogously, if $\qKoe_i (x) \in Q$, then $\psi \parens[\big]{ \qKoe_i (x) } = \qKoe_i \parens[\big]{ \psi(x) }$.
Since $\psi$ is a graph isomorphism, we have that $x_0, x_1, \ldots, x_m \in Q$ form an $i$-labelled walk on $\Gamma_\qcrstQ$ if and only if $\psi(x_0), \psi(x_1), \ldots, \psi(x_m)$ form an $i$-labelled walk on $\Gamma_{\qcrstQ'}$.
Thus, by \comboref{Proposition}{prop:snqcgqKoc}, $\qKoec_i \parens[\big]{ \psi(x) } = \qKoec_i (x)$ and $\qKofc_i \parens[\big]{ \psi(x) } = \qKofc_i (x)$.
Therefore, $\psi$ is a bijective quasi-crystal homomorphism from $\qcrstQ$ to $\qcrstQ'$, and by \comboref{Corollary}{cor:snqciso}, $\psi$ is a quasi-crystal isomorphism between $\qcrstQ$ and $\qcrstQ'$.
\end{proof}

In contrast to the relation between graph homomorphisms and crystal homomorphisms of seminormal crystals associated to semisimple root systems,
we cannot replace the word isomorphism by homomorphism in the previous result.
In \comboref{Example}{exa:ghisnotqch}, we have two seminormal quasi-crystals $\qcrstQ$ and $\qcrstQ'$,
and a map $\psi$ which is a graph homomorphism from $\Gamma_{\qcrstQ}$ to $\Gamma_{\qcrstQ'}$,
but not a quasi-crystal homomorphism from $\qcrstQ$ to $\qcrstQ'$.
Thus, the converse of \comboref{Lemma}{lem:qchqcgh} does not hold even for seminormal quasi-crystals.
On the other hand, we can give a stronger version of \comboref{Proposition}{prop:qccciso} in the particular case of seminormal quasi-crystals.

\begin{prop}
\label{prop:snqccctocc}
Let $\qcrstQ$ and $\qcrstQ'$ be seminormal quasi-crystals of the same type, and let $\psi : \qcrstQ \to \qcrstQ'$ be a quasi-crystal homomorphism.
For each $x \in Q$, if $\psi \parens[\big]{ Q (x) } \subseteq Q'$, then the restriction of $\psi$ to $Q (x)$ is a surjective quasi-crystal homomorphism from $\qcrstQ (x)$ to $\qcrstQ' \parens[\big]{ \psi (x) }$.
\end{prop}

\begin{proof}
Let $x \in Q$ be such that $\psi \parens[\big]{ Q(x) } \subseteq Q'$.
For any $y \in Q(x)$ and $i \in I$,
we have that $\qKofc_i (y) = \qKofc_i \parens[\big]{ \psi(y) }$ by \itmcomboref{Definition}{dfn:qch}{dfn:qchwtc},
and as $\qcrstQ$ and $\qcrstQ'$ are seminormal,
$\qKof_i (y) \in Q$ if and only if $\qKof_i \parens[\big]{ \psi(y) } \in Q'$.
And if so, we get that $\qKof_i (y) \in Q(x)$ by \itmcomboref{Definition}{dfn:qccc}{dfn:qcccmax},
and $\psi (y), \psi \parens[\big]{ \qKof_i (y) } \in Q'$ as $\psi \parens[\big]{ Q(x) } \subseteq Q'$,
which implies by \itmcomboref{Definition}{dfn:qch}{dfn:qchqKof} that $\psi \parens[\big]{ \qKof_i (y) } = \qKof_i \parens[\big]{ \psi(y) }$.
Similarly, $\qKoe_i (y) \in Q$ if and only if $\qKoe_i \parens[\big]{ \psi(y) } \in Q'$,
and if so, $\qKoe_i (y) \in Q(x)$ and $\psi \parens[\big]{ \qKoe_i (y) } = \qKoe_i \parens[\big]{ \psi(y) }$.
Then, given $g_1, \ldots, g_m \in \set{ \qKoe_j, \qKof_j \given j \in I }$,
we have that $g_1 \cdots g_m (x)$ is defined if and only if $g_1 \cdots g_m \parens[\big]{ \psi(x) }$ is defined,
in which case $\psi \parens[\big]{ g_1 \cdots g_m (x) } = g_1 \cdots g_m \parens[\big]{ \psi(x) }$.
This implies that $\psi \parens[\big]{ Q(x) } = Q' \parens[\big]{ \psi(x) }$, and thus, the restriction of $\psi$ to $Q(x)$ induces a well-defined surjective map from $Q(x)$ to $Q' \parens[\big]{ \psi(x) }$.
Since $\psi$ is a quasi-crystal homomorphism, we obtain that the restriction of $\psi$ to $Q(x)$ is a surjective quasi-crystal homomorphism from $\qcrstQ(x)$ to $\qcrstQ' \parens[\big]{ \psi(x) }$.
\end{proof}

Due to the connection between quasi-crystals and graphs, an element of a quasi-crystal $\qcrstQ$ is also a vertex of the graph $\Gamma_\qcrstQ$.
This leads to characterizations based on either perspective and justifies terminology as follows.

\begin{dfn}
\label{dfn:qcie}
Let $\qcrstQ$ be a quasi-crystal.
An element $x \in Q$ is said to be \dtgterm{isolated} if $Q (x) = \set{x}$.
\end{dfn}

An isolated element of a quasi-crystal $\qcrstQ$ is an isolated vertex of the quasi-crystal graph $\Gamma_\qcrstQ$.
Thus, an element $x \in Q$ is isolated if and only if $\qKoe_i (x) = \qKof_i (x) = \undf$, for all $i \in I$, or equivalently, $x$ is isolated if and only if $x$ is of highest and lowest weight.
Furthermore, if $\qcrstQ$ is seminormal, then $x$ is isolated if and only if for each $i \in I$, either $\qKoec_i (x) = \qKofc_i (x) = 0$ or $\qKoec_i (x) = \qKofc_i (x) = +\infty$.

\section{Quasi-tensor product of quasi-crystals}
\label{sec:qcqtp}

A definition of tensor product for quasi-crystals can be given in a similar way as it was originally done for crystals (see\avoidcitebreak \cite{Kas90,Kas91,KN94}).
Such a definition would lead to a generalization to quasi-crystals of the construction of a plactic monoid from a crystal as in\avoidcitebreak \cite[Ch.~5]{Gui22}.
Since we are interested in a general construction of the hypoplactic monoid from a quasi-crystal, in this section we introduce a slightly different definition: the quasi-tensor product.
We then study its properties.
Finally, as this notion will be used in the subsequent sections to relate quasi-crystals and monoids, we describe a combinatorial method to compute the quasi-crystal structure of a quasi-tensor product of quasi-crystals, which is analogous to the signature rule for the tensor product of crystals.

\subsection{Definition and results}
\label{subsec:qcqtpdfnres}

In the following theorem we establish the foundations to introduce the notion of quasi-tensor product of quasi-crystals.

\begin{thm}
\label{thm:qcqtp}
Consider a root system $\Phi$ with weight lattice $\Lambda$ and index set $I$ for the simple roots $(\alpha_i)_{i \in I}$.
Let $\qcrstQ$ and $\qcrstQ'$ be seminormal quasi-crystals of type $\Phi$.
Set $Q \dotimes Q'$ to be the Cartesian product $Q \times Q'$ whose ordered pairs are denoted by $x \dotimes x'$ with $x \in Q$ and $x' \in Q'$.
Define a map ${\wt} : Q \dotimes Q' \to \Lambda$ by
\[
\wt (x \dotimes x') = \wt(x) + \wt(x'),
\]
for $x \in Q$ and $x' \in Q'$. And for each $i \in I$, define maps $\qKoe_i, \qKof_i : Q \dotimes Q' \to (Q \dotimes Q') \sqcup \{\undf\}$ and $\qKoec_i, \qKofc_i : Q \dotimes Q' \to \Z \cup \{ {-\infty}, {+\infty} \}$ as follows:
\begin{enumerate}
\item\label{thm:qcqtpif}
if $\qKofc_i (x) > 0$ and $\qKoec_i (x') > 0$, set
\[
\qKoe_i (x \dotimes x') = \qKof_i (x \dotimes x') = \undf
\quad \text{and} \quad
\qKoec_i (x \dotimes x') = \qKofc_i (x \dotimes x') = {+\infty};
\]

\item\label{thm:qcqtpot}
otherwise, set
\begin{align*}
\qKoe_i (x \dotimes x') &=
   \begin{cases}
      \qKoe_i (x) \dotimes x' & \text{if $\qKofc_i (x) \geq \qKoec_i (x')$}\\
      x \dotimes \qKoe_i (x') & \text{if $\qKofc_i (x) < \qKoec_i (x')$,}
   \end{cases}
\displaybreak[0]\\
\qKof_i (x \dotimes x') &=
   \begin{cases}
      \qKof_i (x) \dotimes x' & \text{if $\qKofc_i (x) > \qKoec_i (x')$}\\
      x \dotimes \qKof_i (x') & \text{if $\qKofc_i (x) \leq \qKoec_i (x')$,}
   \end{cases}
\displaybreak[0]\\
\qKoec_i (x \dotimes x') &= \max \set[\big]{ \qKoec_i (x), \qKoec_i (x') - \innerp[\big]{\wt (x)}{\alphav_i} },
\displaybreak[0]\\
\shortintertext{and}
\qKofc_i (x \dotimes x') &= \max \set[\big]{ \qKofc_i (x) + \innerp[\big]{\wt(x')}{\alphav_i}, \qKofc_i (x') },
\end{align*}
where $x \dotimes \undf = \undf \dotimes x' = \undf$;
\end{enumerate}
for $x \in Q$ and $x' \in Q'$.
Then, $Q \dotimes Q'$ together with the maps $\wt$, $\qKoe_i$, $\qKof_i$, $\qKoec_i$ and $\qKofc_i$ ($i \in I$) forms a seminormal quasi-crystal of type $\Phi$.
\end{thm}

\begin{proof}
Let $x \in Q$, $x' \in Q'$ and $i \in I$.
We have that $\qKoec_i (x), \qKofc_i (x), \qKoec_i (x'), \qKofc_i (x')$ are all non-negative as $\qcrstQ$ and $\qcrstQ'$ are seminormal (\comboref{Definition}{dfn:snqc}).
If $\qKofc_i (x) > 0$ and $\qKoec_i (x') > 0$, it is immediate that $x \dotimes x'$ satisfies all conditions of \comboref{Definition}{dfn:qc}, namely, conditions\avoidrefbreak \itmref{dfn:qcwt} and\avoidrefbreak \itmref{dfn:qcpinfty} which are the ones that apply to this case.
So, assume that $\qKofc_i (x) = 0$ or $\qKoec_i (x') = 0$.

We have that
\begin{align*}
\qKofc_i (x \dotimes x') &= \max \set[\big]{ \qKofc_i (x) + \innerp[\big]{\wt(x')}{\alphav_i}, \qKofc_i (x') }\\
&= \max \set[\big]{ \qKoec_i (x) + \innerp[\big]{\wt(x)}{\alphav_i} + \innerp[\big]{\wt(x')}{\alphav_i}, \qKoec_i (x') + \innerp[\big]{\wt(x')}{\alphav_i} }
\displaybreak[0]\\
&= \max \set[\big]{ \qKoec_i (x), \qKoec_i (x') - \innerp[\big]{\wt(x)}{\alphav_i} } + \innerp[\big]{\wt(x)}{\alphav_i} + \innerp[\big]{\wt(x')}{\alphav_i}\\
&= \qKoec_i (x \dotimes x') + \innerp[\big]{\wt(x \dotimes x')}{\alphav_i}.
\end{align*}
Hence, condition\avoidrefbreak \itmref{dfn:qcwt} of \comboref{Definition}{dfn:qc} is satisfied.

If $\qKofc_i (x) = +\infty$ and $\qKoec_i (x') = 0$, , then $\qKoec_i (x) = +\infty$ and $\qKoe_i (x) = \qKof_i (x) = \undf$,
by\avoidrefbreak \itmref{dfn:qcwt} and\avoidrefbreak \itmref{dfn:qcpinfty} of \comboref{Definition}{dfn:qc},
which implies that
$\qKoec_i (x \dotimes x') = \qKofc_i (x \dotimes x') = +\infty$,
$\qKoe_i (x \dotimes x') = \qKoe_i (x) \dotimes x' = \undf$,
and
$\qKof_i (x \dotimes x') = \qKof_i (x) \dotimes x' = \undf$.
Analogously, if $\qKofc_i (x) = 0$ and $\qKoec_i (x') = +\infty$, we have that $\qKoe_i (x \dotimes x') = \qKof_i (x \dotimes x') = \undf$ and $\qKoec_i (x \dotimes x') = \qKofc_i (x \dotimes x') = +\infty$.
So, besides $\qKofc_i (x) = 0$ or $\qKoec_i (x') = 0$, we may further assume that $\qKofc_i (x) \neq +\infty \neq \qKoec_i (x')$.

We get that
$\innerp[\big]{\wt(x)}{\alphav_i} = \qKofc_i (x) - \qKoec_i (x)$
and
$\innerp[\big]{\wt(x')}{\alphav_i} = \qKofc_i (x') - \qKoec_i (x')$
by \itmcomboref{Definition}{dfn:qc}{dfn:qcwt},
implying that
$\qKoec_i (x \dotimes x') = \max \set[\big]{ \qKoec_i (x), \qKoec_i (x') - \qKofc_i (x) + \qKoec_i (x) }$
and
$\qKofc_i (x \dotimes x') = \max \set[\big]{ \qKofc_i (x) + \qKofc_i (x') - \qKoec_i (x'), \qKofc_i (x') }$.
As $\qKoec_i (x), \qKoec_i (x'), \qKofc_i (x), \qKofc_i (x')$ are all non-negative where $\qKofc_i (x) = 0$ or $\qKoec_i (x) = 0$, we obtain that
\begin{equation}
\qKoec_i (x \dotimes x') = \qKoec_i (x) + \qKoec_i (x')
\quad \text{and} \quad
\qKofc_i (x \dotimes x') = \qKofc_i (x) + \qKofc_i (x').
\label{eq:qcqtpc}
\end{equation}
Now we consider the following cases.
\begin{itemize}
\item Case 1: $\qKofc_i (x) = \qKoec_i (x') = 0$.
We have that
$\qKoe_i (x \dotimes x') = \qKoe_i (x) \dotimes x'$
and
$\qKof_i (x \dotimes x') = x \dotimes \qKof_i (x')$.
Thus, $\qKoe_i (x \dotimes x')$ is defined if and only if $\qKoe_i (x) \in Q$.
If so, then we have that
\[\begin{split}
\wt \parens[\big]{ \qKoe_i (x \dotimes x') } &= \wt \parens[\big]{ \qKoe_i (x) } + \wt (x') = \wt (x) + \alpha_i + \wt (x')\\
&= \wt (x \dotimes x') + \alpha_i,
\end{split}\]
and since $\qKoe_i (x) \dotimes x'$ satisfies the conditions leading to\avoidrefbreak \eqref{eq:qcqtpc}, we deduce that
\[\begin{split}
\qKoec_i \parens[\big]{ \qKoe_i (x \dotimes x') } &= \qKoec_i \parens[\big]{ \qKoe_i (x) } + \qKoec_i (x') = \qKoec_i (x) - 1 + \qKoec_i (x')\\
&= \qKoec_i (x \dotimes x') - 1
\end{split}\]
and
\[\begin{split}
\qKofc_i \parens[\big]{ \qKoe_i (x \dotimes x') } &= \qKofc_i \parens[\big]{ \qKoe_i (x) } + \qKofc_i (x') = \qKofc_i (x) + 1 + \qKofc_i (x')\\
&= \qKofc_i (x \dotimes x') + 1.
\end{split}\]
Also, as $\qKofc_i \parens[\big]{ \qKoe_i (x) } = \qKofc_i (x) + 1 = 1$ and $\qKoec_i (x') = 0$, we get that
\[ \qKof_i \parens[\big]{ \qKoe_i (x \dotimes x') } = \qKof_i \parens[\big]{ \qKoe_i (x) \dotimes x' } = \qKof_i \parens[\big]{ \qKoe_i (x) } \dotimes x' = x \dotimes x'. \]
On the other hand, $\qKof_i (x \dotimes x')$ is defined if and only if $\qKof_i (x') \in Q'$.
If so, we have that $\qKofc_i (x) = 0$ and $\qKoec_i \parens[\big]{ \qKof_i (x') } = \qKoec_i (x') + 1 = 1$, which implies that
\[ \qKoe_i \parens[\big]{ \qKof_i (x \dotimes x') } = \qKoe_i \parens[\big]{ x \dotimes \qKof_i (x') } = x \dotimes \qKoe_i \parens[\big]{ \qKof_i (x') } = x \dotimes x'. \]

\item Case 2: $\qKofc_i (x) > 0$ and $\qKoec_i (x') = 0$.
We have that
$\qKoe_i (x \dotimes x') = \qKoe_i (x) \dotimes x'$
and
$\qKof_i (x \dotimes x') = \qKof_i (x) \dotimes x'$.
Thus, $\qKoe_i (x \dotimes x')$ is defined if and only if $\qKoe_i (x) \in Q$,
and if so, the facts that condition\avoidrefbreak \itmref{dfn:qcqKof} of \comboref{Definition}{dfn:qc} holds
and $\qKof_i \parens[\big]{ \qKoe_i (x \dotimes x') } = x \dotimes x'$ follow as in case 1.
Since $\qcrstQ$ is seminormal and $\qKofc_i (x) > 0$, we get that $\qKof_i (x) \in Q$, which implies that $\qKof_i (x \dotimes x')$ is defined.
As $\qKofc_i \parens[\big]{ \qKof_i (x) } = \qKofc_i (x) - 1 \geq 0$ and $\qKoec_i (x') = 0$, we obtain that
\[ \qKoe_i \parens[\big]{ \qKof_i (x \dotimes x') } = \qKoe_i \parens[\big]{ \qKof_i (x) \dotimes x' } = \qKoe_i \parens[\big]{ \qKof_i (x) } \dotimes x' = x \dotimes x'. \]

\item Case 3: $\qKofc_i (x) = 0$ and $\qKoec_i (x') > 0$.
We have that
$\qKoe_i (x \dotimes x') = x \dotimes \qKoe_i (x')$
and
$\qKof_i (x \dotimes x') = x \dotimes \qKof_i (x')$.
Since $\qcrstQ'$ is seminormal and $\qKoec_i (x') > 0$, then $\qKoe_i (x') \in Q'$, which implies that $\qKoe_i (x \dotimes x')$ is defined.
We have that
\[ \wt \parens[\big]{ \qKoe_i (x \dotimes x') } = \wt (x) + \wt (x') + \alpha_i = \wt (x \dotimes x') + \alpha_i, \]
and since $\qKoe_i (x) \dotimes x'$ satisfies the conditions leading to\avoidrefbreak \eqref{eq:qcqtpc}, we get that
\[ \qKoec_i \parens[\big]{ \qKoe_i (x \dotimes x') } = \qKoec_i (x) + \qKoec_i (x') - 1 = \qKoec_i (x \dotimes x') - 1 \]
and
\[ \qKofc_i \parens[\big]{ \qKoe_i (x \dotimes x') } = \qKofc_i (x) + \qKofc_i (x') + 1 = \qKofc_i (x \dotimes x') + 1. \]
Also, as $\qKofc_i (x) = 0$ and $\qKoec_i \parens[\big]{ \qKoe_i (x') } = \qKoec_i (x') - 1 \geq 0$, we obtain that
\[ \qKof_i \parens[\big]{ \qKoe_i (x \dotimes x') } = \qKof_i \parens[\big]{ x \dotimes \qKoe_i (x') } = x \dotimes \qKof_i \parens[\big]{ \qKoe_i (x') } = x \dotimes x'. \]
The fact that $\qKoe_i \parens[\big]{ \qKof_i (x \dotimes x') } = x \dotimes x'$, whenever $\qKof_i (x \dotimes x')$ is defined, follows as in case 1.
\end{itemize}
In each case we showed that conditions\avoidrefbreak \itmref{dfn:qcqKoe} and\avoidrefbreak \itmref{dfn:qciff} of \comboref{Definition}{dfn:qc} are satisfied.
We also showed that if $x \dotimes x'$ lies in one of these cases, so does $\qKof_i (x \dotimes x')$ when defined.
Thus, by \comboref{Proposition}{prop:dfnqcsimp}, condition\avoidrefbreak \itmref{dfn:qcqKof} of \comboref{Definition}{dfn:qc} also holds.

Therefore, $Q \dotimes Q'$ together with $\wt$, $\qKoe_i$, $\qKof_i$, $\qKoec_i$ and $\qKofc_i$ ($i \in I$) forms a quasi-crystal.
It remains to prove that this quasi-crystal is seminormal (\comboref{Definition}{dfn:snqc}).

Assume that $\qKoec_i (x \dotimes x') \neq +\infty$.
We have that $\qKoec_i (x), \qKofc_i (x), \qKoec_i (x'), \qKofc_i (x') \in \Z_{\geq 0}$ where $\qKofc_i (x) = 0$ or $\qKoec_i (x') = 0$, and so, we have one of the three cases above.
Since $\qcrstQ$ and $\qcrstQ'$ are seminormal, then
$\qKoe_i^{\qKoec_i (x)} (x) \in Q$, $\qKoe_i^{\qKoec_i (x) + 1} (x) = \undf$, $\qKoe_i^{\qKoec_i (x')} (x') \in Q'$,
and $\qKoec_i \parens[\big]{\qKoe_i^{\qKoec_i (x')} (x')} = 0 \leq \qKofc_i \parens[\big]{\qKoe_i^{\qKoec_i (x)} (x)}$.
By\avoidrefbreak \eqref{eq:qcqtpc} and cases 1 and 3 above, we get that
\[\begin{split}
\qKoe_i^{\qKoec_i (x \dotimes x')} (x \dotimes x') &= \qKoe_i^{\qKofc_i (x) + \qKoec_i (x')} (x \dotimes x') = \qKoe_i^{\qKoec_i (x)} \parens[\big]{x \dotimes \qKoe_i^{\qKoec_i (x')} (x')}\\
&= \qKoe_i^{\qKoec_i (x)} (x) \dotimes \qKoe_i^{\qKoec_i (x')} (x')
\end{split}\]
is defined, and
\[ \qKoe_i^{\qKoec_i (x \dotimes x') + 1} (x \dotimes x') = \qKoe_i \parens[\big]{\qKoe_i^{\qKoec_i (x)} (x) \dotimes \qKoe_i^{\qKoec_i (x')} (x')} = \qKoe_i^{\qKoec_i (x) + 1} (x) \dotimes \qKoe_i^{\qKoec_i (x')} (x') = \undf. \]
Similarly, we have that $\qKof_i^{\qKofc_i (x \dotimes x')} (x \dotimes x') = \qKof_i^{\qKofc_i (x)} (x) \dotimes \qKof_i^{\qKofc_i (x')} (x')$ is defined, and $\qKof_i^{\qKofc_i (x \dotimes x') + 1} (x \dotimes x') = \qKof_i^{\qKofc_i (x)} (x) \dotimes \qKof_i^{\qKofc_i (x') + 1} (x') = \undf$.
Hence, $Q \dotimes Q'$ together with the maps $\wt$, $\qKoe_i$, $\qKof_i$, $\qKoec_i$ and $\qKofc_i$ ($i \in I$) is a seminormal quasi-crystal.
\end{proof}

Note that the quasi-crystal structure on $Q \dotimes Q'$ given in\avoidrefbreak \itmref{thm:qcqtpot} of \comboref{Theorem}{thm:qcqtp} is similar to the original definition of the crystal structure for the tensor product of crystals\avoidcitebreak \cite{Kas90,Kas91,KN94}.
Thus, if we omitted\avoidrefbreak \itmref{thm:qcqtpif} and applied\avoidrefbreak \itmref{thm:qcqtpot} to all elements, we would have obtained a generalization to quasi-crystals of the tensor product of crystals, as remarked in the beginning of this section.
Note also that if we apply the maps $\qKoe_i$, $\qKof_i$, $\qKoec_i$ and $\qKofc_i$ as defined in\avoidrefbreak \itmref{thm:qcqtpot} to elements of the form $x \dotimes x'$ with $\qKofc_i (x) = +\infty$ or $\qKoec_i (x') = +\infty$,
then we get the same images as in\avoidrefbreak \itmref{thm:qcqtpif}.
Since $\qcrstQ$ and $\qcrstQ'$ are seminormal, we have that $\qKofc_i (x), \qKoec_i (x') \in \Z_{> 0}$ if and only if $\qKof_i (x) \in Q$ and $\qKoe_i (x') \in Q'$.
Hence, condition\avoidrefbreak \itmref{thm:qcqtpif} of \comboref{Theorem}{thm:qcqtp} is specifying values for the crystal structure on elements of the form $x \dotimes x'$, where $\qKof_i (x)$ and $\qKoe_i (x')$ are defined, different from what they would be if the definitions in\avoidrefbreak \itmref{thm:qcqtpot} would apply to them.
This is a quasi-crystal interpretation of the notion of an $i$-inversion in a word, introduced in\avoidcitebreak \cite[\S~5]{CM17crysthypo}, which justifies the following terminology.

\begin{dfn}
\label{dfn:qcqtp}
Let $\qcrstQ$ and $\qcrstQ'$ be seminormal quasi-crystals of the same type.
The \dtgterm{inverse-free quasi-tensor product of $\qcrstQ$ and $\qcrstQ'$}, or simply the \dtgterm{quasi-tensor product of $\qcrstQ$ and $\qcrstQ'$}, is the seminormal quasi-crystal defined in \comboref{Theorem}{thm:qcqtp} and is denoted by $\qcrstQ \dotimes \qcrstQ'$.
\end{dfn}

We chose to give definitions of the maps of the quasi-crystal structure of a quasi-tensor product in \comboref{Theorem}{thm:qcqtp} to emphasize their resemblance with the maps of the crystal structure of a tensor product of crystals, although in the proof we deduced alternative definitions.
The following result is an immediate consequence of the arguments that led to\avoidrefbreak \eqref{eq:qcqtpc} and the cases that followed it.

\begin{prop}
\label{prop:qcqtpalt}
Let $\qcrstQ$ and $\qcrstQ'$ be seminormal quasi-crystals of the same type.
For $x \in Q$, $x' \in Q'$ and $i \in I$ with $\qKofc_i (x) = 0$ or $\qKoec_i (x') = 0$, we have that
\begin{align*}
\qKoe_i (x \dotimes x') &=
   \begin{cases}
      \qKoe_i (x) \dotimes x' & \text{if $\qKoec_i (x') = 0$}\\
      x \dotimes \qKoe_i (x') & \text{if $\qKoec_i (x') > 0$,}
   \end{cases}
&
\qKof_i (x \dotimes x') &=
   \begin{cases}
      \qKof_i (x) \dotimes x' & \text{if $\qKofc_i (x) > 0$}\\
      x \dotimes \qKof_i (x') & \text{if $\qKofc_i (x) = 0$,}
   \end{cases}\\
\qKoec_i (x \dotimes x') &= \qKoec_i (x) + \qKoec_i (x'),
&
\qKofc_i (x \dotimes x') &= \qKofc_i (x) + \qKofc_i (x').
\end{align*}
\end{prop}

\begin{exa}
\label{exa:qcqtp}
\exaitem\label{exa:qcqtpA32}
The quasi-crystal $\qctA_3^2$, described in \itmcomboref{Example}{exa:qc}{exa:qcA32} is isomorphic to $\qctA_3 \dotimes \qctA_3$ as the map $A_3 \times A_3 \to A_3 \dotimes A_3$, given by $(x, y) \mapsto x \dotimes y$ for each $x, y \in A_3$, is a quasi-crystal isomorphism.
Thus, by \comboref{Theorem}{thm:snqcisoqcg} the quasi-crystal graph $\Gamma_{\qctA_3 \dotimes \qctA_3}$ is isomorphic to the quasi-crystal graph $\Gamma_{\qctA_3^2}$, which is drawn in \itmcomboref{Example}{exa:qcg}{exa:qcgA32}.

\exaitem\label{exa:qcqtpC22}
The quasi-crystal graph $\Gamma_{\qctC_2 \dotimes \qctC_2}$ of the quasi-tensor product $\qctC_2 \dotimes \qctC_2$ (see \itmcomboref{Example}{exa:qc}{exa:qctCn}) is the following.
\[
\begin{tikzpicture}[widecrystal,baseline=(b1b1.base)]
  %
  \node (11)  at (1, 4) { \dotimes 11};
  \node (21)  at (2, 4) {2 \dotimes 1};
  \node (b21) at (3, 4) {\wbar{2} \dotimes 1};
  \node (b11) at (4, 4) {\wbar{1} \dotimes 1};
  \node (12)  at (1, 3) {1 \dotimes 2};
  \node (22)  at (2, 3) {2 \dotimes 2};
  \node (b22) at (3, 3) {\wbar{2} \dotimes 2};
  \node (b12) at (4, 3) {\wbar{1} \dotimes 2};
  \node (1b2)  at (1, 2) {1 \dotimes \wbar{2}};
  \node (2b2)  at (2, 2) {2 \dotimes \wbar{2}};
  \node (b2b2) at (3, 2) {\wbar{2} \dotimes \wbar{2}};
  \node (b1b2) at (4, 2) {\wbar{1} \dotimes \wbar{2}};
  \node (1b1)  at (1, 1) {1 \dotimes \wbar{1}};
  \node (2b1)  at (2, 1) {2 \dotimes \wbar{1}};
  \node (b2b1) at (3, 1) {\wbar{2} \dotimes \wbar{1}};
  \node (b1b1) at (4, 1) {\wbar{1} \dotimes \wbar{1}};
  \path (11) edge node {1} (21)
        (21) edge node {1} (22)
        (21) edge node {2} (b21)
        (b21) edge node {1} (b11)
        (b11) edge node {1} (b12)
        (12) edge node {2} (1b2)
        (22) edge node {2} (b22)
        (b22) edge node {2} (b2b2)
        (b12) edge node {2} (b1b2)
        (1b2) edge node {1} (2b2)
        (2b2) edge node {1} (2b1)
        (b2b2) edge node {1} (b1b2)
        (b1b2) edge node {1} (b1b1)
        (2b1) edge node {2} (b2b1)
        (12) edge [loop left] node {1} ()
        (b22) edge [loop right] node {1} ()
        (2b2) edge [loop right] node {2} ()
        (1b1) edge [loop left] node {1} ()
        (b2b1) edge [loop right] node {1} ();
\end{tikzpicture}
\]
where $\wt (x \dotimes y) = \wt(x) + \wt(y)$, for $x, y \in C_2$.
\end{exa}

From the previous example, we can see that the quasi-tensor product of seminormal crystals may not be a crystal.
Indeed, if $\crstB$ and $\crstB'$ are seminormal crystals with elements $x \in B$ and $x' \in B'$ such that $\qKof_i (x) \in B$ and $\qKoe_i (x') \in B'$, for some $i \in I$, then $\qKoec_i (x \dotimes x') = +\infty$.
Hence, apart from trivial cases, the quasi-tensor product of seminormal crystals is not a crystal.

We only defined quasi-tensor product between quasi-crystals that are seminormal, although if we did not require $\qcrstQ$ and $\qcrstQ'$ in \comboref{Theorem}{thm:qcqtp} to be seminormal, the resulting structure would still be a quasi-crystal, eventually not seminormal too.
The following example shows that this condition is essential to model inversions in words by quasi-crystals, that is, we need both $\qKoe_i$ and $\qKof_i$ to be undefined on $x \dotimes x'$, whenever $\qKof_i (x)$ and $\qKoe_i (x')$ are defined.

\begin{exa}
\label{rmk:qcqtpsn}
Consider the standard quasi-crystal $\qctA_2$ of type $\tA_2$, described in \itmcomboref{Example}{exa:qc}{exa:qctAn}.
Let $\qcrstQ$ be a quasi-crystal of type $\tA_2$ consisting of a set $Q = \set{ {-1}, {-2} }$ and maps given as follows:
\begin{center}
  \begin{tabular}{c|c|c|c|c|c}
    $x$ & $\wt(x)$   & $\qKoe(x)$ & $\qKof(x)$ & $\qKoec(x)$ & $\qKofc(x)$ \\ \hline
    $-1$ & $-\vc{e_{1}}$ & $-2$    & $\undf$    & $0$         & $-1$         \\
    $-2$ & $-\vc{e_{2}}$ & $\undf$    & $-1$    & $-1$         & $0$
  \end{tabular}.
\end{center}
Clearly, $\qctA_2$ is seminormal, but $\qcrstQ$ is not.
Nonetheless, set a quasi-crystal structure on $A_2 \dotimes Q$ as defined in \comboref{Theorem}{thm:qcqtp}.
Then, $\qKof \parens[\big]{1 \dotimes (-1)} = 2 \dotimes (-1)$, which implies that $\qKof$ is defined on an element of the form $x \dotimes y$ where $\qKof (x)$ and $\qKoe (y)$ are defined.
Alternatively, let $\qKoe', \qKof' : A_2 \dotimes Q \to (A_2 \dotimes Q) \sqcup \set{\undf}$ be defined as follows.
\begin{center}
  \begin{tabular}{c|c|c|c|c|c}
    $x \dotimes y$    & $\qKoe' (x \dotimes y)$ & $\qKof' (x \dotimes y)$ \\ \hline
    $1 \dotimes (-1)$ & $\undf$                 & $\undf$                 \\
    $2 \dotimes (-1)$ & $1 \dotimes (-1)$       & $\undf$                 \\
    $1 \dotimes (-2)$ & $\undf$                 & $2 \dotimes (-2)$       \\
    $2 \dotimes (-2)$ & $1 \dotimes (-2)$       & $\undf$
  \end{tabular}
\end{center}
So, $\qKoe'$ and $\qKof'$ are undefined on $x \dotimes y$ whenever $\qKof (x) \in A_2$ and $\qKoe (y) \in Q$, otherwise $\qKoe'$ and $\qKof'$ follow the rule in \itmcomboref{Theorem}{thm:qcqtp}{thm:qcqtpot}.
However, there is no quasi-crystal of type $\tA_2$ whose quasi-Kashiwara operators are $\qKoe'$ and $\qKof'$, because \itmcomboref{Definition}{dfn:qc}{dfn:qciff} is not satisfied, as $\qKoe' \parens[\big]{2 \dotimes (-1)} = 1 \dotimes (-1)$ and $\qKof' \parens[\big]{1 \dotimes (-1)} = \undf$.
This illustrates that requiring quasi-crystals to be seminormal is essential to give an interpretation of an inversion on a word by the quasi-tensor product.
\end{exa}

In the following result we show that the quasi-tensor product $\dotimes$ of quasi-crystals is an associative operation.

\begin{thm}
\label{thm:qcqtpassoc}
Let $\qcrstQ_1$, $\qcrstQ_2$ and $\qcrstQ_3$ be seminormal quasi-crystals of the same type.
The map $(Q_1 \dotimes Q_2) \dotimes Q_3 \to Q_1 \dotimes (Q_2 \dotimes Q_3)$, given by $(x_1 \dotimes x_2) \dotimes x_3 \mapsto x_1 \dotimes (x_2 \dotimes x_3)$, is a quasi-crystal isomorphism between $(\qcrstQ_1 \dotimes \qcrstQ_2) \dotimes \qcrstQ_3$ and $\qcrstQ_1 \dotimes (\qcrstQ_2 \dotimes \qcrstQ_3)$.
\end{thm}

\begin{proof}
Define $\psi : (Q_1 \dotimes Q_2) \dotimes Q_3 \to Q_1 \dotimes (Q_2 \dotimes Q_3)$ by $\psi ((x_1 \dotimes x_2) \dotimes x_3) = x_1 \dotimes (x_2 \dotimes x_3)$ for $x_1 \in Q_1$, $x_2 \in Q_2$ and $x_3 \in Q_3$.
It is immediate that $\psi$ is bijective.
Since $(\qcrstQ_1 \dotimes \qcrstQ_2) \dotimes \qcrstQ_3$ and $\qcrstQ_1 \dotimes (\qcrstQ_2 \dotimes \qcrstQ_3)$ are seminormal by \comboref{Theorem}{thm:qcqtp}, to prove that $\psi$ is a quasi-crystal isomorphism, it suffices to show that $\psi$ is a quasi-crystal homomorphism by \comboref{Corollary}{cor:snqciso}.
Let $x_1 \in Q_1$, $x_2 \in Q_2$ and $x_3 \in Q_3$.
Then,
\[ \wt \parens[\big]{(x_1 \dotimes x_2) \dotimes x_3} = \wt(x_1) + \wt(x_2) + \wt(x_3) = \wt \parens[\big]{x_1 \dotimes (x_2 \dotimes x_3)}. \]
Let $i \in I$. For $k = 1, 2, 3$, we have that $\qKoec_i (x_k), \qKofc_i (x_k) \geq 0$ as $\qcrstQ_k$ is seminormal.
If $\qKoec_i (x_k) = +\infty$, for some $k$, then
$\qKoec_i \parens[\big]{(x_1 \dotimes x_2) \dotimes x_3} = \qKoec_i \parens[\big]{x_1 \dotimes (x_2 \dotimes x_3)} = \qKoec_i (x_k) = +\infty$,
which implies by\avoidrefbreak \itmref{dfn:qcwt} and\avoidrefbreak \itmref{dfn:qcpinfty} of \comboref{Definition}{dfn:qc} that
$\qKofc_i \parens[\big]{(x_1 \dotimes x_2) \dotimes x_3} = \qKofc_i \parens[\big]{x_1 \dotimes (x_2 \dotimes x_3)} = +\infty$,
$\qKoe_i \parens[\big]{(x_1 \dotimes x_2) \dotimes x_3} = \qKoe_i \parens[\big]{x_1 \dotimes (x_2 \dotimes x_3)} = \undf$, and
$\qKof_i \parens[\big]{(x_1 \dotimes x_2) \dotimes x_3} = \qKof_i \parens[\big]{x_1 \dotimes (x_2 \dotimes x_3)} = \undf$.
So assume that $\qKoec_i (x_k) \in \Z_{\geq 0}$ for all $k = 1, 2, 3$.

In the following cases we show that if $\qKofc_i (x_k), \qKoec_i (x_l) > 0$ for some $1 \leq k < l \leq 3$, then $\qKoec_i \parens[\big]{(x_1 \dotimes x_2) \dotimes x_3} = \qKoec_i \parens[\big]{x_1 \dotimes (x_2 \dotimes x_3)} = +\infty$, and thus, it follows as above.
\begin{itemize}
\item Case 1: $\qKofc_i (x_1), \qKoec_i (x_2) > 0$.
Then $\qKoec_i (x_1 \dotimes x_2) = +\infty$ and $\qKoec_i (x_2 \dotimes x_3) \geq \qKoec_i (x_2) > 0$, which imply that
$\qKoec_i \parens[\big]{(x_1 \dotimes x_2) \dotimes x_3} = \qKoec_i \parens[\big]{x_1 \dotimes (x_2 \dotimes x_3)} = +\infty$.

\item Case 2: $\qKofc_i (x_1), \qKoec_i (x_3) > 0$.
Then $\qKofc_i (x_1 \dotimes x_2) \geq \qKofc_i (x_1) > 0$ and $\qKoec_i (x_2 \dotimes x_3) \geq \qKoec_i (x_3) > 0$, which imply that
$\qKoec_i \parens[\big]{(x_1 \dotimes x_2) \dotimes x_3} = \qKoec_i \parens[\big]{x_1 \dotimes (x_2 \dotimes x_3)} = +\infty$.

\item Case 3: $\qKofc_i (x_2), \qKoec_i (x_3) > 0$.
Then $\qKofc_i (x_1 \dotimes x_2) \geq \qKofc_i (x_2) > 0$ and $\qKoec_i (x_2 \dotimes x_3) = +\infty$, which imply that
$\qKoec_i \parens[\big]{(x_1 \dotimes x_2) \dotimes x_3} = \qKoec_i \parens[\big]{x_1 \dotimes (x_2 \dotimes x_3)} = +\infty$.
\end{itemize}
So, we further assume that $\qKoec_i (x_k), \qKofc_i (x_l) > 0$ implies $k \leq l$.

By \comboref{Proposition}{prop:qcqtpalt}, we get that
\[ \qKoec_i \parens[\big]{(x_1 \dotimes x_2) \dotimes x_3} = \qKoec_i (x_1) + \qKoec_i (x_2) + \qKoec_i (x_3) = \qKoec_i \parens[\big]{x_1 \dotimes (x_2 \dotimes x_3)} \]
and
\[ \qKofc_i \parens[\big]{(x_1 \dotimes x_2) \dotimes x_3} = \qKofc_i (x_1) + \qKofc_i (x_2) + \qKofc_i (x_3) = \qKofc_i \parens[\big]{x_1 \dotimes (x_2 \dotimes x_3)}. \]
We also have that
\[
\qKoe_i \parens[\big]{(x_1 \dotimes x_2) \dotimes x_3} =
  \begin{cases}
    \parens[\big]{\qKoe_i (x_1) \dotimes x_2} \dotimes x_3 & \text{if $\qKoec_i (x_2) = \qKoec_i (x_3) = 0$}\\
    \parens[\big]{x_1 \dotimes \qKoe_i (x_2)} \dotimes x_3 & \text{if $\qKoec_i (x_2) > 0 = \qKoec_i (x_3)$}\\
    (x_1 \dotimes x_2) \dotimes \qKoe_i (x_3) & \text{if $\qKoec_i (x_3) > 0$}
  \end{cases}
\]
and
\[
\qKoe_i \parens[\big]{x_1 \dotimes (x_2 \dotimes x_3)} =
  \begin{cases}
    \qKoe_i (x_1) \dotimes (x_2 \dotimes x_3) & \text{if $\qKoec_i (x_2) = \qKoec_i (x_3) = 0$}\\
    x_1 \dotimes \parens[\big]{\qKoe_i (x_2) \dotimes x_3} & \text{if $\qKoec_i (x_2) > 0 = \qKoec_i (x_3)$}\\
    x_1 \dotimes \parens[\big]{x_2 \dotimes \qKoe_i (x_3)} & \text{if $\qKoec_i (x_3) > 0$,}
  \end{cases}
\]
implying that $\psi \parens[\big]{\qKoe_i ((x_1 \dotimes x_2) \dotimes x_3)} = \qKoe_i \parens[\big]{\psi (x_1 \dotimes (x_2 \dotimes x_3))}$.
Similarly,
\[
\qKof_i \parens[\big]{(x_1 \dotimes x_2) \dotimes x_3} =
  \begin{cases}
    \parens[\big]{\qKof_i (x_1) \dotimes x_2} \dotimes x_3 & \text{if $\qKofc_i (x_1) > 0$}\\
    \parens[\big]{x_1 \dotimes \qKof_i (x_2)} \dotimes x_3 & \text{if $\qKofc_i (x_1) = 0 < \qKofc_i (x_2)$}\\
    (x_1 \dotimes x_2) \dotimes \qKof_i (x_3) & \text{if $\qKofc_i (x_1) = \qKofc_i (x_2) = 0$}
  \end{cases}
\]
and
\[
\qKof_i \parens[\big]{x_1 \dotimes (x_2 \dotimes x_3)} =
  \begin{cases}
    \qKof_i (x_1) \dotimes (x_2 \dotimes x_3) & \text{if $\qKofc_i (x_1) > 0$}\\
    x_1 \dotimes \parens[\big]{\qKof_i (x_2) \dotimes x_3} & \text{if $\qKofc_i (x_1) = 0 < \qKofc_i (x_2)$}\\
    x_1 \dotimes \parens[\big]{x_2 \dotimes \qKof_i (x_3)} & \text{if $\qKofc_i (x_1) = \qKofc_i (x_2) = 0$},
  \end{cases}
\]
implying that $\psi \parens[\big]{\qKof_i ((x_1 \dotimes x_2) \dotimes x_3)} = \qKof_i \parens[\big]{\psi (x_1 \dotimes (x_2 \dotimes x_3))}$.

Therefore, $\psi$ is a quasi-crystal isomorphism.
\end{proof}

Due to the previous result, we may omit parenthesis for the quasi-tensor product of seminormal quasi-crystals and simply write $\qcrstQ_1 \dotimes \qcrstQ_2 \dotimes \qcrstQ_3$, whose elements are denoted by $x_1 \dotimes x_2 \dotimes x_3$, for $x_1 \in Q_1$, $x_2 \in Q_2$ and $x_3 \in Q_3$.
From the proofs of \comboref{Theorems}{thm:qcqtp} and\avoidrefbreak \ref{thm:qcqtpassoc}, we deduce the following result, which generalizes \comboref{Proposition}{prop:qcqtpalt} and describes the quasi-crystal structure of a quasi-tensor product of an arbitrary number of seminormal quasi-crystals.

\begin{cor}
\label{cor:qcqtpdesc}
Let $\qcrstQ_1, \ldots, \qcrstQ_m$ be seminormal quasi-crystals of the same type, and let $x_1 \in Q_1, \ldots, x_m \in Q_m$.
Then,
\[ \wt (x_1 \dotimes \cdots \dotimes x_m) = \wt (x_1) + \cdots + \wt (x_m). \]
Also, for $i \in I$, by setting
\[ p = \max \set[\big]{ 1 \leq k \leq m \given \qKoec_i (x_k) > 0 } \]
and
\[ q = \min \set[\big]{ 1 \leq l \leq m \given \qKofc_i (x_l) > 0 }, \]
we have that
\begin{enumerate}
\item\label{cor:qcqtpdescif}
if $p > q$ or $\qKoec_i (x_k) = +\infty$ for some $1 \leq k \leq m$, then
\[ \qKoe_i (x_1 \dotimes \cdots \dotimes x_m) = \qKof_i (x_1 \dotimes \cdots \dotimes x_m) = \undf \]
and
\[ \qKoec_i (x_1 \dotimes \cdots \dotimes x_m) = \qKofc_i (x_1 \dotimes \cdots \dotimes x_m) = +\infty; \]

\item\label{cor:qcqtpdescot}
otherwise,
\begin{gather*}
\qKoe_i (x_1 \dotimes \cdots \dotimes x_m) = x_1 \dotimes \cdots \dotimes x_{p-1} \dotimes \qKoe_i (x_p) \dotimes x_{p+1} \dotimes \cdots \dotimes x_m
\displaybreak[0]\\
\qKof_i (x_1 \dotimes \cdots \dotimes x_m) = x_1 \dotimes \cdots \dotimes x_{q-1} \dotimes \qKof_i (x_q) \dotimes x_{q+1} \dotimes \cdots \dotimes x_m
\displaybreak[0]\\
\qKoec_i (x_1 \dotimes \cdots \dotimes x_m) = \qKoec_i (x_1) + \cdots + \qKoec_i (x_p)
\displaybreak[0]\\
\shortintertext{and}
\qKofc_i (x_1 \dotimes \cdots \dotimes x_m) = \qKofc_i (x_q) + \cdots + \qKofc_i (x_m)
\end{gather*}
\end{enumerate}
\end{cor}

In the following result we show that quasi-crystal homomorphisms between seminormal quasi-crystals give rise to homomorphisms between quasi-tensor products of their domains and images.

\begin{thm}
\label{thm:qcqtph}
Let $\qcrstQ_1$, $\qcrstQ_2$, $\qcrstQ'_1$ and $\qcrstQ'_2$ be seminormal quasi-crystals of the same type, and let $\psi_1 : \qcrstQ_1 \to \qcrstQ'_1$ and $\psi_2 : \qcrstQ_2 \to \qcrstQ'_2$ be quasi-crystal homomorphisms.
The partial map $\psi_1 \dotimes \psi_2 : Q_1 \dotimes Q_2 \to Q'_1 \dotimes Q'_2$, given by $x_1 \dotimes x_2 \mapsto \psi_1 (x_1) \dotimes \psi_2 (x_2)$ for each $x_1 \in Q_1$ and $x_2 \in Q_2$ such that $\psi_1 (x_1) \in Q'_1$ and $\psi_2 (x_2) \in Q'_2$, is a quasi-crystal homomorphism from $\qcrstQ_1 \dotimes \qcrstQ_2$ to $\qcrstQ'_1 \dotimes \qcrstQ'_2$.
Moreover, if $\psi_1$ and $\psi_2$ are quasi-crystal isomorphisms, then $\psi_1 \dotimes \psi_2$ is a quasi-crystal isomorphism.
\end{thm}

\begin{proof}
Let $x_1 \in Q_1$ and $x_2 \in Q_2$ be such that $\psi_1 (x_1) \in Q'_1$ and $\psi_2 (x_2) \in Q'_2$.
We get that
\[ \wt \parens[\big]{\psi_1 (x_1) \dotimes \psi_2 (x_2)} = \wt \parens[\big]{\psi_1 (x_1)} + \wt \parens[\big]{\psi_2 (x_2)} = \wt (x_1) + \wt (x_2) = \wt (x_1 \dotimes x_2). \]
Let $i \in I$ and $k \in \{1, 2\}$.
By \itmcomboref{Definition}{dfn:qch}{dfn:qchwtc}, We have that $\qKoec_i \parens[\big]{\psi_k (x_k)} = \qKoec_i (x_k)$ and $\qKofc_i \parens[\big]{\psi_k (x_k)} = \qKofc_i (x_k)$.
Thus, if $\qKofc_i (x_1) > 0$ and $\qKoec_i (x_2) > 0$, then
\[ \qKoe_i \parens[\big]{\psi_1 (x_1) \dotimes \psi_2 (x_2)} = \qKof_i \parens[\big]{\psi_1 (x_1) \dotimes \psi_2 (x_2)} = \qKoe_i (x_1 \dotimes x_2) = \qKof_i (x_1 \dotimes x_2) = \undf \]
and
\[ \qKoec_i \parens[\big]{\psi_1 (x_1) \dotimes \psi_2 (x_2)} = \qKofc_i \parens[\big]{\psi_1 (x_1) \dotimes \psi_2 (x_2)} = \qKoec_i (x_1 \dotimes x_2) = \qKofc_i (x_1 \dotimes x_2) = +\infty. \]
Otherwise, we get that
\[\begin{split}
\qKoec_i \parens[\big]{\psi_1 (x_1) \dotimes \psi_2 (x_2)} &= \max \set[\big]{ \qKoec_i \parens[\big]{\psi_1 (x_1)}, \qKoec_i \parens[\big]{\psi_2 (x_2)} - \innerp[\big]{\wt \parens[\big]{\psi_1 (x_1)}}{\alphav_i} }\\
&= \max \set[\big]{ \qKoec_i (x_1), \qKoec_i (x_2) - \innerp[\big]{\wt (x_1)}{\alphav_i} }
= \qKoec_i (x_1 \dotimes x_2).
\end{split}\]
Analogously, $\qKofc_i \parens[\big]{\psi_1 (x_1) \dotimes \psi_2 (x_2)} = \qKofc_i (x_1 \dotimes x_2)$.
By \itmcomboref{Definition}{dfn:qch}{dfn:qchqKoe}, if $\qKoe_i (x_k) \in Q_k$ and $\psi_k \parens[\big]{\qKoe_i (x_k)} \in Q'_k$, then $\psi_k \parens[\big]{\qKoe_i (x_k)} = \qKoe_i \parens[\big]{\psi_k (x_k)}$.
Thus, if $\qKoe_i (x_1 \dotimes x_2) \in Q_1 \dotimes Q_2$ and $(\psi_1 \dotimes \psi_2) \parens[\big]{\qKoe_i (x_1 \dotimes x_2)} \in Q'_1 \dotimes Q'_2$, then
\begin{align*}
(\psi_1 \dotimes \psi_2) \parens[\big]{\qKoe_i (x_1 \dotimes x_2)} &=
  \begin{cases}
    \psi_1 \parens[\big]{\qKoe_i (x_1)} \dotimes \psi_2 (x_2) & \text{if $\qKofc_i (x_1) \geq \qKoec_i (x_2)$}\\
    \psi_1 (x_1) \dotimes \psi_2 \parens[\big]{\qKoe_i (x_2)} & \text{if $\qKofc_i (x_1) < \qKoec_i (x_2)$}
  \end{cases}
\displaybreak[0]\\
&=
  \begin{cases}
    \qKoe_i \parens[\big]{\psi_1 (x_1)} \dotimes \psi_2 (x_2) & \text{if $\qKofc_i \parens[\big]{\psi_1 (x_1)} \geq \qKoec_i \parens[\big]{\psi_2 (x_2)}$}\\
    \psi_1 (x_1) \dotimes \qKoe_i \parens[\big]{\psi_2 (x_2)} & \text{if $\qKofc_i \parens[\big]{\psi_1 (x_1)} < \qKoec_i \parens[\big]{\psi_2 (x_2)}$}
  \end{cases}
\\
&= \qKoe_i \parens[\big]{\psi_1 (x_1) \dotimes \psi_2 (x_2)}.
\end{align*}
Similarly, if $\qKof_i (x_1 \dotimes x_2) \in Q_1 \dotimes Q_2$ and $(\psi_1 \dotimes \psi_2) \parens[\big]{\qKof_i (x_1 \dotimes x_2)} \in Q'_1 \dotimes Q'_2$, then $(\psi_1 \dotimes \psi_2) \parens[\big]{\qKof_i (x_1 \dotimes x_2)} = \qKof_i \parens[\big]{\psi_1 (x_1) \dotimes \psi_2 (x_2)}$.
Therefore, $\psi_1 \dotimes \psi_2$ is a quasi-crystal homomorphism from $\qcrstQ_1 \dotimes \qcrstQ_2$ to $\qcrstQ'_1 \dotimes \qcrstQ'_2$.

If $\psi_1$ and $\psi_2$ are quasi-crystal isomorphisms, then $\psi_1 \dotimes \psi_2$ is a bijective quasi-crystal homomorphism between seminormal quasi-crystals, as proved above.
Hence, by \comboref{Corollary}{cor:snqciso}, $\psi_1 \dotimes \psi_2$ is a quasi-crystal isomorphism.
\end{proof}

\subsection{The signature rule}
\label{subsec:qcqtpsr}

We now describe a practical method to compute the quasicrystal structure of the quasi-tensor product of seminormal quasi-crystals.
This method is essentially a combinatorial interpretation of \comboref{Corollary}{cor:qcqtpdesc}, and has a procedure similar to the signature rule for the tensor product of seminormal crystals\avoidcitebreak \cite{HK02}.

Let $Z_0$ be the monoid with zero defined by the following presentation
$\pres{ {-}, {+} \given ({+} {-}, 0) }$.
So, an element of $Z_0$, other than $0$, has the form ${-^a} {+^b}$ with $a, b \in \Z_{\geq 0}$.

Let $\qcrstQ$ be a seminormal quasi-crystal.
For each $i \in I$ define $\qtpsgn_i : Q \to Z_0$ by
\[
  \qtpsgn_i (x) =
  \begin{cases}
    0 & \text{if $\qKoec_i (x) = +\infty$}\\
    {-^{\qKoec_i (x)}} {+^{\qKofc_i (x)}} & \text{otherwise,}
  \end{cases}
\]
for each $x \in Q$.
The map $\qtpsgn_i$ is called the \dtgterm{$i$-signature map} for the quasi-tensor product $\dotimes$, and $\qtpsgn_i (x)$ is called the \dtgterm{$i$-signature} of $x \in Q$.

In comparison with the signature map for the tensor product of crystals, we have that the bicyclic monoid $\pres{ {-} {+} \given ({+} {-}, \ew) }$ (where $\ew$ denotes the empty word) has been replaced by the monoid $Z_0$.
This allows $\qtpsgn_i$ to interact with the quasi-tensor product of seminormal quasi-crystals in the following way.

\begin{prop}
\label{prop:qcqtpsr}
Let $\qcrstQ$ and $\qcrstQ'$ be seminormal quasi-crystals of the same type.
Then,
\[ \qtpsgn_i (x \dotimes x') = \qtpsgn_i (x) \qtpsgn_i (x'), \]
for all $x \in Q$, $x' \in Q'$ and $i \in I$.
\end{prop}

\begin{proof}
Let $x \in Q$, $x' \in Q'$ and $i \in I$.
By \comboref{Corollary}{cor:qcqtpdesc}, if $\qKoec_i (x) = +\infty$ (or equivalently, $\qKofc_i (x) = +\infty$), then $\qtpsgn_i (x) = 0$ and $\qKoec_i (x \dotimes x') = +\infty$, which implies that $\qtpsgn_i (x \dotimes x') = 0 = \qtpsgn_i (x) \qtpsgn_i (x')$.
Similarly, if $\qKoec_i (x') = +\infty$, we have that $\qtpsgn_i (x \dotimes x') = 0 = \qtpsgn_i (x) \qtpsgn_i (x')$.
Thus, assume that $\qKofc_i (x), \qKoec_i (x') \in \Z_{\geq 0}$.
If $\qKofc_i (x), \qKoec_i (x') > 0$, then $\qKoec_i (x \dotimes x') = +\infty$, and
\[
\qtpsgn_i (x) \qtpsgn_i (x')
= {-^{\qKoec_i (x)}} {+^{\qKofc_i (x) - 1}} {+} {-} {-^{\qKoec_i (x') - 1}} {+^{\qKofc_i (x')}}
= 0
= \qtpsgn_i (x \dotimes x').
\]
Finally, assume that $\qKofc_i (x) = 0$ or $\qKoec_i (x') = 0$.
If $\qKofc_i (x) = 0$, then
\[
\qtpsgn_i (x) \qtpsgn_i (x')
= {-^{\qKoec_i (x) + \qKoec_i (x')}} {+^{\qKofc_i (x')}}
= {-^{\qKoec_i (x \dotimes x')}} {+^{\qKofc_i (x \dotimes x')}}
= \qtpsgn_i (x \dotimes x'),
\]
by \comboref{Proposition}{prop:qcqtpalt}.
If $\qKoec_i (x') = 0$, then
\[
\qtpsgn_i (x) \qtpsgn_i (x')
= {-^{\qKoec_i (x)}} {+^{\qKofc_i (x) + \qKofc_i (x')}}
= {-^{\qKoec_i (x \dotimes x')}} {+^{\qKofc_i (x \dotimes x')}}
= \qtpsgn_i (x \dotimes x'),
\]
by \comboref{Proposition}{prop:qcqtpalt}.
\end{proof}

From the previous result, given seminormal quasi-crystals $\qcrstQ_1, \qcrstQ_2, \ldots, \qcrstQ_m$ of the same type and elements $x_1 \in Q_1, x_2 \in Q_2, \ldots, x_m \in Q_m$, we can easily compute the $i$-signature of $x_1 \dotimes x_2 \dotimes \cdots \dotimes x_m$ as
\[ \qtpsgn_i (x_1 \dotimes x_2 \dotimes \cdots \dotimes x_m) = \qtpsgn_i (x_1) \qtpsgn_i (x_2) \cdots \qtpsgn_i (x_m) \]
($i \in I$).
If $\qtpsgn_i (x_1 \dotimes x_2 \dotimes \cdots \dotimes x_m) = 0$, then $\qKoec_i (x_1 \dotimes x_2 \dotimes \cdots \dotimes x_m) = \qKofc_i (x_1 \dotimes x_2 \dotimes \cdots \dotimes x_m) = +\infty$ which implies $\qKoe_i (x_1 \dotimes x_2 \dotimes \cdots \dotimes x_m) = \qKof_i (x_1 \dotimes x_2 \dotimes \cdots \dotimes x_m) = \undf$.
Otherwise,
$\qtpsgn_i (x_1 \dotimes x_2 \dotimes \cdots \dotimes x_m) = {-^{a}} {+^{b}}$,
for some $a, b \in \Z_{\geq 0}$.
Then, $\qKoec_i (x_1 \dotimes x_2 \dotimes \cdots \dotimes x_m) = a$ and $\qKofc_i (x_1 \dotimes x_2 \dotimes \cdots \dotimes x_m) = b$.
From \comboref{Corollary}{cor:qcqtpdesc}, if $a \geq 1$, then
$\qKoe_i (x_1 \dotimes x_2 \dotimes \cdots \dotimes x_m) = x_1 \dotimes \cdots \dotimes x_{p-1} \dotimes \qKoe_i (x_p) \dotimes x_{p+1} \dotimes \cdots \dotimes x_m$,
where $x_p$ originates the right-most symbol $-$ in $\qtpsgn_i (x_1 \dotimes x_2 \dotimes \cdots \dotimes x_m)$.
Also, if $b \geq 1$, then
$\qKof_i (x_1 \dotimes x_2 \dotimes \cdots \dotimes x_m) = x_1 \dotimes \cdots \dotimes x_{q-1} \dotimes \qKof_i (x_q) \dotimes x_{q+1} \dotimes \cdots \dotimes x_m$,
where $x_q$ originates the left-most symbol $+$ in $\qtpsgn_i (x_1 \dotimes x_2 \dotimes \cdots \dotimes x_m)$.
This process is called the \dtgterm{signature rule for the quasi-tensor product}.

\begin{exa}
\label{exa:qcqtpsr}
Consider the quasi-crystal $\qctA_4 \dotimes \qctA_4 \dotimes \qctA_4 \dotimes \qctA_4 \dotimes \qctA_4$, where $\qctA_4$ is the standard quasi-crystal of type $\tA_4$.
We compute $\qKoe_2$, $\qKof_2$, $\qKoec_2$ and $\qKofc_2$ on $3 \dotimes 1 \dotimes 2 \dotimes 2 \dotimes 3$ using the signature rule.
To keep track to which element originates each $-$ and $+$ we write a subscript with the position of the element, this is just an auxiliary notation and the binary operation of $Z_0$ should be applied ignoring the subscripts.
So we have that
\[ \qtpsgn_2 (3 \dotimes 1 \dotimes 2 \dotimes 2 \dotimes 3) = {-_1} {+_3} {+_4} {-_5} = {-_1} {+_3} 0 = 0, \]
and therefore,
\begin{align*}
\qKoe_2 (3 \dotimes 1 \dotimes 2 \dotimes 2 \dotimes 3) &= \undf,
& \qKoec_2 (3 \dotimes 1 \dotimes 2 \dotimes 2 \dotimes 3) &= +\infty,\\
\qKof_2 (3 \dotimes 1 \dotimes 2 \dotimes 2 \dotimes 3) &= \undf,
& \qKofc_2 (3 \dotimes 1 \dotimes 2 \dotimes 2 \dotimes 3) &= +\infty.
\end{align*}
Now we compute $\qKoe_1$, $\qKof_1$, $\qKoec_1$ and $\qKofc_1$ on $2 \dotimes 3 \dotimes 2 \dotimes 3 \dotimes 1$.
Using the same notation as above, we obtain that
\[ \qtpsgn_1 (2 \dotimes 3 \dotimes 2 \dotimes 3 \dotimes 1) = {-_1} {-_3} {+_5}, \]
and therefore,
\begin{align*}
\qKoe_1 (2 \dotimes 3 \dotimes 2 \dotimes 3 \dotimes 1) &= 2 \dotimes 3 \dotimes 1 \dotimes 3 \dotimes 1,
& \qKoec_1 (2 \dotimes 3 \dotimes 2 \dotimes 3 \dotimes 1) &= 2,\\
\qKof_1 (2 \dotimes 3 \dotimes 2 \dotimes 3 \dotimes 1) &= 2 \dotimes 3 \dotimes 2 \dotimes 3 \dotimes 2,
& \qKofc_1 (2 \dotimes 3 \dotimes 2 \dotimes 3 \dotimes 1) &= 1.
\end{align*}
\end{exa}

\section{Quasi-crystal monoids}
\label{sec:qcm}

In this section we study the algebraic framework relating quasi-crystals and monoids, which will be used to give a general definition of hypoplactic monoid.
In \comboref{Subsection}{subsec:qcmh}, we present the definition of quasi-crystal monoid, which is the basic concept for relating quasi-crystals and monoids.
Then in \comboref{Subsection}{subsec:fqcm}, we introduce the definition of free quasi-crystal monoid over a seminormal quasi-crystal, and show that free quasi-crystal monoids satisfy a universal property that defines them up to isomorphism.
Finally, in \comboref{Subsection}{subsec:qcmcq}, we present the notion of congruences on a quasi-crystal monoid, which form a lattice and allow to consider quotients of quasi-crystal monoids, leading to the homomorphism theorems for quasi-crystal monoids.

\subsection{Quasi-crystal monoids and homomorphisms}
\label{subsec:qcmh}


We first introduce the fundamental concept relating quasi-crystals and monoids with respect to the quasi-tensor product $\dotimes$, studied in \comboref{Section}{sec:qcqtp}.

\begin{dfn}
\label{dfn:qcm}
Let $\Phi$ be a root system with weight lattice $\Lambda$ and index set $I$ for the simple roots $(\alpha_i)_{i \in I}$.
A \dtgterm{$\dotimes$-quasi-crystal monoid} $\qcrstM$ of type $\Phi$ consists of a set $M$ together with maps ${\wt} : M \to \Lambda$, $\qKoe_i, \qKof_i : M \to M \sqcup \{\undf\}$, $\qKoec_i, \qKofc_i : M \to \Z \cup \{ {-\infty}, {+\infty} \}$ ($i \in I$) and a binary operation ${\cdot} : M \times M \to M$ satisfying the following conditions:
\begin{enumerate}
\item\label{dfn:qcmqc}
$M$ together with $\wt$, $\qKoe_i$, $\qKof_i$, $\qKoec_i$ and $\qKofc_i$ ($i \in I$) forms a seminormal quasi-crystal;

\item\label{dfn:qcmmon}
$M$ together with $\cdot$ forms a monoid;

\item\label{dfn:qcmhom}
the map $M \dotimes M \to M$, given by $x \dotimes y \mapsto x \cdot y$ for $x, y \in M$, induces a quasi-crystal homomorphism from $\qcrstM \dotimes \qcrstM$ to $\qcrstM$.
\end{enumerate}
\end{dfn}

We stated a definition of quasi-crystal monoid with respect to the quasi-tensor product $\dotimes$, because we shall see that it models the binary operation of the hypoplactic monoid, which we want to generalize.
A similar definition can be given by replacing the quasi-tensor product by other operation on quasi-crystals.
For instance, if we considered the tensor product instead, the subsequent would lead to a notion of plactic monoid over a quasi-crystal.
Since our goal is to introduce the notion of hypoplactic monoid associated to a quasi-crystal, we will only consider quasi-crystal monoids with respect to the quasi-tensor product $\dotimes$, and thus, we will omit $\dotimes$ and just say that $\qcrstM$ is a quasi-crystal monoid.

In a quasi-crystal monoid the interaction between the quasi-crystal structure and the binary operation satisfies rules similar to those satisfied by the quasi-crystal structure of a quasi-tensor product (see \comboref{Theorem}{thm:qcqtp} and \comboref{Proposition}{prop:qcqtpalt}), as shown in the following result.

\begin{lem}
\label{lem:qcmqtpbo}
Let $\qcrstM$ be a quasi-crystal monoid. For $x, y \in M$, we have that
\[ \wt (xy) = \wt(x) + \wt(y), \]
and for $i \in I$, if $\qKofc_i (x) > 0$ and $\qKoec_i (y) > 0$, then
\[
\qKoe_i (xy) = \qKof_i (xy) = \undf
\quad \text{and} \quad
\qKoec_i (xy) = \qKofc_i (xy) = {+\infty},
\]
otherwise,
\begin{align*}
\qKoe_i (xy) &=
   \begin{cases}
      \qKoe_i (x) \cdot y & \text{if $\qKoec_i (y) = 0$}\\
      x \cdot \qKoe_i (y) & \text{if $\qKoec_i (y) > 0$,}
   \end{cases}
&&&
\qKof_i (xy) &=
   \begin{cases}
      \qKof_i (x) \cdot y & \text{if $\qKofc_i (x) > 0$}\\
      x \cdot \qKof_i (y) & \text{if $\qKofc_i (x) = 0$,}
   \end{cases}\\
\qKoec_i (xy) &= \qKoec_i (x) + \qKoec_i (y),
&& \text{and} &
\qKofc_i (xy) &= \qKofc_i (x) + \qKofc_i (y),
\end{align*}
where $x \undf = \undf y = \undf$.
\end{lem}

\begin{proof}
By \itmcomboref{Definition}{dfn:qcm}{dfn:qcmhom}, let $\psi : \qcrstM \dotimes \qcrstM \to \qcrstM$ be the quasi-crystal homomorphism given by $\psi (x \dotimes y) = xy$, for $x, y \in M$.
Let $x, y \in M$ and $i \in I$.
By \itmcomboref{Definition}{dfn:qch}{dfn:qchwtc}, we get that
\[ \wt (xy) = \wt \parens[\big]{\psi (x \dotimes y)} = \wt (x \dotimes y) = \wt (x) + \wt (y). \]
and similarly, $\qKoec_i (xy) = \qKoec_i (x \dotimes y)$ and $\qKofc_i (xy) = \qKofc_i (x \dotimes y)$.
By \itmcomboref{Definition}{dfn:qcm}{dfn:qcmqc}, $\qcrstM$ is seminormal (\comboref{Definition}{dfn:snqc}), and by \comboref{Theorem}{thm:qcqtp}, $\qcrstM \dotimes \qcrstM$ is also seminormal.
Since $\qKoec_i (xy) = \qKoec_i (x \dotimes y)$, we have that $\qKoe_i$ is defined on $xy$ if and only if $\qKoe_i$ is defined on $x \dotimes y$, and since $\qKofc_i (xy) = \qKofc_i (x \dotimes y)$, $\qKof_i$ is defined on $xy$ if and only if $\qKof_i$ is defined on $x \dotimes y$.
Then, as $\psi (M \dotimes M) \subseteq M$, we obtain that $\qKoe_i (xy) = \psi \parens[\big]{ \qKoe_i (x \dotimes y) }$ and $\qKof_i (xy) = \psi \parens[\big]{ \qKof_i (x \dotimes y) }$, by conditions\avoidrefbreak \itmref{dfn:qchqKoe} and\avoidrefbreak \itmref{dfn:qchqKof} of \comboref{Definition}{dfn:qch}.
Therefore, the result follows directly from \comboref{Theorem}{thm:qcqtp} and \comboref{Proposition}{prop:qcqtpalt}.
\end{proof}

The previous result can be generalized to get the values of the quasi-crystal structure on an element of the form $x_1 \cdots x_m$ based only on their values on each $x_k$, $k=1,\ldots,m$.
This leads to an analogue of \comboref{Corollary}{cor:qcqtpdesc}.

\begin{prop}
\label{prop:qcmdesc}
Let $\qcrstM$ be a quasi-crystal monoid, and let $x_1, \ldots, x_m \in M$.
Then,
\[ \wt (x_1 \cdots x_m) = \wt (x_1) + \cdots + \wt (x_m). \]
Also, for $i \in I$, by setting
\[ p = \max \set[\big]{ 1 \leq k \leq m \given \qKoec_i (x_k) > 0 } \]
and
\[ q = \min \set[\big]{ 1 \leq l \leq m \given \qKofc_i (x_l) > 0 }, \]
we have that
\begin{enumerate}
\item\label{prop:qcmdescif}
if $p > q$ or $\qKoec_i (x_k) = +\infty$ for some $1 \leq k \leq m$, then
\[
\qKoe_i (x_1 \cdots x_m) = \qKof_i (x_1 \cdots x_m) = \undf
\quad \text{and} \quad
\qKoec_i (x_1 \cdots x_m) = \qKofc_i (x_1 \cdots x_m) = +\infty;
\]

\item\label{prop:qcmdescot}
otherwise,
\begin{gather*}
\qKoe_i (x_1 \cdots x_m) = x_1 \cdots x_{p-1} \cdot \qKoe_i (x_p) \cdot x_{p+1} \cdots x_m,
\displaybreak[0]\\
\qKof_i (x_1 \cdots x_m) = x_1 \cdots x_{q-1} \cdot \qKof_i (x_q) \cdot x_{q+1} \cdots x_m,
\displaybreak[0]\\
\qKoec_i (x_1 \cdots x_m) = \qKoec_i (x_1) + \cdots + \qKoec_i (x_p),
\displaybreak[0]\\
\shortintertext{and}
\qKofc_i (x_1 \cdots x_m) = \qKofc_i (x_q) + \cdots + \qKofc_i (x_m).
\end{gather*}
\end{enumerate}
\end{prop}

\begin{proof}
We proceed by induction on $m$.
If $m=1$, then the result is trivial, and if $m=2$, then it coincides with \comboref{Lemma}{lem:qcmqtpbo}.
Assume as induction hypothesis (IH) that the result holds for any $x_1, \ldots, x_k \in M$ with $k \leq m$.
Let $y_1, \ldots, y_m, y_{m+1} \in M$.
Since $\cdot$ is associative, we have that $y_1 \cdots y_m y_{m+1} = (y_1 \cdots y_m)y_{m+1}$, where
\[ \wt \parens[\big]{(y_1 \cdots y_m) y_{m+1}} = \wt (y_1 \cdots y_m) + \wt (y_{m+1}) = \wt(y_1) + \cdots + \wt(y_m) + \wt(y_{m+1}), \]
by \comboref{Lemma}{lem:qcmqtpbo} and (IH).
Let $i \in I$.
If $\qKoec_i (y_k) = +\infty$, for some $k \in \set{1, \ldots, m+1}$, then we have when $k \leq m$ that $\qKoec_i (y_1 \cdots y_m) = +\infty$, by (IH), which implies that
\[ \qKoec_i (y_1 \cdots y_m y_{m+1}) = \qKoec_i ((y_1 \cdots y_m) \dotimes y_{m+1}) = +\infty, \]
and by conditions\avoidrefbreak \itmref{dfn:qcwt} and\avoidrefbreak \itmref{dfn:qcpinfty} of \comboref{Definition}{dfn:qc}, $\qKoe_i (y_1 \cdots y_{m+1}) = \qKof_i (y_1 \cdots y_{m+1}) = \undf$ and $\qKofc_i (y_1 \cdots y_{m+1}) = +\infty$.
So, assume that $\qKoec_i (y_k) \in \Z_{\geq 0}$, for $k=1,\ldots,m+1$.

Set
\[ p = \max \set[\big]{ 1 \leq k \leq m+1 \given \qKoec_i (x_k) > 0 } \]
and
\[ q = \min \set[\big]{ 1 \leq l \leq m+1 \given \qKofc_i (x_l) > 0 }. \]
Suppose that $p > q$.
In particular, $p > 1$ and $q < m+1$ implying that the sets where the maximum and minimum are taken are nonempty, and thus, $\qKoec_i (y_p), \qKofc_i (y_q) > 0$.
By (IH),
$\qKofc_i (y_1 \cdots y_{p-1}) \geq \qKofc_i (y_q) > 0$,
and
$\qKoec_i (y_p \cdots y_{m+1}) \geq \qKoec_i (y_p) > 0$.
This implies by \comboref{Lemma}{lem:qcmqtpbo} that
\[ \qKoec_i (y_1 \cdots y_{m+1}) = \qKoec_i \parens[\big]{(y_1 \cdots y_{p-1}) (y_p \cdots y_{m+1})} = +\infty, \]
and by conditions\avoidrefbreak \itmref{dfn:qcwt} and\avoidrefbreak \itmref{dfn:qcpinfty} of \comboref{Definition}{dfn:qc}, $\qKoe_i (y_1 \cdots y_{m+1}) = \qKof_i (y_1 \cdots y_{m+1}) = \undf$ and $\qKofc_i (y_1 \cdots y_{m+1}) = +\infty$.

Finally, suppose that $p \leq q$.
As $\qKofc_i (y_k) = 0$, for $k = 1, \ldots, q-1$, we have by (IH) that $\qKoec_i (y_1 \cdots y_{p-1}) = \qKoec_i (y_1) + \cdots + \qKoec_i (y_{p-1})$ and $\qKofc_i (y_1 \cdots y_{p-1}) = 0$.
By \comboref{Lemma}{lem:qcmqtpbo}, we get that
$\qKoe_i \parens[\big]{(y_1 \cdots y_{p-1}) y_p} = y_1 \cdots y_{p-1} \cdot \qKoe_i (y_p)$
(note that if $\qKoec_i (y_p) = 0$, then $p=1$ and $\qKoe_i (y_p) = \undf$, as $\qcrstM$ is seminormal),
and by (IH),
$\qKoec_i \parens[\big]{(y_1 \cdots y_{p-1}) y_p} = \qKoec_i (y_1 \cdots y_{p-1}) + \qKoec_i (y_p) = \qKoec_i (y_1) + \cdots + \qKoec_i (y_p)$.
Also, since $\qKoec_i (y_k) = 0$, for $k=p+1,\ldots,m+1$, we have by (IH) that $\qKoec_i (y_{p+1} \cdots y_{m+1}) = 0$.
Then, we obtain that
\[\begin{split}
\qKoe_i \parens[\big]{(y_1 \cdots y_p) (y_{p+1} \cdots y_{m+1})} &= \qKoe_i (y_1 \cdots y_p) \cdot y_{p+1} \cdots y_{m+1}\\
&= y_1 \cdots y_{p-1} \cdot \qKoe_i (y_p) \cdot y_{p+1} \cdots y_{m+1}
\end{split}\]
and
\[\begin{split}
\qKoec_i \parens[\big]{(y_1 \cdots y_p) (y_{p+1} \cdots y_{m+1})} &= \qKoec_i (y_1 \cdots y_{p}) + \qKoec_i (y_{p+1} \cdots y_{m+1})\\
&= \qKoec_i (y_1) + \cdots + \qKoec_i (y_p),
\end{split}\]
by \comboref{Lemma}{lem:qcmqtpbo}.
Analogously, we have that
\[
\qKof_i \parens[\big]{(y_1 \cdots y_{q-1}) (y_q \cdots y_{m+1})}
= y_1 \cdots y_{q-1} \cdot \qKof_i (y_q) \cdot y_{q+1} \cdots y_{m+1}
\]
and
\[
\qKofc_i \parens[\big]{(y_1 \cdots y_{q-1}) (y_q \cdots y_{m+1})}
= \qKofc_i (y_q) + \cdots + \qKofc_i (y_{m+1}),
\]
by (IH) and \comboref{Lemma}{lem:qcmqtpbo}.
\end{proof}

In the previous result we saw how the monoid binary operation $\cdot$ interacts with the quasi-crystal structure.
We now show that this allows us to relate some properties of elements.
First, we recall that an element $x$ of a monoid $M$ is called a \dtgterm{commutative element}, also known as \dtgterm{central element}, if $x$ commutes with every element, that is, $xy = yx$, for any $y \in M$.
We also recall that $x$ is called an \dtgterm{idempotent element} if $x^2 = x$.

\begin{prop}
\label{prop:qcmelemprop}
Let $\qcrstM$ be a quasi-crystal monoid and let $x \in M$.
\begin{enumerate}
\item\label{prop:qcmelempropcom}
If $x$ is a commutative element, then $x$ is isolated.

\item\label{prop:qcmelempropidem}
If $x$ is an idempotent element, then $x$ is isolated and $\wt(x) = 0$.
\end{enumerate}
\end{prop}

\begin{proof}
\itmref{prop:qcmelempropcom} Suppose that $x$ is not an isolated element of $\qcrstM$.
Then, take $i \in I$ such that $\qKoe_i$ or $\qKof_i$ is defined on $w$.
As $\qcrstM$ is seminormal, if $\qKoe_i$ is defined on $x$, then $\qKoec_i (x), \qKofc_i (x) \in \Z_{\geq 0}$, where $\qKoec_i (x) > 0$, and the element $y = \qKoe_i^{\qKoec_i (x)} (x)$ satisfies $\qKoec_i (y) = 0$ and $\qKofc_i (y) \in \Z_{>0}$.
By \comboref{Lemma}{lem:qcmqtpbo}, $\qKoec_i (xy) = \qKoec_i (x) \in \Z_{> 0}$ and $\qKoec_i (yx) = +\infty$, which implies that $xy \neq yx$, and thus, $x$ is not commutative.
Otherwise, $\qKof_i$ is defined on $x$, and since $\qcrstM$ is seminormal, we have that $\qKoec_i (x), \qKofc_i (x) \in \Z_{\geq 0}$ where $\qKofc_i (x) > 0$.
The element $z = \qKof_i^{\qKofc_i (x)} (x)$ is such that $\qKoec_i (z) \in \Z_{> 0}$ and $\qKofc_i (z) = 0$.
Then, by \comboref{Lemma}{lem:qcmqtpbo}, $\qKofc_i (xz) = +\infty$ and $\qKofc_i (zx) = \qKofc_i (x) \in \Z_{> 0}$, which implies that $xz \neq zx$, and therefore, $x$ is not commutative.

\itmref{prop:qcmelempropidem} Assume that $x$ is idempotent.
By \comboref{Lemma}{lem:qcmqtpbo}, we have that
\[ \wt (x) = \wt \parens[\big]{ x^2 } = \wt (x) + \wt (x), \]
which implies that $\wt (x) = 0$.
If $\qKoec_i (x) \neq +\infty$ (or equivalently, $\qKofc_i (x) \neq +\infty$), for some $i \in I$, then
\[ \qKoec_i (x) = \qKoec_i \parens[\big]{ x^2 } = \qKoec_i (x) + \qKoec_i (x) \]
and
\[ \qKofc_i (x) = \qKofc_i \parens[\big]{ x^2 } = \qKofc_i (x) + \qKofc_i (x), \]
impliying that $\qKoec_i (x) = \qKofc_i (x) = 0$.
Hence, $\qKoec_i (x), \qKofc_i (x) \in \set{ 0, {+\infty} }$, for any $i \in I$.
As $\qcrstM$ is seminormal, we obtain that $x$ is isolated.
\end{proof}

The monoid identity is in particular both a commutative and an idempotent element, but as we show in the following result, its properties may affect the whole quasi-crystal structure.

\begin{prop}
\label{prop:qcmid}
Let $\qcrstM$ be a quasi-crystal monoid where the monoid identity is denoted by $1$.
Then, $1$ is isolated and $\wt (1) = 0$.
Moreover, for each $i \in I$, either $\qKoec_i (1) = \qKofc_i (1) = 0$ or $\qKoec_i (x) = +\infty$, for all $x \in M$.
\end{prop}

\begin{proof}
Since $1$ is an idempotent element, then $1$ is isolated and $\wt(1) = 0$, by \itmcomboref{Proposition}{prop:qcmelemprop}{prop:qcmelempropidem}.
As $\qcrstM$ is seminormal and $1$ is isolated, we get for each $i \in I$ that either $\qKoec_i (1) = \qKofc_i (1) = 0$ or $\qKoec_i (1) = \qKofc_i (1) = +\infty$.
Suppose there exists $x \in M$ such that $\qKoec_i (x) \neq +\infty$, for some $i \in I$.
Since $\qcrstM$ is seminormal, we get that $\qKoec_i \parens[\big]{\qKoe_i^{\qKoec_i (x)} (x)} = \qKofc_i \parens[\big]{\qKof_i^{\qKofc_i (x)} (x)} = 0$.
Set $y = \qKoe_i^{\qKoec_i (x)} (x)$ and $z = \qKof_i^{\qKofc_i (x)} (x)$.
Then,
\[ 0 \leq \qKoec_i (1) \leq \qKoec_i (1y) = \qKoec_i (y) = 0 \]
and
\[ 0 \leq \qKofc_i (1) \leq \qKofc_i (1z) = \qKofc_i (z) = 0, \]
by \comboref{Lemma}{lem:qcmqtpbo}.
\end{proof}

Note that in the case where for some $i \in I$ we have $\qKoec_i (x) = +\infty$, for all $x \in M$, the quasi-Kashiwara operators $\qKoe_i$ and $\qKof_i$ are undefined on every element in $M$.
Such a case has little interest to study in the context of this paper, and we say that such a quasi-crystal monoid is \dtgterm{degenerate}.
Thus, by \comboref{Proposition}{prop:qcmid}, we get the following characterization.

\begin{dfn}
\label{dfn:qcmnondeg}
A quasi-crystal monoid $\qcrstM$ is said to be \dtgterm{nondegenerate} if $\qKoec_i (1) = \qKofc_i (1) = 0$, for all $i \in I$.
\end{dfn}

Due to the interaction between the binary operation $\cdot$ of a quasi-crystal monoid $\qcrstM$ and the quasi-crystal structure of the quasi-tensor product $\qcrstM \dotimes \qcrstM$ required by \itmcomboref{Definition}{dfn:qcm}{dfn:qcmhom}, we can extend the signature rule described in \comboref{Subsection}{subsec:qcqtpsr} to quasi-crystal monoids.
Let $x, y \in M$ and $i \in I$.
Since $\qKoec_i (xy) = \qKoec_i (x \dotimes y)$ and $\qKofc_i (xy) = \qKofc_i (x \dotimes y)$, we have that the $i$-signature of $xy$ and $x \dotimes y$ coincide.
Hence, by \comboref{Proposition}{prop:qcqtpsr},
$\qtpsgn_i (xy) = \qtpsgn_i (x) \qtpsgn_i (y)$.
Also, by \comboref{Proposition}{prop:qcmid}, either $\qtpsgn_i (1) = \ew$ or $\qtpsgn_i (z) = 0$, for any $z \in M$.
Therefore, we obtained the following result, which can be seen as an improvement of \comboref{Proposition}{prop:qcqtpsr} for nondegenerate quasi-crystal monoids.

\begin{prop}
\label{prop:qcmsr}
Let $\qcrstM$ be a quasi-crystal monoid, and let $i \in I$.
Then,
\[ \qtpsgn_i (xy) = \qtpsgn_i (x) \qtpsgn_i (y), \]
for any $x, y \in M$.
Moreover, $\qtpsgn_i$ is a monoid homomorphism from $M$ to $Z_0$ if and only if $\qKoec_i (x) \in \Z_{\geq 0}$, for some $x \in M$.
\end{prop}

The signature rule for quasi-crystal monoids follows directly from the previous result and \comboref{Proposition}{prop:qcmdesc}.
Consider a quasi-crystal monoid $\qcrstM$.
Let $x_1, \ldots, x_m \in M$ and $i \in I$.
Then, we can compute the $i$-signature of $x_1 \cdots x_m$ based only in the $i$-signature of each $x_k$, $k=1,\ldots,m$, because
\[ \qtpsgn_i (x_1 \cdots x_m) = \qtpsgn_i (x_1) \cdots \qtpsgn_i (x_m). \]
If
$\qtpsgn_i (x_1 \cdots x_m) = 0$,
then $\qKoec_i (x_1 \cdots x_m) = \qKofc_i (x_1 \cdots x_m) = +\infty$ which implies that $\qKoe_i (x_1 \cdots x_m) = \qKof_i (x_1 \cdots x_m) = \undf$.
Otherwise,
$\qtpsgn_i (x_1 \cdots x_m) = {-^{a}} {+^{b}}$,
for some $a, b \in \Z_{\geq 0}$.
Then, $\qKoec_i (x_1 \cdots x_m) = a$ and $\qKofc_i (x_1 \cdots x_m) = b$.
The raising quasi-Kashiwara operator $\qKoe_i$ is defined on $x_1 \cdots x_m$ if and only if $a \geq 1$, in which case
$\qKoe_i (x_1 \cdots x_m) = x_1 \cdots x_{p-1} \cdot \qKoe_i (x_p) \cdot x_{p+1} \cdots x_m$,
where $x_p$ originates the right-most symbol $-$ in $\qtpsgn_i (x_1 \cdots x_m)$.
Similarly, the lowering quasi-Kashiwara operator $\qKof_i$ is defined on $x_1 \cdots x_m$ if and only if $b \geq 1$, in which case
$\qKof_i (x_1 \cdots x_m) = x_1 \cdots x_{q-1} \cdot \qKof_i (x_q) \cdot x_{q+1} \cdots x_m$,
where $x_q$ originates the left-most symbol $+$ in $\qtpsgn_i (x_1 \cdots x_m)$.

We now introduce the notion of a homomorphism between quasi-crystal monoids.

\begin{dfn}
\label{dfn:qcmh}
Let $\qcrstM$ and $\qcrstM'$ be quasi-crystal monoids of the same type.
A \dtgterm{quasi-crystal monoid homomorphism} $\psi$ from $\qcrstM$ to $\qcrstM'$, denoted by $\psi : \qcrstM \to \qcrstM'$, is a map $\psi : M \to M'$ that satisfies the following conditions:
\begin{enumerate}
\item\label{dfn:qcmhqch}
$\psi$ is a quasi-crystal homomorphism;

\item\label{dfn:qcmhmh}
$\psi$ is a monoid homomorphism.
\end{enumerate}
If $\psi$ is also bijective, it is called a \dtgterm{quasi-crystal monoid isomorphism}.
\end{dfn}

Note that in the previous definition we only consider maps from $M$ to $M'$.
But as we observed after \comboref{Definition}{dfn:qch}, when we state that $\psi$ is a quasi-crystal homomorphism in condition\avoidrefbreak \itmref{dfn:qcmhqch} above, we mean that the map $\psi' : M \sqcup \{\undf\} \to M' \sqcup \{\undf\}$, defined by $\psi'(\undf) = \undf$ and $\psi'(x) = \psi(x)$, for each $x \in M$, is a quasi-crystal homomorphism from $\qcrstM$ to $\qcrstM'$.

Also, if $\psi : \qcrstM \to \qcrstM'$ is a quasi-crystal monoid isomorphism, then $\psi$ is both a quasi-crystal isomorphism (by \comboref{Corollary}{cor:snqciso}) and a monoid isomorphism.
The converse is immediate, because a monoid isomorphism is bijective.
Hence, a map $\psi : M \to M'$ is a quasi-crystal monoid isomorphism if and only if $\psi$ is a quasi-crystal isomorphism and a monoid isomorphism.
This implies that if $\psi : \qcrstM \to \qcrstM'$ is a quasi-crystal monoid isomorphism, then $\psi^{-1}$ is a quasi-crystal monoid isomorphism between $\qcrstM'$ and $\qcrstM$.

\subsection{The free quasi-crystal monoid}
\label{subsec:fqcm}


Let $\qcrstQ$ be a seminormal quasi-crystal. For $k \geq 1$, set
\[ \qcrstQ^{\dotimes k} = \underbrace{\qcrstQ \dotimes \cdots \dotimes \qcrstQ}_{\text{$k$ times}}. \]
By \comboref{Corollary}{cor:snqciso} and \comboref{Theorems}{thm:qcqtp} and\avoidrefbreak \ref{thm:qcqtpassoc}, for $k, l \geq 1$ the map $\qcrstQ^{\dotimes k} \dotimes \qcrstQ^{\dotimes l} \to \qcrstQ^{\dotimes (k+l)}$, given by
$(x_1 \dotimes \cdots \dotimes x_k) \dotimes (y_1 \dotimes \cdots \dotimes y_l) \mapsto x_1 \dotimes \cdots \dotimes x_k \dotimes y_1 \dotimes \cdots \dotimes y_l$,
for $x_1, \ldots, x_k, y_1, \ldots, y_l \in Q$, is a quasi-crystal isomorphism.

Let $\zeta$ be an element that does not lie in $Q$.
Set $\qcrstQ^{\dotimes 0}$ to be the seminormal quasi-crystal of the same type as $\qcrstQ$ formed by the set $Q^{\dotimes 0} = \{\zeta\}$ and maps given by $\wt (\zeta) = 0$, $\qKoe_i (\zeta) = \qKof_i (\zeta) = \undf$ and $\qKoec_i (\zeta) = \qKofc_i (\zeta) = 0$ ($i \in I$).
Note that $\qcrstQ^{\dotimes 0} \dotimes \qcrstQ^{\dotimes 0}$ is quasi-crystal isomorphic to $\qcrstQ^{\dotimes 0}$, as both quasi-crystals consist of a single element where the quasi-crystal structure maps coincide.
For any $x \in Q$ and $i \in I$, by \comboref{Theorem}{thm:qcqtp} and \comboref{Proposition}{prop:qcqtpalt}, we have that $\wt(x \dotimes \zeta) = \wt(x)$, $\qKoec_i (x \dotimes \zeta) = \qKoec_i (x)$ and $\qKofc_i (x \dotimes \zeta) = \qKofc_i (x)$.
Also, by the signature rule, since $\qtpsgn_i (\zeta) = \ew$, it is immediate that $\qKoe_i$ (or $\qKof_i$) is defined on $x \dotimes \zeta$ if and only if $\qKoe_i$ (resp., $\qKof_i$) is defined on $x$. And if so, $\qKoe_i (x \dotimes \zeta) = \qKoe_i (x) \dotimes \zeta$ (resp., $\qKof_i (x \dotimes \zeta) = \qKof_i (x) \dotimes \zeta$).
Therefore, the map $Q \dotimes Q^{\dotimes 0} \to Q$, given by $y \dotimes \zeta \mapsto y$ for each $y \in Q$, is a quasi-crystal isomorphism.
Analogously, the map $Q^{\dotimes 0} \dotimes Q \to Q$, given by $\zeta \dotimes y \mapsto y$ for each $y \in Q$, is a quasi-crystal isomorphism.
Since the quasi-tensor product of quasi-crystals is associative (\comboref{Theorem}{thm:qcqtpassoc}), we get that $\qcrstQ^{\dotimes k} \dotimes \qcrstQ^{\dotimes 0}$ and $\qcrstQ^{\dotimes 0} \dotimes \qcrstQ^{\dotimes k}$ are isomorphic to $\qcrstQ^{\dotimes k}$, for any $k \geq 0$.

The sets $Q^{\dotimes k}$ and $Q^{\dotimes l}$ are disjoint, whenever $k \neq l$.
Thus, we can extend the maps $\wt$, $\qKoe_i$, $\qKof_i$, $\qKoec_i$ and $\qKofc_i$ ($i \in I$) defined on each $Q^{\dotimes k}$ to the set
\[ M = \bigcup_{k \geq 0} Q^{\dotimes k} \]
obtaining a seminormal quasi-crystal $\qcrstM$.
By the quasi-crystal isomorphisms $\qcrstQ^{\dotimes k} \dotimes \qcrstQ^{\dotimes l} \to \qcrstQ^{\dotimes (k+l)}$ ($k, l \geq 0$) defined above, we get that $\qcrstM$ becomes a quasi-crystal monoid with the binary operation ${\cdot} : M \times M \to M$ given by
\[ (x_1 \dotimes \cdots \dotimes x_k) \cdot (y_1 \dotimes \cdots \dotimes y_l) = x_1 \dotimes \cdots \dotimes x_k \dotimes y_1 \dotimes \cdots \dotimes y_l \]
and
\[ (x_1 \dotimes \cdots \dotimes x_k) \cdot \zeta = \zeta \cdot (x_1 \dotimes \cdots \dotimes x_k) = x_1 \dotimes \cdots \dotimes x_k, \]
for $x_1, \ldots, x_k, y_1, \ldots, y_l \in Q$.
If we identify $\zeta$ with the empty word $\ew$ and each element of the form $x_1 \dotimes \cdots \dotimes x_k$ in $M$ with the word $x_1 \ldots x_k$ over the alphabet $Q$, we obtain a monoid isomorphism between $M$ and the free monoid $Q^*$ over $Q$, and through this identification we can also define a quasi-crystal structure on $Q^*$.
Therefore, we have constructed a quasi-crystal monoid that leads to the following definition.

\begin{dfn}
\label{dfn:fqcm}
Let $\qcrstQ$ be a seminormal quasi-crystal.
The \dtgterm{free quasi-crystal monoid} $\qcrstQ^\fqcms$ over $\qcrstQ$ is a quasi-crystal monoid of the same type as $\qcrstQ$ consisting of the set $Q^*$ of all words over $Q$, the usual concatenation of words, and quasi-crystal structure maps defined as follows.
For $i \in I$, set
\[
\wt (\ew) = 0,
\quad
\qKoe_i (\ew) = \qKof_i (\ew) = \undf,
\quad \text{and} \quad
\qKoec_i (\ew) = \qKofc_i (\ew) = 0,
\]
and for $u, v \in Q^*$, set
\[ \wt (uv) = \wt(u) + \wt(v), \]
if $\qKofc_i (u) > 0$ and $\qKoec_i (v) > 0$, set
\[
\qKoe_i (u v) = \qKof_i (u v) = \undf
\quad \text{and} \quad
\qKoec_i (u v) = \qKofc_i (u v) = {+\infty},
\]
otherwise, set
\begin{align*}
\qKoe_i (uv) &=
  \begin{cases}
    \qKoe_i (u) v & \text{if $\qKofc_i (u) \geq \qKoec_i (v)$}\\
    u \qKoe_i (v) & \text{if $\qKofc_i (u) < \qKoec_i (v)$,}
  \end{cases}
\displaybreak[0]\\
\qKof_i (uv) &=
  \begin{cases}
    \qKof_i (u) v & \text{if $\qKofc_i (u) > \qKoec_i (v)$}\\
    u \qKof_i (v) & \text{if $\qKofc_i (u) \leq \qKoec_i (v)$,}
  \end{cases}
\displaybreak[0]\\
\qKoec_i (uv) &= \max \set[\big]{ \qKoec_i (u), \qKoec_i (v) - \innerp[\big]{\wt (u)}{\alphav_i} },
\displaybreak[0]\\
\shortintertext{and}
\qKofc_i (uv) &= \max \set[\big]{ \qKofc_i (u) + \innerp[\big]{\wt(v)}{\alphav_i}, \qKofc_i (v) },
\end{align*}
where $u \undf = \undf v = \undf$.
\end{dfn}

As we constructed the free quasi-crystal monoid $\qcrstQ^\fqcms$ based on the quasi-tensor product $\dotimes$ of quasi-crystals, we described its quasi-crystal structure based on the definition of the quasi-crystal structure of a quasi-tensor product (see \comboref{Theorem}{thm:qcqtp} and \comboref{Definition}{dfn:qcqtp}).
Notice that we explicitly gave the values of the quasi-crystal structure maps of $\qcrstQ^\fqcms$ on $\zeta$ (which we identify with the empty word $\ew$), on letters the values follow from the quasi-crystal structure maps of $\qcrstQ$, and on a word of the form $uv$ they depend only on their values on $u$ and $v$.
Thus, the definition of the quasi-crystal structure above is not circular.
Moreover, from \comboref{Proposition}{prop:qcmdesc}, 
we can obtain the values of the quasi-crystal structure maps on a word based only on their values on its letters, which implies the following result.

\begin{prop}
\label{prop:fqcmlng}
Let $\qcrstQ$ be a seminormal quasi-crystal.
For any $w \in Q^*$ and $i \in I$, if $\qKoe_i (w) \in Q^*$ then $\wlng[\big]{\qKoe_i (w)} = \wlng{w}$, and if $\qKof_i (w) \in Q^*$ then $\wlng[\big]{\qKof_i (w)} = \wlng{w}$.
Therefore, $\wlng{u} = \wlng{v}$, whenever $u$ and $v$ lie in the same connected component of $\qcrstQ^\fqcms$.
\end{prop}

\begin{proof}
Let $w \in Q^*$ and $i \in I$.
If $\qKof_i$ is defined on $w$, then $w \neq \ew$, by \comboref{Definition}{dfn:fqcm}, and so, $w = x_1 \ldots, x_m$, for some $x_1, \ldots, x_m \in Q$ and $m \geq 1$.
By \comboref{Proposition}{prop:qcmdesc}, there exists $q \in \set{1, \ldots, m}$ such that
\[ \qKof_i (w) = x_1 \ldots x_{q-1} \qKof_i (x_q) x_{q+1} \ldots x_m. \]
Since $x_q \in Q$, then $\qKof_i (x_q) \in Q$, by \comboref{Definition}{dfn:fqcm}, which implies that $\wlng[\big]{\qKof_i (w)} = \wlng{w}$.
Hence, for any $w \in Q^*$ and $i \in I$, $\wlng[\big]{\qKof_i(w)} = \wlng{w}$, whenever $\qKof_i (w) \in Q^*$.
Analogously, for any $w \in Q^*$ and $i \in I$, if $\qKoe_i$ is defined on $w$, then $\wlng[\big]{\qKoe_i(w)} = \wlng{w}$.

Finally, let $u$ and $v$ be words lying in the same connected component of $\qcrstQ^\fqcms$.
Then, by \comboref{Definition}{dfn:qccc}, there exist
$g_1, \ldots, g_k \in \set{ \qKoe_i, \qKof_i \given i \in I }$
such that
$g_1 \cdots g_k (u) = v$,
and by applying recursively what we have proven above,
$\wlng{g_1 \cdots g_k (u)} = \wlng{u}$.
\end{proof}

From the previous result, we can deduce the following properties of the connected components of $\qcrstQ^\fqcms$ (see \comboref{Definition}{dfn:qccc}).

\begin{prop}
\label{prop:fqcmcc}
Let $\qcrstQ$ be a seminormal quasi-crystal whose underlying set $Q$ is finite,
and let $Q' \subseteq Q^*$ be a connected component of $\qcrstQ^\fqcms$.
Then,
\begin{enumerate}
\item\label{prop:fqcmccf}
$Q'$ is finite,

\item\label{prop:fqcmcchw}
$Q'$ has at least a highest-weight element, and

\item\label{prop:fqcmcclw}
$Q'$ has at least a lowest-weight element.
\end{enumerate}
\end{prop}

\begin{proof}
\itmref{prop:fqcmccf} By \comboref{Proposition}{prop:fqcmlng}, we have that words in $Q'$ have the same length.
Since $Q$ is finite, there are finitely many words of a given length.
Hence, $Q'$ is finite.

\itmref{prop:fqcmcchw} Since $Q'$ is finite, we can take a word $w \in Q'$ whose weight $\wt(x)$ is maximal among weights of words in $Q'$ with respect to the partial order given in\avoidrefbreak \eqref{eq:rswpo}.
Then, by \itmcomboref{Proposition}{prop:qchlwMm}{prop:qchlwMmh}, $w$ is of highest weight.

\itmref{prop:fqcmcclw} Analogously to\avoidrefbreak \itmref{prop:fqcmcchw}, we can take a word $w' \in Q'$ whose weight $\wt(w')$ is minimal among weights of words in $Q'$,
and by \itmcomboref{Proposition}{prop:qchlwMm}{prop:qchlwMml}, we get that $w'$ is of lowest weight.
\end{proof}

We observed before \comboref{Definition}{dfn:qcqtp} that our intention with the (inverse-free) quasi-tensor product was to allow an interpretation in terms of quasi-crystals of the notion of $i$-inversion in a word, introduced in\avoidcitebreak \cite[\S~5]{CM17crysthypo}.
In the context of quasi-crystals of type $\tAn$ (see \itmcomboref{Example}{exa:qc}{exa:qctAn}), for $i \in \set{1,\ldots,n-1}$, a word $w \in A_n^*$ has an $i$-inversion if it admits a decomposition of the form $w = w_1 i w_2 (i+1) w_3$, for some $w_1, w_2, w_3 \in A_n^*$,
and to accomplish our intention, the quasi-Kashiwara operators $\qKoe_i$ and $\qKof_i$ should be undefined on such a word.
In the following example, we check that indeed this happens.

\begin{exa}
\label{exa:fqcmAn}
Consider the standard quasi-crystal $\qctAn$ of type $\tAn$ as described in \itmcomboref{Example}{exa:qc}{exa:qctAn}.
The following is a direct consequence of \comboref{Proposition}{prop:qcmdesc}.
Let $w \in A_n^*$. The weight of $w$ is given by
\[ \wt (w) = \wlng{w}_1 \vc{e_1} + \wlng{w}_2 \vc{e_2} + \cdots + \wlng{w}_n \vc{e_n}. \]
Let $i \in \set{1, \ldots, n-1}$.
If $w$ has a decomposition of the form $w = w_1 i w_2 (i+1) w_3$, for some $w_1, w_2, w_3 \in A_n^*$, then $\qKoec_i (w) = \qKofc_i (w) = +\infty$.
Otherwise, $\qKoec_i (w) = \wlng{w}_{i+1}$ and $\qKofc_i (w) = \wlng{w}_i$.
The raising quasi-Kashiwara operator $\qKoe_i$ is defined on $w$ if and only if $w$ has a decomposition of the form $w = u_1 (i+1) u_2$, for some $u_1, u_2 \in A_n^*$ with $\wlng{u_1}_{i} = \wlng{u_2}_{i+1} = 0$, and if so, $\qKoe_i (w) = u_1 i u_2$.
The lowering quasi-Kashiwara operator $\qKof_i$ is defined on $w$ if and only if $w$ has a decomposition of the form $w = v_1 i v_2$, for some $v_1, v_2 \in A_n^*$ with $\wlng{v_1}_{i} = \wlng{v_2}_{i+1} = 0$, and if so, $\qKof_i (w) = v_1 (i+1) v_2$.
Informally, provided that $w$ does not have a decomposition of the form $w = w_1 i w_2 (i+1) w_3$, we have that $\qKoe_i (w)$ is obtained by replacing the right-most symbol $i+1$ in $w$ by $i$, and $\qKof_i (w)$ is obtained by replacing the left-most symbol $i$ in $w$ by $i+1$.

From \itmcomboref{Examples}{exa:qcg}{exa:qcgA32} and\itmcomboref{\relax}{exa:qcqtp}{exa:qcqtpA32}, We have that the words of length $2$ form the following subgraph of the quasi-crystal graph of $\qctA_3^\fqcms$.
\[
\begin{tikzpicture}[widecrystal,baseline=(33.base)]
  %
  \node (11) at (1, 3) {11};
  \node (21) at (2, 3) {21};
  \node (31) at (3, 3) {31};
  \node (12) at (1, 2) {12};
  \node (22) at (2, 2) {22};
  \node (32) at (3, 2) {32};
  \node (13) at (1, 1) {13};
  \node (23) at (2, 1) {23};
  \node (33) at (3, 1) {33};
  \path (11) edge node {1} (21)
        (21) edge node {2} (31)
        (21) edge node {1} (22)
        (31) edge node {1} (32)
        (12) edge [loop left] node {1} ()
        (12) edge node {2} (13)
        (22) edge node {2} (32)
        (32) edge node {2} (33)
        (13) edge node {2} (23)
        (23) edge [loop right] node {2} ();
\end{tikzpicture}
\]
\end{exa}

The term \dtgterm{free}, used in \comboref{Definition}{dfn:fqcm} to characterize the quasi-crystal monoid $\qcrstQ^\fqcms$ over a seminormal quasi-crystal $\qcrstQ$, is justified by the following universal property.

\begin{thm}
\label{thm:fqcmext}
Let $\qcrstQ$ be a seminormal quasi-crystal and $\qcrstM$ be a nondegenerate quasi-crystal monoid of the same type.
Then, for each quasi-crystal homomorphism $\psi : \qcrstQ \to \qcrstM$ satisfying $\psi (Q) \subseteq M$, there exists a unique quasi-crystal monoid homomorphism $\hat{\psi} : \qcrstQ^\fqcms \to \qcrstM$ such that $\hat{\psi} (x) = \psi (x)$, for all $x \in Q$.
\end{thm}

\begin{proof}
Let $\psi : \qcrstQ \to \qcrstM$ be a quasi-crystal homomorphism such that $\psi(Q) \subseteq M$.
So, we can consider $\psi$ as a map from $Q$ to $M$.
It is well-known that there exists a unique monoid homomorphism $\hat{\psi} : Q^* \to M$ such that $\hat{\psi} (x) = \psi (x)$, for all $x \in Q$.
We also have that $\hat{\psi}$ is given by
\[
\hat{\psi} (\ew) = 1
\quad \text{and} \quad
\hat{\psi} (x_1 \ldots x_m) = \psi (x_1) \cdots \psi (x_m),
\]
for $x_1, \ldots, x_m \in Q$.
It remains to show that $\hat{\psi}$ is a quasi-crystal homomorphism.

Let $i \in I$.
By \comboref{Proposition}{prop:qcmid} and \comboref{Definitions}{dfn:qcmnondeg} and\avoidrefbreak \ref{dfn:fqcm}, we have that $\wt (\ew) = 0 = \wt (1)$, $\qKoe_i (\ew) = \qKof_i (\ew) = \undf = \qKoe_i (1) = \qKof_i (1)$, and $\qKoec_i (\ew) = \qKofc_i (\ew) = 0 = \qKoec_i (1) = \qKofc_i (1)$.
Let $x_1, \ldots, x_m \in Q$, $m \geq 1$, and set $w = x_1 \ldots x_m$.
By \itmcomboref{Definition}{dfn:qch}{dfn:qchwtc}, we have that $\wt \parens[\big]{\psi (x_k)} = \wt (x_k)$, $\qKoec_i \parens[\big]{\psi (x_k)} = \qKoec_i (x_k)$ and $\qKofc_i \parens[\big]{\psi (x_k)} = \qKofc_i (x_k)$ for $k = 1, \ldots, m$.
By \comboref{Proposition}{prop:qcmdesc}, we obtain that
\[ \wt \parens[\big]{\hat{\psi} (w)} = \wt \parens[\big]{\psi (x_1)} + \cdots + \wt \parens[\big]{\psi (x_m)} = \wt (x_1) + \cdots + \wt (x_m) = \wt (w), \]
and since
\[ p = \max \set[\big]{ 1 \leq k \leq m \given \qKoec_i (x_k) > 0 } = \max \set[\big]{ 1 \leq k \leq m \given \qKoec_i \parens[\big]{\psi (x_k)} > 0 } \]
and
\[ q = \min \set[\big]{ 1 \leq l \leq m \given \qKofc_i (x_l) > 0 } = \min \set[\big]{ 1 \leq l \leq m \given \qKofc_i \parens[\big]{\psi (x_l)} > 0 }, \]
we also get that $\qKoec_i \parens[\big]{\hat{\psi} (w)} = \qKoec_i (w)$ and $\qKofc_i \parens[\big]{\hat{\psi} (w)} = \qKofc_i (w)$.
If $\qKoe_i (w) \in Q^*$, then $\qKoe_i (x_p) \in Q^*$, more precisely $\qKoe_i (x_p) \in Q$ as $x_p \in Q$,
and since $\psi (Q) \subseteq M$ we have by \itmcomboref{Definition}{dfn:qch}{dfn:qchqKoe} that $\psi \parens[\big]{\qKoe_i (x_p)} = \qKoe_i \parens[\big]{\psi (x_p)}$,
which implies that
\[\begin{split}
\hat{\psi} \parens[\big]{\qKoe_i (w)} &= \psi (x_1) \cdots \psi (x_{p-1}) \cdot \psi \parens[\big]{\qKoe_i (x_p)} \cdot \psi (x_{p+1}) \cdots \psi (x_m)\\
&= \psi (x_1) \cdots \psi (x_{p-1}) \cdot \qKoe_i \parens[\big]{\psi (x_p)} \cdot \psi (x_{p+1}) \cdots \psi (x_m) = \qKoe_i \parens[\big]{\hat{\psi} (w)}.
\end{split}\]
Analogous reasoning applies if $\qKof_i (w) \in Q^*$, leading to $\hat{\psi} \parens[\big]{\qKof_i (w)} = \qKof_i \parens[\big]{\hat{\psi} (w)}$.
\end{proof}

The property described in the previous result can be used to define the free quasi-crystal monoid up to isomorphism among nondegenerate quasi-crystal monoids.

\begin{cor}
\label{cor:fqcmuti}
Let $\qcrstQ$ be a seminormal quasi-crystal, $\qcrstM$ be a quasi-crystal monoid of the same type, and $\iota : \qcrstQ \to \qcrstM$ be an injective quasi-crystal homomorphism such that
for each nondegenerate quasi-crystal monoid $\qcrstM'$ and each quasi-crystal homomorphism $\psi : \qcrstQ \to \qcrstM'$ satisfying $\psi (Q) \subseteq M'$,
there exists a unique quasi-crystal monoid homomorphism $\hat{\psi} : \qcrstM \to \qcrstM'$ for which $\psi = \hat{\psi} \iota$.
Then, there exists a quasi-crystal monoid isomorphism between $\qcrstM$ and $\qcrstQ^\fqcms$.
\end{cor}

\begin{proof}
Define $\psi : Q \to Q^*$ by $\psi (x) = x$, for each $x \in Q$.
As the quasi-crystal structure maps of $\qcrstQ^\fqcms$ agree on words of length $1$ with the quasi-crystal structure maps of $\qcrstQ$, we have that $\psi$ is a quasi-crystal homomorphism from $\qcrstQ$ to $\qcrstQ^\fqcms$.
Thus, there exists a unique quasi-crystal monoid homomorphism $\hat{\psi} : \qcrstM \to \qcrstQ^\fqcms$ such that $\hat{\psi} \iota (x) = x$, for all $x \in Q$.
Then, $\qcrstM$ is nondegenerate, because $\qKoec_i (1) = \qKoec_i \parens[\big]{\hat{\psi} (1)} = \qKoec_i (\ew) = 0$ and $\qKofc_i (1) = \qKofc_i \parens[\big]{\hat{\psi} (1)} = \qKofc_i (\ew) = 0$, for any $i \in I$.
By \comboref{Theorem}{thm:fqcmext}, there exists a unique quasi-crystal monoid homomorphism $\hat{\iota} : \qcrstQ^\fqcms \to \qcrstM$ such that $\hat{\iota} (x) = \iota (x)$, for all $x \in Q$.
Then, $\hat{\psi} \hat{\iota}$ is a quasi-crystal monoid homomorphism from $\qcrstQ^\fqcms$ to $\qcrstQ^\fqcms$ such that $\hat{\psi} \hat{\iota} (x) = x$, for any $x \in Q$,
and by \comboref{Theorem}{thm:fqcmext}, we obtain that $\hat{\psi} \hat{\iota}$ must be the identity map on $Q^*$.
Also, $\hat{\iota} \hat{\psi}$ is a quasi-crystal monoid homomorphism from $\qcrstM$ to $\qcrstM$ such that $\hat{\iota} \hat{\psi} \iota (x) = \iota (x)$, for any $x \in Q$,
and by uniqueness, we get that $\hat{\iota} \hat{\psi}$ must be the identity map on $M$.
Hence, $\hat{\psi}$ is a quasi-crystal monoid isomorphism between $\qcrstM$ and $\qcrstQ^\fqcms$.
\end{proof}

\subsection{Congruences and quotients}
\label{subsec:qcmcq}


We now study the notion of congruence on a quasi-crystal monoid which leads to the definition of quotient quasi-crystal monoid and to the proof of homomorphism theorems for quasi-crystal monoids.

\begin{dfn}
\label{dfn:qcmc}
Let $\qcrstM$ be a quasi-crystal monoid.
A \dtgterm{quasi-crystal monoid congruence} on $\qcrstM$ is an equivalence relation $\theta \subseteq M \times M$ satisfying the conditions:
\begin{enumerate}
\item\label{dfn:qcmcwtc}
if $(x, y) \in \theta$, then $\wt (x) = \wt (y)$, $\qKoec_i (x) = \qKoec_i (y)$ and $\qKofc_i (x) = \qKofc_i (y)$ for all $i \in I$;

\item\label{dfn:qcmcqKoe}
if $(x, y) \in \theta$ and $\qKoe_i (x) \in M$, then $\parens[\big]{\qKoe_i (x), \qKoe_i (y)} \in \theta$;

\item\label{dfn:qcmcqKof}
if $(x, y) \in \theta$ and $\qKof_i (x) \in M$, then $\parens[\big]{\qKof_i (x), \qKof_i (y)} \in \theta$;

\item\label{dfn:qcmcmul}
if $(x_1, y_1), (x_2, y_2) \in \theta$, then $(x_1 x_2, y_1 y_2) \in \theta$.
\end{enumerate}
\end{dfn}

Let $\qcrstM$ be a quasi-crystal monoid.
It is immediate from the definition that the equality relation $\Delta = \set[\big]{ (x, x) \given x \in M }$ is a quasi-crystal monoid congruence on $\qcrstM$.
We have that $\Delta \subseteq \theta$, for any quasi-crystal monoid congruence $\theta$ on $\qcrstM$.
Also, given a nonempty family $\Theta$ of quasi-crystal monoid congruences on $\qcrstM$, it is straightforward to show that $\bigcap \Theta$ is a quasi-crystal monoid congruence on $\qcrstM$.
From this and the following result, we are able to show that the quasi-crystal monoid congruences on a quasi-crystal monoid form a lattice.

\begin{lem}
\label{lem:qcmccongen}
Let $\qcrstM$ be a quasi-crystal monoid, and let $R \subseteq M \times M$ be a relation on $M$ such that for any $(x, y) \in R$ and $i \in I$, the following conditions are satisfied:
\begin{enumerate}
\item\label{lem:qcmccongenwtc}
$\wt(x) = \wt(y)$, $\qKoec_i (x) = \qKoec_i (y)$, and $\qKofc_i (x) = \qKofc_i (y)$;

\item\label{lem:qcmccongenqKoe}
if $\qKoe_i (x) \in M$, then $\parens[\big]{\qKoe_i (x), \qKoe_i (y)} \in R$; and

\item\label{lem:qcmccongenqKof}
if $\qKof_i (x) \in M$, then $\parens[\big]{\qKof_i (x), \qKof_i (y)} \in R$.
\end{enumerate}
Then, the monoid congruence $\theta_R$ generated by $R$ is a quasi-crystal monoid congruence on $\qcrstM$.
\end{lem}

\begin{proof}
We check that in every step of constructing $\theta_R$ from $R$, properties\avoidrefbreak \itmref{lem:qcmccongenwtc} to\avoidrefbreak \itmref{lem:qcmccongenqKof} are preserved.
Set
\[ R_1 = \set[\big]{ (ux, uy) \given (x, y) \in R, u \in M }. \]
For $(x, y) \in R$, $u \in M$ and $i \in I$, since $\wt(x) = \wt(y)$, $\qKoec_i (x) = \qKoec_i (y)$ and $\qKofc_i (x) = \qKofc_i (y)$, we get that $\wt(ux) = \wt(uy)$, $\qKoec_i (ux) = \qKoec_i (uy)$ and $\qKofc_i (ux) = \qKofc_i (uy)$, by \comboref{Lemma}{lem:qcmqtpbo}.
As $\qcrstM$ is seminormal, we have that $\qKoe_i (ux) \in M$ if and only if $\qKoe_i (uy) \in M$.
And if so, we obtain when $\qKoec_i (x) = 0$ that $\parens[\big]{\qKoe_i (ux), \qKoe_i (uy)} = \parens[\big]{\qKoe_i (u) \cdot x, \qKoe_i (u) \cdot y} \in R_1$,
and when $\qKoec_i (x) > 0$ that $\parens[\big]{\qKoe_i (ux), \qKoe_i (uy)} = \parens[\big]{u \cdot \qKoe_i (x), u \cdot \qKoe_i (y)} \in R_1$, because $\parens[\big]{\qKoe_i (x), \qKoe_i (y)} \in R$, by\avoidrefbreak \itmref{lem:qcmccongenqKoe}.
Similarly, if $\qKof_i (ux) \in M$, we get when $\qKofc_i (u) > 0$ that $\parens[\big]{\qKof_i (ux), \qKof_i (uy)} = \parens[\big]{\qKof_i (u) \cdot x, \qKof_i (u) \cdot y} \in R_1$,
and when $\qKofc_i (u) = 0$ that $\parens[\big]{\qKof_i (ux), \qKof_i (uy)} = \parens[\big]{u \cdot \qKof_i (x), u \cdot \qKof_i (y)} \in R_1$, because $\parens[\big]{\qKof_i (x), \qKof_i (y)} \in R$, by\avoidrefbreak \itmref{lem:qcmccongenqKof}.
Hence, $R_1$ satisfies conditions\avoidrefbreak \itmref{lem:qcmccongenwtc} to\avoidrefbreak \itmref{lem:qcmccongenqKof}.

We can analogously deduce that $R_2 = \set[\big]{ (xv, yv) \given (x, y) \in R_1, v \in M }$ satisfies conditions\avoidrefbreak \itmref{lem:qcmccongenwtc} to\avoidrefbreak \itmref{lem:qcmccongenqKof}.

It is immediate that the reflexive closure $R_3 = R_2 \cup \set[\big]{ (x, x) \given x \in M }$ of $R_2$ satisfies
conditions\avoidrefbreak \itmref{lem:qcmccongenwtc} to\avoidrefbreak \itmref{lem:qcmccongenqKof}.

Set
$R_4 = R_3 \cup \set[\big]{ (x, y) \given (y, x) \in R_3 }$,
which correspondes to the symmetric closure of $R_3$.
Since $R_3$ satisfies condition\avoidrefbreak \itmref{lem:qcmccongenwtc}, so it does $R_4$.
For $(x, y) \in R_3$, if $\qKoe_i (y) \in M$, then $\qKoe_i (x) \in M$, because $\qKoec_i (x) = \qKoec_i (y)$ and $\qcrstM$ is seminormal, and so, $\parens[\big]{\qKoe_i (x), \qKoe_i (y)} \in R_3$, by\avoidrefbreak \itmref{lem:qcmccongenqKoe}.
Similarly, if $\qKof_i (y) \in M$, then $\parens[\big]{\qKof_i (x), \qKof_i (y)} \in R_3$, by\avoidrefbreak \itmref{lem:qcmccongenqKof}.
Hence, $R_4$ also satisfies conditions\avoidrefbreak \itmref{lem:qcmccongenqKoe} and\avoidrefbreak \itmref{lem:qcmccongenqKof}.

Finally, $\theta_R$ corresponds to the transitive closure of $R_4$.
For $(x, y) \in \theta_R$, there exist $x_0, x_1, \ldots, x_m \in M$ such that $x = x_0$, $y = x_m$ and $(x_{k-1}, x_{k}) \in R_4$, for $k=1,\ldots,m$.
For $i \in I$, since $R_4$ satisfies condition\avoidrefbreak \itmref{lem:qcmccongenwtc}, we get that
\[ \wt(x) = \wt(x_0) = \wt(x_1) = \cdots = \wt(x_m) = \wt(y), \]
and similarly, $\qKoec_i (x) = \qKoec_i (y)$ and $\qKofc_i (x) = \qKofc_i (y)$.
As $R_4$ satisfies condition\avoidrefbreak \itmref{lem:qcmccongenqKoe}, if $\qKoe_i (x) \in M$, we get that $\parens[\big]{\qKoe_i (x), \qKoe_i (x_{1})} \in R_4$, and recursively,
$\parens[\big]{\qKoe_i (x_{k-1}), \qKoe_i (x_{k})} \in R_4$, for $k=1, \ldots,m$,
which implies that $\parens[\big]{\qKoe_i (x), \qKoe_i (y)} \in \theta_R$.
Analogously, as $R_4$ satisfies condition\avoidrefbreak \itmref{lem:qcmccongenqKof}, if $\qKof_i (x) \in M$, then $\parens[\big]{\qKof_i (x), \qKof_i (y)} \in \theta_R$.
Therefore, $\theta_R$ satisfies conditions\avoidrefbreak \itmref{lem:qcmccongenwtc} to\avoidrefbreak \itmref{lem:qcmccongenqKof}.
Moreover, $\theta_R$ satisfies \itmcomboref{Definition}{dfn:qcmc}{dfn:qcmcmul}, as by construction $\theta_R$ is a monoid congruence on $M$, and thus, $\theta_R$ is a quasi-crystal monoid congruence on $\qcrstM$.
\end{proof}

\begin{thm}
\label{thm:qcmclat}
Let $\qcrstM$ be a quasi-crystal monoid.
Then, the quasi-crystal monoid congruences on $\qcrstM$ form a lattice with respect to the partial order $\subseteq$ of inclusion.
\end{thm}

\begin{proof}
Let $\Theta$ be the set of all quasi-crystal monoid congruences on $\qcrstM$.
Since the equality relation $\Delta$ lies in $\Theta$, we have that $\Theta$ is nonempty.
Clearly, $\subseteq$ is a partial order on $\Theta$.
Let $\theta, \sigma \in \Theta$.
Since $\theta \cap \sigma$ is a quasi-crystal monoid congruence and the largest set contained in $\theta$ and $\sigma$, we have that $\theta \cap \sigma$ is the infimum of $\theta$ and $\sigma$ in $\Theta$.
Finally, the relation $R = \theta \cup \sigma$ satisfies the conditions of \comboref{Lemma}{lem:qcmccongen}, and so, the monoid congruence $\theta_R$ generated by $R$ is a quasi-crystal monoid congruence on $\qcrstM$.
Also, $\theta_R$ is the smallest equivalence relation on $M$ satisfying \itmcomboref{Definition}{dfn:qcmc}{dfn:qcmcmul} and containing $R$.
Hence, $\theta_R$ is the supremum of $\theta$ and $\sigma$ in $\Theta$.
\end{proof}

Let $\theta$ be a quasi-crystal monoid congruence on a quasi-crystal monoid $\qcrstM$.
From \comboref{Definition}{dfn:qcmc}, it follows that the quasi-crystal structure maps $\wt$, $\qKoe_i$, $\qKof_i$, $\qKoec_i$ and $\qKofc_i$ ($i \in I$), and the monoid binary operation $\cdot$ of $\qcrstM$ give rise in a natural way to a quasi-crystal monoid whose underlying set is the set of all $\theta$-equivalence classes $M / \theta$.
For each $x \in M$, denote the $\theta$-equivalence class of $x$ by $[x]_{\theta}$, or simply, $[x]$.
If $x, y \in M$ are such that $[x] = [y]$, then $\wt (x) = \wt (y)$, $\qKoec_i (x) = \qKoec_i (y)$ and $\qKofc_i (x) = \qKofc_i (y)$, by \itmcomboref{Definition}{dfn:qcmc}{dfn:qcmcwtc}.
As $\qcrstM$ is seminormal and $\qKoec_i (x) = \qKoec_i (y)$, we have that $\qKoe_i$ is defined on $x$ if and only if $\qKoe_i$ is defined on $y$.
And if so, we have by \itmcomboref{Definition}{dfn:qcmc}{dfn:qcmcqKoe} that $[\qKoe_i (x)] = [\qKoe_i (y)]$.
Similarly, by \itmcomboref{Definition}{dfn:qcmc}{dfn:qcmcqKof}, if $\qKof_i$ is defined on $x$ or $y$, then $[\qKof_i (x)] = [\qKof_i (y)]$.
Finally, if $x_1, x_2, y_1, y_2 \in M$ are such that $[x_1] = [y_1]$ and $[x_2] = [y_2]$, then $[x_1 x_2] = [y_1 y_2]$, by \itmcomboref{Definition}{dfn:qcmc}{dfn:qcmcmul}.
Therefore, we obtain the following construction.

\begin{dfn}
\label{dfn:qqcm}
Let $\theta$ be a congruence on a quasi-crystal monoid $\qcrstM$.
The \dtgterm{quotient quasi-crystal monoid} of $\qcrstM$ by $\theta$ is a quasi-crystal monoid $\qcrstM / \theta$ of the same type as $\qcrstM$ consisting of the set $M / \theta$ and maps given by
\begin{align*}
\wt ([x]) &= \wt (x) & \qKoe_i ([x]) &= [\qKoe_i (x)] & \qKof_i ([x]) &= [\qKof_i (x)]\\
\qKoec_i ([x]) &= \qKoec_i (x) & \qKofc_i ([x]) &= \qKofc_i (x) & [x] \cdot [y] &= [x \cdot y],
\end{align*}
where $[\undf] = \undf$, for $x, y \in M$ and $i \in I$.
\end{dfn}

The following result follows directly from the previous definition.

\begin{lem}
\label{lem:qqcmcepi}
Let $\theta$ be a congruence on a quasi-crystal monoid $\qcrstM$.
Then, the map $\pi : M \to M / \theta$, given by $\pi (x) = [x]$ for each $x \in M$, is a surjective quasi-crystal monoid homomorphism from $\qcrstM$ to $\qcrstM / \theta$.
\end{lem}

This leads to the following result that relates congruences and homomorphisms on quasi-crystal monoids.

\begin{thm}
\label{thm:qcmcongker}
Let $\qcrstM$ be a quasi-crystal monoid, and let $\theta \subseteq M \times M$.
Then, $\theta$ is a congruence on $\qcrstM$ if and only if there exist a quasi-crystal monoid $\qcrstM'$ and a quasi-crystal monoid homomorphism $\psi : \qcrstM \to \qcrstM'$ such that $\theta = \ker \psi$.
\end{thm}

\begin{proof}
Let $\theta$ be a quasi-crystal monoid congruence on $\qcrstM$.
Set $\pi : \qcrstM \to \qcrstM / \theta$ to be the quasi-crystal monoid homomorphism defined in \comboref{Lemma}{lem:qqcmcepi}.
Then, for $x, y \in M$, we have that $\pi(x) = \pi(y)$ if and only if $(x, y) \in \theta$, which implies that $\theta = \ker\pi$.

Conversely, let $\psi : \qcrstM \to \qcrstM'$ be a quasi-crystal monoid homomorphism.
It is immediate that $\ker\psi$ is an equivalence relation on $M$.
Let $(x, y) \in {\ker\psi}$ and $i \in I$.
We have that
\[ \wt(x) = \wt \parens[\big]{\psi (x)} = \wt \parens[\big]{\psi (y)} = \wt(y), \]
and similarly, $\qKoec_i (x) = \qKoec_i (y)$ and $\qKofc_i (x) = \qKofc_i (y)$.
If $\qKoe_i (x) \in M$, then $\qKoe_i (y) \in M$, because $\qcrstM$ is seminormal and $\qKoec_i (x) = \qKoec_i (y)$.
Also, since $\psi(M) \subseteq M'$, we get that $\psi \parens[\big]{\qKoe_i (x)}, \psi \parens[\big]{\qKoe_i (y)} \in M'$, and so,
\[ \psi \parens[\big]{\qKoe_i (x)} = \qKoe_i \parens[\big]{\psi (x)} = \qKoe_i \parens[\big]{\psi (y)} = \psi \parens[\big]{\qKoe_i (y)}, \]
which implies that $\parens[\big]{\qKoe_i (x), \qKoe_i (y)} \in {\ker\psi}$.
Analogously, if $\qKof_i (x) \in M$, then we obtain that $\parens[\big]{\qKof_i (x), \qKof_i (y)} \in {\ker\psi}$.
Finally, for $(x_1, y_1), (x_2, y_2) \in {\ker\psi}$, we have that
\[ \psi (x_1 x_2) = \psi(x_1) \psi(x_2) = \psi(y_1) \psi(y_2) = \psi(y_1 y_2), \]
which implies that $(x_1 x_2, y_1 y_2) \in {\ker\psi}$.
Hence, $\ker\psi$ is a quasi-crystal monoid congruence on $\qcrstM$.
\end{proof}

Now, we introduce the homomorphism theorems for quasi-crystal monoids.

\begin{thm}
\label{thm:qcmht1}
Let $\qcrstM$ and $\qcrstM'$ be quasi-crystal monoids of the same type, and let $\psi : \qcrstM \to \qcrstM'$ be a quasi-crystal monoid homomorphism.
Then, for each quasi-crystal monoid congruence $\theta$ on $\qcrstM$ satisfying $\theta \subseteq {\ker\psi}$, there exists a unique quasi-crystal monoid homomorphism $\hat{\psi} : \qcrstM / \theta \to \qcrstM'$ such that $\hat{\psi} ([x]_\theta) = \psi (x)$ for any $x \in M$.

Furthermore, if $\psi$ is surjective, then there exists a unique quasi-crystal monoid isomorphism $\hat{\psi} : \qcrstM / \ker\psi \to \qcrstM'$ such that $\hat{\psi} ([x]) = \psi (x)$, for all $x \in M$.
\end{thm}

\begin{proof}
Let $\theta$ be a quasi-crystal monoid congruence on $\qcrstM$ such that $\theta \subseteq {\ker\psi}$.
For $x, y \in M$, if $[x] = [y]$, then $\psi(x) = \psi(y)$, because $\theta \subseteq {\ker\psi}$.
Thus, we can define a map $\hat{\psi} : M / \theta \to M'$ by $\hat{\psi} ([x]) = \psi(x)$, for each $x \in M$.
Since $\psi$ is a quasi-crystal monoid homomorphism from $\qcrstM$ to $\qcrstM'$, it is immediate from \comboref{Definition}{dfn:qqcm} that $\hat{\psi}$ is a quasi-crystal monoid homomorphism from $\qcrstM / \theta$ to $\qcrstM'$.

Assume that $\psi$ is surjective and take $\theta = \ker\psi$.
Then, given $x' \in M'$, there exists $x \in M$ such that $\psi(x) = x'$, and so, $\hat{\psi} ([x]) = x$.
Also, if $y, z \in M$ are such that $\hat{\psi} ([y]) = \hat{\psi} ([z])$, then $\psi(y) = \psi(z)$, or equivalently, $(y, z) \in {\ker\psi}$, which implies that $[y] = [z]$.
Hence, $\hat{\psi}$ is bijective, and therefore, a quasi-crystal monoid isomorphism.
\end{proof}

\begin{thm}
\label{thm:qcmht2}
Let $\theta$ and $\sigma$ be congruences on a quasi-crystal monoid $\qcrstM$ such that $\theta \subseteq \sigma$.
Define a relation $\sigma / \theta$ on $M / \theta$ by
\[ \sigma / \theta = \set[\big]{ ([x]_\theta, [y]_\theta) \given (x, y) \in \sigma }. \]
Then, $\sigma / \theta$ is a quasi-crystal monoid congruence on $\qcrstM / \theta$.
Moreover, the map $(M / \theta) / (\sigma / \theta) \to M / \sigma$, given by $[[x]_\theta]_{\sigma / \theta} \mapsto [x]_\sigma$ for each $x \in M$, is a quasi-crystal monoid isomorphism between $(\qcrstM / \theta) / (\sigma / \theta)$ and $\qcrstM / \sigma$.
\end{thm}

\begin{proof}
Define a map $\pi : M \to M / \sigma$ by $\pi(x) = [x]_\sigma$, for each $x \in M$.
By \comboref{Lemma}{lem:qqcmcepi}, $\pi$ is a quasi-crystal monoid homomorphism from $\qcrstM$ to $\qcrstM / \sigma$.
Since $\theta \subseteq \sigma = \ker\pi$, by \comboref{Theorem}{thm:qcmht1}, we have a quasi-crystal monoid homomorphism $\hat{\pi} : \qcrstM / \theta \to \qcrstM / \sigma$ given by $\hat{\pi} ([x]_\theta) = [x]_\sigma$, for each $x \in M$.
As $\pi$ is surjective, then $\hat{\pi}$ is also surjective.
For $x, y \in M$, we have that $\hat{\pi} ([x]_\theta) = \hat{\pi} ([y]_\theta)$ if and only if $(x, y) \in \sigma$.
Hence, $\ker\hat{\pi} = \theta / \sigma$.
By \comboref{Theorem}{thm:qcmht1}, we obtain that the map $(M / \theta) / (\sigma / \theta) \to M / \sigma$, given by $[[x]_\theta]_{\sigma / \theta} \mapsto [x]_\sigma$ for each $x \in M$, is a quasi-crystal monoid isomorphism between $(\qcrstM / \theta) / (\sigma / \theta)$ and $\qcrstM / \sigma$.
\end{proof}

\section{The hypoplactic congruence}
\label{sec:hyco}

This section is devoted to study the hypoplactic congruence on a free quasi-crystal monoid.
We start by proving that it results in a quasi-crystal monoid congruence.
Based on this, we give the definition of hypoplactic monoid associated to a seminormal quasi-crystal.
We then characterize the commutative elements of such a monoid.

\begin{dfn}
\label{dfn:hyco}
Let $\qcrstQ$ be a seminormal quasi-crystal.
The \dtgterm{hypoplactic congruence} on $\qcrstQ^\fqcms$ is a relation $\hyco$ on $Q^*$ given as follows.
For $u, v \in Q^*$, $u \hyco v$ if and only if there exists a quasi-crystal isomorphism $\psi : \qcrstQ^\fqcms (u) \to \qcrstQ^\fqcms (v)$ such that $\psi (u) = v$.
\end{dfn}

To prove that $\hyco$ is a quasi-crystal monoid congruence on $\qcrstQ^\fqcms$ we first show the following result.

\begin{lem}
\label{lem:hycocompiso}
Let $\qcrstQ$ be a seminormal quasi-crystal.
For each $u, v \in Q^*$, the map $\parens[\big]{Q^* (u) \dotimes Q^* (v)} (u \dotimes v) \to Q^* (uv)$, given by $x \dotimes y \mapsto xy$, is a quasi-crystal isomorphism between $\parens[\big]{\qcrstQ^\fqcms (u) \dotimes \qcrstQ^\fqcms (v)} (u \dotimes v)$ and $\qcrstQ^\fqcms (uv)$.
\end{lem}

\begin{proof}
Let $u, v \in Q^*$.
Since $\qcrstQ^\fqcms$ is a quasi-crystal monoid, the map $Q^* \dotimes Q^* \to Q^*$, defined by $x \dotimes y \mapsto xy$ for each $x, y \in Q^*$, is a quasi-crystal homomorphism, by \itmcomboref{Definition}{dfn:qcm}{dfn:qcmhom}.
Then, by \comboref{Proposition}{prop:snqccctocc} we have a surjective quasi-crystal homomorphism $\psi : \parens[\big]{\qcrstQ^\fqcms \dotimes \qcrstQ^\fqcms} (u \dotimes v) \to \qcrstQ^\fqcms (uv)$ given by $\psi (x \dotimes y) = xy$, for each $x \dotimes y \in (Q^* \dotimes Q^*) (u \dotimes v)$.

We now show that $\parens[\big]{\qcrstQ^\fqcms \dotimes \qcrstQ^\fqcms} (u \dotimes v)$ and $\parens[\big]{\qcrstQ^\fqcms (u) \dotimes \qcrstQ^\fqcms (v)} (u \dotimes v)$ correspond to the same quasi-crystal.
Since they are formed by connected components of $\qcrstQ^\fqcms \dotimes \qcrstQ^\fqcms$ and $\qcrstQ^\fqcms$, it suffices to prove that their underlying sets coincide.
As $Q^* (u), Q^* (v) \subseteq Q^*$, we get that $\parens[\big]{Q^* (u) \dotimes Q^* (v)} (u \dotimes v) \subseteq (Q^* \dotimes Q^*) (u \dotimes v)$.
Let $x, y \in Q^*$ be such that $x \dotimes y \in (Q^* \dotimes Q^*) (u \dotimes v)$.
By \comboref{Definition}{dfn:qccc}, there exist $g_1, \ldots, g_m \in \set{ \qKoe_i, \qKof_i \given i \in I }$ such that $x \dotimes y = g_1 \cdots g_m (u \dotimes v)$,
and by \comboref{Definition}{dfn:qcqtp}, $x = g'_1 \cdots g'_k (u)$ and $y = g''_1 \cdots g''_l (v)$ for some $g'_1, \ldots, g'_k, g''_1, \ldots, g''_l \in \{g_1, \ldots, g_m\}$.
Hence, $x \in Q^* (u)$ and $y \in Q^* (v)$, and so, $(Q^* \dotimes Q^*) (u \dotimes v) = \parens[\big]{Q^* (u) \dotimes Q^* (v)} (u \dotimes v)$
Therefore, $\psi$ is a surjective quasi-crystal homomorphism from $\parens[\big]{\qcrstQ^\fqcms (u) \dotimes \qcrstQ^\fqcms (v)} (u \dotimes v)$ to $\qcrstQ^\fqcms (uv)$.

Finally, we show that $\psi$ is injective.
Let $x_1 \dotimes y_1, x_2 \dotimes y_2 \in \parens[\big]{Q^* (u) \dotimes Q^* (v)} (u \dotimes v)$.
We have by \comboref{Proposition}{prop:fqcmlng} that $\wlng{x_1} = \wlng{x_2}$ and $\wlng{y_1} = \wlng{y_2}$.
Thus, if $x_1 y_1 = x_2 y_2$, then $x_1 = x_2$ and $y_1 = y_2$ implying $x_1 \dotimes y_1 = x_2 \dotimes y_2$.
Hence, $\psi$ is injective.
By \comboref{Corollary}{cor:snqciso}, $\psi$ is a quasi-crystal isomorphism between $\parens[\big]{\qcrstQ^\fqcms (u) \dotimes \qcrstQ^\fqcms (v)} (u \dotimes v)$ and $\qcrstQ^\fqcms (uv)$.
\end{proof}

\begin{thm}
\label{thm:hycoqcmc}
Let $\qcrstQ$ be a seminormal quasi-crystal.
Then, the hypoplactic congruence $\hyco$ on $\qcrstQ^\fqcms$ is a quasi-crystal monoid congruence on $\qcrstQ^\fqcms$.
\end{thm}

\begin{proof}
It is straightforward to see that $\hyco$ is an equivalence relation.

  Let $u, v \in Q^*$ with $u \hyco v$. Then there exists a quasi-crystal isomorphism $\psi : \qcrstQ^\fqcms (u) \to \qcrstQ^\fqcms (v)$ such that $\psi (u) = v$.
Let $i \in I$. Since $\psi$ is a quasi-crystal isomorphism, we get that
\[ \wt (u) = \wt \parens[\big]{\psi (u)} = \wt (v), \]
and similarly, $\qKoec_i (u) = \qKoec_i (v)$ and $\qKofc_i (u) = \qKofc_i (v)$.
If $\qKoe_i (u) \in Q^*$, then
\[ \psi \parens[\big]{\qKoe_i (u)} = \qKoe_i \parens[\big]{\psi (u)} = \qKoe_i (v), \]
and since $Q^* \parens[\big]{ \qKoe_i (u) } = Q^* (u)$ and $Q^* \parens[\big]{ \qKoe_i (v) } = Q^* (v)$, we obtain $\qKoe_i (u) \hyco \qKoe_i (v)$.
Analogously, if $\qKof_i (u) \in Q^*$, then $\qKof_i (u) \hyco \qKof_i (v)$.

Additionally, let $u', v' \in Q^*$ with $u' \hyco v'$. Then we also have a quasi-crystal isomorphism $\psi' : \qcrstQ^\fqcms (u') \to \qcrstQ^\fqcms (v')$ such that $\psi (u') = v'$.
By \comboref{Lemma}{lem:hycocompiso}, we have quasi-crystal isomorphisms $\psi_1 : \qcrstQ^\fqcms (uv) \to \parens[\big]{\qcrstQ^\fqcms (u) \dotimes \qcrstQ^\fqcms (v)} (u \dotimes v)$ and $\psi_2 : \parens[\big]{\qcrstQ^\fqcms (u') \dotimes \qcrstQ^\fqcms (v')} (u' \dotimes v') \to \qcrstQ^\fqcms (u' v')$ such that $\psi_1 (u v) = u \dotimes v$ and $\psi_2 (u' \dotimes v') = u' v'$.
By \comboref{Theorem}{thm:qcqtph}, we have a quasi-crystal isomorphism $\psi \dotimes \psi'$ between $\qcrstQ^\fqcms (u) \dotimes \qcrstQ^\fqcms (v)$ and $\qcrstQ^\fqcms (u') \dotimes \qcrstQ^\fqcms (v')$ satisfying $(\psi \dotimes \psi') (u \dotimes v) = u' \dotimes v'$.
Set $\psi_3$ to be the restriction of $\psi \dotimes \psi'$ to $\parens[\big]{Q^* (u) \dotimes Q^* (v)} (u \dotimes v)$.
By \comboref{Proposition}{prop:qccciso},  $\psi_3$ is a quasi-crystal isomorphism between $\parens[\big]{\qcrstQ^\fqcms (u) \dotimes \qcrstQ^\fqcms (v)} (u \dotimes v)$ and $\parens[\big]{\qcrstQ^\fqcms (u') \dotimes \qcrstQ^\fqcms (v')} (u' \dotimes v')$.
Then, $\psi_2 \psi_3 \psi_1$ is a quasi-crystal isomorphism between $\qcrstQ^\fqcms (uv)$ and $\qcrstQ^\fqcms (u' v')$ that satisfies
$\psi_2 \psi_3 \psi_1 (uv) = u' v'$.
Hence, $uv \hyco u' v'$.
Therefore, $\hyco$ is a quasi-crystal monoid congruence on $\qcrstQ^\fqcms$.
\end{proof}

We have now set up the framework to present the following definition.

\begin{dfn}
\label{dfn:hypoqcm}
Let $\qcrstQ$ be a seminormal quasi-crystal, and let $\hyco$ be the hypoplactic congruence on $\qcrstQ^\fqcms$.
The quotient quasi-crystal monoid $\qcrstQ^\fqcms / {\hyco}$ is called the \dtgterm{hypoplactic quasi-crystal monoid}, or simply the \dtgterm{hypoplactic monoid}, associated to $\qcrstQ$, and is denoted by $\hypo (\qcrstQ)$.
\end{dfn}

Although $\hypo(\qcrstQ)$ is a quasi-crystal monoid, we are interested in studying its properties as a monoid, and thus, we just refer it as the hypoplactic monoid associated to $\qcrstQ$.
However, we will be constantly considering its quasi-crystal structure, as it plays a fundamental role in the construction of $\hypo(\qcrstQ)$, and consequently, in its properties.

This terminology will make more sense in the following section, where we see how the classical hypoplactic monoid can be placed in context as the hypoplactic monoid associated to the standard quasi-crystal of type $\tAn$.

In a hypoplactic monoid the converse of \itmcomboref{Proposition}{prop:qcmelemprop}{prop:qcmelempropcom} also holds, because the isolated elements (\comboref{Definition}{dfn:qcie}) are commutative, which is a consequence of the following result.

\begin{thm}
\label{thm:hypoiecom}
Let $\qcrstQ$ be a seminormal quasi-crystal,
and let $u, v \in Q^*$ be such that $uv$ is an isolated element of $\qcrstQ^\fqcms$.
Then,
\[ uvw \hyco uwv \hyco wuv, \]
for any $w \in Q^*$.
\end{thm}

\begin{proof}
By \comboref{Proposition}{prop:qcmdesc}, we have that
\[ \wt (uvw) = \wt (u) + \wt (v) + \wt (w) = \wt (uwv), \]
for any $w \in Q^*$.
Since $uv$ is isolated, we have that $\qKoe_i (uv) = \qKof_i (uv) = \undf$, for all $i \in I$.
Then, for each $i \in I$, either $\qKoec_i (uv) = \qKofc_i (uv) = 0$ or $\qKoec_i (uv) = \qKofc_i (uv) = +\infty$, because $\qcrstQ^\fqcms$ is seminormal.
Set $J = \set[\big]{ i \in I \given \qKoec_i (uv) = 0 }$.
By \comboref{Lemma}{lem:qcmqtpbo}, for any $i \in I \setminus J$, since $\qKoec_i (uv) = +\infty$, we have that $\qKoec_i (u) = +\infty$, $\qKoec_i (v) = +\infty$, or $\qKofc_i (u), \qKoec_i (v) \in \Z_{> 0}$.
This implies by \comboref{Proposition}{prop:qcmdesc} that, for any $w \in Q^*$,
\[ \qKoec_i (uvw) = \qKofc_i (uvw) = \qKoec_i (uwv) = \qKofc_i (uwv) = +\infty, \]
and so, $\qKoe_i$ and $\qKof_i$ are undefined on $uvw$ and on $uwv$.
By \comboref{Lemma}{lem:qcmqtpbo}, for any $j \in J$, we have that $\qKoec_j (u) = \qKoec_j (v) = \qKofc_j (u) = \qKofc_j (v) = 0$, because $\qKoec_j (uv) = \qKofc_j (uv) = 0$.
Then, by \comboref{Proposition}{prop:qcmdesc}, for any $w \in Q^*$, we get that
\begin{align*}
\qKoe_j (uvw) &= uv \qKoe_j (w),
& \qKof_j (uvw) &= uv \qKof_j (w),
& \qKoec_j (uvw) &= \qKoec_j (w),
& \qKofc_j (uvw) &= \qKofc_j (w),
\\
\qKoe_j (uwv) &= u \qKof_j (w) v,
& \qKof_j (uwv) &= u \qKof_j (w) v,
& \qKoec_j (uwv) &= \qKoec_j (w),
& \qKofc_j (uwv) &= \qKofc_j (w).
\end{align*}
In particular, given $w \in Q^*$ and $g_1, \ldots, g_m \in \set{ \qKoe_i, \qKof_i \given i \in I}$, we have that $g_1 \cdots g_m$ is defined on $uvw$ if and only if $g_1 \cdots g_m$ is defined on $uwv$ if and only if $g_1, \ldots, g_m \in \set{ \qKoe_j, \qKof_j \given j \in J }$ and $g_1 \cdots g_m$ is defined on $w$.

Let $w \in Q^*$. Define
\[ X = \set[\big]{ w' \in Q^* \given w' = g_1 \cdots g_m (w),\ \text{for some}\ g_1, \ldots, g_m \in \set{ \qKoe_j, \qKof_j \given j \in J } }. \]
Then,
\[ Q^* (uvw) = \set{ uvw' \given w' \in X } \]
and
\[ Q^* (uwv) = \set{ uw'v \given w' \in X }. \]
Also, the map $\psi : Q^* (uvw) \to Q^* (uwv)$ given by $\psi (uvw') = uw'v$, for each $w' \in X$, is a bijective quasi-crystal homomorphism from $\qcrstQ^\fqcms (uvw)$ to $\qcrstQ^\fqcms (uwv)$.
By \comboref{Corollary}{cor:snqciso}, $\psi$ is a quasi-crystal isomorphism between $\qcrstQ^\fqcms (uvw)$ and $\qcrstQ^\fqcms (uwv)$.
As $\psi (uvw) = uwv$, we obtain that $uvw \hyco uwv$.

The fact that $uwv \hyco wuv$ follows analogously.
\end{proof}

We now show that the converse of \itmcomboref{Proposition}{prop:qcmelemprop}{prop:qcmelempropidem} holds for hypoplactic monoids.
This leads to a characterization of the idempotent elements of a hypoplactic monoid.

\begin{thm}
\label{thm:hypoidem}
Let $\qcrstQ$ be a seminormal quasi-crystal,
and let $w \in Q^*$.
Then, $w \hyco w^2$ if and only if $w$ is an isolated element of $\qcrstQ^\fqcms$ and $\wt (w) = 0$.
\end{thm}

\begin{proof}
The direct implication follows from \itmcomboref{Proposition}{prop:qcmelemprop}{prop:qcmelempropidem}.
So, assume that $w$ is an isolated element of $\qcrstQ^\fqcms$ and $\wt (w) = 0$.
We get that
\[ \wt \parens[\big]{ w^2 } = \wt (w) + \wt (w) = 0. \]
Let $i \in I$.
As $\qcrstQ^\fqcms$ is seminormal and $w$ is isolated, we have that either $\qKoec_i (w) = \qKofc_i (w) = 0$ or $\qKoec_i (w) = \qKofc_i (w) = +\infty$.
If $\qKoec_i (w) = \qKofc_i (w) = 0$, then
\[
\qKoec_i \parens[\big]{ w^2 } = \qKoec_i (w) + \qKoec_i (w) = 0
\quad \text{and} \quad
\qKofc_i \parens[\big]{ w^2 } = \qKofc_i (w) + \qKofc_i (w) = 0,
\]
implying that $\qKoe_i$ and $\qKof_i$ are undefined on $w^2$.
Otherwise, $\qKoec_i (w) = \qKofc_i (w) = +\infty$, we get by \comboref{Lemma}{lem:qcmqtpbo} that
$\qKoec_i \parens[\big]{ w^2 } = \qKofc_i \parens[\big]{ w^2 } = +\infty$,
which implies that $\qKoe_i$ and $\qKof_i$ are undefined on $w^2$, by \itmcomboref{Definition}{dfn:qc}{dfn:qcpinfty}.
Hence, $w^2$ is isolated and the map $\psi : q^* (w) \to Q^* \parens[\big]{ w^2 }$, defined by $\psi (w) = w^2$, is a quasi-crystal isomorphism between $\qcrstQ^\fqcms (w)$ and $\qcrstQ^\fqcms \parens[\big]{ w^2 }$.
Therefore, $w \hyco w^2$.
\end{proof}

The following result is a direct consequence of \comboref{Theorems}{thm:hypoiecom} and\avoidrefbreak \ref{thm:hypoidem}.

\begin{cor}
\label{cor:hypoidemcom}
Let $\qcrstQ$ be a seminormal quasi-crystal.
In the hypoplactic monoid $\hypo (\qcrstQ)$, the idempotent elements commute.
\end{cor}

\section{Crystallizing the classical hypoplactic monoid}
\label{sec:crystclassicalhypo}

In this section we prove that the classical hypoplactic monoid $\hypon$ of rank $n$ arises as the hypoplactic monoid $\hypo(\qctAn)$ associated to the standard quasi-crystal $\qctAn$ of type $\tAn$.
This is accomplished by showing that the direct approach in\avoidcitebreak \cite{CM17crysthypo} can be placed in the context developped in the previous sections.

Recall that Kashiwara crystals\avoidcitebreak \cite{Kas90,Kas91} give rise to a plactic monoid anti-isomorphic to the original one\avoidcitebreak \cite{LS81}.
Since we introduced the quasi-tensor product of quasi-crystals (\comboref{Section}{sec:qcqtp}) based on the tensor product of crystals defined by Kashiwara, it is natural to expect the hypoplactic monoid obtained from quasi-crystals to be anti-isomorphic to the original one\avoidcitebreak \cite{KT97}.
Therefore, the results in this section concerning the classical hypoplactic monoid are adaptations of the original results.

\begin{dfn}
\label{dfn:classicalhypo}
Let $n \geq 1$.
The \dtgterm{classical hypoplactic monoid} $\hypon$ of rank $n$ is given by the presentation
$\pres{A_n \given \RpAn_1 \cup \RpAn_2  \cup \RhAn_3 \cup \RhAn_4}$
where
\begin{align*}
\RpAn_1 &= \set[\big]{ (yzx, yxz), (xzy, zxy) \given x < y < z },
\displaybreak[0]\\
\RpAn_2 &= \set[\big]{ (xyx, xxy), (xyy, yxy) \given x < y },
\displaybreak[0]\\
\RhAn_3 &= \set[\big]{ (x z t y, z x y t) \given t \leq x < y \leq z },
\displaybreak[0]\\
\shortintertext{and}
\RhAn_4 &= \set[\big]{ (y t z x, t y x z) \given t < x \leq y < z }.
\end{align*}
The \dtgterm{Knuth relations} consist of $\RpAn_1 \cup \RpAn_2$,
and the \dtgterm{quartic relations} consist of $\RhAn_3 \cup \RhAn_4$.
The \dtgterm{classical hypoplactic congruence} $\hycon$ is the monoid congruence on $A_n^*$ generated by $\RpAn_1 \cup \RpAn_2  \cup \RhAn_3 \cup \RhAn_4$.
\end{dfn}

The Knuth and quartic relations given above are respectively the reverse of the ones given in\avoidcitebreak \cite{Knu70} and\avoidcitebreak \cite{KT97}.
This is part of the adaptations we pointed out in the beginning of this section.

In the rest of this section, fix the root system associated to Cartan type $\tAn$.
The maps $\wt$, $\qKoe_i$, $\qKof_i$, $\qKoec_i$ and $\qKofc_i$, $i=1,\ldots,n-1$, always refer to the quasi-crystal structure of $\qctAn^\fqcms$,
and $\hyco$ always denote the hypoplactic congruence on $\qctAn^\fqcms$.

As we saw in \comboref{Example}{exa:fqcmAn}, for $w \in A_n^*$ and $i \in \set{1,\ldots,n-1}$, $\qKoe_i$ is defined on $w$ if and only if $w$ does not have an $i$-inversion and $\wlng{w}_{i+1} > 0$, and if so, $\qKoe_i (w)$ is obtained from $w$ by replacing the right-most symbol $i+1$ by $i$.
Also, $\qKof_i$ is defined on $w$ if and only if $w$ does not have an $i$-inversion and $\wlng{w}_{i} > 0$, and if so, $\qKof_i (w)$ is obtained from $w$ by replacing the left-most symbol $i$ by $i+1$.
Therefore, we can use the quasi-crystal structure of $\qctAn^\fqcms$ to construct a graph similar to the one in\avoidcitebreak \cite[\S~5]{CM17crysthypo}.

\begin{dfn}
\label{dfn:CMqcg}
Let $\Gamma'_n$ denote the $\Lambda$-weighted $\set{1,\ldots,n-1}$-labelled directed graph consisting of the vertex set $A_n^*$, the weight map of $\qctAn^\fqcms$, and for each $u, v \in A_n^*$ and $i \in \set{1,\ldots,n-1}$, an edge $u \lbedge{i} v$ whenever $\qKof_i (u) = v$.
\end{dfn}

Note that $\Gamma'_n$ can be obtained from the graph constructed in\avoidcitebreak \cite[\S~5]{CM17crysthypo} by reversing the words on each vertex.
This is one of the adaptations pointed out in the beginning of this section.

\begin{thm}[{\cite[Theorem~6.11]{CM17crysthypo}}]
\label{thm:crystchypo}
Let $u, v \in A_n^*$. Then, $u \hycon v$ if and only if there exists a (weight-preserving labelled directed) graph isomorphism $\psi$ between $\Gamma'_n (u)$ and $\Gamma'_n (v)$ such that $\psi(u) = v$.
\end{thm}

We now proceed to prove that $\hyco$ and $\hycon$ are the same relation on $A_n^*$.

\begin{lem}
\label{lem:qcgtAnCMqcg}
The graph $\Gamma'_n$ coincides with the graph that results from $\Gamma_{\qctAn^\fqcms}$ by removing all loops.
\end{lem}

\begin{proof}
From \comboref{Definitions}{dfn:qcg} and \avoidrefbreak \ref{dfn:CMqcg}, we have that the vertex sets and weight maps of $\Gamma'_n$ and $\Gamma_{\qctAn^\fqcms}$ coincide.
For $u, v \in A_n^*$ and $i \in \set{1,\ldots,n-1}$, if $\qKof_i (u) = v$, then $\wt(u) > \wt(v)$, by \comboref{Proposition}{prop:qcrlqKo}, which implies that $u \neq v$, and so, $\Gamma'_n$ is simple.
Finally, we have that $u \lbedge{i} v$ is an edge of $\Gamma'_n$ if and only if $\qKof_i (u) = v$ if and only if $u \lbedge{i} v$ is an edge of $\Gamma_{\qctAn^\fqcms}$ and $u \neq v$.
\end{proof}

\begin{lem}
\label{lem:uqchycoinv}
Let $u, v \in A_n^*$ be such that $u \hycon v$.
Then, for $i \in \set{1,\ldots,n-1}$, $u$ has an $i$-inversion if and only if $v$ has an $i$-inversion.
\end{lem}

\begin{proof}
Let $i \in \set{1,\ldots,n-1}$.
Suppose that $u$ has an $i$-inversion and $v$ does not have an $i$-inversion.
By \comboref{Theorem}{thm:crystchypo}, there exists a graph isomorphism $\psi$ between $\Gamma'_n (u)$ and $\Gamma'_n (v)$ such that $\psi(u) = v$.
Since $u$ has an $i$-inversion, we have that $\wlng{u}_i \geq 1$, and since $\psi$ preserves weights, $\wt(u) = \wt(v)$, which implies that $i$ occurs in $v$.
As $v$ does not have an $i$-inversion, $\qKof_i$ is defined on $v$, and so, $v \lbedge{i} \qKof_i (v)$ is an edge of $\Gamma'_n$.
Since $\psi$ is a graph isomorphism and $\psi(u) = v$, we get that $u \lbedge{i} \psi^{-1} \parens[\big]{\qKof_i (v)}$ is an edge of $\Gamma'_n$,
which is a contradiction, because $\qKof_i$ is undefined on $u$, as $u$ has an $i$-inversion. The other direction is similar.
\end{proof}

\begin{thm}
\label{thm:chypocghypoc}
Let $u, v \in A_n^*$.
Then, $u \hyco v$ if and only if $u \hycon v$.
Therefore, $\hypo(\qctAn)$ and $\hypon$ are isomorphic monoids.
\end{thm}

\begin{proof}
Assume that $u \hycon v$.
By \comboref{Definitions}{dfn:hyco}, there exists a quasi-crystal isomorphism $\psi : \qctAn^\fqcms (u) \to \qctAn^\fqcms (v)$ such that $\psi(u) = v$.
By \comboref{Proposition}{prop:qccciso}, $\psi$ is a graph isomorphism between $\Gamma_{\qctAn^\fqcms} (u)$ and $\Gamma_{\qctAn^\fqcms} (v)$.
By \comboref{Lemma}{lem:qcgtAnCMqcg}, $\Gamma'_n (u)$ and $\Gamma'_n (v)$ can be obtained from $\Gamma_{\qctAn^\fqcms} (u)$ and $\Gamma_{\qctAn^\fqcms} (v)$, respectively, by removing all loops.
Then, $\psi$ is a graph isomorphism between $\Gamma'_n (u)$ and $\Gamma'_n (v)$ satisfying $\psi(u) = v$, which implies by \comboref{Theorem}{thm:crystchypo} that $u \hycon v$.

Conversely, assume that $u \hycon v$.
By \comboref{Theorem}{thm:crystchypo}, there exists a graph isomorphism $\psi'$ between $\Gamma'_n (u)$ and $\Gamma'_n (v)$ such that $\psi' (u) = v$.
To prove that $\psi'$ is a graph isomorphism between $\Gamma_{\qctAn^\fqcms} (u)$ and $\Gamma_{\qctAn^\fqcms} (v)$, by \comboref{Lemma}{lem:qcgtAnCMqcg} it just remains to show that $\psi'$ and $(\psi')^{-1}$ preserve loops.
Let $w \in A_n^*$ and $i \in \set{1,\ldots,n-1}$ such that $w$ has an $i$-labelled loop in $\Gamma_{\qctAn^\fqcms} (u)$.
Then, $\qKoec_i (w) = +\infty$, which implies that $u$ has an $i$-inversion (see \comboref{Example}{exa:fqcmAn}).
By \comboref{Lemma}{lem:qcgtAnCMqcg}, $w$ lies in $\Gamma'_n (u)$,
and since $w \hycon \psi'(w)$ by \comboref{Theorem}{thm:crystchypo}, we get by \comboref{Lemma}{lem:uqchycoinv} that $\psi'(w)$ also has an $i$-inversion.
Hence, $\qKoec_i \parens[\big]{\psi'(w)} = +\infty$, which implies that $\psi'(w)$ has an $i$-labelled loop in $\Gamma_{\qctAn^\fqcms}$.
Analogously, if $w' \in A_n^*$ has an $i$-labelled loop in $\Gamma_{\qctAn^\fqcms} (v)$, then $(\psi')^{-1} (w')$ also has an $i$-labelled loop in $\Gamma_{\qctAn^\fqcms}$, because $(\psi')^{-1}$ is a graph isomorphism between $\Gamma'_n (v)$ and $\Gamma'_n (u)$.
Therefore, $\psi'$ is a graph isomorphism between $\Gamma_{\qctAn^\fqcms} (u)$ and $\Gamma_{\qctAn^\fqcms} (v)$
We have by \comboref{Theorem}{thm:snqcisoqcg} that $\psi'$ is a quasi-crystal isomorphism between $\qctAn^\fqcms (u)$ and $\qctAn^\fqcms (v)$ satisfying $\psi'(u) = v$, which implies that $u \hyco v$.
\end{proof}

The previous result justifies the term \dtgterm{hypoplactic} used in \comboref{Definitions}{dfn:hyco} and\avoidrefbreak \ref{dfn:hypoqcm}, since the classical hypoplactic monoid can be obtained as the hypoplactic monoid associated to the standard quasi-crystal $\qctAn$ of type $\tAn$.
Moreover, it shows that the theory of quasi-crystals presented in \comboref{Sections}{sec:qch} to\avoidrefbreak \ref{sec:hyco} gives rise to a genuine generalization of the classical hypoplactic monoid.
We now have a process of crystallizing the hypoplactic monoid which allows the construction of the classical hypoplactic monoid from the standard quasi-crystal $\qctAn$ of type $\tAn$, and allows the analogous construction of a monoid based on any other seminormal quasi-crystal.

\section{\texorpdfstring%
  {The hypoplactic monoid of type $\tCn$}%
  {The hypoplactic monoid of type C\_n}}
\label{sec:hypotCn}

We described in \comboref{Section}{sec:hyco} a method of obtaining a monoid from a seminormal quasi-crystal.
We then showed in \comboref{Section}{sec:crystclassicalhypo} that for the standard quasi-crystal $\crtAn$ of type $\tAn$ it results in the classical hypoplactic monoid of rank $n$.
A natural way of proceeding is to study the monoids that are obtained for other quasi-crystals.

In this section, we make a detailed study of the hypoplactic monoid $\hypo(\qctCn)$.
We start in \comboref{Subsection}{subsec:hypotCndef} by presenting a description of the free quasi-crystal monoid $\qctCn^\fqcms$ over $\qctCn$, from which $\hypo (\qctCn)$ emerges.
In \comboref{Subsections}{subsec:hypotCnhww} and\avoidrefbreak \ref{subsec:hypotCniw}, we characterize the highest-weight and isolated words of $\qctCn^\fqcms$, which allows us to identify the commutative and idempotent elements of $\hypo (\qctCn)$.
In \comboref{Subsection}{subsec:hypotCnrel}, we investigate whether $\hypo (\qctCn)$ satisfies some well-known relations, such as the Knuth relations.
In \comboref{Subsection}{subsec:hypotCnident}, we show that $\hypo(\qctC_2)$ satisfies nontrivial identities, and describe some of the properties any such identity must have.
We also show that $\hypo (\qctCn)$ does not satisfy nontrivial identities, for $n \geq 3$.
In \comboref{Subsection}{subsec:hypotCnpres}, we prove that $\hypo(\qctCn)$ does not admit a finite presentation, but we identify the connected components of $\qctC_2^\fqcms$ up to isomorphism, leading to a class of representatives for the elements of $\hypo (\qctC_2)$.
Finally, in \comboref{Subsections}{subsec:hypotCnhAn1tohCn} and\avoidrefbreak \ref{subsec:hypotCnhCn1tohCn} we describe monoid embeddings of $\hypo (\qctA_{n-1})$ and $\hypo (\qctC_{n-1})$ into $\hypo (\qctCn)$.

\subsection{\texorpdfstring%
  {The definition of $\hypo(\qctCn)$}%
  {The definition of hypo(C\_n)}}
\label{subsec:hypotCndef}

From \itmcomboref{Examples}{exa:qc}{exa:qctCn} and\itmcomboref{}{exa:qcg}{exa:qcgtCn}, we have that the standard quasi-crystal $\qctCn$ of type $\tCn$ is a seminormal quasi-crystal consisting of an ordered set
\[ C_n = \set[\big]{ 1 < 2 < \cdots < n < \wbar{n} < \wbar{n-1} < \cdots < \wbar{1} }. \]
Its quasi-crystal graph is
\[ 1 \lbedge{1} 2 \lbedge{2} \cdots \lbedge{n-1} n \lbedge{n} \wbar{n} \lbedge{n-1} \wbar{n-1} \lbedge{n-2} \cdots \lbedge{1} \wbar{1}, \]
where the weight map ${\wt} : C_n \to \Z^n$ is defined by
\[
\wt (x) = \vc{e_x}
\quad \text{and} \quad
\wt (\wbar{x}) = -\vc{e_x},
\]
for $x \in \set{1,\ldots,n}$.


Notice that for $x, y \in C_n$ and $i \in \set{1,\ldots,n-1}$, if $\qKofc_i (x) > 0$ and $\qKoec_i (y) > 0$, then $x \in \set[\big]{ i, \wbar{i+1} }$ and $y \in \set[\big]{ i+1, \wbar{i} }$.
If $\qKofc_n (x) > 0$ and $\qKoec_n (y) > 0$, then $x=n$ and $y=\wbar{n}$.

To avoid constant division into cases where $i \neq n$ and $i = n$, for brevity, in the rest of this section we formally consider $n+1$ and $\wbar{n+1}$ to be symbols that never appear in any word.
Thus, the observation in the previous paragraph can be simply re-stated as follows: for any $i \in \set{1,\ldots,n}$, if $\qKofc_i (x) > 0$ and $\qKoec_i (y) > 0$, then $x \in \set[\big]{ i, \wbar{i+1} }$ and $y \in \set[\big]{ i+1, \wbar{i} }$.
This leads to the concept of an $i$-inversion for words over the alphabet $C_n$.

\begin{dfn}
\label{dfn:iinverseCn}
Let $i \in \set{1,\ldots,n}$.
A word $w \in C_n^*$ is said to have an \dtgterm{$i$-inversion} if $w$ admits a decomposition of the form $w = w_1 x w_2 y w_3$, for some $w_1, w_2, w_3 \in C_n^*$, $x \in \set[\big]{ i, \wbar{i+1} }$ and $y \in \set[\big]{ i+1, \wbar{i} }$.

A word $w \in C_n^*$ is said to be \dtgterm{$i$-inversion-free} if $w$ does not have an $i$-inversion.
\end{dfn}

By \comboref{Definition}{dfn:qcm} and \comboref{Proposition}{prop:qcmdesc}, we obtain the following description of the free quasi-crystal monoid $\qctCn^\fqcms$ over $\qctCn$.

\begin{dfn}
\label{dfn:fqcmtCn}
The free quasi-crystal monoid $\qctCn^\fqcms$ over $\qctCn$ consists of the set $C_n^*$ of all words over $C_n$, under the operation of concatenation of words, and a quasi-crystal structure given as follows.
For $w \in C_n^*$, the weight of $w$ is
\[ \wt(w) = \parens[\big]{ \wlng{w}_{1} - \wlng{w}_{\wbar{1}}, \wlng{w}_{2} - \wlng{w}_{\wbar{2}}, \ldots, \wlng{w}_{n} - \wlng{w}_{\wbar{n}} }. \]

For $i \in \set{1,\ldots,n}$, if $w$ has an $i$-inversion, then
\[ \qKoec_i (w) = \qKofc_i (w) = +\infty, \]
otherwise,
\[
\qKoec_i (w) = \wlng{w}_{i+1} + \wlng{w}_{\wbar{i}}
\quad \text{and} \quad
\qKofc_i (w) = \wlng{w}_{i} + \wlng{w}_{\wbar{i+1}}.
\]



The raising quasi-Kashiwara operator $\qKoe_i$ is defined on $w$ if and only if $\qKoec_i (w) \in \Z_{> 0}$.
The lowering quasi-Kashiwara operator $\qKof_i$ is defined on $w$ if and only if $\qKofc_i (w) \in \Z_{> 0}$.
When they are defined, the quasi-Kashiwara operators can be computed (as in \comboref{Proposition}{prop:qcmdesc}) as follows:
Let $i \in \set{1,\ldots,n}$,
and let $w \in C_n^*$ be $i$-inversion-free.
If $i+1$ or $\wbar{i}$ occurs in $w$, let $x$ be the right-most $i+1$ or $\wbar{i}$ occurring in $w$, then $\qKoe_i (w)$ is obtained from $w$ by replacing $x$ by $\qKoe_i (x)$.
If $i$ or $\wbar{i+1}$ occurs in $w$, let $y$ be the left-most $i$ or $\wbar{i+1}$ occurring in $w$, then $\qKof_i (w)$ is obtained from $w$ by replacing $y$ by $\qKof_i (y)$.
\end{dfn}

\begin{exa}
Consider $n=5$.
Take
$w = 2 \wbar{5} 3 5 5 \wbar{3} 2 \wbar{3} 4$.
We have that
$\wt (w) = (0, 2, {-1}, 1, 1)$.
Since $w$ admits a decomposition of the form $w = w_1 2 w_2 3 w_3$, we get that $w$ has a $2$-inversion.
It has a $3$-inversion, as it admits a decomposition of the form $w = w_1 3 w_2 \wbar{3} w_3$.
It has a $4$-inversion, as it admits a decomposition of the form $w = w_1 \wbar{5} w_2 5 w_3$.
Therefore, for $i \in \set{2, 3, 4}$, the quasi-Kashiwara operators $\qKoe_i$ and $\qKof_i$ are undefined on $w$.
On the other hand, $w$ is $1$-inversion-free and $5$-inversion-free.
We have that $\qKof_1$ is undefined on $w$, as neither $1$ nor $\wbar{2}$ occurs in $w$,
$\qKoe_1 (w) = 2 \wbar{5} 3 5 5 \wbar{3} 1 \wbar{3} 4$,
$\qKoe_5 (w) = 2 5 3 5 5 \wbar{3} 2 \wbar{3} 4$,
and
$\qKof_5 (w) = 2 \wbar{5} 3 \wbar{5} 5 \wbar{3} 2 \wbar{3} 4$.
\end{exa}

In \comboref{Section}{sec:qcg}, we showed that a seminormal quasi-crystal can be described by its quasi-crystal graph.
Thus, to study the hypoplactic monoid $\hypo(\qctCn)$, we will frequently resort to the quasi-crystal graph $\Gamma_{\qctCn^\fqcms}$ of $\qctCn^\fqcms$

\begin{exa}
\label{exa:qcgfqcmtCn}
The empty word $\ew$ is an isolated vertex in $\Gamma_{\qctCn^\fqcms}$ without loops.
The set of letters $C_n$ forms a connected component which is described in \itmcomboref{Example}{exa:qcg}{exa:qcgtCn}.
We now turn our attention to the case $n=2$.
Words of length $2$ form a subgraph of $\Gamma_{\qctC_2^\fqcms}$ that is isomorphic to the quasi-crystal graph of $\qctC_2 \dotimes \qctC_2$, which is described in \itmcomboref{Example}{exa:qcqtp}{exa:qcqtpC22}.
The connected component $\Gamma_{\qctC_2^\fqcms} (121)$ of $\Gamma_{\qctC_2^\fqcms}$ containing $121$ is the following:
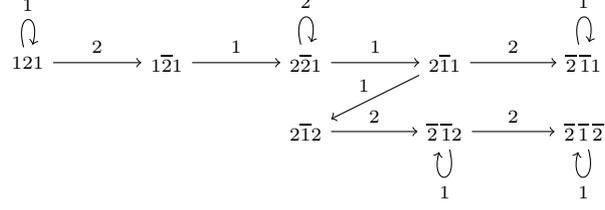
\[
\begin{tikzpicture}[widecrystal,baseline=(b2b1b2.base)]
  %
  \node (121) at (1, 2) {121};
  \node (1b21) at (2, 2) {1\wbar{2}1};
  \node (2b21) at (3, 2) {2\wbar{2}1};
  \node (2b11) at (4, 2) {2\wbar{1}1};
  \node (b2b11) at (5, 2) {\wbar{2}\, \wbar{1}1};
  \node (2b12) at (3, 1) {2\wbar{1}2};
  \node (b2b12) at (4, 1) {\wbar{2}\, \wbar{1}2};
  \node (b2b1b2) at (5, 1) {\wbar{2}\, \wbar{1}\, \wbar{2}};
  \path (121) edge node {2} (1b21)
        (1b21) edge node {1} (2b21)
        (2b21) edge node {1} (2b11)
        (2b11) edge [swap] node[inner sep=2pt] {1} (2b12)
        (2b11) edge node {2} (b2b11)
        (2b12) edge node {2} (b2b12)
        (b2b12) edge node {2} (b2b1b2)
        (121) edge [loop above] node {1} ()
        (2b21) edge [loop above] node {2} ()
        (b2b11) edge [loop above] node {1} ()
        (b2b12) edge [loop below] node {1} ()
        (b2b1b2) edge [loop below] node {1} ();
\end{tikzpicture}
\]
The connected component $\Gamma_{\qctC_2^\fqcms} (212)$ of $\Gamma_{\qctC_2^\fqcms}$ containing $212$ is the following:
\[
\begin{tikzpicture}[widecrystal,baseline=(b1b2b1.base)]
  %
  \node (212) at (1, 2) {212};
  \node (b212) at (2, 2) {\wbar{2}12};
  \node (b21b2) at (3, 2) {\wbar{2}1\wbar{2}};
  \node (b112) at (1, 1) {\wbar{1}12};
  \node (b11b2) at (2, 1) {\wbar{1} 1 \wbar{2}};
  \node (b12b2) at (3, 1) {\wbar{1}2\wbar{2}};
  \node (b12b1) at (4, 1) {\wbar{1}2\wbar{1}};
  \node (b1b2b1) at (5, 1) {\wbar{1}\, \wbar{2}\, \wbar{1}};
  \path (212) edge node {2} (b212)
        (b212) edge node {2} (b21b2)
        (b21b2) edge [swap] node[inner sep=2pt] {1} (b11b2)
        (b112) edge node {2} (b11b2)
        (b11b2) edge node {1} (b12b2)
        (b12b2) edge node {1} (b12b1)
        (b12b1) edge node {2} (b1b2b1)
        (212) edge [loop above] node {1} ()
        (b212) edge [loop above] node {1} ()
        (b112) edge [loop below] node {1} ()
        (b12b2) edge [loop below] node {2} ()
        (b1b2b1) edge [loop below] node {1} ();
\end{tikzpicture}
\]
\end{exa}

By \comboref{Definitions}{dfn:hyco} and\avoidrefbreak \ref{dfn:hypoqcm}, the hypoplactic monoid $\hypo(\qctCn)$ is the quotient monoid of $C_n^*$ by the hypoplactic congruence $\hyco$.
Although we omitted the weight map $\wt$ in the example above, recall that $\Gamma_{\qctCn^\fqcms}$ is a weighted labelled graph, from which $\hyco$ can be obtained, since, for two words $u, v \in C_n^*$, we have by \comboref{Theorem}{thm:snqcisoqcg} that $u \hyco v$ if and only if there exists a graph isomorphism $\psi$ between the connected components $\Gamma_{\qctCn^\fqcms} (u)$ and $\Gamma_{\qctCn^\fqcms} (v)$ of $\Gamma_{\qctCn^\fqcms}$ such that $\psi(u) = v$.

The following result shows that the hypoplactic congruence $\hyco$ respects inversions.

\begin{prop}
\label{prop:hypotCniinversionresp}
Let $u, v \in C_n^*$ with $u \hyco v$,
and let $i \in \set{1,\ldots,n}$.
Then, $u$ has an $i$-inversion if and only if $v$ has an $i$-inversion.
\end{prop}

\begin{proof}
If $u$ has an $i$-inversion, then $\qKoec_i (u) = +\infty$.
Since $u \hyco v$, then $\qKoec_i (v) = \qKoec_i (u) = +\infty$, which implies that $v$ has an $i$-inversion.
The converse follows from the fact that $u \hyco v$ implies $v \hyco u$.
\end{proof}

This result is analogous to one obtained for the classical hypoplactic monoid in \comboref{Lemma}{lem:uqchycoinv}, where, for each $i \in \set{1,\ldots,n-1}$, either all words in a congruence class of the hypoplactic congruence have an $i$-inversion, or all of them are $i$-inversion-free.

From \comboref{Definition}{dfn:CMqcg}, \comboref{Lemma}{lem:qcgtAnCMqcg} and\avoidcitebreak \cite[Proposition~5.2]{CM17crysthypo}, we can deduce a construction of the quasi-crystal graph $\Gamma_{\qctAn^\fqcms}$ from the crystal graph over $A_n^*$ of type $\tAn$.
An analogous construction of the quasi-crystal graph $\Gamma_{\qctCn^\fqcms}$ from the crystal graph over $C_n^*$ of type $\tCn$\avoidcitebreak \cite{KN94,Lec02} can also be given.
First, note that the weight maps coincide.
By \comboref{Definition}{dfn:fqcm}, we have that when the quasi-Kashiwara operators are defined, they coincide with the Kashiwara operators.
Thus, the quasi-crystal graph $\Gamma_{\qctCn^\fqcms}$ is obtained from the crystal graph over $C_n^*$ of type $\tCn$ by deleting all $i$-labelled edges starting or ending on a word with an $i$-inversion, and then, adding $i$-labelled loops on all words with an $i$-inversion, for $i=1,\ldots,n$.

The empty word $\ew$ and the word $1\wbar{1}$ are related by the plactic congruence on $C_n^*$\avoidcitebreak \cite{Lec02}.
Since $1 \wbar{1}$ has a $1$-inversion, we get by \comboref{Proposition}{prop:hypotCniinversionresp} that $1 \wbar{1} \nhyco \ew$.
Hence, the plactic congruence on $C_n^*$ is not contained in the hypoplactic congruence on $C_n^*$.
This contrasts with the well-known result for type $\tAn$, where the plactic congruence on $A_n^*$ is contained in the hypoplactic congruence on $A_n^*$.

Finally, notice that a word $w$ in $C_n^*$ may have unbarred symbols and barred symbols. For the sake of simplicity, we introduce the following notation.

\begin{dfn}
\label{dfn:wbarCn}
Set $\wbar{\ew} = \ew$.
For each $x \in \set{1,\ldots,n}$, set $\wbar{\wbar{x}} = x$.
Given a word $w = x_1 x_2 \ldots x_m$, with $x_1, x_2, \ldots, x_m \in C_n$, set $\wbar{w} = \wbar{x_m}\,\wbar{x_{m-1}} \ldots \wbar{x_1}$.
\end{dfn}

The following result shows that this notation preserves inversions.

\begin{lem}
\label{lem:hypotCniinversionbar}
Let $w \in C_n^*$ and $i \in \set{1,\ldots,n}$.
Then, $w$ has an $i$-inversion if and only if $\wbar{w}$ has an $i$-inversion.
\end{lem}

\begin{proof}
If $w$ has an $i$-inversion, then $w = w_1 x w_2 y w_3$, for some $w_1, w_2, w_3 \in C_n^*$, $x \in \set[\big]{ i, \wbar{i+1} }$ and $y \in \set[\big]{ i+1, \wbar{i} }$.
Thus, $\wbar{w} = \wbar{w_3}\,\wbar{y}\,\wbar{w_2}\,\wbar{x}\,\wbar{w_1}$, where $\wbar{y} \in \set[\big]{ \wbar{i+1}, i }$ and $\wbar{x} \in \set[\big]{ \wbar{i}, i+1 }$, and so, $\wbar{w}$ has an $i$-inversion.

The converse is immediate since $\wbar{\wbar{w}} = w$.
\end{proof}

In the quasi-crystal monoid $\qctCn^\fqcms$ we have the following relation between a word $w \in C_n^*$ and its barred version $\wbar{w}$.

\begin{prop}
\label{prop:hypotCnbarqcs}
Let $w \in C_n^*$,
and let $i \in \set{1,\ldots,n}$.
Then, $\wt (\wbar{w}) = -\wt(w)$, $\qKoec_i (\wbar{w}) = \qKofc_i (w)$, and $\qKofc_i (\wbar{w}) = \qKoec_i (w)$.
Furthermore, if $\qKoe_i (\wbar{w})$ or $\qKof_i (w)$ are defined, then $\qKoe_i (\wbar{w}) = \wbar{\qKof_i(w)}$,
and if $\qKof_i (\wbar{w})$ or $\qKoe_i (w)$ are defined, then $\qKof_i (\wbar{w}) = \wbar{\qKoe_i (w)}$.
\end{prop}

\begin{proof}
From \comboref{Definition}{dfn:fqcmtCn}, we have that
\[\begin{split}
\wt (\wbar{w})
                      & = \parens[\big]{ \wlng{\wbar{w}}_{1} - \wlng{\wbar{w}}_{\wbar{1}}, \ldots, \wlng{\wbar{w}}_{n} - \wlng{\wbar{w}}_{\wbar{n}} }\\
                      & = \parens[\big]{ \wlng{w}_{\wbar{1}} - \wlng{w}_{1}, \ldots, \wlng{w}_{\wbar{n}} - \wlng{w}_{n} }
= -\wt (w).
\end{split}\]
By \comboref{Lemma}{lem:hypotCniinversionbar}, we get that $\qKoec_i (\wbar{w}) = \qKofc_i (\wbar{w}) = +\infty$ if and only if $\qKoec_i (w) = \qKofc_i (w) = +\infty$.
And if so, the result follows trivially.
Thus, assume that neither $w$ nor $\wbar{w}$ has an $i$-inversion.
For $i \neq n$, we have that
\[
  \qKoec_i (\wbar{w}) = \wlng{\wbar{w}}_{i+1} + \wlng{\wbar{w}}_{\wbar{i}} = \wlng{w}_{\wbar{i+1}} + \wlng{w}_{i} = \qKofc_i (w);
\]
furthermore, $\qKoec_n (\wbar{w}) = \wlng{\wbar{w}}_{\wbar{n}} = \wlng{w}_n = \qKofc_n (w)$.

As $\qctCn^\fqcms$ is seminormal, $\qKoe_i (\wbar{w})$ is defined if and only if $\qKof_i (w)$ is defined.
If so, then there exist $w_1, w_2 \in C_n^*$ and $x \in \set[\big]{ i, \wbar{i+1} }$ such that $\wlng{w_1}_{i} = \wlng{w_1}_{\wbar{i+1}} = 0$, $w = w_1 x w_2$, and $\qKof_i (w) = w_1 \qKof_i (x) w_2$.
Since $\wbar{w} = \wbar{w_2}\,\wbar{x}\,\wbar{w_1}$, where $\wbar{x} \in \set[\big]{ i+1, \wbar{i} }$ and $\wlng{\wbar{w_1}}_{i+1} = \wlng{\wbar{w_1}}_{\wbar{i}} = 0$, we get that $\qKoe_i (\wbar{w}) = \wbar{w_2} \qKoe_i (\wbar{x}) \wbar{w_1}$.
Note that if $x = i$, then
$\qKoe_i (\wbar{x}) = \wbar{i+1} = \wbar{\qKof_i (x)}$,
and otherwise,
$x = \wbar{i+1}$
and
$\qKoe_i (\wbar{x}) = i = \wbar{\qKof_i (x)}$.
Hence, $\qKoe_i (\wbar{w}) = \wbar{\qKof_i (w)}$.
Analogously, $\qKof_i (\wbar{w}) = \wbar{\qKoe_i (w)}$, whenever $\qKof_i (\wbar{w})$ or $\qKoe_i (w)$ are defined.
\end{proof}

The following results are straightforward consequences of the previous result:

\begin{cor}
\label{cor:fqcmtCnccbar}
Given $u, v \in C_n^*$, there is an edge $u \lbedge{i} v$ in the quasi-crystal graph $\Gamma_{\qctCn^\fqcms}$ if and only if there is an edge $\wbar{v} \lbedge{i} \wbar{u}$.
Thus, for any $w \in C_n^*$,
\[ C_n^* (\wbar{w}) = \set[\big]{ \wbar{u} \given u \in C_n (w) }. \]
\end{cor}

\begin{cor}
\label{cor:hypotCnuvbarubarv}
Let $u, v \in C_n^*$.
Then, $u \hyco v$ if and only if $\wbar{u} \hyco \wbar{v}$.
\end{cor}

\subsection{Highest-weight words}
\label{subsec:hypotCnhww}

In the study of plactic monoids for the infinite Cartan types\avoidcitebreak \cite{Lec02,Lec03}, words of highest weight are extremely relevant as they are used to index connected components of crystal graphs.
An analogous relation was proven in\avoidcitebreak \cite{CM17crysthypo} for the classical hypoplactic monoid.
Thus, we now characterize the highest-weight words of $\qctCn^\fqcms$, and check whether they satisfy properties similar to highest-weight words in the mentioned contexts.

From \itmcomboref{Definition}{dfn:qchlwe}{dfn:qchlweh}, we have that a word $w \in C_n^*$ is of highest weight if $\qKoe_i$ is undefined on $w$, for all $i \in \set{1,\ldots,n}$.
Equivalently, $w$ is of highest weight if the only edges in $\Gamma_{\qctCn^\fqcms}$ ending on $w$ are loops.

\begin{prop}
\label{prop:fqcmtCnhww}
Let $w \in C_n^*$.
Then, $w$ is of highest weight if and only if for each letter $x \in C_n$ occurring in $w$, the following conditions are satisfied:
\begin{enumerate}
\item if $x \in \set{2,\ldots,n}$, then $w$ has an $(x-1)$-inversion;

\item if $x \in \set[\big]{\wbar{n},\ldots,\wbar{1}}$, then $w$ has an $\wbar{x}$-inversion.
\end{enumerate}
\end{prop}

\begin{proof}
Suppose that $w$ is of highest weight.
Let $x \in C_n$ be a letter occurring in $w$.
If $x \in \set{2,\ldots,n}$, then $\qKoec_{x-1} (w) > 0$, and since $\qctCn^\fqcms$ is seminormal and $\qKoe_{x-1}$ is undefined on $w$, we get that $\qKoec_{x-1} (w) = +\infty$, or equivalently, $w$ has an $(x-1)$-inversion.
If $x \in \set[\big]{\wbar{n},\ldots,\wbar{1}}$, then $\qKoec_{\wbar{x}} (w) > 0$, which implies as in the previous case that $w$ has an $\wbar{x}$-inversion.

Conversely, suppose that $w$ is not of highest weight.
Then, take $i \in \set{1,\ldots,n}$ such that $\qKoe_i$ is defined on $w$.
By \comboref{Definition}{dfn:fqcmtCn}, we have that $\qKoe_i (w) = w_1 \qKoe_i (x) w_2$, for some $w_1, w_2 \in C_n^*$ and $x \set[\big]{ i+1, \wbar{i} }$.
Therefore, the letter $i+1$ or the letter $\wbar{i}$ occurs in $w$, and $w$ does not have an $i$-inversion.
\end{proof}

\begin{exa}
\label{exa:fqcmtC4hww}
Consider $n=4$.
In $\qctC_4^\fqcms$, the following words are of highest weight:
$1$,
$12$,
$1 \wbar{1}$,
$3 \wbar{3} \wbar{2}$,
$3 \wbar{3} 3$,
and
$1 2 3 4 \wbar{4}\,\wbar{3}\,\wbar{2}\,\wbar{1}$.
\end{exa}

In the crystal graphs studied in\avoidcitebreak \cite{KN94}, which led to the construction of the plactic monoids for the infinite Cartan types\avoidcitebreak \cite{Lec02,Lec03}, each connected component has exactly one highest-weight element and exactly one lowest-weight element.
Due to the results in\avoidcitebreak \cite{CM17crysthypo} and in \comboref{Section}{sec:crystclassicalhypo}, we also have that the connected components of the free quasi-crystal monoid $\qctAn^\fqcms$ have exactly one highest-weight word and exactly one lowest-weight word.
The free quasi-crystal monoid $\qctCn^\fqcms$ does not have this property,
as we can see in \comboref{Example}{exa:qcgfqcmtCn} that $212$ and $\wbar{1}12$ are highest-weight words of $\qctC_2^\fqcms$ which belong to the same connected component.
The same happens in $\qctCn^\fqcms$ for any $n \geq 2$, because
\begin{align*}
  & \qKoe_2 \qKof_1 \qKof_2 \qKof_2 \qKof_3 \cdots \qKof_{n-1} \qKof_n \qKof_{n-1} \cdots \qKof_2 (212) \\
  &\qquad {}= \qKoe_2 \qKof_1 \qKof_2 \qKof_2 \qKof_3 \cdots \qKof_{n-1} \qKof_n (n12) \displaybreak[0]\\
  &\qquad {}= \qKoe_2 \qKof_1 \qKof_2 \qKof_2 \qKof_3 \cdots \qKof_{n-1} (\wbar{n}12) \displaybreak[0]\\
  &\qquad {}= \qKoe_2 \qKof_1 \qKof_2 \parens[\big]{ \wbar{2}12 } \displaybreak[0]\\
  &\qquad {}= \qKoe_2 \qKof_1 \parens[\big]{ \wbar{2}1 \qKof_2 (2) } \displaybreak[0]\\
  &\qquad {}= \qKoe_2 \parens[\big]{ \wbar{1}1 \qKof_2 (2) } \\
  &\qquad {}= \wbar{1}12.
\end{align*}
We can also see that $\wbar{2}\,\wbar{1}1$ and $\wbar{2}\,\wbar{1}\,\wbar{2}$ are lowest-weight words of $\qctCn^\fqcms$ which belong to the same connected component.
Although we have that a connected component of $\qctCn^\fqcms$ may have more than one highest-weight word or more then one lowest-weight word, we can guarantee by \comboref{Proposition}{prop:fqcmcc} that it has at least one of each.
Moreover, in the following result we describe a one-to-one correspondence between highest-weight words and lowest-weight words of $\qctCn^\fqcms$.

\begin{prop}
\label{prop:fqcmtCnhwwlww}
Let $w \in C_n^*$.
Then, $w$ is of highest weight if and only if $\wbar{w}$ is of lowest weight.
Also, $w$ is of lowest weight if and only if $\wbar{w}$ is of highest weight.
\end{prop}

\begin{proof}
For any $i \in \set{1,\ldots,n}$, we have by \comboref{Corollary}{cor:fqcmtCnccbar} that $\qKoe_i$ (or $\qKof_i$) is defined on $w$ if and only if $\qKof_i$ (resp., $\qKoe_i$) is defined on $\wbar{w}$.
This implies that $w$ is of highest (resp., lowest) weight if and only if $\wbar{w}$ is of lowest (resp., highest) weight.
\end{proof}

From \comboref{Example}{exa:fqcmtC4hww}, we have that $1 2 3 4 \wbar{4}\,\wbar{3}\,\wbar{2}\,\wbar{1}$ is a highest-weight word in $\qctC_4^\fqcms$.
Also, $\wt \parens[\big]{1 2 3 4 \wbar{4}\,\wbar{3}\,\wbar{2}\,\wbar{1}} = 0$ and $\qKoec_i \parens[\big]{1 2 3 4 \wbar{4}\,\wbar{3}\,\wbar{2}\,\wbar{1}} = +\infty$, for all $i \in \set{1,2,3,4}$.
Thus, for any $w \in C_4^*$, we get that $1 2 3 4 \wbar{4}\,\wbar{3}\,\wbar{2}\,\wbar{1} w$ is a highest-weight word with weight $\wt (w)$.
In the following result, we generalize this reasoning, from which we can see that highest weights (\comboref{Definition}{dfn:qchlw}) in $\qctCn^\fqcms$ do not identify a relevant subset of weights.

\begin{prop}
\label{prop:fqcmtCnallhw}
Any element $\lambda \in \Z^n$ is a highest weight in $\qctCn^\fqcms$.
\end{prop}

\begin{proof}
Let $\lambda = (\lambda_1, \ldots, \lambda_n) \in \Z^n$.
For each $i \in \set{1,\ldots,n}$, if $\lambda_i \geq 0$, set $a_i = \lambda_i + 1$ and $b_i = 1$, otherwise, set $a_i = 1$ and $b_i = -\lambda_i + 1$.
The word
\[ w = 1^{a_1} 2^{a_2} \ldots n^{a_n} \wbar{n}^{b_n} \wbar{n-1}^{b_{n-1}} \ldots \wbar{1}^{b_1} \]
is such that
\[ \wt (w) = (a_1 - b_1, a_2 - b_2, \ldots, a_n - b_n) = \lambda \]
and $\qKoec_i (w) = +\infty$, for any $i \in \set{1,\ldots,n}$, because $w$ has a decomposition of the form $w = w_1 i w_2 \wbar{i} w_3$, for some $w_1, w_2, w_3 \in C_n^*$.
Hence, $w$ is a highest-weight word with weight $\lambda$, which implies that $\lambda$ is a highest weight.
\end{proof}

The previous result together with \comboref{Propositions}{prop:hypotCnbarqcs} and\avoidrefbreak \ref{prop:fqcmtCnhwwlww} implies that any element $\lambda \in \Z^n$ is a lowest weight in $\qctCn^\fqcms$.

\subsection{Isolated words}
\label{subsec:hypotCniw}

By \comboref{Proposition}{prop:qcmelemprop} and \comboref{Theorem}{thm:hypoiecom}, we have that the commutative elements of the hypoplactic monoid $\hypo(\qctCn)$ correspond to the hypoplactic congruence classes of isolated words of $\qctCn^\fqcms$.
By \comboref{Theorem}{thm:hypoidem}, we also have that the idempotent elements of the hypoplactic monoid $\hypo(\qctCn)$ correspond to the hypoplactic congruence classes of isolated words of $\qctCn^\fqcms$ with weight $0$.
Therefore, we now turn our attention to characterizing the isolated words in $\qctCn^\fqcms$, and consequently, obtain some relations in $\hypo(\qctCn)$.

By \comboref{Definition}{dfn:qcie}, a word $w \in C_n^*$ is isolated if it is an isolated vertex in the quasi-crystal graph $\Gamma_{\qctCn^\fqcms}$.
In other words, $w$ is isolated if and only if it is both of highest and of lowest weight in $\qctCn^\fqcms$.

\begin{prop}
\label{prop:fqcmtCniwbar}
Let $w \in C_n^*$.
Then, $w$ is an isolated word if and only if $\wbar{w}$ is an isolated word.
Also, $w$ is an isolated word if and only if both $w$ and $\wbar{w}$ are of highest weight.
\end{prop}

\begin{proof}
By \comboref{Proposition}{prop:fqcmtCnhwwlww}, $w$ is of highest and of lowest weight if and only if $\wbar{w}$ is of highest and of lowest weight.
Also, $w$ and $\wbar{w}$ are of highest weight if and only if $w$ and $\wbar{w}$ are of lowest weight.
And so, the result follows.
\end{proof}

\begin{prop}
\label{prop:hypoCniwiinv}
Let $w \in C_n^*$. Then, $w$ is an isolated word if and only if both of the following conditions hold:
\begin{enumerate}
\item $w$ has a $1$-inversion if $1$ or $\wbar{1}$ occurs in $w$;

\item $w$ has an $(i-1)$-inversion and an $i$-inversion, for all $i \in \set{2,\ldots,n}$ such that $i$ or $\wbar{i}$ occurs in $w$.
\end{enumerate}
\end{prop}

\begin{proof}
Suppose that $w$ is an isolated word.
By \comboref{Proposition}{prop:fqcmtCniwbar}, $w$ and $\wbar{w}$ are of highest weight.
Let $i \in \set{1,\ldots,n}$.
If $i$ occurs in $w$, or equivalently, $\wbar{i}$ occurs in $\wbar{w}$, then we have by \comboref{Proposition}{prop:fqcmtCnhww} that $w$ has an $(i-1)$-inversion when $i \geq 2$, and that $\wbar{w}$ has an $i$-inversion which implies by \comboref{Lemma}{lem:hypotCniinversionbar} that $w$ has an $i$-inversion.
Analogously, if $\wbar{i}$ occurs in $w$, or equivalently, $i$ occurs in $\wbar{w}$, then $w$ has an $(i-1)$-inversion and an $i$-inversion.

Conversely, suppose that $w$ is not an isolated word.
Take $i \in \set{1,\ldots,n}$ such that $\qKoe_i$ or $\qKof_i$ is defined on $w$.
By \comboref{Definition}{dfn:fqcmtCn}, $w$ does not have an $i$-inversion, and some letter among $i$, $i+1$, $\wbar{i+1}$ and $\wbar{i}$ occurs in $w$.
\end{proof}

\begin{exa}
Consider $n = 4$.
In $\qctC_4^\fqcms$, the following are isolated words:
$1 \wbar{1}$,
$1 2 \wbar{2}$,
$2 \wbar{2}\,\wbar{1}$,
$3 \wbar{3} 3$,
and
$1 2 3 4 \wbar{4}\,\wbar{3}\,\wbar{2}\,\wbar{1}$.
\end{exa}

In the following result, we show how to obtain commutative and idempotent elements of $\hypo(\qctCn)$ from each word in $C_n^*$.

\begin{prop}
\label{prop:hypotCnwbwwwbwwbw}
Let $w \in C_n^*$.
Then, $w \wbar{w} w$ and $w \wbar{w} w \wbar{w}$ are isolated words in $\qctCn^\fqcms$.
Therefore, $w \wbar{w} w$ is a commutative element of $\hypo(\qctCn)$,
and $w \wbar{w} w \wbar{w}$ is a commutative and idempotent element of $\hypo(\qctCn)$.
\end{prop}

\begin{proof}
For each $i \in \set{1,\ldots,n}$, if $i$ occurs in $w$, then $w \wbar{w} w$ and $w \wbar{w} w \wbar{w}$ have decompositions of the form $w_1 i w_2 \wbar{i} w_3 i w_4$, for some $w_1, w_2, w_3, w_4 \in C_n^*$, which implies that $w \wbar{w} w$ and $w \wbar{w} w \wbar{w}$ have an $(i-1)$-inversion and an $i$-inversion.
If $\wbar{i}$ occurs in $w$, then $w \wbar{w} w$ and $w \wbar{w} w \wbar{w}$ have decompositions of the form $w_1 \wbar{i} w_2 i w_3 \wbar{i} w_4$, for some $w_1, w_2, w_3, w_4 \in C_n^*$, which implies that $w \wbar{w} w$ and $w \wbar{w} w \wbar{w}$ have an $(i-1)$-inversion and an $i$-inversion.
By \comboref{Proposition}{prop:hypoCniwiinv}, $w \wbar{w} w$ and $w \wbar{w} w \wbar{w}$ are isolated words,
and by \comboref{Theorem}{thm:hypoiecom}, they are commutative elements of $\hypo(\qctCn)$.
By \comboref{Propositions}{prop:qcmdesc} and\avoidrefbreak \ref{prop:hypotCnbarqcs}, we have that
\[ \wt (w \wbar{w} w \wbar{w}) = \wt(w) - \wt(w) + \wt(w) -\wt(w) = 0, \]
and by \comboref{Theorem}{thm:hypoidem}, we get that $w \wbar{w} w \wbar{w}$ is an idempotent element of $\hypo(\qctCn)$.
\end{proof}

To give a complete characterization of the commutative and idempotent elements of $\hypo(\qctCn)$, we first introduce the following notation.

\begin{dfn}
For each word $w \in C_n^*$, define an $n$-tuple $\inv (w) = (\delta_1, \ldots, \delta_n) \in \set{0, 1}^n$, where, for $i \in \set{1,\ldots,n}$, $\delta_i = 1$ if and only if $w$ has an $i$-inversion.
\end{dfn}

\begin{lem}
\label{lem:hypoCniwhycoiw}
Let $u, v \in C_n^*$ be isolated words in $\qctCn^\fqcms$.
Then, $u \hyco v$ if and only if $\wt (u) = \wt (v)$ and $\inv (u) = \inv (v)$.
\end{lem}

\begin{proof}
Since $u$ and $v$ are isolated words, we get that $C_n^* (u) = \set{u}$ and $C_n^* (v) = \set{v}$, which implies that $\qKoec_i (u), \qKofc_i (u), \qKoec_i (v), \qKofc_i (v) \in \set{ 0, {+\infty} }$, for all $i \in \set{1,\ldots,n}$, because $\qctCn^\fqcms$ is seminormal.
By \comboref{Proposition}{prop:hypotCniinversionresp}, for each $i \in \set{1,\ldots,n}$, we have that $\qKoec_i (u) = \qKofc_i (u) = +\infty$ (or $\qKoec_i (u) = \qKofc_i (v) = +\infty$) if and only if $u$ (resp., $v$) has an $i$-inversion.
Therefore, the map $\psi : C_n^* (u) \to C_n^* (v)$, given by $\psi (u) = v$, is a quasi-crystal isomorphism between $\qctCn^\fqcms (u)$ and $\qctCn^\fqcms (v)$ if and only if $\wt (u) = \wt (v)$ and $\inv (u) = \inv (v)$.
\end{proof}

\begin{thm}
\label{thm:hypoCniwbij}
The map that sends each isolated word $w \in C_n^*$ to $\parens[\big]{ \wt(w), \inv(w) }$ induces a bijection between the set of commutative elements of $\hypo(\qctCn)$ and the set of pairs $(\lambda, \delta)$ with $\lambda = (\lambda_1, \ldots, \lambda_n) \in \Z^n$ and $\delta = (\delta_1, \ldots, \delta_n) \in \set{0, 1}^n$ satisfying the following conditions:
\begin{enumerate}
\item\label{thm:hypoCniwbijl}
if $\lambda_i \neq 0$, for some $i \in \set{1,\ldots,n}$, then $\delta_i = 1$, and $\delta_{i-1} = 1$ when $i \geq 2$;

\item\label{thm:hypoCniwbijd}
if $\delta_i = 1$, for some $i \in \set{2,\ldots,n}$, then $\delta_{i-1} = 1$, or $\delta_{i+1} = 1$ when $i \leq n-1$.
\end{enumerate}
\end{thm}

\begin{proof}
By \comboref{Proposition}{prop:qcmelemprop} and \comboref{Theorem}{thm:hypoiecom}, we have that the commutative elements of $\hypo(\qctCn)$ correspond to the hypoplactic congruence classes of isolated words of $\qctCn^\fqcms$.
By \comboref{Lemma}{lem:hypoCniwhycoiw}, the map that sends each isolated word $w \in C_n^*$ to $\parens[\big]{ \wt(w), \inv(w) }$ induces a well-defined injective map from the commutative elements of $\hypo(\qctCn)$ to $\Z^n \times \set{0, 1}^n$.

We now show that the pairs $\parens[\big]{ \wt(w), \inv(w) }$, where $w \in C_n^*$ is an isolated word of $\qctCn^\fqcms$, satisfy conditions\avoidrefbreak \itmref{thm:hypoCniwbijl} and\avoidrefbreak \itmref{thm:hypoCniwbijd}.
Let $w \in C_n^*$ be an isolated word of $\qctCn^\fqcms$.
For each $i \in \set{1,\ldots,n}$, set $\lambda_i = \wlng{w}_{i} - \wlng{w}_{\wbar{i}}$, and if $w$ has an $i$-inversion, take $\delta_i = 1$, otherwise, take $\delta_i = 0$.
So, $\wt(w) = (\lambda_1, \ldots, \lambda_n)$ and $\inv(w) = (\delta_1, \ldots, \delta_n)$.
If $\lambda_i \neq 0$, for some $i \in \set{1,\ldots,n}$, then $i$ or $\wbar{i}$ occurs in $w$ implying that $w$ has an $(i-1)$-inversion (if $i \geq 2$) and an $i$-inversion, by \comboref{Proposition}{prop:hypoCniwiinv}, and so, $\delta_{i-1} =1$ (if $i \geq 2$) and $\delta_i = 1$.
If $\delta_i = 1$, for some $i \in \set{2,\ldots,n}$, then some letter among $i$, $i+1$, $\wbar{i+1}$ and $\wbar{i}$ occurs in $w$ implying that $w$ has an $(i-1)$-inversion, or when $i \leq n-1$, an $(i+1)$-inversion, by \comboref{Proposition}{prop:hypoCniwiinv}, and thus, $\delta_{i-1} = 1$ or $\delta_{i+1} = 1$.
Therefore, the pair $\parens[\big]{ \wt(w), \inv(w) }$ satisfies conditions\avoidrefbreak \itmref{thm:hypoCniwbijl} and\avoidrefbreak \itmref{thm:hypoCniwbijd}.

Finally, we show that for each pair $(\lambda, \delta) \in \Z^n \times \set{0, 1}^n$ satisfying conditions\avoidrefbreak \itmref{thm:hypoCniwbijl} and\avoidrefbreak \itmref{thm:hypoCniwbijd}, there exists an isolated word $w \in C_n^*$ such that $\wt(w) = \lambda$ and $\inv(w) = \delta$.
Let $\lambda = (\lambda_1, \ldots, \lambda_n) \in \Z^n$ and $\delta = (\delta_1, \ldots, \delta_n) \in \set{0, 1}^n$ satisfying conditions\avoidrefbreak \itmref{thm:hypoCniwbijl} and\avoidrefbreak \itmref{thm:hypoCniwbijd}.
If $\delta_1 = 1$, set $w_1 = 1 \wbar{1}$, otherwise, set $w_1 = \ew$.
For each $i \in \set{2,\ldots,n}$, if $\delta_{i-1} = \delta_{i} = 1$, take $w_i = i \wbar{i} i \wbar{i}$, otherwise, take $w_i = \ew$.
By\avoidrefbreak \itmref{thm:hypoCniwbijd}, if $\delta_i = 1$, for some $i \in \set{2,\ldots,n}$, then $w_i \neq \ew$ or $w_{i+1} \neq \ew$, which implies that $w_i w_{i+1}$ has an $i$-inversion.
Also, for $i \in \set{2,\ldots,n}$, we have that $\qKoec_{i-1} (w_i) = \qKoec_i (w_i) = +\infty$, and $\qKoec_j (w_i) = 0$, whenever $j \in \set{1,\ldots,n} \setminus \set{i-1,i}$.
Then, the word $w' = w_1 w_2 \ldots w_n$ has a $1$-inversion if and only if $\delta_1 = 1$,
and for $i \in \set{2,\ldots,n}$, $w'$ has an $i$-inversion if and only if $\delta_{i} = 1$.
Hence, $\inv(w') = \delta$.
Since $\wt (w_i) = 0$, for any $i \in \set{1,\ldots,n}$, we get that $\wt(w') = 0$.

For each $i \in \set{1,\ldots,n}$, if $\lambda_i \geq 0$, set $a_i = \lambda_i$ and $b_i = 0$, otherwise, set $a_i = 0$ and $b_i = -\lambda_i$.
Let
\[ w = w' 1^{a_1} 2^{a_2} \ldots n^{a_n} \wbar{n}^{b_n} \wbar{n-1}^{b_{n-1}} \ldots \wbar{1}^{b_1}. \]
By\avoidrefbreak \itmref{thm:hypoCniwbijl}, if $a_i \neq 0$ or $b_i \neq 0$, for some $i \in \set{1,\ldots,n}$, then $\delta_{i-1} = 1$ when $i \geq 2$, and $\delta_i = 1$, implying that $w'$ has an $(i-1)$-inversion (if $i \geq 2$) and an $i$-inversion.
Hence, $\inv(w) = \inv(w') = \delta$.
Since $\wt(w') = 0$, we get that
\[ \wt (w) = (a_1 - b_1, a_2 - b_2, \ldots, a_n - b_n) = \lambda. \]
Therefore, $\parens[\big]{ \wt(w), \inv(w) } = (\lambda, \delta)$.
\end{proof}

\begin{cor}
The map that sends each isolated word $w \in C_n^*$ to $\inv(w)$ induces a bijection between the set of idempotent elements of $\hypo(\qctCn)$ and the set of $n$-tuples $\delta = (\delta_1, \ldots, \delta_n) \in \set{0, 1}^n$ such that for each $i \in \set{2,\ldots,n}$, if $\delta_i = 1$, then $\delta_{i-1} = 1$, or $\delta_{i+1} = 1$ when $i \leq n-1$.
\end{cor}

\begin{proof}
By \comboref{Theorem}{thm:hypoidem}, the idempotent elements of $\hypo(\qctCn)$ correspond to the hypoplactic congruence classes of isolated words of $\qctCn^\fqcms$ with weight $0$.
Thus, the result follows directly from \comboref{Theorem}{thm:hypoCniwbij}.
\end{proof}

\subsection{Relations}
\label{subsec:hypotCnrel}

In this subsection, we first prove some results for $\hypo(\qctC_2)$ that allow a deeper understanding of this monoid,
which will be necessary to deduce some properties in the following subsections. Motivated by the fact that the plactic
monoid of type $\tCn$ satisfies the Knuth relations (see \comboref{Definition}{dfn:classicalhypo}) with the restriction
that $x \neq \wbar{z}$ \cite[Definition~3.2.1]{Lec02}, we then study whether the hypoplactic monoid $\hypo(\qctCn)$
satisfies the Knuth relations. In fact, we show that the Knuth relations only hold for one choice of generators.

\begin{lem}
\label{lem:hypoc2121121}
Let $m_1, m_2, p_1, p_2 \in \Z_{\geq 0}$ and $n_1, n_2 \in \Z_{> 0}$.
Then, $1^{m_1} 2^{n_1} 1^{p_1} \hyco 1^{m_2} 2^{n_2} 1^{p_2}$ in $\qctC_2^\fqcms$ implies that $m_1 = m_2$, $n_1 = n_2$ and $p_1 = p_2$.
\end{lem}

\begin{proof}
Assume that $1^{m_1} 2^{n_1} 1^{p_1} \hyco 1^{m_2} 2^{n_2} 1^{p_2}$.
Then, $m_1 + p_1 = m_2 + p_2$ and $n_1 = n_2$, because $\wt (1^{m_1} 2^{n_1} 1^{p_1}) = \wt (1^{m_2} 2^{n_2} 1^{p_2})$.
Suppose $m_1 \neq m_2$.
Without loss of generality, assume $m_1 < m_2$.
Set
\[\begin{split}
u &= \qKof_1^{m_2 + n_1 - 1} \qKof_2^{n_1} (1^{m_1} 2^{n_1} 1^{p_1}) = \qKof_1^{m_2 + n_1 - 1} \parens[\big]{ 1^{m_1} \wbar{2}^{n_1} 1^{p_1} }\\
&= 2^{m_1} \wbar{1}^{n_1} 2^{m_2 - m_1 - 1} 1^{p_1 - m_2 + m_1 + 1}
\end{split}\]
and
\[\begin{split}
v &= \qKof_1^{m_2 + n_2 - 1} \qKof_2^{n_2} (1^{m_2} 2^{n_2} 1^{p_2}) = \qKof_1^{m_2 + n_2 - 1} \parens[\big]{ 1^{m_2} \wbar{2}^{n_2} 1^{p_2} }\\
&= 2^{m_2} \wbar{1}^{n_2 - 1} \wbar{2} 1^{p_2}.
\end{split}\]
Since $1^{m_1} 2^{n_1} 1^{p_1} \hyco 1^{m_2} 2^{n_2} 1^{p_2}$ and $n_1 = n_2$, we get that $u \hyco v$.
By \comboref{Proposition}{prop:hypotCniinversionresp}, this is a contradiction, because $u$ is $2$-inversion-free and $v$ has a $2$-inversion.
\end{proof}

\begin{lem}
\label{lem:hypoc212121122}
Let $n_1, n_2, p_1, p_2, \in \Z_{\geq 0}$ and $m_1, m_2, q_1, q_2 \in \Z_{> 0}$.
Then, in $\qctC_2^\fqcms$, $1^{m_1} 2^{n_1} 1^{p_1} 2^{q_1} \hyco 1^{m_2} 2^{n_2} 1^{p_2} 2^{q_2}$ if and only if $m_1 + p_1 = m_2 + p_2$ and $n_1 + q_1 = n_2 + q_2$.
\end{lem}

\begin{proof}
If we first suppose that $1^{m_1} 2^{n_1} 1^{p_1} 2^{q_1} \hyco 1^{m_2} 2^{n_2} 1^{p_2} 2^{q_2}$ then we get that $\wt (1^{m_1} 2^{n_1} 1^{p_1} 2^{q_1}) = \wt (1^{m_2} 2^{n_2} 1^{p_2} 2^{q_2})$, which implies that $m_1 + p_1 = m_2 + p_2$ and $n_1 + q_1 = n_2 + q_2$.

We now show that $1 2^{k} 1^{l} 2 \hyco 1^{l + 1} 2^{k + 1}$, for any $k, l \in \Z_{\geq 0}$.
The aim is to show that each connected component $\Gamma_{\qctC_2^\fqcms} \parens[\big]{ 1 2^k 1^l 2 }$ and $\Gamma_{\qctC_2^\fqcms} \parens[\big]{ 1^{l+1} 2^{k+1} }$ is a path with $2k + 2l + 5$ vertices.
This will allow us to define a bijection $\psi : C_2^* \parens[\big]{ 1 2^k 1^l 2 } \to C_2^* \parens[\big]{ 1^{l+1} 2^{k+1} }$ that maps each word $w \in C_2^* \parens[\big]{ 1 2^k 1^l 2 }$ to the word $\psi(w) \in C_2^* \parens[\big]{ 1^{l+1} 2^{k+1} }$ such that the position of $w$ in $\Gamma_{\qctC_2^\fqcms} \parens[\big]{ 1 2^k 1^l 2 }$ is the same as $\psi (w)$ in $\Gamma_{\qctC_2^\fqcms} \parens[\big]{ 1^{l+1} 2^{k+1} }$.

The paths $\Gamma_{\qctC_2^\fqcms} \parens[\big]{ 1 2^k 1^l 2 }$ and $\Gamma_{\qctC_2^\fqcms} \parens[\big]{ 1^{l+1} 2^{k+1} }$ start in $1 2^k 1^l 2$ and $1^{l+1} 2^{k+1}$, which are of highest weight.
From these starting-points, there is a sequence of $k+1$ edges labelled by $2$, each of which transforms a symbol $2$ to a symbol $\wbar{2}$, in order from left to right through the word; at each step except the last, there is a $1$-inversion in the word and so a loop labelled by $1$ at that vertex.
There are then $k+l+2$ edges labelled by $1$, each of which transforms a symbol $1$ to a symbol $2$ or a symbol $\wbar{2}$ to a symbol $\wbar{1}$, in order from left to right through the word; again, in each step except the first and the last, there is a $2$-inversion in the word so a loop labelled by $2$ at that vertex.
Finally, there is a sequence of $l+1$ edges labelled by $2$, each transforming a symbol $2$ to a symbol $\wbar{2}$, in order from left to right throughout the word; again, there is a loop labelled by $1$ at each vertex.

Hence, $u \lbedge{i} v$ is an edge in $\Gamma_{\qctC_2^\fqcms} \parens[\big]{ 1 2^k 1^l 2 }$ if and only if $\psi (u) \lbedge{i} \psi(v)$ is an edge of $\Gamma_{\qctC_2^\fqcms} \parens[\big]{ 1^{l+1} 2^{k+1} }$.
And since $\psi \parens[\big]{ 1 2^k 1^l 2 } = 1^{l+1} 2^{k+1}$, where
\[ \wt \parens[\big]{ 1 2^k 1^l 2 } = (l+1, k+1) = \wt \parens[\big]{ 1^{l+1} 2^{k+1} }, \]
we have that $\psi$ preserves weights.
Therefore, by \comboref{Theorem}{thm:snqcisoqcg}, $\psi$ is a quasi-crystal isomorphism, which implies that $1 2^k 1^l 2 \hyco 1^{l+1} 2^{k+1}$.

Finally, as $\hyco$ is a monoid congruence, we can iterately apply $1 2^k 1^l 2 \hyco 1^{l+1} 2^{k+1}$ to see that
$1^{m_1} 2^{n_1} 1^{p_1} 2^{q_1} \hyco 1^{m_1 + p_1} 2^{n_1 + q_1}$ and
$1^{m_2} 2^{n_2} 1^{p_2} 2^{q_2} \hyco 1^{m_2 + p_2} 2^{n_2 + q_2}$. If $m_1 + p_1 = m_2 + p_2$ and
$n_1 + q_1 = n_2 + q_2$, we then obtain that $1^{m_1} 2^{n_1} 1^{p_1} 2^{q_1} \hyco 1^{m_2} 2^{n_2} 1^{p_2} 2^{q_2}$.
\end{proof}

\begin{prop}
\label{prop:hypotC2w2121}
Let $w \in \set{1, 2}^*$.
Then, $w \hyco 2^{m_1} 1^{m_2} 2^{m_3} 1^{m_4}$ in $\qctC_2^\fqcms$, for some $m_1, m_2, m_3, m_4 \in \Z_{\geq 0}$.
\end{prop}

\begin{proof}
Suppose that $w \neq 2^{m_1} 1^{m_2} 2^{m_3} 1^{m_4}$, for any $m_1, m_2, m_3, m_4 \in \Z_{\geq 0}$.
Then,
\[ w = 2^{q_0} 1^{p_1} 2^{q_1} 1^{p_2} 2^{q_2} \ldots 1^{p_k} 2^{q_k} 1^{p_{k+1}}, \]
for some $p_{k+1}, q_0 \in \Z_{\geq 0}$ and $p_1, \ldots, p_k, q_1, \ldots, q_k \in \Z_{> 0}$.
By \comboref{Lemma}{lem:hypoc212121122}, we have that $1^{p_1} 2^{q_1} 1^{p_2} 2^{q_2} \hyco 1^{p_1 + p_2} 2^{q_1 + q_2}$, and by iterating this process, we obtain that
\[ 1^{p_1} 2^{q_1} 1^{p_2} 2^{q_2} \ldots 1^{p_k} 2^{q_k} \hyco 1^{p_1 + p_2 + \cdots + p_k} 2^{q_1 + q_2 + \cdots + q_k}, \]
which implies that
$w \hyco 2^{q_0} 1^{p_1 + p_2 + \cdots + p_k} 2^{q_1 + q_2 + \cdots + q_k} 1^{p_{k+1}}$.
\end{proof}

We will see in \comboref{Theorem}{thm:hypoCnabfree} that the previous result does not hold in $\qctCn^\fqcms$ when $n \geq 3$.
We now study some properties satisfied by words $u, v \in C_n^*$ that are hypoplactic congruent $u \hyco v$ in $\qctCn^\fqcms$, for any $n \geq 2$.

\begin{lem}
\label{lem:hypotCnuv2elem}
Let $u, v \in C_n^*$ with $u \hyco v$ in $\qctCn^\fqcms$,
and let $x \in \set{1,\ldots,n}$.
\begin{enumerate}
\item\label{lem:hypotCnuv2elemxx1}
If $x \leq n-1$ and $u \in \set{ x, x+1 }^*$, then $v \in \set{ x, x+1 }^*$, $\wlng{u}_x = \wlng{v}_x$ and $\wlng{u}_{x+1} = \wlng{v}_{x+1}$.

\item\label{lem:hypotCnuv2elembxbx1}
If $x \leq n-1$ and $u \in \set[\big]{ \wbar{x+1}, \wbar{x} }^*$, then $v \in \set[\big]{ \wbar{x+1}, \wbar{x} }^*$, $\wlng{u}_{\wbar{x+1}} = \wlng{v}_{\wbar{x+1}}$ and $\wlng{u}_{\wbar{x}} = \wlng{v}_{\wbar{x}}$.

\item\label{lem:hypotCnuv2elemx}
If $u \in \set{ x }^*$, then $v \in \set{ x }^*$ and $\wlng{u}_x = \wlng{v}_x$.

\item\label{lem:hypotCnuv2elembx}
If $u \in \set{ \wbar{x} }^*$, then $v \in \set{ \wbar{x} }^*$ and $\wlng{u}_{\wbar{x}} = \wlng{v}_{\wbar{x}}$.

\item\label{lem:hypotCnuv2elemxbx}
If $x \neq 2$ and $u \in \set{ x, \wbar{x} }^*$, then $v \in \set{ x, \wbar{x} }^*$ and $\wlng{u}_{x} - \wlng{u}_{\wbar{x}} = \wlng{v}_{x} - \wlng{v}_{\wbar{x}}$.
\end{enumerate}
\end{lem}

\begin{proof}
\itmref{lem:hypotCnuv2elemxx1}
Assume $x \leq n-1$ and $u \in \set{ x, x+1 }^*$.
For $i \in \set{1,\ldots,n} \setminus \set{ x, x+1 }$, since neither $i$ nor $\wbar{i+1}$ occurs in $u$, we have that $u$ is $i$-inversion-free, and as $u \hyco v$,
\[ \wlng{v}_{i} \leq \qKofc_i (v) = \qKofc_i (u) = \wlng{u}_{i} + \wlng{u}_{\wbar{i+1}} = 0, \]
which implies that $i$ does not occur in $v$.
And since $\wt(u) = \wt(v)$, we get that $-\wlng{v}_{\wbar{i}} = \wlng{u}_{i} - \wlng{u}_{\wbar{i}} = 0$, which implies that $\wbar{i}$ does not occur in $v$.
Hence, $v \in \set[\big]{ x, x+1, \wbar{x+1}, \wbar{x} }$.
Since neither $x+2$ nor $\wbar{x+1}$ occurs in $u$, we have that $u$ is $(x+1)$-inversion-free, and so,
\[ \wlng{v}_{\wbar{x+1}} \leq \qKoec_{x+1} (v) = \qKoec_{x+1} (u) = \wlng{u}_{x+2} + \wlng{u}_{\wbar{x+1}} = 0, \]
implying that $\wbar{x+1}$ does not occur in $v$.
As $\wt(u) = \wt(v)$, we get that $\wlng{u}_{x+1} = \wlng{v}_{x+1}$.
Then, $u' \hyco v'$ where
\[ u' = \qKof_{x+2}^{\wlng{u}_{x+1}} \qKof_{x+3}^{\wlng{u}_{x+1}} \cdots \qKof_{n-1}^{\wlng{u}_{x+1}} \qKof_{n}^{\wlng{u}_{x+1}} \qKof_{n-1}^{\wlng{u}_{x+1}} \cdots \qKof_{x+1}^{\wlng{u}_{x+1}} (u) \]
and
\[ v' = \qKof_{x+2}^{\wlng{v}_{x+1}} \qKof_{x+3}^{\wlng{v}_{x+1}} \cdots \qKof_{n-1}^{\wlng{u}_{x+1}} \qKof_{n}^{\wlng{v}_{x+1}} \qKof_{n-1}^{\wlng{v}_{x+1}} \cdots \qKof_{x+1}^{\wlng{v}_{x+1}} (v). \]
Note that $u'$ and $v'$ are respectively obtained from $u$ and $v$ by replacing each $x+1$ by $\wbar{x+1}$.
In particular, $u' \in \set[\big]{ x, \wbar{x+1} }$ and $v' \in \set[\big]{ x, \wbar{x+1}, \wbar{x} }$.
Since neither $x+1$ nor $\wbar{x}$ occurs in $u'$, we get that $u'$ is $x$-inversion-free, and since $u' \hyco v'$,
\[ \wlng{v}_{\wbar{x}} = \wlng{v'}_{\wbar{x}} \leq \qKoec_{x} (v') = \qKoec_{x} (u') = \wlng{u'}_{x+1} + \wlng{u'}_{\wbar{x}} = 0, \]
which implies that $\wbar{x}$ does not occur in $v$.
Therefore, $v \in \set{ x, x+1 }$ which implies that $\wlng{u}_{x} = \wlng{v}_{x}$ and $\wlng{u}_{x+1} = \wlng{v}_{x+1}$, because $\wt(u) = \wt(v)$.

\itmref{lem:hypotCnuv2elembxbx1}
If $x \leq n-1$ and $u \in \set[\big]{ \wbar{x+1}, \wbar{x} }^*$, then $\wbar{u} \in \set{ x, x+1 }^*$, and as $\wbar{u} \hyco \wbar{v}$ by \comboref{Corollary}{cor:hypotCnuvbarubarv}, we get by\avoidrefbreak \itmref{lem:hypotCnuv2elemxx1} that $\wbar{v} \in \set{ x, x+1 }^*$, $\wlng{\wbar{u}}_{x} = \wlng{\wbar{v}}_{x}$ and $\wlng{\wbar{u}}_{x+1} = \wlng{\wbar{v}}_{x+1}$.
This implies that $v \in \set[\big]{ \wbar{x+1}, \wbar{x} }^*$, $\wlng{u}_{\wbar{x+1}} = \wlng{v}_{\wbar{x+1}}$ and $\wlng{u}_{\wbar{x}} = \wlng{v}_{\wbar{x}}$.

\itmref{lem:hypotCnuv2elemx}
Suppose $u \in \set{ x }^*$. If $x > 1$ then $u$ lies in $\set{ x-1, x }^*$, which implies by\avoidrefbreak \itmref{lem:hypotCnuv2elemxx1} that $v$ lies in $\set{ x-1, x }^*$ , where $\wlng{u}_{x} = \wlng{u}_{x}$ and $\wlng{v}_{x-1} = \wlng{v}_{x+1} = 0$ and so $v \in \set{ x }^*$. If $x = 1$, then $u \in \set{1,2}^*$ and the result follows similarly from \avoidrefbreak \itmref{lem:hypotCnuv2elemxx1}.

\itmref{lem:hypotCnuv2elembx}
If $u \in \set{ \wbar{x} }^*$, then $\wbar{u} \in \set{ x }^*$ and the result follows from \comboref{Corollary}{cor:hypotCnuvbarubarv} and\avoidrefbreak \itmref{lem:hypotCnuv2elemx}.

\itmref{lem:hypotCnuv2elemxbx}
Assume $x \neq 2$ and $u \in \set{ x, \wbar{x} }^*$.
As $u \hyco v$, we have that $\wt(u) = \wt(v)$, which implies that $\wlng{u}_{x} - \wlng{u}_{\wbar{x}} = \wlng{v}_{x} - \wlng{v}_{\wbar{x}}$.
For $i \in \set{1,\ldots,n} \setminus \set{ x-1, x }$, since neither $i$ nor $\wbar{i+1}$ occurs in $u$, we have that $u$ is $i$-inversion-free, and as $u \hyco v$,
\[ \wlng{v}_{i} \leq \qKofc_i (v) = \qKofc_i (u) = \wlng{u}_{i} + \wlng{u}_{\wbar{i+1}} = 0, \]
which implies that $i$ does not occur in $v$.
And since $\wt(u) = \wt(v)$, we get that $-\wlng{v}_{\wbar{i}} = \wlng{u}_{i} - \wlng{u}_{\wbar{i}} = 0$, which implies that $\wbar{i}$ does not occur in $v$.
If $x=1$, then the result is proven.
Otherwise, $x \geq 3$, we have that
\[ \wlng{v}_{x-1} \leq \qKoec_{x-2} (v) = \qKoec_{x-2} (u) = \wlng{u}_{x-1} + \wlng{u}_{\wbar{x-2}} = 0, \]
and then, $-\wlng{v}_{\wbar{x-1}} = \wlng{u}_{x-1} - \wlng{u}_{\wbar{x-1}} = 0$, implying that neither $x-1$ nor $\wbar{x-1}$ occurs in $v$.
Therefore, $v \in \set{ x, \wbar{x} }^*$.
\end{proof}

\begin{prop}
\label{prop:hypotCnuvxy}
Let $x, y \in \set{1,\ldots,n}$ with $\set{x, y} \neq \set[\big]{ 2, \wbar{2} }$,
and let $u, v \in C_n^*$ with $u \hyco v$ in $\qctCn^\fqcms$.
If $u \in \set{x, y}^*$, then $v \in \set{x, y}^*$.
\end{prop}

\begin{proof}
Assume that $u \in \set{x, y}^*$.
If $x = y$, the result follows from items\avoidrefbreak \itmref{lem:hypotCnuv2elemx} and\avoidrefbreak \itmref{lem:hypotCnuv2elembx} of \comboref{Lemma}{lem:hypotCnuv2elem}.
If $x = \wbar{y}$, then $x \neq 2$, because $\set{x, y} \neq \set[\big]{ 2, \wbar{2} }$, and so, the result follows from \itmcomboref{Lemma}{lem:hypotCnuv2elem}{lem:hypotCnuv2elemxbx}.
Thus, in the following, we assume that $x \neq y$ and $x \neq \wbar{y}$.

Take $i, j \in \set{1,\ldots,n}$ such that $x \in \set[\big]{ i, \wbar{i} }$ and $y \in \set[\big]{ j, \wbar{j} }$.
Since $x \neq y$ and $x \neq \wbar{y}$, we have that $i \neq j$.
Without loss of generality, we assume that $i < j$.
We thus consider the following cases.
\begin{itemize}
\item Case 1: $x=i$ and $y=j$.
If $j > i+1$, then $u$ is $(j-1)$-inversion-free, because neither $j-1$ nor $\wbar{j}$ occurs in $u$.
In this case, as $u \hyco v$, we get that $u' \hyco v'$ for
\[
u' = \qKoe_{i+1}^{\wlng{u}_j} \qKoe_{i+2}^{\wlng{u}_j} \cdots \qKoe_{j-1}^{\wlng{u}_j} (u)
\quad \text{and} \quad
v' = \qKoe_{i+1}^{\wlng{u}_j} \qKoe_{i+2}^{\wlng{u}_j} \cdots \qKoe_{j-1}^{\wlng{u}_j} (v).
\]
Note that $u'$ is obtained from $u$ by replacing each $j$ by $i+1$.
If $j = i+1$, we take $u' = u$ and $v' = v$; trivially, $u' \hyco v'$.
In either case, $u' \in \set{ i, i+1 }^*$.
We get by \itmcomboref{Lemma}{lem:hypotCnuv2elem}{lem:hypotCnuv2elemxx1} that $v' \in \set{ i, i+1 }^*$ and $\wlng{v'}_{i+1} = \wlng{u'}_{i+1} = \wlng{u}_{j}$.
In the case $j = i+1$, this establishes the result immediately since $v = v'$.
In the case $j > i+1$,
\[ v = \qKof_{j-1}^{\wlng{u}_j} \qKof_{j-2}^{\wlng{u}_j} \cdots \qKof_{i+1}^{\wlng{u}_j} (v'), \]
we have that $v$ is obtained from $v'$ by replacing each $i+1$ by $j$, as $\wlng{v'}_{i+1} = \wlng{u}_{j}$, and so, $v \in \set{ i, j }^*$.

\item Case 2: $x = i$ and $y = \wbar{j}$.
Note that $u$ is $j$-inversion-free, because neither $j$ nor $\wbar{j+1}$ occurs in $u$, as $i < j$.
Since $u \hyco v$, we get that $u' \hyco v'$ for
\[ u' = \qKoe_{i+1}^{\wlng{u}_{\wbar{j}}} \qKoe_{i+2}^{\wlng{u}_{\wbar{j}}} \cdots \qKoe_{n-1}^{\wlng{u}_{\wbar{j}}} \qKoe_{n}^{\wlng{u}_{\wbar{j}}} \qKoe_{n-1}^{\wlng{u}_{\wbar{j}}} \cdots \qKoe_{j}^{\wlng{u}_{\wbar{j}}} (u) \]
and
\[ v' = \qKoe_{i+1}^{\wlng{u}_{\wbar{j}}} \qKoe_{i+2}^{\wlng{u}_{\wbar{j}}} \cdots \qKoe_{n-1}^{\wlng{u}_{\wbar{j}}} \qKoe_{n}^{\wlng{u}_{\wbar{j}}} \qKoe_{n-1}^{\wlng{u}_{\wbar{j}}} \cdots \qKoe_{j}^{\wlng{u}_{\wbar{j}}} (v). \]
Note that $u'$ is obtained from $u$ by replacing each $\wbar{j}$ by $i+1$.
With a reasoning analogous to case 1, we obtain that $v \in \set[\big]{ i, \wbar{j} }^*$.

\item Case 3: $x = \wbar{i}$ and $y = \wbar{j}$.
If $j > i+1$, then $u$ is $(j-1)$-inversion-free, as neither $j$ nor $\wbar{j-1}$ occurs in $u$,
and since $u \hyco v$, we get that $u' \hyco v'$ for
\[
u' = \qKof_{i+1}^{\wlng{u}_{\wbar{j}}} \qKof_{i+2}^{\wlng{u}_{\wbar{j}}} \cdots \qKof_{j-1}^{\wlng{u}_{\wbar{j}}} (u)
\quad \text{and} \quad
v' = \qKof_{i+1}^{\wlng{u}_{\wbar{j}}} \qKof_{i+2}^{\wlng{u}_{\wbar{j}}} \cdots \qKof_{j-1}^{\wlng{u}_{\wbar{j}}} (v).
\]
Note that $u'$ is obtained from $u$ by replacing each $\wbar{j}$ by $\wbar{i+1}$.
If $j = i+1$, we take $u' = u$ and $v' = v$; trivially, $u' \hyco v'$.
In either case, $u' \in \set[\big]{ \wbar{i+1}, \wbar{i} }^*$.
We get by \itmcomboref{Lemma}{lem:hypotCnuv2elem}{lem:hypotCnuv2elembxbx1} that $v' \in \set[\big]{ \wbar{i+1}, \wbar{i} }^*$ and $\wlng{v'}_{\wbar{i+1}} = \wlng{u'}_{\wbar{i+1}} = \wlng{u}_{\wbar{j}}$.
In the case $j = i+1$, this establishes the result immediately since $v = v'$.
In the case $j > i+1$,
\[ v = \qKoe_{j-1}^{\wlng{u}_{\wbar{j}}} \qKoe_{j-2}^{\wlng{u}_{\wbar{j}}} \cdots \qKof_{i+1}^{\wlng{u}_{\wbar{j}}} (v'), \]
we have that $v$ is obtained from $v'$ by replacing each $\wbar{i+1}$ by $\wbar{j}$, as $\wlng{v'}_{\wbar{i+1}} = \wlng{u}_{\wbar{j}}$, and thus, $v \in \set[\big]{ \wbar{j}, \wbar{i} }^*$.

\item Case 4: $x = \wbar{i}$ and $y = j$.
As $i < j$, note that $u$ is $j$-inversion-free, because neither $j+1$ nor $\wbar{j}$ occurs in $u$.
Since $u \hyco v$, we get that $u' \hyco v'$ for
\[ u' = \qKof_{i+1}^{\wlng{u}_j} \qKof_{i+2}^{\wlng{u}_j} \cdots \qKof_{n-1}^{\wlng{u}_j} \qKof_{n}^{\wlng{u}_j} \qKof_{n-1}^{\wlng{u}_j} \cdots \qKof_{j}^{\wlng{u}_j} (u) \]
and
\[ v' = \qKof_{i+1}^{\wlng{u}_j} \qKof_{i+2}^{\wlng{u}_j} \cdots \qKof_{n-1}^{\wlng{u}_j} \qKof_{n}^{\wlng{u}_j} \qKof_{n-1}^{\wlng{u}_j} \cdots \qKof_{j}^{\wlng{u}_j} (v). \]
Note that $u'$ is obtained from $u$ by replacing each $j$ by $\wbar{i+1}$.
With a reasoning analogous to case 3, we obtain that $v \in \set[\big]{ j, \wbar{i} }^*$.
\end{itemize}
In either case, we get that $v \in \set{ x, y }^*$.
\end{proof}

\begin{prop}
\label{prop:hypotCnuvbegend}
Let $m \in \Z_{\geq 0}$,
and let $i, j \in \set{1,\ldots,n}$ with $i < j$.
In $\qctC_n^\fqcms$,
for any $u, v \in \set[\big]{ 1, 2, \ldots, i, j, \wbar{i-1}, \wbar{i-2}, \ldots, \wbar{1} }^*$,
we have that
\begin{enumerate}
\item\label{prop:hypotCnuvbegendb}
$j^{m} i u \nhyco j^{m+1} v$; and

\item\label{prop:hypotCnuvbegende}
$u j i^{m} \nhyco v i^{m+1}$.
\end{enumerate}
\end{prop}

\begin{proof}
\itmref{prop:hypotCnuvbegendb}
Suppose there exist $u, v \in \set[\big]{ 1, 2, \ldots, i, j, \wbar{i-1}, \wbar{i-2}, \ldots, \wbar{1} }^*$ such that $j^m i u \hyco j^{m+1} v$ in $\qctCn^\fqcms$.
So, $\wt (j^m i u) = \wt \parens[\big]{ j^{m+1} v }$ which implies that $\wlng{v}_{i} = \wlng{u}_{i} + 1$ and $\wlng{u}_{j} = \wlng{v}_{j} + 1$.
In particular, $j$ occurs in $u$.
Since neither $j+1$ nor $\wbar{j}$ occurs in $u$ or $v$, then $j^m i u$ and $j^{m+1} v$ are $j$-inversion-free.
Set
\[\begin{split}
w_1 &= \qKof_{i+1}^{\wlng{u}_j + m} \qKof_{i+2}^{\wlng{u}_j + m} \cdots \qKof_{n-1}^{\wlng{u}_j + m} \qKof_{n}^{\wlng{u}_j + m} \qKof_{n-1}^{\wlng{u}_j + m} \cdots \qKof_{j}^{\wlng{u}_j + m} (j^m i u)\\
&= \wbar{i+1}^m i \parens[\big]{ \qKof_{i+1}^{\wlng{u}_j} \qKof_{i+2}^{\wlng{u}_j} \cdots \qKof_{n-1}^{\wlng{u}_j}\qKof_{n}^{\wlng{u}_j} \qKof_{n-1}^{\wlng{u}_j} \cdots \qKof_{j}^{\wlng{u}_j} (u) }\\
&= \wbar{i+1}^m i u'
\end{split}\]
and
\[\begin{split}
w_2 &= \qKof_{i+1}^{\wlng{u}_j + m} \qKof_{i+2}^{\wlng{u}_j + m} \cdots \qKof_{n-1}^{\wlng{u}_j + m}\qKof_{n}^{\wlng{u}_j + m} \qKof_{n-1}^{\wlng{u}_j + m} \cdots \qKof_{j}^{\wlng{u}_j + m} \parens[\big]{ j^{m+1} v }\\
&= \wbar{i+1}^{m+1} \parens[\big]{ \qKof_{i+1}^{\wlng{u}_j - 1} \qKof_{i+2}^{\wlng{u}_j - 1} \cdots \qKof_{n-1}^{\wlng{u}_j - 1} \qKof_{n}^{\wlng{u}_j - 1} \qKof_{n-1}^{\wlng{u}_j - 1} \cdots \qKof_{j}^{\wlng{u}_j - 1} (v) }\\
&= \wbar{i+1}^{m+1} v'.
\end{split}\]
As $j^m i u \hyco j^{m+1} v$, we get that $w_1 \hyco w_2$.
Note that $u'$ is obtained from $u$ by replacing each $j$ by $\wbar{i+1}$,
and as $\wlng{v}_j = \wlng{u}_j - 1$, $v'$ is obtained from $v$ by replacing each $j$ by $\wbar{i+1}$.
In particular, $\wbar{i+1}$ occurs in $u'$, as $j$ occurs in $u$.
Since neither $i+1$ nor $\wbar{i}$ occurs in $w_1$ or $w_2$, we have that $w_1$ and $w_2$ are $i$-inversion-free.
Set
\[
w'_1 = \qKof_{i}^{m+1} (w_1) = \wbar{i}^m (i+1) u'
\quad \text{and} \quad
w'_2 = \qKof_{i}^{m+1} (w_2) = \wbar{i}^{m+1} v'.
\]
As $w_1 \hyco w_2$, we get that $w'_1 \hyco w'_2$.
Since $\wbar{i+1}$ occurs in $u'$, then $w'_1$ has an $(i+1)$-inversion.
And since neither $i+1$ nor $\wbar{i+2}$ occurs in $w'_2$, then $w'_2$ is $(i+1)$-inversion-free.
By \comboref{Proposition}{prop:hypotCniinversionresp}, this is a contradiction, because we obtained that $w'_1$ has an $(i+1)$-inversion, $w'_2$ is $(i+1)$-inversion-free, and $w'_1 \hyco w'_2$.

\itmref{prop:hypotCnuvbegende}
Suppose there exist $u, v \in \set[\big]{ 1, 2, \ldots, i, j, \wbar{i-1}, \wbar{i-2}, \ldots, \wbar{1} }^*$ such that $u j i^m \hyco v i^{m+1}$ in $\qctCn^\fqcms$.
So, $\wt (u j i^m) = \wt \parens[\big]{ v i^{m+1} }$ which implies that $\wlng{u}_i = \wlng{v}_i + 1$ and $\wlng{v}_j = \wlng{u}_j + 1$.
As justified in\avoidrefbreak \itmref{prop:hypotCnuvbegendb}, set
\begin{align*}
w_1 &= \qKof_{i+1}^{\wlng{u}_j + 1} \qKof_{i+2}^{\wlng{u}_j + 1} \cdots \qKof_{n-1}^{\wlng{u}_j + 1} \qKof_{n}^{\wlng{u}_j + 1} \qKof_{n-1}^{\wlng{u}_j + 1} \cdots \qKof_{j}^{\wlng{u}_j + 1} (u j i^m)
= u' \wbar{i+1} i^m
\displaybreak[0]\\
w'_1 &= \qKof_{i}^{\wlng{u}_i + \wlng{u}_j + m + 1} (w_1)
= u'' \wbar{i} (i+1)^m
\displaybreak[0]\\
w''_1 &= \qKof_{j+1}^{\wlng{u}_i + m} \qKof_{j+2}^{\wlng{u}_i + m} \cdots \qKof_{n-1}^{\wlng{u}_i + m} \qKof_{n}^{\wlng{u}_i + m} \qKof_{n-1}^{\wlng{u}_i + m} \cdots \qKof_{i+1}^{\wlng{u}_i + m} (w'_1)
= u''' \wbar{i} \wbar{j}^m
\displaybreak[0]\\
\intertext{and}
w_2 &= \qKof_{i+1}^{\wlng{u}_j + 1} \qKof_{i+2}^{\wlng{u}_j + 1} \cdots \qKof_{n-1}^{\wlng{u}_j + 1} \qKof_{n}^{\wlng{u}_j + 1} \qKof_{n-1}^{\wlng{u}_j + 1} \cdots \qKof_{j}^{\wlng{u}_j + 1} \parens[\big]{ v i^{m+1} }
= v' i^{m+1}
\displaybreak[0]\\
w'_2 &= \qKof_{i}^{\wlng{u}_i + \wlng{u}_j + m + 1} (w_2)
= v'' (i+1)^{m+1}
\displaybreak[0]\\
w''_2 &= \qKof_{j+1}^{\wlng{u}_i + m} \qKof_{j+2}^{\wlng{u}_i + m} \cdots \qKof_{n-1}^{\wlng{u}_i + m} \qKof_{n}^{\wlng{u}_i + m} \qKof_{n-1}^{\wlng{u}_i + m} \cdots \qKof_{i+1}^{\wlng{u}_i + m} (w'_2)
= v''' \wbar{j}^{m+1}.
\end{align*}
As $u j i^m \hyco v i^{m+1}$, we get that $w''_1 \hyco w''_2$.
Note that $w''_1$ is obtained from $u j i^m$ by replacing each $i$ by $\wbar{j}$ and each $j$ by $\wbar{i}$,
and since $\wlng{v}_i = \wlng{u}_i - 1$ and $\wlng{v}_j = \wlng{u}_j + 1$, $w''_2$ is obtained from $v i^{m+1}$ by replacing each $i$ by $\wbar{j}$ and each $j$ by $\wbar{i}$.
In particular, $u''', v''' \in \set[\big]{ 1, 2, \ldots, i-1, \wbar{j}, \wbar{i}, \wbar{i-1}, \ldots, \wbar{1} }^*$.
We have by \comboref{Corollary}{cor:hypotCnuvbarubarv} that
\[ j^m i \wbar{u'''} = \wbar{w''_1} \hyco \wbar{w''_2} = j^{m+1} \wbar{v'''}, \]
which is a contradiction by\avoidrefbreak \itmref{prop:hypotCnuvbegendb}.
\end{proof}

From \comboref{Propositions}{prop:hypotCnuvxy} and\avoidrefbreak \ref{prop:hypotCnuvbegend}, we have for words $u, v \in C_n^*$ with length at most $2$ that $u \hyco v$ implies $u = v$.
In the following result, we identify which words of length $3$ are hypoplactic congruent, and obtain that for distinct words, it comes under the statement of \comboref{Theorem}{thm:hypoiecom}.

\begin{thm}
Let $x, y, z, x', y', z' \in C_n$.
Then, $x y z \hyco x' y' z'$ in $\qctCn^\fqcms$ if and only if $x y z = x' y' z'$ or $x y z, x' y' z' \in \set[\big]{ a 1 \wbar{1}, 1 a \wbar{1}, 1 \wbar{1} a }$, for some $a \in C_n$.
\end{thm}

\begin{proof}
As $1 \wbar{1}$ is an isolated word, we have by \comboref{Theorem}{thm:hypoiecom} that $a 1 \wbar{1} \hyco 1 a \wbar{1} \hyco 1 \wbar{1} a$, for any $a \in C_n$.
And so, the converse implication holds.

Assume that $x y z \hyco x' y' z'$, that is, there exists a quasi-crystal isomorphism between the connected components $\qctCn^\fqcms (x y z)$ and $\qctCn^\fqcms (x' y' z')$ mapping $x y z$ to $x' y' z'$.
By \comboref{Propositions}{prop:fqcmlng} and\itmcomboref{\relax}{prop:fqcmcc}{prop:fqcmcchw}, we have that all words in $\qctCn^\fqcms (x y z)$ have exactly three letters, and $\qctCn^\fqcms (x y z)$ has at least one highest-weight word.
So, we first suppose that $xyz$ is a highest-weight word.
As $x y z \hyco x' y' z'$, we get that if $x y z$ is isolated, so is $x' y' z'$, and if $x y z$ is of highest weight but not isolated, so is $x' y' z'$.
So, in the following, we consider this two cases separately.

By \comboref{Proposition}{prop:hypoCniwiinv}, the isolated words in $\qctCn^\fqcms$ with three letters are
$1 1 \wbar{1}$, $1 \wbar{1}\, \wbar{1}$,
$1 2 \wbar{2}$, $2 \wbar{2}\, \wbar{1}$,
or of the form $t \wbar{t} t$, for $t \in C_n$.
If $x y z$ and $x' y' z'$ are among these words and $x y z \neq x' y' z'$, then as $\wt (x y z) = \wt (x' y' z')$, we must have that $x y z$ and $x' y' z'$ lie in
$\set[\big]{ 1 1 \wbar{1}, 1 2 \wbar{2}, 1 \wbar{1} 1 }$
or
$\set[\big]{ 1 \wbar{1}\, \wbar{1}, 2 \wbar{2}\, \wbar{1}, \wbar{1} 1 \wbar{1} }$.
Since both $1 2 \wbar{2}$ and $2 \wbar{2} \wbar{1}$ have a $2$-inversion, we get by \comboref{Corollary}{cor:hypotCnuvbarubarv} that $1 1 \wbar{1} \nhyco 1 2 \wbar{2} \nhyco 1 \wbar{1} 1$ and $1 \wbar{1}\, \wbar{1} \nhyco 2 \wbar{2}\, \wbar{1} \nhyco \wbar{1} 1 \wbar{1}$.
Thus, if $x y z$ and $x' y' z'$ are isolated and $x y z \neq x' y' z'$, then they lie in $\set[\big]{ 1 1 \wbar{1}, 1 \wbar{1} 1 }$ or $\set[\big]{ 1 \wbar{1}\, \wbar{1}, \wbar{1} 1 \wbar{1} }$.

By \comboref{Propositions}{prop:fqcmtCnhww} and\avoidrefbreak \ref{prop:hypoCniwiinv}, the highest-weight words in $\qctCn^\fqcms$, which are not isolated and consist of three letters, are
$1 1 1$,
$1 1 2$, $1 2 1$,
$1 2 2$, $2 1 2$,
$1 2 3$ (if $n \geq 3$),
$1 \wbar{1} 2$, $1 2 \wbar{1}$, $2 1 \wbar{1}$,
of the form $\wbar{i} i (i+1)$, for $i \in \set{1,\ldots,n-1}$,
or of the form $(j+1) \wbar{j+1}\, \wbar{j}$, for $j \in \set{2,\ldots,n-1}$.
If $x y z$ and $x' y' z'$ are among these words and $x y z \neq x' y' z'$, then as $\wt (x y z) = \wt (x' y' z')$, we must have that $x y z$ and $x' y' z'$ lie in
$\set{ 1 1 2, 1 2 1}$,
$\set{ 1 2 2, 2 1 2}$,
$\set[\big]{ 1 \wbar{1} 2, 1 2 \wbar{1}, 2 1 \wbar{1}, \wbar{1} 1 2 }$.
We get by \itmcomboref{Proposition}{prop:hypotCnuvbegend}{prop:hypotCnuvbegendb} that $2 1  2 \nhyco 1 2 2$
and by \itmcomboref{Proposition}{prop:hypotCnuvbegend}{prop:hypotCnuvbegende} that $1 1 2 \nhyco 1 2 1$.
Since $\wbar{1} 1 2$ is $1$-inversion-free, we have by \comboref{Corollary}{cor:hypotCnuvbarubarv} that $\wbar{1} 1 2 \nhyco w$, for any $w \in \set[\big]{ 1 \wbar{1} 2, 1 2 \wbar{1}, 2 1 \wbar{1} }$.
Thus, if $x y z$ and $x' y' z'$ are of highest weight, but not isolated, and $x y z \neq x' y' z'$, then they lie in $\set[\big]{ 1 \wbar{1} 2, 1 2 \wbar{1}, 2 1 \wbar{1} }$.

We have that
\begin{align*}
C_n^* \parens[\big]{ 1 \wbar{1} 2 } &= \set[\big]{ 1 \wbar{1} a \given a \in C_n, 2 \leq a \leq \wbar{2} },
\displaybreak[0]\\
C_n^* \parens[\big]{ 1 2 \wbar{1} } &= \set[\big]{ 1 a \wbar{1} \given a \in C_n, 2 \leq a \leq \wbar{2} },
\displaybreak[0]\\
\shortintertext{and}
C_n^* \parens[\big]{ 2 1 \wbar{1} } &= \set[\big]{ a 1 \wbar{1} \given a \in C_n, 2 \leq a \leq \wbar{2} }.
\end{align*}
And so, if $x y z$ and $x' y' z'$ lie in some of these connected components, then as $\wt (x y z) = \wt (x' y' z')$, we obtain that $x y z$ and $x' y' z'$ lie in $\set[\big]{ 1 \wbar{1} a, 1 a \wbar{1}, a 1 \wbar{1} }$, for some $a \in C_n$ with $2 \leq a \leq \wbar{2}$.

Therefore, for any $x, y, z, x', y', z' \in C_n$ such that $x y z \neq x' y' z'$ and $x y z \hyco x' y' z'$, we have that $x y z$ and $x' y' z'$ lie in
$\set[\big]{ 1 \wbar{1} a, 1 a \wbar{1}, a 1 \wbar{1} }$, for some $a \in C_n$.
\end{proof}

From the previous result, we get that the Knuth relations (\comboref{Definition}{dfn:classicalhypo}) only hold in $\hypo(\qctCn)$ for instances that come under the statement of \comboref{Theorem}{thm:hypoiecom}.

\begin{cor}
\label{cor:hypotCnKnuthrel}
Let $x, y, z \in C_n$.
Then, $yzx \hyco yxz$ in $\qctCn^\fqcms$ if and only if $x = y = z$ or $(y, x) = \parens[\big]{ 1, \wbar{1} })$ or $(y, z) = \parens[\big]{ 1, \wbar{1} })$.
Also, $xzy \hyco zxy$ in $\qctCn^\fqcms$ if and only if $x = y = z$ or $(x, y) = \parens[\big]{ 1, \wbar{1} })$ or $(z, y) = \parens[\big]{ 1, \wbar{1} })$.
\end{cor}

\subsection{Identities}
\label{subsec:hypotCnident}

We start by checking some properties of the identities satisfied by $\hypo (\qctC_2)$.

\begin{thm}
\label{thm:hypotC2identproperties}
Let $X$ be a finite alphabet, and let $u, v \in X^*$.
If $\hypo (\qctC_2)$ satisfies the identity $u = v$, then the following conditions are satisfied:
\begin{enumerate}
\item\label{thm:hypotC2identpropertiesletters}
$\wlng{u}_x = \wlng{v}_x$, for all $x \in X$;

\item\label{thm:hypotC2identpropertiesfirst}
until the first occurrence of a letter $x \in X$ in $u$ and $v$, each letter of $X$ occurs exactly the same number of times in $u$ and $v$, that is, if $u = u_1 x u_2$ and $v = v_1 x v_2$, where $u_1, u_2, v_1, v_2 \in X^*$ are such that $\wlng{u_1}_x = \wlng{v_1}_x = 0$, then $\wlng{u_1}_y = \wlng{v_1}_y$, for all $y \in X$;

\item\label{thm:hypotC2identpropertieslast}
after the last occurrence of a letter $x \in X$ in $u$ and $v$, each letter of $X$ occurs exactly the same number of times in $u$ and $v$, that is, if $u = u_1 x u_2$ and $v = v_1 x v_2$, where $u_1, u_2, v_1, v_2 \in X^*$ are such that $\wlng{u_2}_x = \wlng{v_2}_x = 0$, then $\wlng{u_2}_y = \wlng{v_2}_y$, for all $y \in X$.
\end{enumerate}
\end{thm}

\begin{proof}
Let $x \in X$.

\itmref{thm:hypotC2identpropertiesletters} If we consider the map from $X$ to $C_2^*$ that sends $x$ to $1$ and each other letter of $X$ to $\ew$, we obtain that $1^{\wlng{u}_x} \hyco 1^{\wlng{v}_x}$.
So, $\wt \parens[\big]{ 1^{\wlng{u}_x} } = \wt \parens[\big]{ 1^{\wlng{v}_x} }$ which implies that $\wlng{u}_x = \wlng{v}_x$.

\itmref{thm:hypotC2identpropertiesfirst} Since $\wlng{u}_x = \wlng{v}_x$, we have that $x$ occurs in $u$ if and only if $x$ occurs in $v$.
And if so, there exist $u_1, u_2, v_1, v_2 \in X^*$ such that $u = u_1 x u_2$, $v = v_1 x v_2$, and $x$ does not occur in $u_1$ or $v_1$.
Given $y \in X$, consider the monoid homomorphism $\psi : X^* \to C_2^*$ induced by $\psi (x) = 1$, $\psi (y) = 2$ and $\psi (z) = \ew$, for each $z \in X \setminus \set{x, y}$.
Then,
\[ 2^{\wlng{u_1}_y} 1 \psi (u_2) = \psi (u) \hyco \psi (v) = 2^{\wlng{v_1}_y} 1 \psi (v_2), \]
which implies that $\wlng{u_1}_y = \wlng{v_1}_y$, by \itmcomboref{Proposition}{prop:hypotCnuvbegend}{prop:hypotCnuvbegendb}.

\itmref{thm:hypotC2identpropertieslast} Since $\wlng{u}_x = \wlng{v}_x$, we have that $x$ occurs in $u$ if and only if $x$ occurs in $v$.
And if so, there exist $u_1, u_2, v_1, v_2 \in X^*$ such that $u = u_1 x u_2$, $v = v_1 x v_2$, and $x$ does not occur in $u_2$ or $v_2$.
Given $y \in X$, consider the monoid homomorphism $\psi : X^* \to C_2^*$ induced by $\psi (x) = 2$, $\psi (y) = 1$ and $\psi (z) = \ew$, for each $z \in X \setminus \set{x, y}$.
Then,
\[ \psi (u_1) 2 1^{\wlng{u_2}_y} = \psi (u) \hyco \psi (v) = \psi (v_1) 2 1^{\wlng{v_2}_y}, \]
which implies that $\wlng{u_2}_y = \wlng{v_2}_y$, by \itmcomboref{Proposition}{prop:hypotCnuvbegend}{prop:hypotCnuvbegende}.
\end{proof}

%
%
%

\begin{thm}
The hypoplactic monoid $\hypo (\qctC_2)$ satisfies the identity $xyxyxy = xyyxxy$,
that is, $uvuvuv \hyco uvvuuv$ in $\qctC_2^\fqcms$, for any $u, v \in C_2^*$.
\end{thm}

\begin{proof}
Let $u, v \in C_2^*$.
We first assume that $uvuvuv$ is of highest weight, that is, $\qKoec_1, (uvuvuv), \qKoec_2 (uvuvuv) \in \set{ 0, {+\infty} }$.
If $\wbar{2}$ occurs in $uvuvuv$, then $uvuvuv$ has a $2$-inversion, because it is of highest weight. Hence $2$ also occurs in $uvuvuv$.
In this case, $2$ and $\wbar{2}$ occur in $uv$ and $vu$, implying that both $uvuvuv$ and $uvvuuv$ have a decomposition of the form $w_1 2 w_2 \wbar{2} w_3 2 w_4$, for some $w_1, w_2, w_3, w_4 \in C_2^*$.
Hence, $uvuvuv$ and $uvvuuv$ are isolated words as they have $1$- and $2$-inversions.
Since
\[ \wt(uvuvuv) = 3 \wt(u) + 3 \wt(v) = \wt(uvvuuv), \]
we get by \comboref{Lemma}{lem:hypoCniwhycoiw} that $uvuvuv \hyco uvvuuv$.
So, we now assume that $\wbar{2}$ does not occur in $uvuvuv$.

If $\wbar{1}$ occurs in $uvuvuv$, then $uvuvuv$ has a $1$-inversion, because it is of highest weight.
Since $\wbar{2}$ does not occur in $uvuvuv$, we get that $1$ and $\wbar{1}$ occur in $uv$ and $vu$, implying that both $uvuvuv$ and $uvvuuv$ have a decomposition of the form $w_1 1 w_2 \wbar{1} w_3$, for some $w_1, w_2, w_3 \in \set[\big]{ 1, 2, \wbar{1} }^*$.
As $1\wbar{1}$ is an isolated word, we have by \comboref{Theorem}{thm:hypoiecom} that
\[ uvuvuv \hyco 1^{3\wlng{u}_1+3\wlng{v}_1} 2^{3\wlng{u}_2+3\wlng{v}_2} \wbar{1}^{3\wlng{u}_{\wbar{1}}+3\wlng{v}_{\wbar{1}}} \hyco uvvuuv. \]
Thus, we further assume that $\wbar{1}$ does not occur in $uvuvuv$.

If $2$ occurs in $uvuvuv$, then $uvuvuv$ has a $1$-inversion, because it is of highest weight.
Since $\wbar{2}$ and $\wbar{1}$ do not occur in $uvuvuv$, we get that $uv, vu \in \set{1, 2}^*$, and $1$ and $2$ occur in $uv$, implying that there exist $w_1, w_2, w_3, w_4 \in \set{1, 2}^*$ such that $uv = w_1 1 w_2$ and $uv = w_3 2 w_4$.
As $\wlng{w_2 uv w_3}_1 = \wlng{w_2 vu w_3}_1$ and $\wlng{w_2 uv w_3}_2 = \wlng{w_2 vu w_3}_2$, we have by \comboref{Lemma}{lem:hypoc212121122} that
\[\begin{split}
uvuvuv
= w_1 1 w_2 uv w_3 2 w_4
&\hyco w_1 1^{\wlng{w_2 uv w_3}_1+1} 2^{\wlng{w_2 uv w_3}_2+1} w_4
\\
&\hyco w_1 1 w_2 vu w_3 2 w_4
= uvvuuv.
\end{split}\]
Finally, if we also assume that $2$ does not occur in $uvuvuv$, then $u = 1^{\wlng{u}}$ and $v = 1^{\wlng{v}}$, which implies that $uvuvuv = uvvuuv$.

Therefore, we obtain that $uvuvuv \hyco uvvuuv$, for any $u, v, \in C_2^*$ such that $uvuvuv$ is of highest weight.
We now show that this also holds when $uvuvuv$ is not of highest weight.

Suppose there exist $u, v \in C_2^*$ such that $uvuvuv \nhyco uvvuuv$.
The set
\[ W = \set[\big]{ u'v'u'v'u'v' \given u', v' \in C_2^*, u'v'u'v'u'v' \nhyco u'v'v'u'u'v', \wlng{u'} = \wlng{u}, \wlng{v'} = \wlng{v} } \]
is nonempty and finite, so we can take words $u', v' \in C_2^*$ such that $u'v'u'v'u'v' \nhyco u'v'v'u'u'v'$, $\wlng{u'} = \wlng{u}$, $\wlng{v'} = \wlng{v}$, and $u'v'u'v'u'v'$ has maximal weight among weights of words in $W$, that is, if $w \in W$ and $\wt(w) \geq \wt(u'v'u'v'u'v')$, then $\wt(w) = \wt(u'v'u'v'u'v')$.
As shown above, we have that $u'v'u'v'u'v'$ is not of highest weight.
Take $i \in \set{1,2}$ such that $\qKoe_i$ is defined on $u'v'u'v'u'v'$.
Then, $u'v'u'v'u'v'$ does not have an $i$-inversion, $\qKoec_i (u'v'u'v'u'v') = 3 \qKoec_i (u') + 3 \qKoec_i (v') \in \Z_{> 0}$, and
\[\begin{split}
&\qKoe_i^{\qKoec_i (u'v'u'v'u'v')} (u'v'u'v'u'v')\\
&\qquad {}= \qKoe_i^{\qKoec_i (u')} (u') \qKoe_i^{\qKoec_i (v')} (v') \qKoe_i^{\qKoec_i (u')} (u') \qKoe_i^{\qKoec_i (v')} (v') \qKoe_i^{\qKoec_i (u')} (u') \qKoe_i^{\qKoec_i (v')} (v').
\end{split}\]
Set $u'' = \qKoe_i^{\qKoec_i (u')} (u')$ and $v'' = \qKoe_i^{\qKoec_i (v')} (v')$.
By \comboref{Proposition}{prop:qcrlqKo},
$\wt (u''v''u''v''u''v'') > \wt (u'v'u'v'u'v')$,
and by \comboref{Proposition}{prop:fqcmlng},
$\wlng{u''} = \wlng{u'} = \wlng{u}$
and
$\wlng{v''} = \wlng{v'} = \wlng{v}$.
As the weight of $u'v'u'v'u'v'$ is maximal among weights of words in $W$, we get that $u''v''u''v''u''v'' \notin W$, which implies that
$u''v''u''v''u''v'' \hyco u''v''v''u''u''v''$.
We have by \itmcomboref{Definition}{dfn:qc}{dfn:qciff} that $\qKof_i$ is defined on $u''v''u''v''u''v''$, and so, $u''v''u''v''u''v''$ does not have an $i$-inversion.
Since
$u''v''u''v''u''v'' \hyco u''v''v''u''u''v''$,
we get that $\qKofc_i (u''v''v''u''u''v'') = \qKofc_i (u''v''u''v''u''v'') = 3 \qKofc_i (u'') + 3 \qKofc_i (v'')$, and as
\[\begin{split}
&\qKof_i^{\qKofc_i (u''v''u''v''u''v'')} (u''v''u''v''u''v'')\\
&\qquad {}= \qKof_i^{\qKofc_i (u''v''u''v''u''v'')} \parens[\big]{ \qKoe_i^{\qKoec_i (u'v'u'v'u'v')} (u'v'u'v'u'v') }
= u'v'u'v'u'v'
\end{split}\]
and
\[\begin{split}
&\qKof_i^{\qKofc_i (u''v''u''v''u''v'')} (u''v''v''u''u''v'')\\
&\qquad {}= \qKof_i^{\qKofc_i (u'')} (u'') \qKof_i^{\qKofc_i (v'')} (v'') \qKof_i^{\qKofc_i (v'')} (v'') \qKof_i^{\qKofc_i (u'')} (u'') \qKof_i^{\qKofc_i (u'')} (u'') \qKof_i^{\qKofc_i (v'')} (v'')\\
&\qquad {}= u'v'v'u'u'v',
\end{split}\]
we obtain that $u'v'u'v'u'v' \hyco u'v'v'u'u'v'$, which is a contradiction.
Therefore, for any $u, v \in C_2^*$, we have that $uvuvuv \hyco uvvuuv$, that is, $\hypo(\qctC_2)$ satisfies the identity $xyxyxy = xyyxxy$.
\end{proof}

We now turn our attention for whether $\hypo(\qctCn)$ satisfies identities when $n \geq 3$.
In the following results, we prove that $\hypo(\qctCn)$ does not satisfy any nontrivial identity, for $n \geq 3$.
This is achieved by showing that it contains free submonoids with more than one generator.

\begin{lem}
\label{lem:hypoCn1b2free}
Consider $n \geq 3$.
Let $u, v \in \set[\big]{1, \wbar{2}}^*$.
Then, $u \hyco v$ in $\qctCn^\fqcms$ if and only if $u = v$.
\end{lem}

\begin{proof}
Suppose that $u \hyco v$ and $u \neq v$.
Since $\wt(u) = \wt(v)$, we have that $\wlng{u}_1 = \wlng{v}_1$ and $\wlng{u}_{\wbar{2}} = \wlng{v}_{\wbar{2}}$, and as $u \neq v$, both $1$ and $\wbar{2}$ occur in $u$ and $v$.
Take $u', v', w \in \set[\big]{ 1, \wbar{2} }^*$ such that $u = w1u'$ and $v = w\wbar{2}v'$.
Since $\wlng{u}_1 = \wlng{v}_1$ and $\wlng{u}_{\wbar{2}} = \wlng{v}_{\wbar{2}}$, we have that $\wlng{v'}_1 = \wlng{u'}_1 + 1$ and $\wlng{u'}_{\wbar{2}} = \wlng{v'}_{\wbar{2}} + 1$.

The words $u$ and $v$ do not have $2$-inversions, because neither $2$ nor $\wbar{3}$ occurs in them.
Set
\[
u_1 = \qKoe_2^{\wlng{v'}_{\wbar{2}}} (u) = w 1 \qKoe_2^{\wlng{v'}_{\wbar{2}}} (u')
\quad \text{and} \quad
v_1 = \qKoe_2^{\wlng{v'}_{\wbar{2}}} (v) = w \wbar{2} \qKoe_2^{\wlng{v'}_{\wbar{2}}} (v').
\]
Note that $\qKoe_2^{\wlng{v'}_{\wbar{2}}} (v')$ is obtained from $v'$ by replacing each $\wbar{2}$ by $\wbar{3}$,
and as $\wlng{v'}_{\wbar{2}} = \wlng{u'}_{\wbar{2}} - 1$, $\qKoe_2^{\wlng{v'}_{\wbar{2}}} (u')$ is obtained from $u'$ by replacing each $\wbar{2}$ by $\wbar{3}$ except for the left-most $\wbar{2}$ that remains unchanged.
In particular, $\wbar{2}$ occurs in $\qKoe_2^{\wlng{v'}_{\wbar{2}}} (u')$ and does not occur in $\qKoe_2^{\wlng{v'}_{\wbar{2}}} (v')$.
As $u \hyco v$, we also have that $u_1 \hyco v_1$.

The words $u_1$ and $v_1$ do not have $1$-inversions, because neither $2$ nor $\wbar{1}$ occurs in them.
Set
\[
u_2 = \qKof_1^{\wlng{w}+1} (u_1) = \qKof_1^{\wlng{w}} (w) 2 \qKoe_2^{\wlng{v'}_{\wbar{2}}} (u')
\quad \text{and} \quad
v_2 = \qKof_1^{\wlng{w}+1} (v_1) = \qKof_1^{\wlng{w}} (w) \wbar{1} \qKoe_2^{\wlng{v'}_{\wbar{2}}} (v').
\]
Note that $\qKof_1^{\wlng{w}} (w)$ is obtained from $w$ by replacing each $1$ by $2$ and each $\wbar{2}$ by $\wbar{1}$.
Since $\wbar{2}$ occurs in $\qKoe_2^{\wlng{v'}_{\wbar{2}}} (u')$, we have that $u_2$ has a $2$-inversion,
and since neither $3$ nor $\wbar{2}$ occurs in $v_2$, we have that $v_2$ does not have a $2$-inversion.
By \comboref{Proposition}{prop:hypotCniinversionresp}, this is a contradiction, because $u_1 \hyco v_1$ implies that $u_2 \hyco v_2$.
\end{proof}

\begin{thm}
\label{thm:hypoCnabfree}
Consider $n \geq 3$.
Let $a, b \in C_n$ with $a \neq \wbar{b}$,
and let $u, v \in \set{a, b}^*$.
Then, $u \hyco v$ in $\qctCn^\fqcms$ if and only if $u = v$.
\end{thm}

\begin{proof}
Take $i, j \in \set{1,\ldots,n}$ such that $a \in \set[\big]{ i, \wbar{i} }$ and $b \in \set[\big]{ j, \wbar{j} }$.
Since $a \neq \wbar{b}$ and the case $a = b$ is trivial, without loss of generality we assume $i < j$.
By \comboref{Corollary}{cor:hypotCnuvbarubarv}, we have that $u \hyco v$ if and only if $\wbar{u} \hyco \wbar{v}$,
and so, without loss of generality, we further assume that $a = i$.

Suppose that $u \hyco v$.
As $\wt(u) = \wt(v)$ and $i \neq j$, we get that $\wlng{u}_a = \wlng{v}_a$ and $\wlng{u}_b = \wlng{v}_b$.
If $i=1$, set $u' = u$ and $v' = v$,
otherwise, $u$ and $v$ do not have $(i-1)$-inversions, because neither $i-1$ nor $\wbar{i}$ occurs in them as $i < j$,
and so, set
\[
u' = \qKoe_1^{\wlng{u}_a} \qKoe_2^{\wlng{u}_a} \cdots \qKoe_{i-1}^{\wlng{u}_a} (u)
\quad \text{and} \quad
v' = \qKoe_{1}^{\wlng{v}_a} \qKoe_{2}^{\wlng{v}_a} \cdots \qKoe_{i-1}^{\wlng{v}_a} (v).
\]
As $a=i$, note that $u'$ and $v'$ are respectively obtained from $u$ and $v$ by replacing each $a$ by $1$.
In particular, $u', v' \in \set{1, b}^*$ and $\wlng{u'}_b = \wlng{v'}_b$.
Since $u \hyco v$ and $\wlng{u}_a = \wlng{v}_a$, we get that $u' \hyco v'$.

In the following cases, we obtain words $u'', v'' \in \set[\big]{ 1, \wbar{2} }$ by applying the same and in the same order quasi-Kashiwara operators to $u'$ and $v'$.
\begin{itemize}
\item Case 1: $b=j$.
The words $u'$ and $v'$ do not have $j$-inversions, because neither $j+1$ nor $\wbar{j}$ occurs in them as $1 \leq i < j$.
Set
\[
u'' = \qKof_j^{\wlng{u'}_b} \qKof_{j+1}^{\wlng{u'}_b} \cdots \qKof_{n}^{\wlng{u'}_b} \qKof_{n-1}^{\wlng{u'}_b} \cdots \qKof_{2}^{\wlng{u'}_b} (u')
\]
and
\[
v'' = \qKof_j^{\wlng{v'}_b} \qKof_{j+1}^{\wlng{v'}_b} \cdots \qKof_{n}^{\wlng{v'}_b} \qKof_{n-1}^{\wlng{v'}_b} \cdots \qKof_{2}^{\wlng{v'}_b} (v').
\]

\item Case 2: $b=\wbar{j}$.
If $j=2$, set $u'' = u'$ and $v'' = v'$,
otherwise, we have that $j \geq 3$, as $j > i \geq 1$, and the words $u'$ and $v'$ do not have $(j-1)$-inversions, because neither $j$ nor $\wbar{j-1}$ occurs in them, and so, set
\[
u'' = \qKof_{j-1}^{\wlng{u'}_b} \qKof_{j-2}^{\wlng{u'}_b} \cdots \qKof_{2}^{\wlng{u'}_b} (u')
\quad \text{and} \quad
v'' = \qKof_{j-1}^{\wlng{v'}_b} \qKof_{j-2}^{\wlng{v'}_b} \cdots \qKof_{2}^{\wlng{v'}_b} (v').
\]
\end{itemize}
In either case, note that $u''$ and $v''$ are respectively obtained from $u'$ and $v'$ by replacing each $b$ by $\wbar{2}$.
In particular, $u'', v'' \in \set[\big]{ 1, \wbar{2} }^*$.
Since $u' \hyco v'$ and $\wlng{u'}_b = \wlng{v'}_b$, we get that $u'' \hyco v''$.
By \comboref{Lemma}{lem:hypoCn1b2free}, we obtain that $u'' = v''$, and since the quasi-Kashiwara operators are injective when defined (\itmcomboref{Definition}{dfn:qc}{dfn:qciff}), we deduce that $u=v$.
\end{proof}

From the previous result we have when $n \geq 3$ that for $a, b \in C_n$ with $a \neq \wbar{b}$, the set $\set{a, b}$ is free on $\hypo(\qctCn)$.
This marks a difference when compared to $\hypo(\qctC_2)$, where $\set{1, 2}$ is not free, as shown in \comboref{Lemma}{lem:hypoc212121122}.
This also implies the following result.

\begin{cor}
Let $n \geq 3$.
Then, $\hypo(\qctCn)$ does not satisfy nontrivial identities.
\end{cor}

%
%

\subsection{Presentations}
\label{subsec:hypotCnpres}

We first show that the hypoplactic monoid $\hypo(\qctC_2)$ does not admit a finite presentation.

\begin{thm}
\label{thm:hypoC2fpres}
The hypoplactic congruence $\hyco$ on the free monoid $C_2^*$ is not finitely generated.
Therefore, $\hypo (\qctC_2)$ has no finite presentation.
\end{thm}

\begin{proof}
Let $R$ be a finite subset of $\hyco$, and denote by $\sim_R$ the monoid congruence on $C_2^*$ generated by $R$.
Thus, ${\sim_R} \subseteq {\hyco}$.
As $R$ is finite, set
\[ m = 1 + \max_{(u,v) \in R} \set[\big]{ \wlng{u}, \wlng{v} }. \]

We first show that for any subword $u$ of $1 2 1^{m} 2$ with length at most $m$ (that is, any word $u$ consisting of at most $m$ consecutive letters of $1 2 1^{m} 2$) and $v \in C_2^*$, if $u \hyco v$, then $u = v$.
Let $u, v \in C_2^*$ be such that $u \hyco v$ and $u$ is a subword of $1 2 1^{m} 2$.
Then, $u \in \set{1, 2}^*$, and by \comboref{Lemma}{lem:hypotCnuv2elem}, $v \in \set{1, 2}^*$.
As $u$ is a subword of $1 2 1^{m} 2$ with length at most $m$, we get that $u$ lies in one of the following cases.
\begin{itemize}
\item Case 1: $u = 1^p$ where $p \in \Z_{\geq 0}$.
Since $u \hyco v$, we have that $\wt(u) = \wt(v)$, and since $v \in \set{1, 2}^*$, we get that $\wlng{v}_1 = \wlng{u}_1 = p$ and $\wlng{v}_2 = \wlng{u}_2 = 0$.
Hence, $v = 1^p = u$.

\item Case 2: $u = 1^{p_1} 2 1^{p_2}$ where $p_1, p_2 \in \Z_{\geq 0}$.
As in the previous case, we have that $\wlng{v}_1 = p_1 + p_2$ and $\wlng{v}_2 = 1$.
So, $v = 1^{q_1} 2 1^{q_2}$, for some $q_1, q_2 \in \Z_{\geq 0}$ such that $q_1 + q_2 = p_1 + p_2$.
By \comboref{Lemma}{lem:hypoc2121121}, $1^{p_1} 2 1^{p_2} \hyco 1^{q_1} 2 1^{q_2}$ implies that $p_1 = q_1$ and $p_2 = q_2$.
Hence, $v = u$.
\end{itemize}

Since $R \subseteq {\hyco}$ and every pair $(u, v) \in R$ satisfies $\wlng{u} < m$ and $\wlng{v} < m$, we get that if $(u, v) \in R$ and $u$ or $v$ are subwords of $1 2 1^{m} 2$, then $u = v$.
As $R$ generates $\sim_R$, we obtain for $w \in C_2^*$ that $1 2 1^{m} 2 \sim_R w$ if and only if $w = 1 2 1^{m} 2$.
On the other hand, we have that $1 2 1^{m} 2 \neq 1^{m+1} 2^2$, and by \comboref{Lemma}{lem:hypoc212121122}, $1 2 1^{m} 2 \hyco 1^{m+1} 2^2$.
Hence, ${\sim_R} \neq {\hyco}$.
And therefore, $\hyco$ is not finitely generated.
\end{proof}

Although, $\hypo(\qctC_2)$ does not have a finite presentation, we are able to describe connected components of $\Gamma_{\qctC_2^\fqcms}$, and thus, find representatives for the hypoplactic congruence classes.

\begin{lem}
\label{lem:hypoc22ub1v1w212b1}
Let $u, v, w \in \set[\big]{ 1, 2, \wbar{1} }^*$.
Then, $2 u \wbar{1} v 1 w 2 \hyco 1^{m+1} 2^{p+2} \wbar{1}^{q+1}$ in $\qctC_2^\fqcms$, for some $m, p, q \in \Z_{\geq 0}$ where $m = 0$ or $q = 0$.
\end{lem}

\begin{proof}
Set $m = \max \set[\big]{ 0, \wlng{u}_1 + \wlng{v}_1 + \wlng{w}_1 - \wlng{u}_{\wbar{1}} - \wlng{v}_{\wbar{1}} - \wlng{w}_{\wbar{1}} }$,
$p = \wlng{u}_2 + \wlng{v}_2 + \wlng{w}_2$,
and $q = \max \set[\big]{ 0, \wlng{u}_{\wbar{1}} + \wlng{v}_{\wbar{1}} + \wlng{w}_{\wbar{1}} - \wlng{u}_1 - \wlng{v}_1 - \wlng{w}_1 }$.
We first show that each connected component $\Gamma_{\qctC_2^\fqcms} \parens[\big]{ 2 u \wbar{1} v 1 w 2 }$ and $\Gamma_{\qctC_2^\fqcms} \parens[\big]{ 1^{m+1} 2^{p+2} \wbar{1}^{q+1} }$ is a path with $p+3$ vertices.
The paths $\Gamma_{\qctC_2^\fqcms} \parens[\big]{ 2 u \wbar{1} v 1 w 2 }$ and $\Gamma_{\qctC_2^\fqcms} \parens[\big]{ 1^{m+1} 2^{p+2} \wbar{1}^{q+1} }$ start respectively in $2 u \wbar{1} v 1 w 2$ and $1^{m+1} 2^{p+2} \wbar{1}^{q+1}$, which are highest-weight words without $2$-inversions, as $\wbar{2}$ does not occur in them.
From these starting-points, there is a sequence of $p+2$ edges labelled by $2$, each of which transforms a symbol $2$ to a symbol $\wbar{2}$, in order from left to right through the word.
The end-points of the paths $\Gamma_{\qctC_2^\fqcms} \parens[\big]{ 2 u \wbar{1} v 1 w 2 }$ and $\Gamma_{\qctC_2^\fqcms} \parens[\big]{ 1^{m+1} 2^{p+2} \wbar{1}^{q+1} }$ are
$\wbar{2} \qKof_2^{\wlng{u}_2} (u) \wbar{1} \qKof_2^{\wlng{v}_2} (v) 1 \qKof_2^{\wlng{w}_2} (w) \wbar{2}$
and
$1^{m+1} \wbar{2}^{p+2} \wbar{1}^{q+1}$,
respectively, which are lowest-weight words.
Also, each vertex of the paths has a loop labelled by $1$, because it admits a decomposition of the form $w_1 1 w_2 2 w_3$ or $w_1 \wbar{2} w_2 \wbar{1} w_3$, for some $w_1, w_2, w_3 \in C_2^*$.

Define a bijection $\psi : C_2^* \parens[\big]{ 2 u \wbar{1} v 1 w 2 } \to C_2^* \parens[\big]{ 1^{m+1} 2^{p+2} \wbar{1}^{q+1} }$ that maps each word $w \in C_2^* \parens[\big]{ 2 u \wbar{1} v 1 w 2 }$ to the word $\psi(w) \in C_2^* \parens[\big]{ 1^{m+1} 2^{p+2} \wbar{1}^{q+1} }$ such that the position of $w$ in $\Gamma_{\qctC_2^\fqcms} \parens[\big]{ 2 u \wbar{1} v 1 w 2 }$ is the same as $\psi (w)$ in $\Gamma_{\qctC_2^\fqcms} \parens[\big]{ 1^{m+1} 2^{p+2} \wbar{1}^{q+1} }$.
As shown above, for $u, v \in C_2^* \parens[\big]{ 2 u \wbar{1} v 1 w 2 }$ and $i \in \set{1,2}$, we have that $u \lbedge{i} v$ is an edge in $\Gamma_{\qctC_2^\fqcms} \parens[\big]{ 2 u \wbar{1} v 1 w 2 }$ if and only if $\psi(u) \lbedge{i} \psi(v)$ is an edge of $\Gamma_{\qctC_2^\fqcms} \parens[\big]{ 1^{m+1} 2^{p+2} \wbar{1}^{q+1} }$.
And since $\psi \parens[\big]{ 2 u \wbar{1} v 1 w 2 } = 1^{m+1} 2^{p+2} \wbar{1}^{q+1}$, where
\[ \wt \parens[\big]{ 2 u \wbar{1} v 1 w 2 } = (m-q, p+2) = \wt \parens[\big]{ 1^{m+1} 2^{p+2} \wbar{1}^{q+1} }, \]
we get that $\psi$ preserves weights.
Therefore, by \comboref{Theorem}{thm:snqcisoqcg}, $\psi$ is a quasi-crystal isomorphism, which implies that
$2 u \wbar{1} v 1 w 2 \hyco 1^{m+1} 2^{p+2} \wbar{1}^{q+1}$.
\end{proof}

\begin{thm}
\label{thm:fqcmtC2isocc}
Any connected component of $\qctC_2^\fqcms$ is quasi-crystal isomorphic to one and only one of the following:
\begin{enumerate}
\item\label{thm:fqcmtC2isocc1}
$\qctC_2^\fqcms \parens[\big]{ 1^{m} }$, $m \in \Z_{\geq 0}$;

\item\label{thm:fqcmtC2isocc2121}
$\qctC_2^\fqcms \parens[\big]{ 2^{m_1} 1^{m_2 + 1} 2^{m_3 + 1} 1^{m_4} }$, $m_1, m_2, m_3, m_4 \in \Z_{\geq 0}$;

\item\label{thm:fqcmtC2isocc12b1}
$\qctC_2^\fqcms \parens[\big]{ 1^{m_1 + 1} 2^{m_2} \wbar{1}^{m_3 + 1} }$, $m_1, m_2, m_3 \in \Z_{\geq 0}$ with $m_1 = 0$ or $m_3 = 0$;

\item\label{thm:fqcmtC2isocc12b2b1}
$\qctC_2^\fqcms \parens[\big]{ 1^{m_1 + 1} 2^{m_2 + 1} \wbar{2}^{m_3 + 1} \wbar{1}^{m_4 + 1} }$, $m_1, m_2, m_3, m_4 \in \Z_{\geq 0}$ with $m_1 = 0$ or $m_4 = 0$, and $m_2 = 0$ or $m_3 = 0$.
\end{enumerate}
Therefore, the elements in these connected components form a minimal set of representatives for the hypoplactic congruence classes on $C_2^*$.
\end{thm}

\begin{proof}
By \itmcomboref{Proposition}{prop:fqcmcc}{prop:fqcmcchw} any connected component of $\qctC_2^\fqcms$ has at least a highest-weight word.
Let $w \in C_2^*$ be a highest-weight word of $\qctC_2^\fqcms$.
If $\wbar{2}$ occurs in $w$, then $w$ has a $2$-inversion, because it is of highest weight.
This implies that $2$ occurs in $w$, and as $w$ is of highest weight, $w$ has a $1$-inversion.
Hence, $w$ is an isolated word.
For each $i \in \set{1, 2}$, if $\wlng{w}_{i} \geq \wlng{w}_{\wbar{i}}$, set $m_i = \wlng{w}_{i} - \wlng{w}_{\wbar{i}}$ and $m_{5-i} = 0$, otherwise, set $m_i = 0$ and $m_{5-i} = \wlng{w}_{\wbar{i}} - \wlng{w}_{i}$.
Then,
\[ \wt(w) = (m_1 - m_4, m_2 - m_3) = \wt \parens[\big]{ 1^{m_1 + 1} 2^{m_2 + 1} \wbar{2}^{m_3 + 1} \wbar{1}^{m_4 + 1} }, \]
which implies by \comboref{Lemma}{lem:hypoCniwhycoiw} that $w \hyco 1^{m_1+1} 2^{m_2+1} \wbar{2}^{m_3+1} \wbar{1}^{m_4+1}$,
and so, we have a quasi-crystal isomorphism between $\qctC_2^\fqcms (w)$ and $\qctC_2^\fqcms \parens[\big]{ 1^{m_1+1} 2^{m_2+1} \wbar{2}^{m_3+1} \wbar{1}^{m_4+1} }$.
So, in the following, we assume that $\wbar{2}$ does not occur in $w$.

If $\wbar{1}$ occurs in $w$, then $w$ has a $1$-inversion, because it is of highest weight.
Since $\wbar{2}$ does not occur in $w$, then we have in $w$ that a $\wbar{1}$ appears to the right of a $1$, or otherwise, $2$ appears to the right of a $1$.
In this second case, as $\wbar{1}$ occurs in $w$, we may have that a $\wbar{1}$ appears to the right of a $2$, or otherwise, every $\wbar{1}$ occurs to the left of any $1$ and $2$.
Based on these decompositions, we get that $w$ lies in one of the following cases.
\begin{itemize}
\item Case 1: $w = w_1 1 w_2 \wbar{1} w_3$, for some $w_1, w_2, w_3 \in \set[\big]{ 1, 2, \wbar{1} }^*$.
Since $1\wbar{1}$ is an isolated word, we get by \comboref{Theorem}{thm:hypoiecom} that $w \hyco 1 w_1 w_2 w_3 \wbar{1}$, and by iterating this process,
$w \hyco 1^{\wlng{w}_1} 2^{\wlng{w}_2} \wbar{1}^{\wlng{w}_{\wbar{1}}}$.
Set $m_2 = \wlng{w}_2$.
If $\wlng{w}_1 \geq \wlng{w}_{\wbar{1}}$, set $m_1 = \wlng{w}_1 - \wlng{w}_{\wbar{1}}$ and $m_3 = 0$, otherwise, set $m_1 = 0$ and $m_3 = \wlng{w}_{\wbar{1}} - \wlng{w}_1$.
Since $\wt \parens[\big]{ 1\wbar{1} } = 0$, we get by \comboref{Theorems}{thm:hypoiecom} and\avoidrefbreak \ref{thm:hypoidem} that
\[\begin{split}
w &\hyco 1^{\wlng{w}_1} 2^{m_2} \wbar{1}^{\wlng{w}_{\wbar{1}}}
\hyco 1^{\wlng{w}_1} \wbar{1}^{\wlng{w}_{\wbar{1}}} 2^{m_2}
\hyco 1^{m_1 + 1} \wbar{1}^{m_3 + 1} 2^{m_2}
\\
&\hyco 1^{m_1 + 1} 2^{m_2} \wbar{1}^{m_3 + 1}.
\end{split}\]
In particular, there exists a quasi-crystal isomorphism between $\qctC_2^\fqcms (w)$ and $\qctC_2^\fqcms \parens[\big]{ 1^{m_1 + 1} 2^{m_2} \wbar{1}^{m_3 + 1} }$.

\item Case 2: $w = w_1 2 w_2 \wbar{1} w_3 1 w_4 2 w_5$, for some $w_1, w_2, w_3, w_4, w_5 \in \set[\big]{ 1, 2, \wbar{1} }^*$.
By \comboref{Lemma}{lem:hypoc22ub1v1w212b1}, we have that $2 w_2 \wbar{1} w_3 1 w_4 2 \hyco 1^{p_1+1} 2^{p_2+2} \wbar{1}^{p_3+1}$, for some $p_1, p_2, p_3 \in \Z_{\geq 0}$.
As in the previous case, we get that
\[ w = w_1 1^{p_1+1} 2^{p_2+2} \wbar{1}^{p_3+1} w_5 \hyco 1^{m_1+1} 2^{m_2} \wbar{1}^{m_3+1}, \]
for some $m_1, m_2, m_3 \in \Z_{\geq 0}$ with $m_1 = 0$ or $m_3 = 0$.
Hence, there exists a quasi-crystal isomorphism between $\qctC_2^\fqcms (w)$ and $\qctC_2^\fqcms \parens[\big]{ 1^{m_1+1} 2^{m_2} \wbar{1}^{m_3+1} }$.

\item Case 3: $w = \wbar{1}^{p} u$, for some $u \in \set{1, 2}^*$ and $p \in \Z_{> 0}$.
By \comboref{Proposition}{prop:hypotC2w2121}, we have that $u \hyco 2^{q_1} 1^{q_2} 2^{q_3} 1^{q_4}$, for some $q_1, q_2, q_3, q_4 \in \Z_{\geq 0}$.
In particular, $\wlng{w}_2 = \wlng{u}_2 = q_1 + q_3$.
Since $w$ has a $1$-inversion, then $u$ has a $1$-inversion, which implies by \comboref{Proposition}{prop:hypotCniinversionresp} that $q_2 > 0$ and $q_3 > 0$.
Note that
\[\begin{split}
\qKoe_2^{p + q_1 + q_3} \qKoe_1^{p} \qKof_2^{q_1 + q_3} \parens[\big]{ \wbar{1}^{p} 2^{q_1} 1^{q_2} 2^{q_3} 1^{q_4} }
&= \qKoe_2^{p + q_1 + q_3} \qKoe_1^{p} \parens[\big]{ \wbar{1}^{p} \wbar{2}^{q_1} 1^{q_2} \wbar{2}^{q_3} 1^{q_4} }\\
&= \qKoe_2^{p + q_1 + q_3} \parens[\big]{ \wbar{2}^{p} \wbar{2}^{q_1} 1^{q_2} \wbar{2}^{q_3} 1^{q_4} }\\
&= 2^{p+q_1} 1^{q_2} 2^{q_3} 1^{q_4}.
\end{split}\]
Since $u \hyco 2^{q_1} 1^{q_2} 2^{q_3} 1^{q_4}$, we get that $w \hyco \wbar{1}^{p} 2^{q_1} 1^{q_2} 2^{q_3} 1^{q_4}$, and as $\wlng{w}_2 = q_1 + q_3$, we obtain that
\[ \qKoe_2^{\wlng{w}_2 + p} \qKoe_1^{p} \qKof_2^{\wlng{w}_2} (w) \hyco 2^{p+q_1} 1^{q_2} 2^{q_3} 1^{q_4}. \]
Set $m_1 = p + q_1$, $m_2 = q_2 - 1$, $m_3 = q_3 - 1$ and $m_4 = q_4$.
Then, there exists a quasi-crystal isomorphism between $\qctC_2^\fqcms (w)$ and $\qctC_2^\fqcms \parens[\big]{ 2^{m_1} 1^{m_2+1} 2^{m_3+1} 1^{m_4} }$.
\end{itemize}
So, we further assume that $\wbar{1}$ does not occur in $w$.

If $2$ occurs in $w$, then $w$ has a $1$-inversion, because it is of highest weight.
As neither $\wbar{1}$ nor $\wbar{2}$ occurs in $w$, we have that $w \in \set{1, 2}^*$, and by \comboref{Proposition}{prop:hypotC2w2121}, $w \hyco 2^{p_1} 1^{p_2} 2^{p_3} 1^{p_4}$, for some $p_1, p_2, p_3, p_4 \in \Z_{\geq 0}$.
Since $w$ has a $1$-inversion, we get by \comboref{Proposition}{prop:hypotCniinversionresp} that $p_2 > 0$ and $p_3 > 0$.
Set $m_1 = p_1$, $m_2 = p_2 - 1$, $m_3 = p_3 - 1$ and $m_4 = p_4$.
Then, there exists a quasi-crystal isomorphism between $\qctC_2^\fqcms (w)$ and $\qctC_2^\fqcms \parens[\big]{ 2^{m_1} 1^{m_2+1} 2^{m_3+1} 1^{m_4} }$.

Finally, if $2$ does not occur in $w$, then $w = 1^m$, where $m = \wlng{w}$.
And so, $\qctC_2^\fqcms (w)$ coincides with $\qctC_2^\fqcms (1^m)$.

We have thus proved that any connected component of $\qctC_2^\fqcms$ is quasi-crystal isomorphic to some connected component lying in\avoidrefbreak \itmref{thm:fqcmtC2isocc1} to\avoidrefbreak \itmref{thm:fqcmtC2isocc12b2b1}.
It remains to show that it is quasi-crystal isomorphic to only one of such connected components.
So, we now show that there are no quasi-crystal isomorphic connected components among the ones in\avoidrefbreak \itmref{thm:fqcmtC2isocc1} to\avoidrefbreak \itmref{thm:fqcmtC2isocc12b2b1}.

Each connected component $\qctC_2^\fqcms \parens[\big]{ 1^{m_1 + 1} 2^{m_2 + 1} \wbar{2}^{m_3 + 1} \wbar{1}^{m_4 + 1} }$ in\avoidrefbreak \itmref{thm:fqcmtC2isocc12b2b1} consists of an isolated word with weight $(m_1 - m_4, m_2 - m_3)$, a $1$-inversion and a $2$-inversion.
By the condition that $m_1 = 0$ or $m_4 = 0$, and $m_2 = 0$ or $m_3 = 0$, we have that all words in\avoidrefbreak \itmref{thm:fqcmtC2isocc12b2b1} have different weights, which implies by \comboref{Lemma}{lem:hypoCniwhycoiw} that they are not hypoplactic congruent.
Also, all connected components in\avoidrefbreak \itmref{thm:fqcmtC2isocc1} to\avoidrefbreak \itmref{thm:fqcmtC2isocc12b1} contain a word without a $2$-inversion, and so, there is no quasi-crystal isomorphism between any of them and some connected component in\avoidrefbreak \itmref{thm:fqcmtC2isocc12b2b1}.

Each connected component $\qctC_2^\fqcms \parens[\big]{ 1^{m_1 + 1} 2^{m_2} \wbar{1}^{m_3 + 1} }$ in\avoidrefbreak \itmref{thm:fqcmtC2isocc12b1} is formed by $m_2 + 1$ words of the form $1^{m_1 + 1} \wbar{2}^{p_1} 2^{p_2} \wbar{1}^{m_3 + 1}$, for $p_1, p_2 \in \Z_{\geq 0}$ such that $p_1 + p_2 = m_2$.
In particular, each connected component has exactly one highest-weight word, namely: $1^{m_1 + 1} 2^{m_2} \wbar{1}^{m_3 + 1}$.
So, if $\qctC_2^\fqcms \parens[\big]{ 1^{m_1 + 1} 2^{m_2} \wbar{1}^{m_3 + 1} }$ and $\qctC_2^\fqcms \parens[\big]{ 1^{m'_1 + 1} 2^{m'_2} \wbar{1}^{m'_3 + 1} }$ are quasi-crystal isomorphic connected components lying in\avoidrefbreak \itmref{thm:fqcmtC2isocc12b1}, then we have that $1^{m_1 + 1} 2^{m_2} \wbar{1}^{m_3 + 1} \hyco 1^{m'_1 + 1} 2^{m'_2} \wbar{1}^{m'_3 + 1}$, which implies that $m_1 - m_3 = m'_1 - m'_3$ and $m_2 = m'_2$, and by the condition that $m_1 = 0$ or $m_3 = 0$ and the condition that $m'_1 = 0$ or $m'_3 = 0$, we obtain that $m_1 = m'_1$ and $m_3 = m'_3$.
Hence, there are no quasi-crystal isomorphic connected components among the ones in\avoidrefbreak \itmref{thm:fqcmtC2isocc12b1}.
Also, all words lying in the connected components in\avoidrefbreak \itmref{thm:fqcmtC2isocc12b1} have $1$-inversions,
each connected component in\avoidrefbreak \itmref{thm:fqcmtC2isocc1} or\avoidrefbreak \itmref{thm:fqcmtC2isocc2121} have at least one word without a $1$-inversion (respectively, $1^m$ or $\wbar{2}^{m_1} 1^{m_2 + 1} \wbar{2}^{m_3 + 1} 1^{m_4}$),
and so, there is no quasi-crystal isomorphism between some connected component in\avoidrefbreak \itmref{thm:fqcmtC2isocc12b1} and some connected component in\avoidrefbreak \itmref{thm:fqcmtC2isocc1} or\avoidrefbreak \itmref{thm:fqcmtC2isocc2121}.

The only word lying in a connected component $\qctC_2^\fqcms \parens[\big]{ 2^{m_1} 1^{m_2 + 1} 2^{m_3 + 1} 1^{m_4} }$ in\avoidrefbreak \itmref{thm:fqcmtC2isocc2121} and the set $\set{1, 2}^*$ is $2^{m_1} 1^{m_2 + 1} 2^{m_3 + 1} 1^{m_4}$.
So, if $\qctC_2^\fqcms \parens[\big]{ 2^{m_1} 1^{m_2 + 1} 2^{m_3 + 1} 1^{m_4} }$ and $\qctC_2^\fqcms \parens[\big]{ 2^{m'_1} 1^{m'_2 + 1} 2^{m'_3 + 1} 1^{m'_4} }$ are quasi-crystal isomorphic connected components in\avoidrefbreak \itmref{thm:fqcmtC2isocc2121}, then we have by \comboref{Lemma}{lem:hypotCnuv2elem} that $2^{m_1} 1^{m_2 + 1} 2^{m_3 + 1} 1^{m_4} \hyco 2^{m'_1} 1^{m'_2 + 1} 2^{m'_3 + 1} 1^{m'_4}$, which implies by \comboref{Proposition}{prop:hypotCnuvbegend} that $m_1 = m'_1$, $m_2 = m'_2$, $m_3 = m'_3$ and $m_4 = m'_4$.
Hence, there are no quasi-crystal isomorphic connected components among the ones in\avoidrefbreak \itmref{thm:fqcmtC2isocc2121}.
Also, each connected component in\avoidrefbreak \itmref{thm:fqcmtC2isocc2121} contains at least one word with a $2$-inversion (for instance, $\wbar{1}^{m_1} 2^{m_2 + 1} \wbar{2}^{m_3 + 1} 1^{m_4}$),
while no word lying in some connected component in\avoidrefbreak \itmref{thm:fqcmtC2isocc1} has a $2$-inversion,
and so, there is no quasi-crystal isomorphism between some connected component in\avoidrefbreak \itmref{thm:fqcmtC2isocc2121} and some connected component in\avoidrefbreak \itmref{thm:fqcmtC2isocc1}.

Finally, note that the only highest-weight word lying in a connected component $\qctC_2^\fqcms \parens[\big]{ 1^{m} }$ in\avoidrefbreak \itmref{thm:fqcmtC2isocc1} and the set $\set{1, 2}^*$ is $1^{m}$.
If $\qctC_2^\fqcms \parens[\big]{ 1^{m} }$ and $\qctC_2^\fqcms \parens[\big]{ 1^{m'} }$ are quasi-crystal isomorphic connected components in\avoidrefbreak \itmref{thm:fqcmtC2isocc1}, then we have by \comboref{Lemma}{lem:hypotCnuv2elem} that $1^{m} \hyco 1^{m'}$, which implies that $m = m'$.
Hence, there is no quasi-crystal isomorphism between distinct connected components among the ones in\avoidrefbreak \itmref{thm:fqcmtC2isocc1}.
\end{proof}

%
%
%
%
%
%
%
%
%

Finally, we show that $\hypo(\qctCn)$ does not admit a finite presentation for $n \geq 3$.

\begin{thm}
\label{thm:hypotCnfpres}
Consider $n \geq 3$.
The hypoplactic congruence $\hyco$ on $C_n^*$ is not finitely generated.
Therefore, $\hypo (\qctCn)$ has no finite presentation.
\end{thm}

\begin{proof}
Let $R$ be a finite subset of $\hyco$, and denote by $\sim_R$ the monoid congruence on $C_n^*$ generated by $R$.
Thus, ${\sim_R} \subseteq {\hyco}$.
As $R$ is finite, take
\[ m = 1 + \max_{(u, v) \in R} \set[\big]{ \wlng{u}, \wlng{v} }. \]
Set $w = 1^m 2^m 1^m \wbar{2}^m \wbar{1}^m$.
If $u$ is a subword of $w$ with length at most $m$ (that is, a word $u$ consisting of at most $m$ consecutive letters of $w$), then $u$ lies in $\set{a, b}^*$, for some $a, b \in \set[\big]{ 1, 2, \wbar{2}, \wbar{1} }$ with $a \neq \wbar{b}$,
and if $v \in C_n^*$ is such that $u \hyco v$, we get by \comboref{Proposition}{prop:hypotCnuvxy} that $v \in \set{ a, b }^*$,
and then, we obtain by \comboref{Theorem}{thm:hypoCnabfree} that $u = v$.
Since $R \subseteq {\hyco}$ and every pair $(u, v) \in R$ satisfies $\wlng{u} < m$ and $\wlng{v} < m$, we get that if $(u, v) \in R$ and $u$ or $v$ are subwords of $w$, then $u = v$.
As $R$ generates $\sim_R$, we have for $w' \in C_n^*$ that $w \sim_R w'$ if and only if $w' = w$.
On the other hand, we have that $w \neq 1^{2m} 2^m \wbar{2}^m \wbar{1}^m$, and as $1^m 2^m \wbar{2}^m \wbar{1}^m$ is an isolated word, $w \hyco 1^{2m} 2^m \wbar{2}^m \wbar{1}^m$, by \comboref{Theorem}{thm:hypoiecom}.
This implies that ${\sim_R} \neq {\hyco}$.
And therefore, $\hyco$ is not finitely generated.
\end{proof}

\subsection{\texorpdfstring%
  {From $\hypo (\qctA_{n-1})$ to $\hypo (\qctCn)$}%
  {From hypo(A\_\{n-1\}) to hypo(C\_n)}}
\label{subsec:hypotCnhAn1tohCn}

We first show that an embedding of $\hypo(\qctAn)$ into $\hypo(\qctCn)$ cannot map each letter of $A_n$ to a letter of $C_n$.

\begin{prop}
\label{prop:hypotAnhypotCnnot22}
For $m \geq 2$ and $n \geq 2$, there exists no injective monoid homomorphism $\psi : {\hypo(\qctA_{m})} \to {\hypo (\qctCn)}$ such that $\psi(x), \psi(y) \in C_n$, for some $x, y \in A_{m}$ with $x \neq y$.
\end{prop}

\begin{proof}
  Suppose $\psi : {\hypo(\qctA_{m})} \to {\hypo (\qctCn)}$ is an injective monoid homomorphism such that $\psi(x), \psi(y) \in C_n$, for some $x, y \in A_{m}$ with $x \neq y$.
Without loss of generality assume $x < y$.
Since $xyx = xxy$ in $\hypo(\qctA_{m})$, we get that $\psi(x)\psi(y)\psi(x) = \psi(x)\psi(x)\psi(y)$ in $\hypo (\qctCn)$, which implies by \comboref{Corollary}{cor:hypotCnKnuthrel} that $\psi(x) = \psi(y)$, or $\psi(x) = 1$ and $\psi(y) = \wbar{1}$.
As $\psi$ is injective, we must have that $\psi(x) = 1$ and $\psi(y) = \wbar{1}$.
Since $1\wbar{1}$ is an isolated word of $\qctCn^\fqcms$ and $\wt \parens[\big]{ 1\wbar{1} } = 0$, we get by \comboref{Theorem}{thm:hypoidem} that
\[ \psi(xyxy) = 1\wbar{1}1\wbar{1} = 1\wbar{1} = \psi(xy) \]
in $\hypo(\qctCn)$, which is a contradiction as $\psi$ is injective.
\end{proof}

We now show that an injective map between the relevant alphabets cannot be extended to a (not necessarily injective) homomorphism from the hypoplactic monoid of type $\tAn$ to that of type $\tCn$.

\begin{prop}
\label{prop:hypotAnhypotCnnot32}
For $m \geq 3$ and $n \geq 2$, no injective map from $A_{m}$ to $C_n$ can be extended to a monoid homomorphism from $\hypo(\qctA_{m})$ to $\hypo (\qctCn)$.
\end{prop}

\begin{proof}
Suppose $\psi : {\hypo(\qctA_{m})} \to {\hypo (\qctCn)}$ is a monoid homomorphism where its restriction $\psi|_{A_{m}}$ is an injective map from $A_{m}$ to $C_n$.
So, we can take $y \in \set{2, 3}$ such that $\psi(y) \neq \wbar{1}$ and $\psi(y) \neq \psi(1)$.
Since $1y1 = 11y$ in $\hypo(\qctA_{m})$, we obtain that $\psi(1)\psi(y)\psi(1) = \psi(1)\psi(1)\psi(y)$ in $\hypo (\qctCn)$, contradicting \comboref{Corollary}{cor:hypotCnKnuthrel}.
\end{proof}

Now, we show that $\hypo(\qctA_{n-1})$ can be embedded in $\hypo(\qctCn)$.

\begin{thm}
Let $n \geq 3$.
Define a map $\psi : A_{n-1}^* \to C_n^*$ by
\[
\psi(w) =
\begin{cases}
\ew & \text{if $w = \ew$}\\
w n\wbar{n}n\wbar{n} & \text{otherwise,}
\end{cases}
\]
for each $w \in A_{n-1}^*$.
Then $\psi$ factors to give an injective monoid homomorphism from $\hypo(\qctA_{n-1})$ to $\hypo(\qctCn)$.
\end{thm}

\begin{proof}
Denote the quasi-crystal structure of $\qctA_{n-1}^\fqcms$ by $\wt^{\qctA}$, $\qKoe_i^{\qctA}$, $\qKof_i^{\qctA}$, $\qKoec_i^{\qctA}$ and $\qKofc_i^{\qctA}$ ($i \in \set{1,\ldots,n-2}$),
and denote the hypoplactic congruence on $\qctA_{n-1}^\fqcms$ by $\hyco_{\qctA}$.
Similarly, denote the quasi-crystal structure of $\qctCn^\fqcms$ by $\wt^{\qctC}$, $\qKoe_i^{\qctC}$, $\qKof_i^{\qctC}$, $\qKoec_i^{\qctC}$ and $\qKofc_i^{\qctC}$ ($i \in \set{1,\ldots,n}$),
and denote the hypoplactic congruence on $\qctCn^\fqcms$ by $\hyco_{\qctC}$.
From \comboref{Example}{exa:fqcmAn} and \comboref{Definition}{dfn:fqcmtCn}, it is immediate that $g^{\qctA} (w) = g^{\qctC} (w)$, for any $w \in A_{n-1}^*$ and $g \in \set[\big]{ \qKoe_i, \qKof_i, \qKoec_i, \qKofc_i \given i \in \set{1,\ldots,n-2} }$.

We now show that for $u, v \in A_{n-1}^*$, $u \hyco_{\qctA} v$ if and only if $\psi(u) \hyco_{\qctC} \psi(v)$,
which implies that $\psi$ induces a well-defined injective map from $\hypo(\qctA_{n-1})$ to $\hypo(\qctCn)$.
First, note that if $u \hyco_{\qctA} \ew$, for some $u \in A_{n-1}^*$, then $\wt^{\qctA}(u) = 0$, which implies that $u = \ew$.
If $u' \hyco_{\qctC} \ew$, for some $u' \in C_n^*$, then
\[ \wlng{u'}_{i} + \wlng{u'}_{\wbar{i+1}} \leq \qKofc_i^{\qctC} (u') = \qKofc_i^{\qctC} (\ew) = 0, \]
for any $i \in \set{1,\ldots,n}$, implying that $u' = \ew$.
Let $w \in A_{n-1}^*$ with $w \neq \ew$.
We have that
\[ \wt^{\qctC} (wn\wbar{n}n\wbar{n}) = \parens[\big]{ \wlng{w}_1, \wlng{w}_2, \ldots, \wlng{w}_{n-1}, 0 }, \]
which implies for $w' \in A_{n-1}^*$ that $\wt^{\qctC} (wn\wbar{n}n\wbar{n}) = \wt^{\qctC} (w'n\wbar{n}n\wbar{n})$ if and only if $\wt^{\qctA} (w) = \wt^{\qctA} (w')$.
Also,
\[ \qKoec_{n-1}^{\qctC} (wn\wbar{n}n\wbar{n}) = \qKoec_{n}^{\qctC} (wn\wbar{n}n\wbar{n}) = +\infty, \]
which implies that $\qKoe_{n-1}^{\qctC}$, $\qKof_{n-1}^{\qctC}$, $\qKoe_{n}^{\qctC}$ and $\qKof_{n}^{\qctC}$ are undefined on $wn\wbar{n}n\wbar{n}$,
and for $i \in \set{1,\ldots,n-2}$,
\[
\qKoec_i^{\qctC} (wn\wbar{n}n\wbar{n}) = \qKoec_i^{\qctC} (w) = \qKoec_i^{\qctA} (w)
\quad \text{and} \quad
\qKofc_i^{\qctC} (wn\wbar{n}n\wbar{n}) = \qKofc_i^{\qctC} (w) = \qKofc_i^{\qctA} (w),
\]
because $\qKoec_i^{\qctC} (n\wbar{n}n\wbar{n}) = \qKofc_i^{\qctC} (n\wbar{n}n\wbar{n}) = 0$.
Since $\qctA_{n-1}^\fqcms$ and $\qctCn^\fqcms$ are seminormal, we get that $\qKof_i^{\qctC}$ is defined on $wn\wbar{n}n\wbar{n}$ if and only if $\qKof_i^{\qctA}$ is defined on $w$,
and if so,
\[ \qKof_i^{\qctC} (wn\wbar{n}n\wbar{n}) = \qKof_i^{\qctC} (w) n\wbar{n}n\wbar{n} = \qKof_i^{\qctA} (w) n\wbar{n}n\wbar{n}. \]
Hence, for $u, v \in A_{n-1}^*$ and $i \in \set{1,\ldots,n-2}$, we have an edge $u \lbedge{i} v$ in $\Gamma_{\qctA_{n-1}^\fqcms}$ if and only if we have an edge $un\wbar{n}n\wbar{n} \lbedge{i} vn\wbar{n}n\wbar{n}$ in $\Gamma_{\qctCn^\fqcms}$.
This implies that $\Gamma_{\qctCn^\fqcms} (wn\wbar{n}n\wbar{n})$ is obtained from $\Gamma_{\qctA_{n-1}^\fqcms} (w)$ by concatenating $n\wbar{n}n\wbar{n}$ to each vertex, and adding $(n-1)$-labelled and $n$-labelled loops to each vertex.
Equivalently, $\Gamma_{\qctA_{n-1}^\fqcms} (w)$ is obtained from $\Gamma_{\qctCn^\fqcms} (wn\wbar{n}n\wbar{n})$ by removing the last four letters of each vertex, and removing all $(n-1)$-labelled and $n$-labelled loops.
Therefore, for any $u, v \in A_{n-1}^*$, there exists a graph isomorphism between $\Gamma_{\qctA_{n-1}^\fqcms} (u)$ and $\Gamma_{\qctA_{n-1}^\fqcms} (v)$ mapping $u$ to $v$ if and only if there exists a graph isomorphism between $\Gamma_{\qctCn^\fqcms} (un\wbar{n}n\wbar{n})$ and $\Gamma_{\qctCn^\fqcms} (vn\wbar{n}n\wbar{n})$ mapping $un\wbar{n}n\wbar{n}$ to $vn\wbar{n}n\wbar{n}$.
By \comboref{Theorem}{thm:snqcisoqcg}, we have that $u \hyco_{\qctA} v$ if and only if $u n\wbar{n}n\wbar{n} \hyco_{\qctC} v n\wbar{n}n\wbar{n}$.

To obtain that $\psi$ induces an injective monoid homomorphism from $\hypo (\qctA_{n-1})$ to $\hypo (\qctCn)$, it remains to prove that $\psi (uv) \hyco_{\qctC} \psi(u) \psi (v)$, for any $u, v \in A_{n-1}^*$, since $\psi(\ew) = \ew$ follows from the definition of $\psi$.
As shown above, if $uv \hyco_{\qctA} \ew$, then $uv = \ew$, which implies that $u=v=\ew$, and so, $\psi (uv) \hyco_{\qctC} \psi(u) \psi(v)$.
By \comboref{Proposition}{prop:hypotCnwbwwwbwwbw}, we have that $n\wbar{n}n\wbar{n}$ is a commutative and idempotent element of $\hypo(\qctCn)$.
So, for any $u', v' \in C_n^*$, we get that
\[ u' n\wbar{n}n\wbar{n} v' n\wbar{n}n\wbar{n} \hyco_{\qctC} u'v' (n\wbar{n}n\wbar{n})^2 \hyco_{\qctC} u'v' n\wbar{n}n\wbar{n}, \]
in particular, for $u, v \in A_{n-1}^*$ with $u \neq \ew$ or $v \neq \ew$, we obtain that $\psi(uv) \hyco_{\qctC} \psi(u) \psi(v)$.
Therefore, we get that $\psi$ induces an injective monoid homomorphism from $\hypo(\qctA_{n-1})$ to $\hypo(\qctCn)$.
\end{proof}

\subsection{\texorpdfstring%
  {From $\hypo(\qctC_{n-1})$ to $\hypo (\qctCn)$}%
  {From hypo(C\_\{n-1\}) to hypo(C\_n)}}
\label{subsec:hypotCnhCn1tohCn}

The following result shows that we have a monoid embedding from $\hypo(\qctC_{n-1})$ to $\hypo(\qctCn)$.

\begin{thm}
Let $n \geq 3$.
Consider $\psi$ to be the monoid homomorphism from $C_{n-1}^*$ to $C_n^*$ such that
\[
\psi(x) = (x+1) 1\wbar{1}
\quad \text{and} \quad
\psi(\wbar{x}) = \parens[\big]{ \wbar{x+1} } 1\wbar{1},
\]
for each $x \in \set{1,\ldots,n-1}$.
Then, $\psi$ induces an injective monoid homomorphism from $\hypo(\qctC_{n-1})$ to $\hypo(\qctCn)$.
\end{thm}

\begin{proof}
Let $\tau$ be the monoid homomorphism from $C_{n-1}^*$ to $C_n^*$ such that
\[
\tau(x) = x+1
\quad \text{and} \quad
\tau(\wbar{x}) = \wbar{x+1},
\]
for each $x \in \set{1,\ldots,n-1}$.
Equivalently, for $w \in C_{n-1}^*$, $\tau (w)$ is obtained  from $w$ by replacing each $x$ by $x+1$ and each $\wbar{x}$ by $\wbar{x+1}$, for $x \in \set{1,\ldots,n-1}$.

For each $m \in \set{n-1,n}$, denote the quasi-crystal structure of $\qctC_m^\fqcms$ by $\wt^{(m)}$, $\qKoe_i^{(m)}$, $\qKof_i^{(m)}$, $\qKoec_i^{(m)}$ and $\qKofc_i^{(m)}$ ($i \in \set{1,\ldots,m}$),
and denote the hypoplactic congruence on $\qctC_m^\fqcms$ by $\hyco_{(m)}$.
From \comboref{Definition}{dfn:fqcmtCn}, for $w \in C_{n-1}^*$ and $i \in \set{1,\ldots,n-1}$, note that $w$ has an $i$-inversion if and only if $\tau(w)$ has an $(i+1)$-inversion,
and so, we get that
$\qKoec_{i+1}^{(n)} \parens[\big]{\tau(w)} = \qKoec_i^{(n-1)} (w)$
and
$\qKofc_{i+1}^{(n)} \parens[\big]{\tau(w)} = \qKofc_i^{(n-1)} (w)$.
Moreover, $\qKoe_{i+1}^{(n)}$ is defined on $\tau(w)$ if and only if $\qKoe_{i}^{(n-1)}$ is defined on $w$,
and if so, $\qKoe_{i+1}^{(n)} \parens[\big]{ \tau(w) } = \tau \parens[\big]{ \qKoe_{i}^{(n-1)} (w) }$.
Analogously, $\qKof_{i+1}^{(n)}$ is defined on $\tau(w)$ if and only if $\qKof_{i}^{(n-1)}$ is defined on $w$,
and if so, $\qKof_{i+1}^{(n)} \parens[\big]{ \tau(w) } = \tau \parens[\big]{ \qKof_{i}^{(n-1)} (w) }$.

Since $\psi$ is a monoid homomorphism from $C_{n-1}^*$ to $C_n^*$, to prove that $\psi$ induces an injective monoid homomorphism from $\hypo(\qctC_{n-1})$ to $\hypo(\qctCn)$, it suffices to show that for any $u, v \in C_{n-1}^*$, $u \hyco_{(n-1)} v$ if and only if $\psi(u) \hyco_{(n)} \psi(v)$.
Note that for $m \in \set{n-1,n}$ and $u \in C_m^*$, if $u \hyco_{(m)} \ew$, then
\[ \wlng{u}_{i} + \wlng{u}_{\wbar{i+1}} \leq \qKofc_i^{(m)} (u) = \qKofc_i^{(m)} (\ew) = 0, \]
for any $i \in \set{1,\ldots,m}$, implying that $u = \ew$.
Let $w \in C_{n-1}^*$ with $w \neq \ew$.
Take $x_1, \ldots, x_k \in C_{n-1}$ such that $w = x_1 \ldots x_k$.
Since $1\wbar{1}$ is an isolated word of $\qctCn^\fqcms$ and $\wt \parens[\big]{ 1\wbar{1} } = 0$, we have by \comboref{Theorems}{thm:hypoiecom} and\avoidrefbreak \ref{thm:hypoidem} that
\[\begin{split}
\psi (w)
&= \tau (x_1) 1\wbar{1} \tau(x_2) 1\wbar{1} \ldots \tau(x_k) 1\wbar{1}\\
&= \tau(x_1) \tau(x_2) \ldots \tau(x_k) \parens[\big]{ 1\wbar{1} }^k = \tau(w) 1\wbar{1}.
\end{split}\]
We have that
\[ \wt^{(n)} \parens[\big]{ \psi(w) } = \parens[\big]{ 0, \wlng{w}_1 - \wlng{w}_{\wbar{1}}, \wlng{w}_2 - \wlng{w}_{\wbar{2}}, \ldots, \wlng{w}_{n-1} - \wlng{w}_{\wbar{n-1}} }, \]
which implies that for $w' \in C_{n-1}^*$, $\wt^{(n)} \parens[\big]{ \psi (w) } = \wt^{(n)} \parens[\big]{ \psi (w') }$ if and only if $\wt^{(n-1)} (w) = \wt^{(n-1)} (w')$.
Also,
$\qKoec_1^{(n)} \parens[\big]{ \psi (w) } = +\infty$,
which implies that $\qKoe_1^{(n)}$ and $\qKof_1^{(n)}$ are undefined on $\psi(w)$,
and for $i \in \set{2,\ldots,n}$,
\[
\qKoec_{i}^{(n)} \parens[\big]{ \psi (w) } = \qKoec_{i}^{(n)} \parens[\big]{ \tau (w) } = \qKoec_{i-1}^{(n-1)} (w)
\]
and
\[
\qKofc_{i}^{(n)} \parens[\big]{ \psi (w) } = \qKofc_{i}^{(n)} \parens[\big]{ \tau (w) } = \qKofc_{i-1}^{(n-1)} (w),
\]
because $\qKoec_{i}^{(n)} \parens[\big]{ 1\wbar{1} } = \qKofc_{i}^{(n)} \parens[\big]{ 1\wbar{1} } = 0$.
Since $\qctC_{n-1}^\fqcms$ and $\qctCn^\fqcms$ are seminormal, we have that $\qKof_{i}^{(n)}$ is defined on $\psi(w)$ if and only if $\qKof_{i-1}^{(n-1)}$ is defined on $w$,
and if so,
\[ \qKof_i^{(n)} \parens[\big]{ \psi(w) } = \qKof_{i}^{(n)} \parens[\big]{ \tau(w) } 1\wbar{1} = \tau \parens[\big]{ \qKof_{i-1}^{(n-1)} (w) } 1\wbar{1} = \psi \parens[\big]{ \qKof_{i-1}^{(n-1)} (w) }. \]
Hence, for $u, v \in C_{n-1}^*$ and $i \in \set{1,\ldots,n-1}$, we have an edge $u \lbedge{i} v$ in $\Gamma_{\qctC_{n-1}^\fqcms}$ if and only if we have an edge $\psi(u) \lbedge{i} \psi(v)$ in $\Gamma_{\qctCn^\fqcms}$.
This implies that $\Gamma_{\qctCn^\fqcms} \parens[\big]{ \psi(w) }$ is obtained from $\Gamma_{\qctC_{n-1}^\fqcms} (w)$ by applying $\psi$ to each vertex, and adding $1$-labelled loops to each vertex.
As $\psi$ is injective, $\Gamma_{\qctC_{n-1}^\fqcms} (w)$ can also be obtained from $\Gamma_{\qctCn^\fqcms} \parens[\big]{ \psi(w) }$ by reversing the described process.
Therefore, for any $u, v \in C_{n-1}^*$, there exists a graph isomorphism between $\Gamma_{\qctC_{n-1}^\fqcms} (u)$ and $\Gamma_{\qctC_{n-1}^\fqcms} (v)$ mapping $u$ to $v$ if and only if there exists a graph isomorphism between $\Gamma_{\qctCn^\fqcms} \parens[\big]{ \psi(u) }$ and $\Gamma_{\qctCn^\fqcms} \parens[\big]{ \psi(v) }$ mapping $\psi(u)$ to $\psi(v)$.
By \comboref{Theorem}{thm:snqcisoqcg}, we have that $u \hyco_{(n-1)} v$ if and only if $\psi(u) \hyco_{(n)} \psi(v)$.
\end{proof}

By composing the homomorphisms from the previous result, we get the following.

\begin{cor}
Let $n > m \geq 2$.
Consider $\psi$ to be the monoid homomorphism from $C_{m}^*$ to $C_n^*$ such that
\[
\psi(x) = (x+n-m) 1 2 \ldots (n-m) (\wbar{n-m}) \parens[\big]{ \wbar{n-m-1} } \ldots \wbar{1}
\]
and
\[
\psi(\wbar{x}) = (\wbar{x+n-m}) 1 2 \ldots (n-m) (\wbar{n-m}) \parens[\big]{ \wbar{n-m-1} } \ldots \wbar{1},
\]
for each $x \in \set{1,\ldots,m}$.
Then, $\psi$ induces an injective monoid homomorphism from $\hypo(\qctC_{m})$ to $\hypo(\qctCn)$.
\end{cor}

\bibliography{\jobname}

\end{document}